\documentclass[11pt, reqno]{amsart}

\usepackage{style}

\title[Asymptotic stability of an anisotropic micropolar fluid]{\vspace{-10cm}Anisotropic micropolar fluids subject to a uniform microtorque: the stable case}

\author{Antoine Remond-Tiedrez}
\address{
Department of Mathematical Sciences\\
Carnegie Mellon University\\
Pittsburgh, PA 15213, USA
}
\email[A. Remond-Tiedrez]{aremondt@andrew.cmu.edu}

\author{Ian Tice}
\address{
Department of Mathematical Sciences\\
Carnegie Mellon University\\
Pittsburgh, PA 15213, USA
}
\email[I. Tice]{iantice@andrew.cmu.edu}
\thanks{I. Tice was supported by an NSF CAREER Grant (DMS \#1653161).}


\begin{document}

	\begin{abstract}
		We study a three-dimensional, incompressible, viscous, micropolar fluid with anisotropic microstructure on a periodic domain.
		Subject to a uniform microtorque, this system admits a unique nontrivial equilibrium.
		We prove that when the microstructure is inertially oblate  (i.e. pancake-like) this equilibrium is nonlinearly asymptotically stable.

        Our proof employs a nonlinear energy method built from the natural energy dissipation structure of the problem.  Numerous difficulties arise due to the dissipative-conservative structure of the problem.  Indeed, the dissipation fails to be coercive over the energy, which itself is weakly coupled in the sense that, while it provides estimates for the fluid velocity and microstructure angular velocity, it only provides control of two of the six components of the microinertia tensor.  To overcome these problems, our method relies on a delicate combination of two distinct tiers of energy-dissipation estimates, together with transport-like advection-rotation estimates for the microinertia.  When combined with a quantitative rigidity result for the microinertia, these allow us to deduce the existence of global-in-time decaying solutions near equilibrium.
	\end{abstract}

	\maketitle

	\vspace{-2em}
	{\small\tableofcontents}
	\newpage
	\newgeometry{margin=0.75in}

%----------------------------------------------------------------------------------------------------
%	INTRODUCTION
%----------------------------------------------------------------------------------------------------

\section{Introduction}
\label{sec:intro}

\subsection{Brief description of the model}
\label{sec:discuss_model}

	Following the tradition of generalised continuum mechanics dating back to the Cosserat brothers \cite{cosserats},
	micropolar fluids were introduced by Eringen in \cite{eringen_first}.
	The theory of micropolar fluids extends classical continuum mechanics by taking into account the effects due to microstructure present in the continuum.
	For viscous, incompressible fluids, this results in a model coupling the Navier-Stokes equations to an evolution equation for the rigid microstructure present
	at every point in the continuum.
	This theory has been used to describe aerosols and colloidal suspensions, such as those appearing in biological fluids \cite{maurya_biological_fluids},
	blood flow \cite{ramkissoon_air_bubble_blood_flow, beg_al_blood_flow, mekheimer_kot_blood_flow},
	lubrication \cite{allen_kline_lubrication, bayada_lukaszewicz_lubrication, rajasekhar_nicodemus_sharma_lubrication}
	and the lubrication of human joints \cite{sinha_singh_prasad_human_joints},
	as well as liquid crystals \cite{eringen_liquid_crystals, lhuillier_rey_liquid_crystals, gay_balmaz_ratiu_tronci_liquid_crystals}
	and ferromagnetic fluids \cite{nochetto_salgado_tomas_ferrohydrodynamics}.

	We now provide a brief description of the model where we only introduce new terminology and concepts if they are necessary to formulate the main result.
	For a more thorough discussion of the model and for its careful derivation we direct the reader's attention to Section 2 of \cite{rt_tice_amf_iut}
	and Chapter 1 of \cite{thesis} respectively.

	The state of a three-dimensional micropolar fluid at a point in space-time is described by the following variables:
	the fluid's velocity is a vector $u\in\R^3$,
	the fluid's pressure is a scalar $p\in\R^3$,
	the microstructure's angular velocity is a vector $\omega\in\R^3$,
	and the microstructure's moment of inertia is a positive definite symmetric matrix $J\in\R^{3\times 3}$
	which is called the \emph{microinertia tensor}.
	Here we study \emph{homogeneous} micropolar fluids, meaning that the microstructures at any two points of the fluid are identical up to a proper rotation.
	Equivalently, this means that the microinertia tensors at any two points of the fluid are equal up to conjugation (by that same rotation).
	Note that the shape of the microstructure determines the microinertia tensor, but the converse fails since the same microinertia tensor may be achieved by microstructures of differing shapes.

	We restrict our attention to problems in which the microinertia plays a significant role,
	and so in this paper we only consider \emph{anisotropic} micropolar fluids.
	This means that the microinertia is not isotropic, or in other words that $J$ has at least two distinct eigenvalues.
	To be precise, we study micropolar fluids whose microstructure has an \emph{inertial axis of symmetry}.
	That is to say that there are physical constants $\lambda, \nu > 0$ which depend on the microstructure such that, at every point,
	$J$ is a symmetric matrix with spectrum $\cbrac{\lambda, \lambda, \nu}$.
	Studying microstructures with an inertial axis of symmetry may be viewed as the intermediate case between
	the isotropic case where the microinertia has a repeated eigenvalues of multiplicity three and
	the ``fully'' anisotropic case where the microinertia has three distinct eigenvalues.

	The equations governing the motion of a micropolar fluid in the periodic spatial domain $\T^3 = \R^3 / \Z^3$ subject to an external microtorque $\tau e_3$ are
	\begin{subnumcases}{}
		\pdt u + \brac{u \cdot \nabla} u = (\mu + \kappa/2) \Delta u + \kappa \nabla\times\omega - \nabla p
		&on $\brac{0, T} \times \T^3$,
		\label{eq:full_sys_u}\\
		\nabla\cdot u = 0
		&on $\brac{0, T} \times \T^3$,
		\label{eq:full_sys_div}\\
		J\brac{ \pdt\omega + \brac{u\cdot\nabla}\omega } + \omega\times J\omega
			= \kappa\nabla\times u - 2\kappa\omega + \brac{\tilde{\alpha} - \tilde{\gamma}} \nabla\brac{\nabla\cdot\omega} + \tilde{\gamma} \Delta\omega + \tau e_3
		\hspace{-1.2em}&on $\brac{0, T} \times \T^3$, { \hspace{1.5em} }
		\label{eq:full_sys_omega}\\
		\pdt J + \brac{u\cdot \nabla} J = \sbrac{\Omega, J}
		&on $\brac{0, T} \times \T^3$,
		\label{eq:full_sys_J}
	\end{subnumcases}
	where $\sbrac{ \,\cdot\,, \,\cdot\,}$ denotes the commutator between two matrices,
	$\tilde{\alpha} = \alpha + 4\beta/3$ and $\tilde{\gamma} = \beta+\gamma$ where $\mu$, $\kappa$, $\alpha$, $\beta$, and $\gamma$ are non-negative physical viscosity constants,
	$\tau$ denotes the magnitude of the microtorque, and $\Omega$ is the $3$-by-$3$ antisymmetric matrix identified with $\omega$ via the identity $\Omega v = \omega\times v$
	for every $v\in\R^3$.

	We are considering the situation in which external forces are absent and the external microtorque is constant, namely equal to $\tau e_3$ for some constant $\tau > 0$.
	In particular note that the choice of $e_3$ as the direction of the microtorque is made without loss of generality since the equations of motion are equivariant under proper rotations.
	More precisely: if $\brac{u, p, \omega, J}$ is a solution of \eqref{eq:full_sys_u}--\eqref{eq:full_sys_J} then, for any $\mathcal{R} \in SO(3)$,
	$\brac{u, p, \mathcal{R}\omega, \mathcal{R} J \mathcal{R}^T}$ is a solution of \eqref{eq:full_sys_u}--\eqref{eq:full_sys_J} \emph{provided} that the
	microtorque $\tau e_3$ is replaced by $\tau \mathcal{R} e_3$.
	
	We can motivate this choice to have no external forces and a constant microtorque in two ways.
	On one hand it it reminiscent of certain chiral active fluids constituted of self-spinning particles which continually drive energy into the system
	\cite{banerjee_souslov_et_at_chiral_active_fluids}, as our constant microtorque does.
	On the other hand this choice of an external force-microtorque pair is motivated by the lack of analytical results on anisotropic micropolar fluids.
	As a first foray into the world of anisotropic micropolar fluids, it is natural to look for the simplest external force-microtorque pair which gives rise
	to non-trivial equilibria for the angular velocity $\omega$ and the microinertia $J$.
	The simplest such external force-microtorque pair is precisely $(0, \tau e_3)$.
	
\paragraph{\textbf{The equilibrium and its stability}}

	\vspace{0.5em}
	Let us now turn our attention to the aforementioned equilibrium.
	Subject to a constant and uniform microtorque, the unique equilibrium of the system is the following:
	the fluid's velocity is quiescent ($u_{eq} = 0$),
	the pressure is null ($p_{eq} = 0$),
	the angular velocity is aligned with the microtorque ($\omega_{eq} = \frac{\tau}{2\kappa}$),
	and the inertial axis of symmetry of the microstructure is aligned with the microtorque such that the microinertia is $J_{eq} = \diag\brac{\lambda, \lambda, \nu}$.

	A physically-motivated heuristic suggests that the stability of the equilibrium depends on the microstructure,
	and more precisely that the equilibrium is stable if $\nu > \lambda$ and unstable if $\lambda > \nu$.
	This heuristic explanation is based on the analysis of the energy associated with the system
	and with a comparison with the ODE describing the rotation of a damped rigid body subject to an external torque.
	While we defer to the companion paper \cite{rt_tice_amf_iut} for a detail discussion of this heuristic,
	the core of the argument based on the analysis of the energy can be seen from the energy-dissipation relation recorded later in this paper.
	In particular, the energy recorded in \eqref{eq:def_E_lin} below (which then appears first in a rigorous setting in \fref{Proposition}{prop:gen_pert_ED_rel})
	only remains positive-definite when $\nu > \lambda$, which suggests that this may characterize the stable regime.

	In the former case where $\nu > \lambda$ we say that the microstructure is \emph{inertially oblate}, or pancake-like,
	and in the latter case where $\lambda > \nu$ we say that the microstructure is \emph{inertially oblong}, or rod-like.
	This nomenclature is justified by the following fact.
	For rigid bodies with an axis of symmetry and a uniform mass density, the terms ``oblate'',
	which essentially means that the body is shorter along its axis of symmetry than it is wide across it,
	and ``inertially oblate'' describe the same thing (and similarly for the terms ``oblong'' and ``inertially oblong'').
	Examples of inertially oblong and oblate rigid bodies are provided in \fref{Figure}{fig:rigid_bodies}.

	In a companion paper \cite{rt_tice_amf_iut} we prove the instability of inertially oblong microstructure.
	In this paper we prove the asymptotic stability of inertially oblate microstructure in \fref{Theorem}{thm:early_statement_main_result}
	In particular, combining the main results of these two papers produces a \emph{sharp} nonlinear stability criterion, recorded in \fref{Theorem}{thm:sharp_criterion}.
	\begin{figure}[b]
		\centering
		\begin{subfigure}[h!]{0.45\textwidth}
			\centering
			\includegraphics{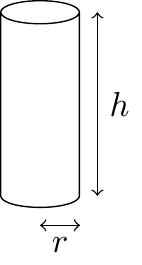}
			\caption{This rigid body is inertially oblong if $h^2 > 6 r^2$.}
			\label{fig:oblong}
		\end{subfigure}
		\begin{subfigure}[h!]{0.45\textwidth}
			\centering
			\includegraphics{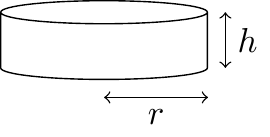}
			\caption{This rigid body is inertially oblate if $h^2 < 6 r^2$.}
			\label{fig:oblate}
		\end{subfigure}
		\caption{Two rigid bodies with uniform density which possess an inertial axis of symmetry.}
		\label{fig:rigid_bodies}
	\end{figure}
\subsection{A brief summary of techniques and difficulties}
\label{sec:mixed_type}

	The main thrust of this paper is to prove that if the microstructure is inertially oblate then the equilibrium is nonlinearly asymptotically stable
	with almost exponential decay to equilibrium.
	Here ``almost exponential'' means that the rate of decay is algebraic and grows unboundedly as further smallness and regularity assumptions are imposed on the initial data.
	In order to provide context for our main result and to motivate the presence of the various functionals used in it,
	we will now attempt a quick overview of the difficulties associated to \eqref{eq:full_sys_u}--\eqref{eq:full_sys_J} and our techniques for dealing with them.  A more detailed discussion is presented in  \fref{Section}{sec:strategy_and_difficulties}.

	As in many viscous fluid problems, the system \eqref{eq:full_sys_u}--\eqref{eq:full_sys_J} is of mixed dissipative-conservative type, with some of the unknowns having dissipation mechanisms and others not.  This manifests as the mixed parabolic-hyperbolic structure of the PDEs.  Such systems usually have a physical dissipation functional, $\mathfrak{D}$, that couples to a physical energy functional, $\mathfrak{E}$, via the energy-dissipation relation
\begin{equation}\label{eq:toy_ed}
 \frac{d}{dt} \mathfrak{E} + \mathfrak{D} = 0.
\end{equation}
    We won't need the precise form of $\mathfrak{E}$ and $\mathfrak{D}$ for our problem here, so we don't state them precisely, but they can be found in $(2.8)$ of \cite{rt_tice_amf_iut}.  Our technique for analyzing the problem is based on higher-regularity versions of this structure, and since differentiating linearizes the PDEs, it's actually the linearized versions, $\mathfrak{E}_{lin}$ and $\mathfrak{D}_{lin}$ that are most relevant in our discussion.  Indeed, the questions of if and how the unknowns appear in the linearized versions of $\mathfrak{E}_{lin}$ and $\mathfrak{D}_{lin}$ become paramount.

    In general energy-dissipation relations, if we have a bound $\mathfrak{E}_{lin} \leqslant C \mathfrak{D}_{lin}$, then the dissipation is said to be \emph{coercive}, and we expect to be able to prove the exponential decay of $\mathfrak{E}_{lin}$ via a linearized version of \eqref{eq:toy_ed} and a Gronwall argument.  However, if this inequality does not hold, we say the dissipation fails to be coercive, and the decay of solutions is no longer obvious.  As we will see below, the latter holds for our problem, \eqref{eq:full_sys_u}--\eqref{eq:full_sys_J}.
     
    Without coercivity, the role of the energy becomes more complicated.  On the one hand, more terms in the energy means more a priori control, but on the other it means more things that the dissipation may fail to control, further complicating a proof of decay.  If all of the unknowns appear in $\mathfrak{E}_{lin}$ we say that there is \emph{strong coupling}, and otherwise we say that there is \emph{weak coupling}.  Based on the above relation to the dissipation, it may seem that weak coupling is preferable, but this is only true from the point of view of exploiting the dissipation for decay information.  
    
    In the  case of strong coupling we can estimate all of the system's unknowns at the same time, with the same derivative counts.  In the case of weak coupling only some of the unknowns appear in the energy, and the remaining quantities must be estimated in other ways.  This typically entails exploiting some sort of conservative hyperbolic structure that scales differently in terms of derivative counts than the main energy-dissipation part.  At the linear level this isn't a problem because the weak coupling actually leads to a decoupling of these parts of the linearized problem, and we just get different estimates for each part.  However, the essential difficulty with weak coupling comes at the nonlinear level, where the scaling mismatch can make dealing with the high-regularity interaction terms extremely delicate.

    Let's now focus this general discussion onto the specifics of the problem \eqref{eq:full_sys_u}--\eqref{eq:full_sys_J}.  In this case, if $(v,\theta,K)$ denotes the linearization of $(u,\omega,J)$ about the equilibrium, then 
\begin{align}
\mathfrak{E}_{lin} &=  \int_{\T^3} \frac{1}{2} \abs{v}^2
				+ \frac{1}{2} J\theta\cdot\theta
				+ \frac{\tilde{\tau}^2}{\nu-\lambda}  \frac{1}{2} \abs{a}^2, \text{ and }
\label{eq:def_E_lin}\\
\mathfrak{D}_{lin} &=
 \int_{\T^3} \frac{\mu}{2} \vbrac{\symgrad v}^2
			+ 2\kappa \vbrac{ \frac{1}{2} \nabla\times v - \theta}^2
			+ \alpha\vbrac{\nabla\cdot\theta}^2
			+ \frac{\beta}{2} \vbrac{\symgrad^0 \theta}^2
			+ 2\gamma \vbrac{\nabla\times\theta}^2,
\label{eq:def_D_lin}
\end{align}
    where $a = (K_{13},K_{23})$ and $\symgrad$, $\symgrad^0$
    denote twice the symmetric part and twice the traceless symmetric part of the gradient, respectively, and are defined precisely below \eqref{eq:intro_stress_tensors_def}.  From these expressions it's clear that coercivity fails and that we have weak coupling.  Indeed, only two of the six components of the symmetric $K \in \mathbb{R}^{3 \times 3}$ appear in $\mathfrak{E}_{lin}$ (see \fref{Section}{sec:weak_coupling} below for a more detailed discussion of the special role played by $a$).  To estimate the entirety of $K$ we are forced to appeal to the advection-rotation equation \eqref{eq:full_sys_J} and its linearization.  This is a hyperbolic equation coupling to both $u$ and $\omega$, but at different levels of regularity, which already reveals potentially problematic mismatches with energy-dissipation estimates.  On the plus side, we can readily obtain $L^q-$based estimates from \eqref{eq:full_sys_J} for values of $q$ other than $2$.  On the down side, the estimates provided at the highest level of regularity are quite bad, as they grow linearly in time, which makes using them globally in time a delicate proposition.
	
	Our strategy for getting around these problems is to employ a version of the two-tier energy method that was introduced in \cite{guo_tice_periodic,guo_tice_infinite} to handle the viscous surface wave problem, which is a strongly coupled problem with coercivity failure.  Roughly, the idea behind this scheme of a priori estimates is that control of high regularity terms (the high tier) can be synthesized with decay estimates of low regularity terms (the low tier) to simultaneously overcome coercivity and interaction difficulties and prove the existence of global-in-time algebraically-decaying solutions.  The two-tier method is a strategy and not a black box, so it must be adapted to the particulars of each problem.  In our case, due to the weak coupling, the complicated structure of the hyperbolic equation for $K$, and troubles in interfacing with the local existence theory, this requires significant work.
	
	To see how decay information can be recovered in the two-tier scheme, consider the following.  The energy-dissipation structure at low regularity will tell us that (assuming that the nonlinear interactions are wrought under control)
	\begin{equation}
	\label{eq:intro_low_ED_rel}
		\frac{\mathrm{d}}{\mathrm{d}t} \mathcal{E}_\text{low} + \frac{1}{2} \mathcal{D}_\text{low} \leqslant 0
	\end{equation}
	where $ \mathcal{E}_\text{low} $ and $ \mathcal{D}_\text{low} $ are low regularity energy and dissipation functionals built from $\mathfrak{E}_{lin}$ and $\mathfrak{D}_{lin}$, respectively.
	To be concrete:
	\begin{equation}
	\label{eq:discuss_func_E_low}
		\mathcal{E}_\text{low} \sim \normtyp{(u, \theta)}{H}{2}^2 + \normtyp{a}{H}{2}^2 + \normtyp{ \pdt(u,\theta) }{L}{2}^2 + \normtyp{ \pdt a }{L}{2}^2
	\end{equation}
	where we have introduced the perturbative angular velocity $\theta = \omega - \omega_{eq}$.
	The exact form of $ \mathcal{D}_\text{low} $ is not relevant here; all that matters is that $ \mathcal{E}_\text{low} \lesssim \mathcal{E}_\text{high}^{1-\sigma} \mathcal{D}_\text{low}^\sigma$
	for some high regularity energy functional $ \mathcal{E}_\text{high} $ and some $\sigma\in (0,1)$ which behaves as $\sigma \sim \frac{high - low}{high - low + 1}$.
	Here $low$ and $high$ are placeholders for regularities indices precisely measuring the regularity of the solution at each level.
	Crucially, this observation may be combined with \eqref{eq:intro_low_ED_rel},
	\emph{provided} that $ \mathcal{E}_\text{high} $ is bounded,
	to deduce the algebraic decay of $ \mathcal{E}_\text{low} $ at a rate proportional to $high - low$.
	Note that this is precisely almost exponential decay since the growing rate of decay is dependent on the regularity of the solution.
	More concretely, for some non-negative integer $M$, $\mathcal{E}_\text{high}$ and $ \mathcal{D}_\text{high}$ are given by
	\begin{align}
		\mathcal{E}_\text{high} = \sum_{j=0}^M \normtypns{\pdt^j (u, \theta, a) }{H}{2M-2j}^2
			+ \normtypns{(K, \pdt K, \pdt^2 K)}{H}{2M-3}^2 
			+ \sum_{j=3}^M \normtypns{\pdt^j K}{H}{2M-2j+2}^2,
	\label{eq:discuss_func_E_high}\\
		\text{ and } 	
		\mathcal{D}_\text{high} = \sum_{j=0}^M \normtypns{\pdt^j (u,\theta)}{H}{2M-2j+1}^2
			+ \sum_{j=0}^3 \normtypns{\pdt^j a}{H}{2M-j-1}^2 + \sum_{j=4}^M \normtypns{\pdt^j a}{H}{2M-2j+3}^2
	\label{eq:discuss_func_D_high}
	\end{align}
	where we have introduced the perturbative microinertia $K = J - J_{eq}$
	and where $ \mathcal{D}_\text{high} $ is a high regularity dissipation functional whose integral in time will remain bounded.  Remarkably, although $\mathcal{E}_{low}$ provides no direct control of $K$ except its components in $a = (K_{13}, K_{23})$, a special algebraic identity for symmetric matrices with spectrum $\{\lambda,\lambda,\nu\}$ leads to a quantitative rigidity result that will allow us to obtain decay information about all of $K$ from $a$ alone.

	We have now witnessed the first key idea of the two-tier energy method:
	the decay of the low level energy is intimately tied to the boundedness of the high level energy.  In the above sketch this dependence only goes one way, but in practice it also goes the other way since the transport estimates for $K$ at the highest derivative count result in an upper bound that grows linearly in time (see \fref{Section}{sec:strategy_and_difficulties} for a more thorough discussion). This warrants the introduction of the last functional we need in order to state the main result. We define $ \mathcal{F}_\text{high} $ to contain all terms for which the only control we have is growing in time, namely
	\begin{equation}
	\label{eq:discuss_func_F_high}
		\mathcal{F}_\text{high} = \normtyp{K}{H}{2M+1}^2 + \normtyp{\pdt K}{H}{2M}^2 + \normtyp{\pdt^2 K}{H}{2M-1}^2.
	\end{equation}

\subsection{Statement of the main result}
\label{sec:intro_main_result}
	
	In order to state the main result we introduce the global assumptions at play throughout this paper.
	\begin{definition}[Global assumptions]
	\label{def:global_assumptions}
		We assume that the initial microinertia $J_0$ has an inertial axis of symmetry and is inertially oblate,
		i.e. for every $x\in\T^3$ the spectrum of $J_0 (x)$ is $\cbrac{\lambda, \lambda, \nu}$ where $\nu > \lambda > 0$.
		We also assume that the initial velocity $u_0$ has average zero
		and that the viscosity constants $\mu$, $\kappa$, $\alpha$, $\beta$, and $\gamma$ are strictly positive.
	\end{definition}
	Note that the assumption that $u$ have average zero is justified by the invariance of \eqref{eq:full_sys_u}--\eqref{eq:full_sys_J}
	under Galilean transformations $u(t,y) \mapsto u(t,y + t\bar{u}) - \bar{u}$ for any constant $\bar{u}\in\R$.
	We may now state the main result. A more precise form of this result is found in \fref{Theorem}{thm:gwp_decay_clean}.
	\begin{thm}[Nonlinear asymptotic stability and decay]
	\label{thm:early_statement_main_result}
		Suppose that the global assumptions of \fref{Definition}{def:global_assumptions} hold
		and let $X_{eq} = \brac{u_{eq}, \omega_{eq}, J_{eq}} = \brac{0, \frac{\tau}{2\kappa} e_3, \diag\brac{\lambda, \lambda, \nu}}$ and $p_{eq} = 0$ be the equilibrium solution
		of \eqref{eq:full_sys_u}--\eqref{eq:full_sys_J}.
		For every integer $M\geqslant 4$ there exists $\eta, C > 0$ such that solutions to \eqref{eq:full_sys_u}--\eqref{eq:full_sys_J}
		exist globally in-time for every initial condition in the $\eta$-ball defined by
		\begin{equation*}
			\normtyp{(u_0, \omega_0 - \omega_{eq})}{H}{2M}^2 + \normtyp{J_0 - J_{eq}}{H}{2M+1}^2 < \eta.
		\end{equation*}
		Moreover the solutions satisfy the estimate
		\begin{equation*}
			\sup_{t\geqslant 0}
			\mathcal{E}_\text{low} (t) {\brac{1+t}}^{2M-2} + \mathcal{E}_\text{high} (t) + \frac{ \mathcal{F}_\text{high} (t) }{ \brac{1+t} } + \int_0^\infty \mathcal{D}_\text{high} (s) ds
			\leqslant C \brac{ \normtyp{(u_0, \omega_0 - \omega_{eq}}{H}{2M}^2 + \normtyp{J_0 - J_{eq}}{H}{2M+1}^2 }.
		\end{equation*}
		Recall that the functionals present on the left-hand side are defined in \eqref{eq:discuss_func_E_low}--\eqref{eq:discuss_func_F_high}.
	\end{thm}

	Note that in \fref{Theorem}{thm:early_statement_main_result} the pressure has disappeared from consideration in the estimates.
	This is because the pressure only plays an auxiliary role in the problem and may be eliminated altogether from \eqref{eq:full_sys_u}
	by projection onto the space of divergence-free vector fields.

	At face value \fref{Theorem}{thm:early_statement_main_result} only provides us with decay of $u$, $\theta$, and $a$ in terms of the norms appearing in $ \mathcal{E}_\text{low}$.
	However we may interpolate between $ \mathcal{E}_\text{low}$ and $ \mathcal{E}_\text{high}$ to obtain decay estimates on intermediate norms of $u$, $\theta$, and $a$.
	Algebraic identities may then be used to show that, if $ \normtyp{K}{L}{\infty} $ is sufficiently small, then $\abs{K} \lesssim \abs{a}$ pointwise,
	from which we may deduce the decay of $K$.
	The ensuing decay rates are recorded precisely in \fref{Corollary}{cor:add_decay} below
	(which is proved at the end of \fref{Section}{sec:gwp}).
	Note that this corollary only records the decay of the unknowns and their first time derivative.
	The decay rates of higher-order temporal derivatives can then be established by differentiating \eqref{eq:full_sys_u}--\eqref{eq:full_sys_J},
	however since they are not necessary for our purposes here and so we omit them.
	Crucially, with these decay rates in hand we deduce that \fref{Theorem}{thm:early_statement_main_result} is indeed a proof of \emph{asymptotic stability}.

	\begin{cor}[Decay rates]
	\label{cor:add_decay}
		Under the hypotheses of \fref{Theorem}{thm:early_statement_main_result} the global solution $(u, \theta, K)$ satisfies
		\begin{align*}
			\sup_{t\geqslant 0} \Bigg(
			\sup_{0\leqslant s\leqslant 2M+1} \normtyp{K(t)}{H}{s}^2 {\brac{1+t}}^{2M-4-s(2M-3)/(2M+1)}
			+ \sup_{0\leqslant s\leqslant 2M} \normtyp{\pdt K (t)}{H}{s}^2 {\brac{1+t}}^{2M-4-s(2M-3)/(2M)}
		\\
			+ \sup_{2 \leqslant s \leqslant 2M} \brac{ \normtyp{(u, \theta, a)(t) }{H}{s}^2 + \normtyp{\pdt(u, \theta, a)(t)}{H}{s-2}^2 } {\brac{1+t}}^{2M-s}
			\Bigg)
		\lesssim \brac{ \normtyp{(u_0, \theta_0, K_0)}{H}{2M}^2 + \normtyp{K_0}{H}{2M+1}^2 }.
		\end{align*}
	\end{cor}
 
\paragraph{\textbf{Sharp nonlinear stability criterion}}

	\vspace{0.5em}
	We may combine the main result of this paper, namely \fref{Theorem}{thm:early_statement_main_result}, with the decay rates of \fref{Corollary}{cor:add_decay}
	and the main result of \cite{rt_tice_amf_iut} to deduce a sharp nonlinear stability criterion recorded in \fref{Theorem}{thm:sharp_criterion} below.
	In order to formulate \fref{Theorem}{thm:sharp_criterion} in a clean way we define appropriate spaces, namely
	\begin{align*}
		\mathcal{H}_0 &= H^{2M} \brac{ \T^3;\, \R^3} \times H^{2M} \brac{\T^3;\, \R^3} \times H^{2M+1} \brac{\T^3;\, \sym(3)},\\
		\mathcal{H}_s &= H^{2M} \brac{ \T^3;\, \R^3} \times H^{2M} \brac{\T^3;\, \R^3} \times H^{2M-4/(2M-3)} \brac{\T^3;\, \sym(3)}, \text{ and }\\
		\mathcal{H}_{as} &= H^{2M-\varepsilon} \brac{ \T^3;\, \R^3} \times H^{2M-\varepsilon} \brac{\T^3;\, \R^3} \times H^{2M-4/(2M-3)-\varepsilon} \brac{\T^3;\, \sym(3)},
	\end{align*}
	where $\varepsilon > 0$ may be taken to be arbitrarily small.
	We  may now state the sharp nonlinear stability criterion.
	\begin{thm}
	\label{thm:sharp_criterion}
		Let $X_{eq} = \brac{u_{eq}, \omega_{eq}, J_{eq}} = \brac{0, \frac{\tau}{2\kappa} e_3, \diag\brac{\lambda, \lambda, \nu}}$ be the equilibrium solution
		of \eqref{eq:full_sys_u}--\eqref{eq:full_sys_J}.
		\begin{itemize}
			\item	If the microstructure is inertially oblong ($\lambda > \nu$)
				then the equilibrium is nonlinearly unstable in $L^2$.
			\item	If the microstructure is inertially oblate ($\nu > \lambda$)
				then the equilibrium is nonlinearly $\mathcal{H}_s$-stable in $\mathcal{H}_0$
				and nonlinearly asymptotically $\mathcal{H}_{as}$-stable in $\mathcal{H}_0$.
		\end{itemize}
	\end{thm}
	\noindent
	The notions of nonlinear stability and instability above are those familiar from dynamical systems.
	\emph{Nonlinear stability} in $L^2$ means that there exists a radius $\delta > 0$ and a sequence of initial data ${\cbrac{ X_n^0 }}_{n=0}^\infty$
	which converge to $X_{eq}$ in $L^2$ such that the solutions to \eqref{eq:full_sys_u}--\eqref{eq:full_sys_J} starting from $X_n^0$ exit the $\delta$-ball about $X_{eq}$
	in finite time (which depends on $n$).
	\emph{Nonlinear $\mathcal{H}_s$-stability in $\mathcal{H}_0$} means that for every $\varepsilon > 0$
	there exists a $\delta$-ball about $X_{eq}$ in $\mathcal{H}_0$ in which \eqref{eq:full_sys_u}--\eqref{eq:full_sys_J} is globally well-posed
	and such that solutions remain in the $\varepsilon$-ball about $X_{eq}$ in $\mathcal{H}_s$ for all time.
	about $X_{eq}$ in $\mathcal{H}_0$ also converge to $X_{eq}$ in $\mathcal{H}_{as}$ as time $t\to\infty$.
	
\subsection{Previous work}
\label{sec:prev_work}
	
	The continuum mechanics community has actively and extensively studied micropolar fluids over the past fifty years.
	While an exhaustive literature review is beyond the scope of this paper, we highlight the mathematics literature here.
	To the best of our knowledge, current mathematical results only consider \emph{isotropic} microstructure,
	which means that the microinertia $J$ is a scalar multiple of the identity.
	In particular, when a micropolar fluid is isotropic the precession term $\omega \times J\omega$ which appears in \eqref{eq:full_sys_omega} now vanishes
	and \eqref{eq:full_sys_J}, which governs the dynamics of the microinertia, is trivially satisfied.
	Note that in two-dimensions the microinertia is a scalar, such that all two-dimensional micropolar fluids are isotropic.

	The results known about isotropic micropolar fluids follow the pattern of what is known about viscous fluids.
	In two-dimensions global well-posedness holds \cite{lukaszewiscz_01} and quantitative rates of decay are obtained in \cite{dong_chen}.
	In three-dimensions, where well-posedness was first discussed by Galdi and Rionero \cite{galdi_rionero},
	weak solutions where constructed globally in time by {\L}ukaszewicz \cite{lukaszewicz_90}, who also proved that strong solutions are unique \cite{lukaszewicz_89}.
	More recent work has established global well-posedness for small data in critical Besov spaces \cite{chen_miao}
	and in the space of pseudomeasures \cite{ferreira_villamizar_roa}, and a blow-up criterion was derived \cite{yuan}.
	There is also a body of works dedicated to the study of partially inviscid limits taking one or more of the viscosity coefficients to zero.
	We refer to \cite{dong_zhang} for an illustrative example.

	Various extensions of the model of incompressible micropolar fluids presented here have been considered.
	These extensions have treated the compressible case \cite{liu_zhang}, coupled the system to heat transfer \cite{tarasinska, kalita_langa_lukaszewicz},
	and to magnetic fields \cite{ahmadi_shahinpoor, rojas_medar_marko}.
	Again, to the best of our knowledge all of these works consider \emph{isotropic} micropolar fluids.

	As mentioned above, we employ a two-tier nonlinear energy method as our scheme of a priori estimates.  This was originally used by Guo and Tice \cite{guo_tice_periodic,guo_tice_infinite} in the analysis of the viscous surface wave problem, where the conservative variable is the free surface function, which is strongly coupled.  This technique was also used to deal with the strongly coupled mass density for the compressible surface wave problem by Jang, Tice, Wang \cite{jang_tice_wang_compressible_internal_wave}.  Two-tier schemes have also proved useful in MHD problems without magnetic viscosity, where the magnetic field is the conservative unknown and is weakly coupled according to our above classification: see, for instance, the works of Ren, Wu, Ziang, and Zhang    \cite{ren_wu_xiang_zhang_MHD}, Abidi and Zhang \cite{abidi_zhang_MHD}, Tan and Wang \cite{tan_wang_MHD}, and Wang \cite{wang_MHD}.  The weak coupling of our present problem is more complicated than in these MHD results since some of the components of $K$, namely $a = (K_{13},K_{23})$, are strongly coupled, which means that $K$ cannot be conveniently ``integrated out'' by solving for it in terms of the other unknowns.
	
%----------------------------------------------------------------------------------------------------
%	STRATEGY AND DIFFICULTIES
%----------------------------------------------------------------------------------------------------
	
\section{Strategy and difficulties}
\label{sec:strategy_and_difficulties}

	In this section we describe the various obstructions to proving a stability result like \fref{Theorem}{thm:early_statement_main_result} and we discuss our strategy to overcome them.
	Since we study the nonlinear stability of a non-trivial equilibrium it is natural to change variables and use perturbative unknowns.
	This is done in \fref{Section}{sec:perturb_form}.

	In \fref{Sections}{sec:no_spectral_gap} and \ref{sec:weak_coupling} we then discuss the two main obstructions, namely the lack of a spectral gap and the weak coupling.
	In a nutshell, the difficulties are as follows.
	The lack of spectral gap leads to a failure of coercivity.
	The key in overcoming that is to prove a $\theta$-coercivity estimate.
	On top of that, weak coupling means that, even with $\theta$-coercivity, we cannot immediately deduce decay of all the unknowns
	(and so this comes after the fact via an algebraic identity and interpolation).

	We conclude in \fref{Sections}{sec:discuss_ap}--\ref{sec:discuss_gwp} with a discussion of the various moving pieces of our proof of \fref{Theorem}{thm:early_statement_main_result}.
	The centerpiece of our proof is the scheme of a priori estimates introduced in \fref{Section}{sec:discuss_ap}.
	\fref{Section}{sec:discuss_lwp} describes the local well-posedness theory
	and \fref{Section}{sec:discuss_cont} discusses how to ``glue'' the local well-posedness theory and the a priori estimates by means of a continuation argument.
	Finally \fref{Section}{sec:discuss_gwp} explains how to synthesize the various pieces of the proof in order to deduce global well-posedness and decay, and hence asymptotic stability.

\subsection{Perturbative formulation and overall strategy}
\label{sec:perturb_form}

	Since we study the nonlinear stability of \eqref{eq:full_sys_u}--\eqref{eq:full_sys_J} about the equilibrium
	$
		(u_{eq}, p_{eq}, \omega_{eq}, J_{eq}) = (0, 0, \frac{\tau}{2\kappa} e_3, \diag(\lambda, \lambda, \nu))
	$
	we naturally seek to write this system in terms of the perturbative variables
	$
		(u, p, \theta, K) = (u, p, \omega, J) - (\omega_{eq}, p_{eq}, \omega_{eq}, J_{eq}).
	$
	We may then write \eqref{eq:full_sys_u}--\eqref{eq:full_sys_J} equivalently as
	\begin{subnumcases}{}
		\pdt u + u\cdot\nabla u = (\mu + \kappa/2) \Delta u + \kappa\nabla\times\theta - \nabla p
		&on $(0,T)\times \T^n$,
		\label{eq:pertub_sys_no_ten_pdt_u}\\
		\nabla\cdot u = 0
		&on $(0,T)\times \T^n$,
		\label{eq:pertub_sys_no_ten_div}\\
		(J_{eq} + K)(\pdt\theta + u\cdot\nabla\theta) + (\omega_{eq} + \theta) \times (J_{eq} + K)(\omega_{eq} + \theta)
		\nonumber\\\qquad
			= \kappa\nabla\times u - 2\kappa\theta + (\tilde{\alpha} - \tilde{\gamma}) \nabla(\nabla\cdot\theta) + \tilde{\gamma}\Delta\theta
		&on $(0,T)\times \T^n$,
		\label{eq:pertub_sys_no_ten_pdt_theta}\\
		\pdt K + u\cdot\nabla K = \sbrac{\Omega_{eq} + \Theta, J_{eq} + K}.
		&on $(0,T)\times \T^n$,
		\label{eq:pertub_sys_no_ten_pdt_K}
	\end{subnumcases}
	where recall that $a = (K_{12}, K_{13})$.
	This is the system that will be studied in this paper,
	and there are two important remarks to make about \eqref{eq:pertub_sys_no_ten_pdt_u}--\eqref{eq:pertub_sys_no_ten_pdt_K}.

	The first remark is that the pressure plays a different role from the other unknowns, since it is essentially the Lagrange multiplier associated with the incompressibility constraint.
	We will therefore remove the pressure from consideration by the usual trick of projecting \eqref{eq:pertub_sys_no_ten_pdt_u} onto the space of divergence-free vector fields.
	This is done using the Leray projection $\mathbb{P}_L$ which on the three-dimensional torus takes the simple form $\mathbb{P}_L = - \nabla\times\Delta\inv\nabla\times$.
	We may then deduce from \eqref{eq:pertub_sys_no_ten_pdt_u} that
	\begin{equation}
	\label{eq:full_sys_u_Leray_proj}
		\pdt u + \mathbb{P}_L ( u\cdot\nabla u) = (\mu + \kappa/2)\Delta u + \kappa\nabla\times \theta.
	\end{equation}
	This equation will often be useful, in particular when it comes to the local well-posedness theory where it is convenient to
	view \eqref{eq:pertub_sys_no_ten_div}--\eqref{eq:pertub_sys_no_ten_pdt_K} and \eqref{eq:full_sys_u_Leray_proj} as an ODE.

	The second remark builds off of the fact that, as hinted at in \fref{Section}{sec:mixed_type} above and as discussed in more detail in \fref{Section}{sec:weak_coupling} below,
	$a$ is a component of $K$ which plays a particularly important role.
	It will therefore be crucial, when performing energy estimates, to read off from \eqref{eq:pertub_sys_no_ten_pdt_K} the equation governing the dynamics of $a$, namely
	\begin{equation}
	\label{eq:pertub_sys_pdt_a}
		\pdt a + u\cdot\nabla a = -(\nu - \lambda) \bar{\theta}^\perp + (\bar{K} - K_{33}I_2) \bar{\theta}^\perp + \frac{\tau}{2\kappa} a^\perp + \theta_3 a^\perp,
	\end{equation}
	where $\bar{\theta} = (\theta_1, \theta_2)$, $v^\perp = (-v_2, v_1)$ for any $v\in\R^2$, and $\bar{K} = \begin{pmatrix} K_{11} & K_{12} \\ K_{21} & K_{22} \end{pmatrix}$.
	
	To conclude this section we note that it is often useful to consider an alternative formulation of \eqref{eq:pertub_sys_no_ten_pdt_u}--\eqref{eq:pertub_sys_no_ten_pdt_K}
	involving the stress tensor $T$ and the couple-stress tensor $M$.
	Note that, just as the classical stress tensor $T$ encodes a fluid's response to forces,
	the couple-stress tensor $M$ encodes a micropolar fluid's response to torques acting on the microstructure.
	These tensors are given by
	\begin{equation}
	\label{eq:intro_stress_tensors_def}
		T = \mu \symgrad u + \kappa\ten\brac{ \frac{1}{2} \nabla\times u - \omega} - pI
		\text{ and } 
		M = \alpha (\nabla\cdot\omega) I + \beta \symgrad^0 \omega + \gamma\ten\nabla\times\omega.
	\end{equation}
	Here $\symgrad v$ denotes (twice) the symmetric part of the derivative of a vector field $v$, i.e. $\symgrad v = \nabla v + {\nabla v}^T$,
	and $\symgrad^0 v$ denotes its trace-free part, i.e. $\symgrad^0 v = \symgrad v - \frac{2}{n} (\nabla\cdot v) I$.
	We may then formulate \eqref{eq:pertub_sys_no_ten_pdt_u}--\eqref{eq:pertub_sys_no_ten_pdt_K} equivalently as
	\begin{subnumcases}{}
		\pdt u + u\cdot\nabla u = (\nabla\cdot T)(u, p, \theta)
		&on $(0,T)\times \T^n$,
		\label{eq:pertub_sys_ten_pdt_u}\\
		\nabla\cdot u = 0
		&on $(0,T)\times \T^n$,
		\label{eq:pertub_sys_ten_div}\\
		(J_{eq} + K)(\pdt\theta + u\cdot\nabla\theta) + (\omega_{eq} + \theta) \times (J_{eq} + K)(\omega_{eq} + \theta)
		\nonumber\\\qquad
			= 2\vc T(u, p, \theta) + (\nabla\cdot M)(\theta)
		&on $(0,T)\times \T^n$,
		\label{eq:pertub_sys_ten_pdt_theta}\\
		\pdt K + u\cdot\nabla K = \sbrac{\Omega_{eq} + \Theta, J_{eq} + K}.
		&on $(0,T)\times \T^n$.
		\label{eq:pertub_sys_ten_pdt_K}
	\end{subnumcases}
	This formulation is particularly convenient when it comes to identifying the energy-dissipation relation
	since it makes it clear which terms contribute to the energy and which contribute to the dissipation.
	To be precise, we see that the dissipation comes precisely from the stress and couple-stress tensors since
	\begin{equation}
		\int_{\mathbb{T}^3} T:(\Theta - \nabla u) + M : \nabla\theta
			= \int_{\T^3} \frac{\mu}{2} \vbrac{\symgrad u}^2
			+ 2\kappa \vbrac{ \frac{1}{2} \nabla\times u - \theta}^2
			+ \alpha\vbrac{\nabla\cdot\theta}^2
			+ \frac{\beta}{2} \vbrac{\symgrad^0 \theta}^2
			+ 2\gamma \vbrac{\nabla\times\omega}^2
	\label{eq:dissip_first_apparition}
	\end{equation}
	where the right-hand side denotes the dissipation $D(u, \theta)$.
	In particular note that the dissipation does not provide any control over the perturbative microinertia $K$.

\subsection{Lack of spectral gap}
\label{sec:no_spectral_gap}

	In this section we describe the first of the two main obstructions in proving the asymptotic stability of the equilibrium, namely the \emph{provable} absence of a spectral gap.
	To prove that the linearization of \eqref{eq:pertub_sys_no_ten_pdt_u}--\eqref{eq:pertub_sys_no_ten_pdt_K} about equilibrium does not have a spectral gap we may leverage
	the careful spectral analysis carried out in the companion paper \cite{rt_tice_amf_iut} where the instability of inertially oblong microstructure is established.

	In order to describe this spectral analysis we must first recall the block structure of the linearization.
	Upon linearizing \eqref{eq:pertub_sys_no_ten_pdt_theta}--\eqref{eq:pertub_sys_no_ten_pdt_K} and \eqref{eq:full_sys_u_Leray_proj}
	and pre-multiplying the linearization of \eqref{eq:pertub_sys_no_ten_pdt_theta} by $J_{eq}\inv$ we see that the resulting linear operator may
	be written in block form as $\mathcal{B} \oplus \frac{\tau}{2\kappa} \sbrac{R, \,\cdot\,} \oplus 0$.
	Here the first block acts on the variables $(u, \theta, a)$, the second block acts on $\bar{K} = \begin{pmatrix} K_{11} & K_{12} \\ K_{21} & K_{22} \end{pmatrix}$,
	and the third block acts on $K_{33}$.
	Crucially, the last two blocks have trivial dynamics since $\frac{\tau}{2\kappa} \sbrac{R, \,\cdot\,}$ is a three-dimensional anti-symmetric operator
	which gives rise to one trivial mode and two conjugate oscillatory modes.
	The linear stability of \eqref{eq:pertub_sys_no_ten_pdt_theta}--\eqref{eq:pertub_sys_no_ten_pdt_K} and \eqref{eq:full_sys_u_Leray_proj} is therefore dictated by the spectrum of $\mathcal{B}$.

	In practice we study the spectrum of the symbol $\hat{\mathcal{B}}(k)$, $k\in{(2\pi\Z)}^3$, of $\mathcal{B}$.
	Note that the torus is rescaled for convenience since this particular scaling means that $\hat{\nabla} (k) = ik$.
	A careful spectral analysis (the details of which can be found in \cite{rt_tice_amf_iut}) then allows us to prove the following.
	For any $k\in {(2\pi\Z)}^3$ the spectrum of $\hat{\mathcal{B}}(k)$ is contained in the half-slab $H = \cbrac{z\in\C : \re z \leqslant 0 \text{ and } \abs{\im z} \leqslant C}$,
	for some constant $C>0$.
	In particular we may find a radius $R > C$ and a cutoff $k_* > 0$ such that if $\abs{k} > k_*$ then there are precisely three eigenvalues of $\hat{\mathcal{B}}(k)$ in $H \cap B_R(0)$:
	zero (which is associated with the incompressibility constraint) and a conjugate pair of eigenvalues $z(k)$ and $\bar{z}(k)$.
	Crucially, this pair of eigenvalues satisfies $\re z(k) \to 0$ as $\abs{k}\to\infty$.
	This analysis, summarized pictorially in \fref{Figure}{fig:spectral_analysis}, proves that $\mathcal{B}$, and hence the linearization itself, does not admit a spectral gap.

	\begin{figure}[h!]
		\centering
		\includegraphics{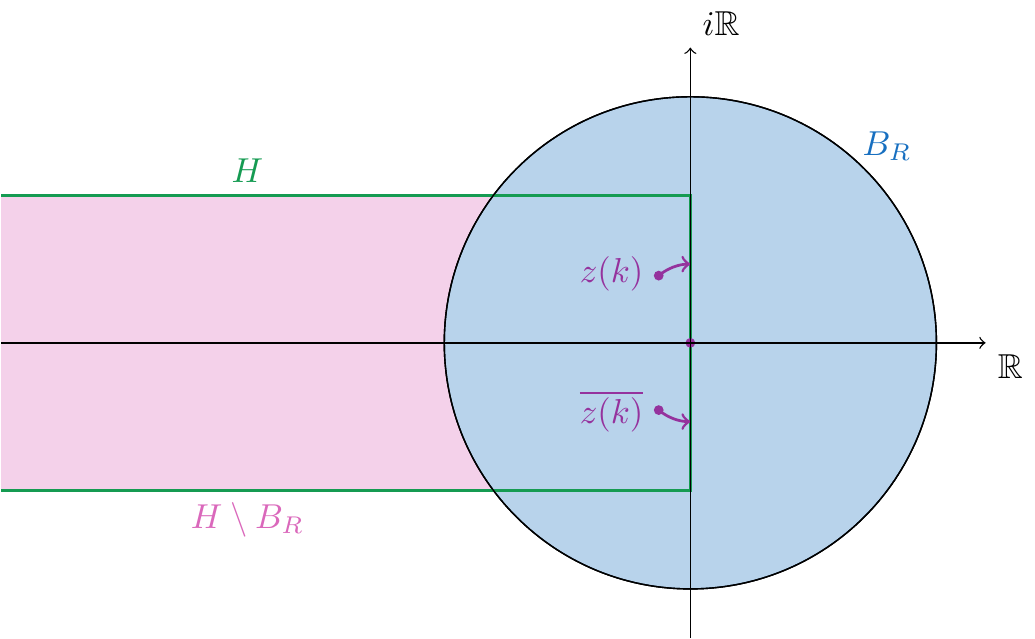}
		\caption{A pictorial summary of the argument carried out in \fref{Section}{sec:no_spectral_gap} which proves the lack of a spectral gap.}
		\label{fig:spectral_analysis}
	\end{figure}

	At the nonlinear level, the manifestation of the lack of a spectral gap is a failure of coercivity, meaning that an estimate of the form $\mathcal{E} \leqslant \mathcal{D}$ is out of reach.
	To overcome this, we prove a $\theta$-coercivity estimate, which takes the form $ \mathcal{E} \lesssim \mathcal{D}^\theta$ for some $\theta\in (0,1)$
	and leads to algebraic decay. The implementation of this is discussed in more details in \fref{Section}{sec:discuss_linear_analysis} below.

\subsection{Weak coupling}
\label{sec:weak_coupling}
	
	In this section we discuss the second major obstruction to proving the stability of the equilibrium, namely the so-called \emph{weak coupling}
	of the parabolic part of the problem to its hyperbolic part.

	Recall that as discussed in \fref{Section}{sec:mixed_type}, the term ``weak coupling'' describes the fact that only \emph{some} of the unknowns appear in the energy.
	The remaining unknowns are solely controlled by conservative hyperbolic-type estimates.
	As seen in \eqref{eq:def_E_lin}, here the weak coupling manifests itself at the level of the perturbative microinertia $K$ since only
	two of its components, denote by $a = (K_{13}, K_{23})$, appear in the energy.
	The reason why only $a$ appears in the energy is due to the precession term $\omega\times J\omega$ from \eqref{eq:full_sys_omega}.
	More precisely: when writing the precession term in perturbative form, as is done in \eqref{eq:pertub_sys_no_ten_pdt_theta},
	we notice the appearance of the term $\omega_{eq} \times J \omega_{eq} = {\brac{\frac{\tau}{2\kappa}}}^2 (-a_2, a_1)$.

	At first glance, this means that we only expect decay of $a$, and not of the remaining components of $K$.
	This is how our scheme of a priori estimates is built (see \fref{Section}{sec:discuss_ap} below for more details).
	This means that the nonlinear estimates are particularly delicate.
	Indeed, since some terms appearing in the nonlinear interactions are not assumed to decay, at any level of regularity,
	it follows that we must be very careful about playing off these non-decaying terms against terms that decay sufficiently fast.

	That being said, it most also be noted that, independently of our scheme of a priori estimates, a minor miracle of linear algebra occurs.
	This allows us to prove a \emph{quantitative rigidity} result: if $ \normtyp{K}{L}{\infty} \leqslant \nu - \lambda $, then $ \abs{K} \leqslant  2\abs{a}$ pointwise.
	In particular, we may then deduce the $L^2$ decay of $K$, and then bootstrap via interpolation to obtain the decay of norms of $K$ at higher regularity.
	Note that this ``minor miracle'' is recorded in \fref{Proposition}{prop:quant_rigidity}.

	Crucially, since the decay of $K$ can be recovered a posteriori via algebraic identities, we do not incorporate it into our scheme of a priori estimates.
	Indeed, doing so would not strengthen the final statement of the theorem, and while it would give us another lever to pull when performing nonlinear estimates,
	it would further complicate  our scheme of a priori estimates since the numerology of the precise decay rates of $K$ and $\pdt K$ is not particularly pleasant.	
	
\subsection{A priori estimates}
\label{sec:discuss_ap}

	In this section we discuss the a priori estimates, which are carried out in \fref{Section}{sec:a_prioris}.
	As mentioned previously,
	a fundamental observation about the problem at hand is that it is of \emph{mixed} type.
	On one hand, the equations driving the dynamics of the velocity $u$ and the (perturbative) angular velocity $\theta$ are parabolic.
	On the other hand, the equation driving the dynamics of the (perturbative) microinertia $K$ is hyperbolic.
	Having made this fundamental observation, the two questions we seek to answer are the following.
	\begin{enumerate}
		\item	What kind of decay does the linearized problem possess?
		\item	How can we massage the nonlinear structure to push this decay through to the nonlinear problem?
	\end{enumerate}

	We will address the decay of the linearized problem in \fref{Section}{sec:discuss_linear_analysis}
	and turn our attention to the nonlinear effects in \fref{Section}{sec:discuss_nonlinear_effects}.
	Throughout this discussion, we will underscore how the four pieces of the a priori estimates, namely
	(1) closing the energy estimates at the low regularity,
	(2) closing the energy estimates at the high regularity,
	(3) deriving advection-rotation estimates for $K$, and
	(4) obtaining the decay of intermediate norms,
	are related to one another. This is also summarized pictorially in \fref{Figure}{fig:a_priori_discuss}.

	Note that throughout this discussion we will use various versions of the energy and the dissipation.
	Their precise forms may be found in \fref{Section}{sec:notation} below.
	At first pass, the following heuristics may be useful for the reader:
	$\overline{\mathcal{E}}$ and $\overline{\mathcal{D}}$ denote the energy and dissipation functionals that naturally arise when
	performing the energy estimates whereas $\mathcal{E}$ and $\mathcal{D}$ denote \emph{improved} versions of these functionals.

\subsubsection{Linear analysis}
\label{sec:discuss_linear_analysis}

	\paragraph{\textbf{$\theta$-coercivity and the two-tier energy structure.}}
	We begin our discussion with the analysis of the linearized system, and how it leads to almost-exponential decay.
	We will emphasize that it naturally gives rise to the aforementioned two-tier energy structure where the decay of the low level energy is tied to the boundedness of the high level energy.

	The starting point is the energy-dissipation relation, which tells us that
	\begin{equation}
	\label{eq:discuss_ap_ED_low}
		\frac{\mathrm{d}}{\mathrm{d}t} \overline{\mathcal{E}}_\text{low} + \overline{\mathcal{D}}_\text{low} = 0.
	\end{equation}
	Note that the mixed parabolic-hyperbolic structure already manifests itself here: the energy $ \overline{\mathcal{E}}_\text{low} $ is a function of $(u, \theta, a)$
	whereas the dissipation $ \overline{\mathcal{D}}_\text{low} $ is only a function of $(u, \theta)$.
	To derive any decay estimate from \eqref{eq:discuss_ap_ED_low} we need the dissipation to control the energy in some fashion,
	which at first glance is out of reach due to the absence of $a$ in the dissipation.

	Given the structure of the equations, this gap between the energy and the dissipation may not be fully closed.
	In particular, note that this energy-dissipation gap is consistent with the lack of spectral gap discussed in \fref{Section}{sec:no_spectral_gap}.
	To partially close this energy-dissipation gap we \emph{improve} the dissipation,
	i.e. leverage auxiliary estimates for $a$ to see that
	\begin{equation*}
		\overline{\mathcal{D}}_\text{low} (u,\theta) \gtrsim \mathcal{D}_\text{low} (u, \theta, a).
	\end{equation*}
	With this improvement in hand there is hope for the dissipation to control the energy.
	More precisely, what we can show is that
	\begin{equation}
	\label{eq:discuss_ap_hypo}
		\overline{\mathcal{E}}_\text{low} \lesssim \overline{\mathcal{E}}_M^{1-\theta} \mathcal{D}_\text{low}^{\theta},\;
		\text{ for } \theta = \frac{2M-2}{2M-1} \in (0,1),
	\end{equation}
	where $ \overline{\mathcal{E}}_M $ is a high regularity counterpart to the low regularity energy $ \overline{\mathcal{E}}_\text{low}$.
	When \eqref{eq:discuss_ap_hypo} holds we say that the dissipation is \emph{$\theta$-coercive} over the energy.

	Crucially, this $\theta$-coercivity estimate is only useful if we know that the high regularity energy $ \overline{\mathcal{E}}_M $ remains bounded.
	Thankfully this is immediate from the high regularity version of the energy-dissipation relation, which reads
	\begin{equation*}
		\frac{\mathrm{d}}{\mathrm{d}t} \overline{\mathcal{E}}_M + \overline{\mathcal{D}}_M = 0.
	\end{equation*}
	The non-negativity of $ \overline{\mathcal{D}}_M $ then tells us that $ \overline{\mathcal{E}}_M (t) \leqslant \overline{\mathcal{E}}_M (0)$.

	Combining the $\theta$-coercivity estimate at low regularity, the boundedness of the high regularity energy,
	and a nonlinear Gronwall argument allows us to deduce the decay of the low regularity energy.
	Indeed, we have that, for some $C>0$,
	\begin{equation}
	\label{eq:discuss_ap_decay_E_low}
		\frac{\mathrm{d}}{\mathrm{d}t} \overline{\mathcal{E}}_\text{low}
		+ \frac{C \overline{\mathcal{E}}_\text{low}^{1/\theta}}{ {\overline{\mathcal{E}}_M (0)}^{\frac{1}{\theta}-1} }
		\leqslant 0,
		\text{ and hence }
		\overline{\mathcal{E}}_\text{low} (t) \lesssim \frac{ \overline{\mathcal{E}}_M (0) }{ {\brac{1+t}}^{2M-2} }.
	\end{equation}

	It is important to make two remarks here.
	First we note that as shown above the almost-exponential decay is \emph{not} a nonlinear effect. It is the best rate of decay we can expect given the structure of the linearized problem.
	Second we recall that, as mentioned previously, the two-tier energy structure is significantly more intricate for the nonlinear problem,
	since in that case the decay of $ \overline{\mathcal{E}}_\text{low} $ and the boundedness of $ \overline{\mathcal{E}}_M $ are \emph{inter-dependent} on one another,
	whereas here in this linear setting only the decay of the low regularity energy is predicated on the boundedness of the high regularity energy.

\subsubsection{Nonlinear effects}
\label{sec:discuss_nonlinear_effects}

	\paragraph{\textbf{Decay of intermediate norms.}}
	We begin our discussion of the nonlinear effects with an description of how
	the decay of the low level energy and the boundedness of the high level energy lead to the (slower) decay of intermediate norms.
	While this is not, in essence, a nonlinear feature, it is crucial in order to wrest some of the nonlinear effects under control, as described further below in this section.
	We note that this interpolation argument is carried out precisely in \fref{Section}{sec:decay_int_norms}.

	Interpolation theory tells us that if the low regularity energy decays as in \eqref{eq:discuss_ap_decay_E_low} and the high regularity energy is bounded by its initial value, then
	\begin{equation}
	\label{eq:discuss_ap_decay_int_norms}
		\overline{\mathcal{K}}_I (t) 
		\lesssim \frac{ \overline{\mathcal{E}}_M (0) }{ {\brac{1+t}}^{2M-2I} },\;
		\text{ for } 1\leqslant I\leqslant M,
	\end{equation}
	where $\overline{\mathcal{K}}_I = \normtyp{(u,\theta,a) (t) }{H}{2I}^2 + \normtyp{(\pdt u,\pdt\theta,\pdt a)(t)}{H}{2I-2}^2$ is the sum of the (squared) norms for which we expect decay.
	Note that since only $(u, \theta, a)$ and $\pdt(u, \theta, a)$ appear in $ \overline{\mathcal{E}}_\text{low}$,
	these are also the only terms appearing in $ \overline{\mathcal{K}}_I$.
	In other words: \eqref{eq:discuss_ap_decay_int_norms} does not yield any decay information on higher temporal derivatives,
	a fact which will be important later.

	\hspace{0.1em}
	\paragraph{\textbf{Controlling $K$: advection-rotation estimates.}}
	We continue our discussion of the nonlinear effect with an explanation of why energy estimates are not sufficient to close the scheme of a priori estimates.
	Then we discuss how the advection-rotation estimates which give us control over $K$ give rise to a dichotomy between ``good'' terms and ``bad'' terms.
	The advection-rotation estimates for $K$ are carried out in \fref{Section}{sec:adv_rot_est_K} and culminate in \fref{Proposition}{prop:est_K}.

	Energy estimates are not sufficient to close the scheme of a priori estimates for a simple reason:
	they produce interactions which are out of control due to the absence of $K$ in the dissipation.
	Indeed, suppose that instead of using the equation which governs the dynamics of $a$ in the energy estimates
	we used the equation which governs the dynamics of the full perturbative microinertia $K$.
	Schematically, we could then obtain an energy-dissipation relation of the form
	\begin{equation*}
		\frac{\mathrm{d}}{\mathrm{d}t} \overline{\mathcal{E}} (u, \theta, K) + \overline{\mathcal{D}} (u, \theta) = \overline{\mathcal{I}} (u, \theta, K),
	\end{equation*}
	where $\overline{\mathcal{I}}$ denotes the interaction terms.
	However, as described in \fref{Section}{sec:discuss_linear_analysis}, even after improving the dissipation we can only wrest $a$ under control, and not $K$.
	This is due to the fact that only $a$ appears in the equation governing the dynamics of $\theta$, and this is where our auxiliary estimates for $a$ begin.
	Ultimately, as described previously this is because $K$ only appears in that equation through the precession term, and we have the identity
	$\omega_{eq} \times K \omega_{eq} = {\brac{ \frac{\tau}{2\kappa} }}^2 \tilde{a}^\perp$.

	This lack of dissipative control over $K$ is fatal when it comes to gaining control over the interaction terms.
	More precisely, when taking $\alpha$ many derivatives we see that one of the interaction terms is
	\begin{equation}
	\label{eq:discuss_bad_interaction}
		\int_{\T^3} \partial^\alpha \brac{ \sbrac{\Theta, K} } : \partial^\alpha K =\vcentcolon \mathcal{I}^\alpha.
	\end{equation}
	We cannot hope to control this term since we are after an estimate that would allow us to absorb the interaction term into the dissipation, provided the energy is small,
	i.e. an estimate of the form $\vbrac{\mathcal{I}^\alpha} \lesssim \mathcal{E}^{1/2} \mathcal{D}$.
	In light of this inability to close the scheme of a priori estimates by solely relying on energy estimates, we turn our attention to the equation governing the dynamics of $K$.
	This is essentially a reminder that since the problem is of mixed parabolic-hyperbolic type, we cannot build a complete scheme of a priori estimates leveraging only the parabolic
	structure of the problem (i.e. the structure that gives rise to the energy-dissipation relation) and must also take into account the hyperbolic structure embedded in the equation governing
	the dynamics of $K$.

	The equation satisfied by $K$ is an advection-rotation equation since it involves both advective effects due to the velocity $u$
	and rotational effects due to the (perturbative) angular velocity $\theta$.
	The fundamental observation is the following: if $v$ is divergence-free, $A$ is anti-symmetric, and $S$ is symmetric, then (provided all unknowns are sufficiently regular)
	\begin{equation*}
		\int_{\T^3} \brac{ \pdt + v\cdot\nabla - \sbrac{A, \,\cdot\,} } S : S = 0.
	\end{equation*}
	This leads to the following $L^p$ estimate.
	If $S$ solves
	$
		\brac{\pdt + v\cdot\nabla - \sbrac{A, \,\cdot\,} } S = F
	$
	for some forcing $F$, then
	\begin{equation*}
		\normtyp{ S(t) }{L}{p} \leqslant \normtyp{ S(0) }{L}{p} + \int_0^t \normtyp{ \sym(F) (s) }{L}{p} ds.
	\end{equation*}
	This immediately grants us control over $L^p$ norms of $K$.
	To gain control over $K$ in $H^k$ we must couple this $L^p$ estimate with high-low estimates.
	In particular, provided that some low regularity norms decay sufficiently fast (which, as discussed in \fref{Section}{sec:discuss_linear_analysis}, is expected),
	we may combine such high-low estimates with the $L^p$ estimate above to deduce that
	\begin{equation*}
		\normtyp{K(t)}{H}{k} \lesssim \normtyp{K(0)}{H}{k} + \int_0^t \normtyp{(u,\theta)(s)}{H}{k} ds.
	\end{equation*}

	Crucially, there are only two ways in which we can control $\int_0^t \normtyp{(u,\theta)}{H}{k} $:
	(1) through the boundedness of $\int_0^t \overline{\mathcal{D}}_M $ and
	(2) through the decay of the intermediate norms in $ \overline{\mathcal{K}}_I $.
	Comparing items (1) and (2), the following trade-off comes to light:
	item (1) gives us control of $K$ at a higher regularity, at the cost of an upper bound growing in time.
	Indeed, on one hand it follows from item (1) and the Cauchy-Schwarz inequality that, for all $k\leqslant 2M+1$,
	\begin{equation*}
		\int_0^t \normtyp{(u,\theta)}{H}{k}
		\leqslant t {\brac{ \fint_0^t \normtyp{(u,\theta)}{H}{k}^2 }}^{1/2}
		\leqslant {\brac{ t \overline{\mathcal{D}}_M }}^{1/2}.
	\end{equation*}
	On the other hand combining item (2) and the decay of the intermediate norms of \eqref{eq:discuss_ap_decay_int_norms} tells us that, for $k\leqslant 2M-3$,
	\begin{equation*}
		\int_0^t \normtyp{(u, \theta)}{H}{k}
		\lesssim \int_0^t \frac{ \overline{\mathcal{E}}_M (0) }{ {\brac{1+s}}^{M-\frac{k}{2}} } ds
		\lesssim \overline{\mathcal{E}}_M (0).
	\end{equation*}

	Note that this trade-off is only at play when two or fewer temporal derivatives hit $K$.
	This is because control of time derivatives of $K$ does not come from advection-rotation estimates.
	Instead, it comes from applying derivatives to the equation satisfied by $K$.
	For the sake of exposition let us discuss this process under the assumption that $K$ solves the linearized equation
	\begin{equation*}
		\pdt K = \sbrac{\Omega_{eq}, K} + \sbrac{\Theta, J_{eq}}.
	\end{equation*}
	To control $\pdt^j K$ we apply $j-1$ temporal derivatives to the equation.
	The crux of the argument is this: since $\pdt^{j-1} \theta$ is controlled through the high regularity energy $ \overline{\mathcal{E}}_M $ in the space $H^{2M-2j+2}$,
	we see that this derivative count is below $2M-3$, i.e. $2M-2j+2 \leqslant 2M-3$, precisely when $j\geqslant 3$.
	So indeed this trade-off only concerns the first two temporal derivatives of $K$.

	The practical implication of this trade-off is the following dichotomy between ``good'' and ``bad'' terms.
	If we seek to control $K$ or one of its time derivatives in $H^k$ for $k\leqslant 2M-3$ then we are dealing with a ``good'' term which is bounded in time.
	If we seek to control $K$, $\pdt K$, or $\pdt^2 K$ in $H^k$ for $k > 2M-3$ then we are dealing with a ``bad'' term for which the only bound we have is growing in time.

	Note that this distinction is by no means purely academic: nonlinear interaction terms appear that require us to control $K$ (and its temporal derivatives)
	at high regularity, and for example it is critical to be able to control $K$ in $H^{2M+1}$ due to interactions of the form
	$
		\int_{\mathbb{T}^3} ( \partial^\alpha K ) \theta \cdot \partial^\alpha a
	$
	when $\abs{\alpha} = 2M$.
	Since we seek an upper bound of the form $\mathcal{E}^{1/2} \mathcal{D}$ even though $K$ is not in the dissipation and $a$ is only controlled dissipatively up to $H^{2M-1}$
	(this is precisely the manifestation of the lack of coercivity), we must integrate by parts, which requires control of $K$ in $H^{2M+1}$.

	As a concluding note regarding the advection-rotation estimates, it is essential to remember that this control of $K$ is conditioned on the decay of intermediate norms.
	This is precisely what informs how the advection-rotation estimates fit in the overall scheme of a priori estimates.

	\hspace{0.1em}
	\paragraph{\textbf{Closing the energy estimates at the low level.}}
	We continue our discussion of the nonlinear effects and sketch how to close the energy estimates at the low level.
	The key observation here is that we may proceed as we did in the linear case (discussed in \fref{Section}{sec:discuss_linear_analysis} above),
	with two differences.
	Note that the closure of the energy estimates at the low level is done in \fref{Proposition}{prop:close_est_low_level},
	which puts together all the pieces from \fref{Section}{sec:close_at_low}.

	The first difference is that the microinertia appears as a weight in the energy in the energy.
	This is readily addressed by the propagation in time of the spectrum of the microinertia, since then $ \int_{\mathbb{T}^3} J\theta\cdot\theta \asymp \int_{\mathbb{T}^3} \abs{\theta}^2$.
	The second difference s that nonlinear interactions appear on the right-hand side of the energy-dissipation relation of \eqref{eq:discuss_ap_ED_low}.
	As was the case in the linear setting of \fref{Section}{sec:discuss_linear_analysis} we leverage the boundedness of the high level energy,
	which is used here to control these interactions.
	We may then deduce the decay of $ \mathcal{E}_\text{low} $ as in \eqref{eq:discuss_ap_decay_E_low}.

	Crucially, this decay is once again (as was the case in the linear analysis) predicated on the boundedness of the high level energy $ \overline{\mathcal{E}}_M $.
	However, by contrast with the linear case, it is very delicate to ensure that the high level energy remains bounded in the nonlinear setting.
	This is discussed in detail below.

	\hspace{0.1em}
	\paragraph{\textbf{Closing the energy estimates at the high level.}}
	We near the end of our discussion of the a priori estimates and provide a sketch of how to close the energy estimates at the high level,
	noting in particular the difficulties that arise due to the presence of $K$, and describing how to overcome these challenges.
	This is carried out rigorously in \fref{Section}{sec:close_at_high}, leading up to the closure of the energy estimates at the high level in \fref{Proposition}{prop:close_est_high_level}.

	The fundamental principles used to close the energy estimates at the high level are the same as those used to close the estimates at the low level:
	improve the dissipation and control the interactions.
	However, difficulties arise due to the presence of $K$ and the fact that, as discussed above, the only control we have over $K$, $\pdt K$, and $\pdt^2 K$
	at regularity above $2M-3$ is growing in time.

	To be precise, let us write the energy-dissipation relation at the high level as
	\begin{equation}
	\label{eq:discuss_ap_ED_high}
		\frac{\mathrm{d}}{\mathrm{d}t} \overline{\mathcal{E}}_M + \overline{\mathcal{D}}_M = \overline{\mathcal{I}}_M,
	\end{equation}
	where $ \overline{\mathcal{I}}_M$ denotes the interactions.
	Immediately, when improving the dissipation and controlling the interactions, ``bad'' terms from the advection-rotation estimates for $K$ appear.
	Since the upper bound on these bad terms is growing in time, our only hope that their appearance does not break the scheme of a priori estimates
	is that they may be counter-balanced by terms which decay in time.
	The decay of intermediate norms therefore plays an essential role in the closure of the energy estimates at the high level.
	With this careful balancing act in mind, between ``bad'' terms involving $K$ and terms decaying sufficiently fast,
	the estimates establishing the improvement of the dissipation $ \overline{\mathcal{D}}_M$
	and the control of the interactions $ \overline{\mathcal{I}}_M$ can be shown to take the form
	\begin{equation}
	\label{eq:discuss_ap_close_high_first_est}
			\mathcal{D}_M \lesssim \overline{\mathcal{D}}_M + \overline{\mathcal{K}}_2 \mathcal{F}_M
			\text{ and }
			\vbrac{ \overline{\mathcal{I}}_M } \lesssim \mathcal{E}_M^{1/2} \mathcal{D}_M + \mathcal{K}_\text{low}^{1/2} \mathcal{F}_M^{1/2} \mathcal{D}_M^{1/2}
	\end{equation}
	where $ \mathcal{K}_\text{low}$ contains all the terms whose decay is needed to counteract the potential growth of $ \mathcal{F}_M $.

	A particular subtlety, which is worth pointing out, arises when identifying $ \mathcal{K}_\text{low}$.
	Indeed, it turns out that
	$
		\mathcal{K}_\text{low} = \overline{\mathcal{K}}_2 + \normtyp{ \pdt^2 \theta }{L}{2}^2,
	$
	where the appearance of $\pdt^2 \theta$ is ultimately due to the interaction with the commutator with $J\pdt$.
	The first term, $ \overline{\mathcal{K}}_2$, is immediately known to decay from the decay of intermediate norms discussed at the beginning of \fref{Section}{sec:discuss_nonlinear_effects}.
	The decay of $\pdt^2 \theta$ is not quite so immediate since the norms in $ \overline{\mathcal{K}}_I $ only involve one temporal derivative.
	We are therefore required to perform an auxiliary estimate for $\pdt^2 \theta$,
	which hinges on the structure of the equation governing the dynamics of $\theta$ and the propagation in time of the spectrum of the microinertia $J$,
	to establish that $\pdt^2 \theta$ decays when $ \overline{\mathcal{K}}_2$ decays.

	Having established \eqref{eq:discuss_ap_close_high_first_est}, the heuristic which guides our next step is, as discussed above,
	that the decay of $ \mathcal{K}_\text{low}$ will balance out the potential growth of $ \mathcal{F}_M $.
	It turns out that establishing such a result rigorously can only be carried out in a time-integrated fashion.
	We end up proving that
	\begin{equation}
	\label{eq:discuss_ap_close_high_second_est}
			\int_0^t \overline{\mathcal{K}}_2 \mathcal{F}_M \lesssim \alpha \brac{ 1 + \int_0^t \overline{\mathcal{D}}_M } \text{ and }
			\int_0^t \mathcal{K}_\text{low}^{1/2} \mathcal{F}_M^{1/2} \mathcal{D}_M^{1/2} \lesssim \alpha \int_0^t \overline{\mathcal{D}}_M
	\end{equation}
	where $\alpha > 0$, which depends on the initial conditions and the decay of intermediate norms, can be made small.
	Crucially, the estimates above are obtained by leveraging the control of $ \mathcal{F}_M $ afforded to us by the advection-rotation estimates for $K$.
	This shows that closing the energy estimates at the high level is a delicate affair which relies directly on two of the other three pieces of our scheme of a priori estimates:
	the decay of intermediate norms and the advection-rotation estimates for $K$.

	To conclude it suffices to combine \eqref{eq:discuss_ap_close_high_first_est} and \eqref{eq:discuss_ap_close_high_second_est} with the energy-dissipation relation
	\eqref{eq:discuss_ap_ED_high}, from which we deduce the boundedness of the high level energy $ \overline{\mathcal{E}}_M $.

	\hspace{0.1em}
	\paragraph{\textbf{Synthesis.}}
	We conclude our discussion of the a priori estimates with a brief note on how to put all the pieces together.
	Each of the four pieces of the a priori estimates, namely
	closing the energy estimates at the low regularity,
	closing the energy estimates at the high regularity,
	deriving advection-rotation estimates for $K$, and
	obtaining the decay of intermediate norms,
	depends on one or more of the other.
	A careful assembly is therefore required to ensure that the argument does not end up being circular.
	This is summarized pictorially in \fref{Figure}{fig:a_priori_discuss} and done carefully in \fref{Section}{sec:ap_synthesis},
	culminating in the main a priori estimates result recorded in \fref{Theorem}{thm:a_priori}.

	The key insight is to kick off the scheme of a priori estimates by assuming the smallness of the solution.
	From there we can take two passes at the estimates:
	in the first pass we use the smallness assumption to ensure that all the pieces of our scheme of a priori estimates are in play,
	and in the second pass we obtain structured estimates where the smallness parameter disappears from the estimates
	and all the estimates obtained are in terms of the initial data.

	\begin{figure}
	\begin{center}
	\begin{tikzpicture}
		% Define the styles
		\tikzstyle{sblock} = [rectangle, draw, fill=RoyalBlue!20, text width = 6cm, text centered, minimum height = 1cm];
		\tikzstyle{block} = [rectangle, draw, fill=RoyalBlue!20, text width = 6cm, text centered, minimum height = 2cm];
		\tikzstyle{arrow} = [draw, ->, thick];
		% Boxes
		\node [sblock]		(smallAssum)	at (-4.5,  3.5)	{Smallness\\ assumption on $ \overline{\mathcal{E}}_M $};
		\node [block]		(closeLowLvl)	at (-4.5,  0)	{
			\textbf{Closing at the low level}\\
			Boundedness of $ \overline{\mathcal{E}}_M $\\
			$\implies$ Decay of $ \mathcal{E}_\text{low}$
		};
		\node [block]		(decayIntNorms)	at (-4.5, -3.5)	{
			\textbf{Decay of intermediate norms}\\
			Boundedness of $ \overline{\mathcal{E}}_M $\\
			\& decay of $ \mathcal{E}_\text{low} $\\
			$\implies$ Decay of intermediate norms
		};
		\node [block]		(advRotEq)	at ( 4.5, -3.5)	{
			\textbf{Advection-rotation equations}\\
			Decay of intermediate norms\\
			$\implies$ control of $ \mathcal{F}_M $
		};
		\node [block]		(closeHighLvl)	at ( 4.5, 0)	{
			\textbf{Closing at the high level}\\
			Decay of intermediate norms\\
			\& control of $ \mathcal{F}_M $\\
			$\implies$ boundedness of $ \overline{\mathcal{E}}_M $
		};
		% Arrows
		\draw [arrow, dashed, color=Thistle]	(smallAssum.270-57)	--		node [left]						{1}	(closeLowLvl.90+40);
		\draw [arrow, dashed, color=Thistle]	(closeLowLvl.270-40)	--		node [left]						{2}	(decayIntNorms.90+40);
		\draw [arrow, color=ForestGreen]	(decayIntNorms.270-40)	-- + (0, -1) -|	node [below, xshift = -4.5cm]				{3}	(advRotEq.270);
		\draw [arrow, color=ForestGreen]	(advRotEq.90)		--		node [right]						{4}	(closeHighLvl.270);
		\draw [arrow, color=ForestGreen]	(closeHighLvl.90)	-- + (0,  1) -| node [above, xshift = 3.8cm]				{5}	(closeLowLvl.90-40);
		\draw [arrow, color=ForestGreen]	(closeLowLvl.270+40)	--		node [right]						{6}	(decayIntNorms.90-40);
	\end{tikzpicture}
	\end{center}
	\caption{
		Pictorial summary of how the various pieces of the a priori estimates depend on one another.
		The arrows indicate the steps taken to close our scheme of a priori estimates -- c.f. \fref{Theorem}{thm:a_priori} for details.
		In the first pass all estimates obtained are in terms of the smallness parameter, this is indicated by the pink dashed arrows.
		In the second pass all estimates obtained are in terms of the initial conditions, this is indicated by the green solid arrows.
		Note that the seventh step of \fref{Theorem}{thm:a_priori} is omitted here since
		it plays an essential role in the propagation of the estimates over time, but is not essential for their closure.
	}
	\label{fig:a_priori_discuss}
	\end{figure}
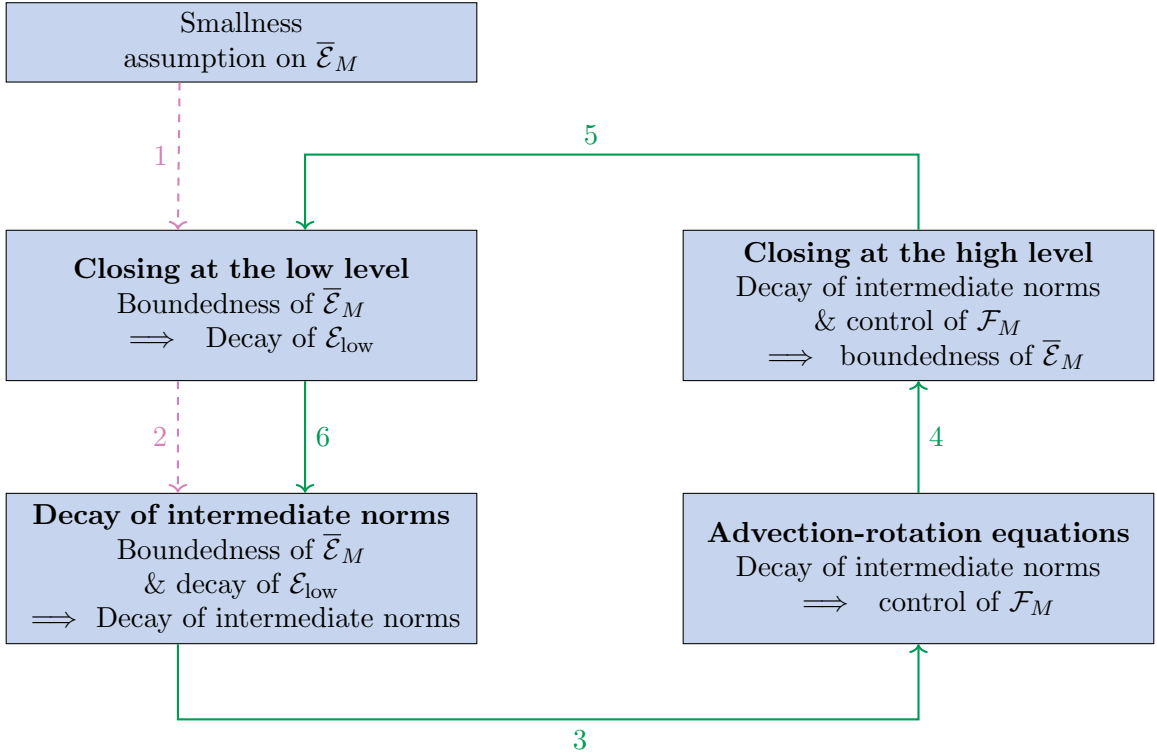

\subsection{Local well-posedness}
\label{sec:discuss_lwp}

	In this section we discuss the local well-posedness. In a nutshell, the key question is:
	``how much of the nonlinear structure do we keep in order to be able to obtain good estimates on the sequence of approximate solutions''.
	The local-posedness theory is developed in \fref{Section}{sec:lwp}, whose main take-away is \fref{Theorem}{thm:lwp}.

	\hspace{0.1em}
	\paragraph{\textbf{Strategy.}}
	We will produce solutions locally-in-time via a Galerkin scheme.
	We will
	(1) solve a sequence of approximate problems on finite-dimensional subspaces of the solution space,
	(2) obtain uniform estimates on the sequence of approximate solutions, and
	(3) pass to the limit by compactness.
	Since the domain is the torus, it is natural to approximate by cutting off at the first $n$ Fourier modes.
	More precisely: writing $W_n \subseteq L^2$ such that $f \in W_n$ if and only if $\hat{f}(k) = 0$ for all $\abs{k} > n$,
	we are looking to solve the approximate problem
	\begin{equation}
	\label{eq:lwp_discuss_A}
		\widetilde{T}_n (K) \pdt Z = L Z + P_n \mathcal{N} (Z)
		\text{ for } Z = (u, \theta, K) \in W_n
	\end{equation}
	where $L$ is a linear operator with constant coefficients, $\mathcal{N}$ accounts for the nonlinearities, and $\widetilde{T}_n (K)$ is an appropriate approximation of
	$I_3 \oplus (J_{eq} + K) \oplus I_{3\times 3}$, namely
	$
		\widetilde{T}_n (K) \vcentcolon= I_3 \oplus (J_{eq} + P_n \circ K) \oplus I_{3\times 3}
	$
	where $(P_n \circ K)\theta \vcentcolon= P_n (K\theta)$ for every $\theta \in L^2$,
	for $P_n$ denoting the $L^2$-orthogonal projection onto $W_n$.

	\hspace{0.1em}
	\paragraph{\textbf{A subtle point.}}
	Due to the presence of $\widetilde{T}_n (K)$ we will need to invert $J_{eq} + P_n \circ K$.
	Whilst fairly straightforward to do, this must nonetheless be done carefully since we are no longer merely inverting the matrix $J_{eq} + K$ pointwise,
	but rather we are inverting the \emph{operator} $J_{eq} + P_n \circ K$ as an operator from $W_n$ to itself.
	The corresponding results are recorded in \fref{Section}{sec:inv_T_K}, where we also obtain $H^k$-to-$H^k$ bounds on ${ \widetilde{T}_n (K) }\inv$.

	\hspace{0.1em}
	\paragraph{\textbf{Nonlinear structure}}
	Constructing a sequence of approximate solutions solving \eqref{eq:lwp_discuss_A} is easy,
	however we run into trouble when looking for estimates of the approximate solutions.
	The issue is that in \eqref{eq:lwp_discuss_A} we have stripped away the nonlinear structure of the problem which helps us by providing good energy estimates.

	To make this idea precise let us compare the two systems below.
	Both systems are cartoon versions of \eqref{eq:lwp_discuss_A} where we neglect the velocity $u$,
	dismiss the dissipative contributions, and omit the projection $P_n$.
	Note that we write $J = J_{eq} + K$ and $\omega = \omega_{eq} + \theta$.
	We consider
	\begin{equation}
	\label{eq:lwp_discuss_cartoon}
		(1)	\left\{
			\begin{aligned}
				&J \pdt \theta = f_1,\\
				&\pdt K = F_2
			\end{aligned}
			\right.
		\;\text{ and }\;
		(2)	\left\{
			\begin{aligned}
				&(J(\pdt + u\cdot\nabla) + \omega\times J)\pdt\theta = f_1,\\
				&(\pdt + u\cdot\nabla)K = \sbrac{\Omega, J}.
			\end{aligned}
			\right.
	\end{equation}
	The energy associated with both systems is
	\begin{equation*}
		\frac{1}{2} \int_{\mathbb{T}^3}  J \theta\cdot\theta + \frac{1}{2} \int_{\mathbb{T}^3}  \abs{K}^2,
	\end{equation*}
	however the interaction terms differ.
	To be precise, the issue is this: when taking $\alpha$ many derivatives the first system gives rise to an interaction of the form
	$
		I^\alpha = \int_{\mathbb{T}^3} \partial^\alpha (\sbrac{\Theta, K}) \partial^\alpha \theta \cdot \partial^\alpha \theta.
	$
	However, the only grants us control of $\theta$ and $K$ in $H^{\abs{\alpha}}$, which is not sufficient to control $I^\alpha$.
	Crucially, this interaction is not present when performing energy estimates with the second system.
	The moral of the story is that some nonlinear structure is optional while some is not.
	In particular, it is essential to keep the full nonlinear advection-rotation equation satisfied by $K$.

	\hspace{0.1em}
	\paragraph{\textbf{A final wrinkle.}}
	In the cartoon \eqref{eq:lwp_discuss_cartoon} above we brazenly dismissed any mention of the projection $P_n$.
	Of course, since we seek to frame the approximate problem as an ODE on the finite-dimensional space where only finitely many Fourier modes are nonzero
	and since that space is not closed under multiplication, the nonlinearities of \eqref{eq:lwp_discuss_cartoon} will require the presence of projections.
	However, this must be done carefully.
	Due to the fact that some nonlinear structure must be kept in the approximate problem (as discussed above),
	it turns that we need to approximate $K$ by using (schematically) twice as many Fourier modes as are used for the velocities $u$ and $\theta$.

	\hspace{0.1em}
	\paragraph{\textbf{Discrepancy in the energies.}}
	It is important here to note that the energies of the local well-posedness differ from those of the main scheme of a priori estimates.
	Schematically, these energies are of the form
	\begin{equation}
	\label{eq:discrep_energies_def}
		E_{loc} \sim \normtyp{(u, \theta, K)}{H}{2M}^2
		\text{ and }
		E_{ap} \sim \normtyp{(u, \theta, a)}{H}{2M}^2 + \normtyp{K}{H}{2M-3}^2
	\end{equation}
	respectively, where for simplicity we have omitted norms involving temporal derivatives.

	In order to explain this discrepancy recall that, as discussed in \fref{Section}{sec:discuss_nonlinear_effects} above,
	energy estimates are not sufficient to close the a priori estimates due to the absence of $K$ from the dissipation
	and the appearance of interactions terms as in \eqref{eq:discuss_bad_interaction}.
	We are thus led to employing advection-rotation estimates to control $K$ in the scheme of a priori estimates,
	which means that we control $K$ in $H^{2M-3}$ and not $H^{2M}$ (as would be the case employing energy estimates).

	However, when it comes to the local well-posedness theory, the interaction $I^\alpha$ of \eqref{eq:discuss_bad_interaction} is harmless
	since it can be estimated as $\abs{I^\alpha} \lesssim E_{loc}^3$.
	Such an estimate would be fatal for the a priori estimates since it cannot be absorbed into the dissipation
	but it is harmless locally-in-time since it is amenable to a nonlinear Gronwall argument.

	The consequence of this discrepancy is that some additional work is required in order to ensure that the a priori estimates and the local well-posedness theory
	``glue'' together nicely. This is discussed in \fref{Section}{sec:discuss_cont} below.

\subsection{Continuation argument}
\label{sec:discuss_cont}

	In this section we discuss the continuation argument whose purpose is to allow us to glue together the a priori estimates and the local well-posedness theory.
	This gluing is nontrivial, in the sense that it requires a new set of estimates,
	precisely because of the mismatch between the energies used for the local well-posedness and the energies used for the a priori estimates
	(as discussed at the end of \fref{Section}{sec:discuss_lwp} above).
	The gluing is carried out in \fref{Section}{sec:cont}, where the key continuation argument it leads to is recorded in \fref{Theorem}{thm:cont}.

	In order to justify the necessity of this additional set of estimates let us consider $E_{ap}$ and $E_{loc}$ defined as in \eqref{eq:discrep_energies_def}
	as cartoons of the energies used in the a priori estimates and in the local well-posedness theory, respectively.
	In particular note that for the sake of exposition we have omitted any mention of norms controlling temporal derivatives of the unknowns.
	Let us also consider the following cartoons of the a priori estimates and of the local well-posedness
	(which are simplified to the point of technical inaccuracy, but remain informative nonetheless)
	\begin{align}
		\sup_{0\leqslant t\leqslant T} E_{ap}(t) \leqslant \delta
		\Rightarrow
		\sup_{0\leqslant t\leqslant T} E_{ap} (t) + \frac{\normtyp{K(t)}{H}{2M}^2}{1+t}  \leqslant C_1 E_{ap} (0)
	\label{eq:AP}\tag{AP}\\
		\text{ and }
		\sup_{0\leqslant t\leqslant T} E_{loc} (t) \leqslant \rho\brac{ E_{loc}(0) }
	\label{eq:LWP}\tag{LWP}
	\end{align}
	where $\rho : (0,\infty) \to (0,\infty)$ is a strictly increasing function vanishing asymptotically at zero
	(whose appearance comes from the nonlinear Gronwall argument used in the local well-posedness theory).
	Note that here $ \normtyp{K}{H}{2M}^2 $ is a placeholder for the ``bad'' terms comprising $ \mathcal{F}_M$
	(whose appearance is discussed in detail in \fref{Section}{sec:discuss_nonlinear_effects}).
	To glue the a priori estimates and the local well-posedness theory it suffices to fulfill the following goal.
	\hspace{0.1em}
	\paragraph{\textbf{Goal:}} If \eqref{eq:AP} holds on the time interval $\sbrac{0, T}$, find a sufficiently small timescale $\tau > 0$
	such that \eqref{eq:AP} continues to hold on the interval $\sbrac{0, T+\tau}$.
	\hspace{0.1em}
	\paragraph{\textbf{Difficulty:}}
	For $\tau$ small enough the local well-posedness theory guarantees that we can always continue the solution from $\sbrac{0, T}$ to $\sbrac{0, T+\tau}$.
	The crux of the argument is therefore to ensure that the smallness hypothesis of \eqref{eq:AP} remains satisfied on $\sbrac{0, T+\tau}$.
	However, the growth of the bad term $ \normtyp{K}{H}{2M}^2 $ in \eqref{eq:AP} renders this impossible.
	Indeed, combining \eqref{eq:AP} and \eqref{eq:LWP} tells us that
	\begin{equation*}
		\sup_{T\leqslant t\leqslant T+\tau} E_{ap} (t)
		\leqslant \sup_{T\leqslant t\leqslant T+\tau} C_2 E_{loc} (t)
		\leqslant C_2 \rho\brac{ E_{loc} (T) }
		\leqslant C_2 \rho\brac{ C_3 (1+T) E_{ap} (0) }
	\end{equation*}
	and we cannot guarantee that the right-hand side be small independently of the time horizon $T$.
	\hspace{0.1em}
	\paragraph{\textbf{Solution:}}
	The remedy is to prove an estimate of the form
	\begin{equation*}
		\sup_{T\leqslant t\leqslant T+\tau} E_{ap} (t)
		\leqslant \tilde{\rho}\brac{ E_{ap} (T)}
	\label{eq:E}
	\tag{E}
	\end{equation*}
	for $\tau > 0$ sufficiently small, where $\tilde{\rho}$ is another strictly increasing function which vanishes asymptotically at zero.
	We may then couple \eqref{eq:E} to \eqref{eq:AP} to deduce that
	\begin{equation*}
		\sup_{0\leqslant t\leqslant T+\tau} E_{ap}(t)
		\leqslant \tilde{\rho}\brac{E_{ap}(T)}
		\leqslant \tilde{\rho} \brac{C_1 E_{ap} (0)}
		\leqslant \delta
	\end{equation*}
	provided the initial condition is sufficiently small.
	Note that this estimate is referred to in the sequel as a \emph{reduced energy estimates}
	since it estimates the ``reduced'' unknown $(u, \theta, a)$, by contrast with the ``full'' unknown $(u, \theta, K)$.
	In practice, performing the estimate \eqref{eq:E} relies on the same fundamental estimates as those used to prove \eqref{eq:AP},
	with one fundamental difference: whereas \eqref{eq:AP} relies on the smallness of the energy,
	\eqref{eq:E} relies instead on the smallness of the timescale on which it holds.

\subsection{Global well-posedness and decay}
\label{sec:discuss_gwp}

	In this section we discuss how to put together all the pieces of the puzzle to deduce the main result of \fref{Theorem}{thm:gwp_decay_clean}.
	This is carried out in \fref{Section}{sec:gwp}.
	In a nutshell: the local well-posedness developed in \fref{Section}{sec:lwp} couples to the a priori estimates of \fref{Section}{sec:a_prioris} to produce a solution which
	lives in the small energy regime, at which point the continuation argument recorded in \fref{Section}{sec:cont}
	kicks in to tell us that the solution lives in the small energy regime globally-in-time.

	The only subtlety in this process comes from coupling the local well-posedness theory to the a priori estimates.
	Indeed, the estimates provided by the local well-posedness theory are not quite strong enough to invoke the a priori estimates, due to insufficient control over $K$.
	To close that gap we rely on an auxiliary estimate for $K$, which is recorded in \fref{Lemma}{lemma:local_in_time_aux_est_K}.

%----------------------------------------------------------------------------------------------------
%	NOTATION
%----------------------------------------------------------------------------------------------------

\section{Notation}
\label{sec:notation}

For the reader's convenience we record here the notation used throughout the paper.

Throughout, the unknown $Z$ comprises all perturbative variables, i.e. $Z = (u, \theta, K)$, while $Y = (u, \theta, a)$ comprises all variables that are proved to decay.

The constant $\tilde{\tau}$ is defined to be $\tilde{\tau} = \frac{\tau}{2\kappa}$. It is omnipresent in the paper since it is equal to the magnitude of the angular velocity at equilibrium $\omega_{eq}$.

Now we record some notation from linear algebra.
\begin{itemize}
	\item	For any vectors $a$ and $b$, $a \otimes b$ denotes the matrix acting via $(a \otimes b) v = (b\cdot v)a$ for any vector $v$.
	\item 	For any $v = (v_1, v_2, v_3) \in \R^3$ and any $w\in\R^2$
		$
			\bar{v} \vcentcolon= (v_1, v_2)
		$,$
			\tilde{w} \vcentcolon= (w_1, w_2, 0)
		$,$
			\bar{v}^\perp \vcentcolon= (-v_2, v_1)
		$ and $
			\tilde{w}^\perp \vcentcolon= (-w_2, w_1, 0)
		$.
		In other words: $v \mapsto \bar{v}$ is the projection onto the $e_1 - e_2$ plane, $w\mapsto \tilde{w}$ is its canonical right-inverse,
		and the superscript $\perp$ denotes a $\frac{\pi}{2}$ (counterclockwise) rotation in the $e_1-e_2$ plane.
	\item	$\ten$ and $\vc$ denote the canonical identification of $\R^3$ with the space of anti-symmetric $3$-by-$3$ matrices using the cross product, and vice-versa.
		To be precise: $(\ten a) v = a \times v$ and $(\vc A) \times v = A v$ for any vectors $a, v \in \R^3$ and any $3$-by-$3$ anti-symmetric matrix $A$.
	\item	$\sym (n)$ denotes the space of real symmetric $n$-by-$n$ matrices.
		Moreover, for any matrix $M$, $\sym(M)$ denotes its symmetric part, i.e. $\sym(M) \vcentcolon= \frac{1}{2}(M + M^T)$.
	\item	Given two linear maps $L_1 : V_1 \to W_1$ and $L_2 : V_2 \to W_2$, the linear map $L_1 \oplus L_2 : V_1 \times V_2 \to W_1 \times W_2$
		is defined as $(L_1 \oplus L_2)(v_1, v_2) \vcentcolon= (L_1 v_1, L_2 v_2)$ for every $v_1 \in V_1$ and $v_2 \in V_2$.
\end{itemize}

We now record the various functionals present throughout the paper. In order to do so, we first introduce parabolic norms.
\begin{itemize}
	\item	For any multi-index $\alpha \in\N^{1+3}$ we define the \emph{parabolic count of derivatives} $\abs{\alpha}_P$ to be
		$\abs{\alpha}_P = 2\alpha_0 + \abs{\bar{\alpha}}$, where we have written $\alpha = (\alpha_0, \bar{\alpha}) \in \N \times \N^3$.
	\item	For $i,j,k\in\N$ satisfying $0\leqslant i\leqslant j \leqslant k/2$ we define the parabolic norms
		\begin{equation}
		\label{eq:not_parabolic_norms}
			\normtyp{f}{P}{k}^2 = \sum_{\abs{\alpha}_P \leqslant k} \normtyp{ \partial^\alpha f }{L}{2}^2,\;
			\norm{f}{P^k_j}^2 = \sum_{\substack{ \abs{\alpha}_P \leqslant k \\ \alpha_0 \leqslant j }} \normtyp{ \partial^\alpha f }{L}{2}^2, \text{ and }
			\norm{f}{P^k_{i,j}}^2 = \sum_{\substack{ \abs{\alpha}_P \leqslant k \\ i \leqslant \alpha_0 \leqslant j }} \normtyp{ \partial^\alpha f }{L}{2}^2.
		\end{equation}
\end{itemize}
	First we record some energy and energy-like functionals. For any non-negative integer $M$ we define
	\begin{equation}
	\label{eq:not_E_M_bar}
		\widetilde{\mathcal{E}}_M =
		\sum_{\abs{\alpha}_P \leqslant 2M}
			\frac{1}{2} \int_{\T^3} \abs{ \partial^\alpha u}^2
			+ \frac{1}{2} \int_{\T^3} J \partial^\alpha \theta \cdot \partial^\alpha \theta
			+ \frac{\tilde{\tau}^2}{\nu-\lambda} \frac{1}{2} \int_{\T^3} \abs{ \partial^\alpha a}^2
		\text{ and }
		\overline{\mathcal{E}}_M = \normtyp{(u, \theta, a)}{P}{2M}^2.
	\end{equation}
	In particular when $M = 1$ we define
	\begin{equation}
	\label{eq:not_E_low}
		\widetilde{\mathcal{E}}_\text{low} = \widetilde{\mathcal{E}}_1,\,
		\overline{\mathcal{E}}_\text{low} = \overline{\mathcal{E}}_1, \text{ and }
		\mathcal{E}_\text{low} = \overline{\mathcal{E}}_\text{low} + \normtyp{\pdt a}{H}{1}^2 + \normtypns{ \pdt^2 a }{L}{2}^2.
	\end{equation}
	When $M\geqslant 3$ we define
	\begin{align}
		\mathcal{E}^{(K)}_M = \normtyp{K}{H}{2M-3}^2 + \normtyp{\pdt K}{H}{2M-3}^2 + \normtypns{\pdt^2 K}{H}{2M-3}^2 + \sum_{j=3}^M \normtypns{\pdt^j K}{H}{2M-2i+2}^2,
	\label{eq:not_E_M_K}
	\\
		\mathcal{E}_M = \overline{\mathcal{E}}_M + \mathcal{E}_M^{(K)}
		\text{ and }
		\mathcal{F}_M = \normtyp{K}{H}{2M+1}^2 + \normtyp{\pdt K}{H}{2M}^2 + \normtyp{\pdt^2 K}{H}{2M-2}^2.
	\label{eq:not_E_M_and_F_M}
	\end{align}
	We also define the intermediate energy functionals
	\begin{equation}
	\label{eq:not_K_bar}
		\overline{\mathcal{K}}_I
		= \norm{(u, \theta, a)}{P^{2I}_{1}}^2
		\text{ and } 
		\mathcal{K}_\text{low} = \overline{\mathcal{K}}_2 + \normtypns{ \pdt^2 \theta }{L}{2}^2.
	\end{equation}
	We now record the dissipation functionals. The dissipation is given by
	\begin{equation}
	\label{eq:not_dissip}
		D(u, \theta)
			= \int_{\T^3} \frac{\mu}{2} \vbrac{\symgrad u}^2
			+ 2\kappa \vbrac{ \frac{1}{2} \nabla\times u - \theta}^2
			+ \alpha\vbrac{\nabla\cdot\theta}^2
			+ \frac{\beta}{2} \vbrac{\symgrad^0 \theta}^2
			+ 2\gamma \vbrac{\nabla\times\theta}^2
	\end{equation}
	and we define
	\begin{equation}
	\label{eq:not_D_M}
		\overline{\mathcal{D}}_M = \normtyp{(u,\theta)}{P}{2M+1}^2,\;
		\mathcal{D}_M^a = \sum_{j=0}^3 \normtypns{\pdt^j a}{H}{2M-j-1}^2 + \sum_{j=4}^M \normtypns{\pdt^j a}{H}{2M-2i+3}^2,
		\text{ and }
		\mathcal{D}_M = \overline{\mathcal{D}}_M + \mathcal{D}_M^a.
	\end{equation}
	When $M=1$ we also define
	\begin{equation}
	\label{eq:not_D_low}
		\overline{\mathcal{D}}_\text{low} = \overline{\mathcal{D}}_1 \text{ and }
		\mathcal{D}_\text{low} = \overline{\mathcal{D}}_\text{low} + \normtyp{a}{H}{1}^2 + \normtyp{ \pdt a}{L}{2}^2.
	\end{equation}
	Finally we write the interaction terms as
	\begin{equation}
	\label{eq:not_I}
		\overline{\mathcal{I}}_I \vcentcolon= \sum_{\abs{\alpha}_P \leqslant 2M} \mathcal{I}^\alpha
		\text{ and }
		\overline{\mathcal{I}}_\text{low} \vcentcolon= \overline{\mathcal{I}}_1
	\end{equation}
	for $\mathcal{I}^\alpha$ as in \fref{Lemma}{lemma:record_form_interactions}.

%----------------------------------------------------------------------------------------------------
%	A PRIORIS
%----------------------------------------------------------------------------------------------------

\section{A priori estimates}
\label{sec:a_prioris}

In this section we develop the scheme of a priori estimates central to the stability result proven in this paper.
We begin with advection-rotation estimates for $K$ in \fref{Section}{sec:adv_rot_est_K} and then turn our attention to energy estimates in \fref{Sections}{sec:ED_structure}-\ref{sec:close_at_high}.
More precisely: in \fref{Section}{sec:ED_structure} we identify of the energy-dissipation structure of the problem and use it in \fref{Sections}{sec:close_at_low} and \ref{sec:close_at_high}
to close the energy estimates at the low and high level, respectively.
We then record the interpolation result giving us decay of intermediate norms provided both the low and high level energies are controlled in \fref{Section}{sec:decay_int_norms}.
We conclude this section by putting all the pieces of the scheme of a priori estimates together in \fref{Section}{sec:ap_synthesis}.

\subsection{Advection-rotation estimates for $K$}
\label{sec:adv_rot_est_K}

	In this section we record the advection-rotation estimates we may derive for $K$ based on the advection-rotation equation \eqref{eq:pertub_sys_no_ten_pdt_K}.
	The culmination of this section is \fref{Proposition}{prop:est_K}, which synthesizes the estimates obtained in this section.
	We begin with $L^p$ estimates for the advection-rotation operator encountered in \eqref{eq:pertub_sys_no_ten_pdt_K}
	which are foundational for all other advection-rotation estimates obtained here.

	\begin{prop}[$L^p$ estimates for advection-rotation equations]
	\label{prop:Lp_est_adv_rot_eq_diff_form}
		Let $T>0$ be a finite time horizon and let $1\leqslant p < \infty$.
		Let $v$ be a continuously differentiable vector field on $\cobrac{0,T}\times\T^n$,
		let $M$ be a continuous matrix field on $\cobrac{0,T}\times\T^n$,
		and let $F \in L^\infty\brac{ \cobrac{0,T};\, L^p\brac{ \T^n;\, \R^{n\times n}}}$.
		If $S \in L^\infty\brac{ \cobrac{0,T};\, L^p\brac{ \T^n;\, \sym(n)}}$ is a distributional solution of
		\begin{equation*}
			\brac{\pdt + u\cdot\nabla - \sbrac{M,\,\cdot\,}} S = F \text{ on } (0,T)\times \T^n \text{ and }
			S(t=0) = S_0
		\end{equation*}
		for some $S_0 \in L^p$ then it satisfies the estimate
		\begin{equation*}
			\norm{S}{L^\infty_T L^p}
			\leqslant \exp\brac{\int_0^t \frac{1}{p} \normtyp{(\nabla\cdot v)(s)}{L}{\infty} ds} \normtyp{S_0}{L}{p}
			+ \int_0^t \exp\brac{\int_0^s \frac{1}{p} \normtyp{(\nabla\cdot v)(r)}{L}{\infty} dr } \normtyp{\sym(F)(s)}{L}{p} ds.
		\end{equation*}
	\end{prop}
	\begin{proof}
		The fundamental idea behind this estimate is the following formal computation.
		First we compute, in light of \fref{Lemma}{lemma:commut_are_antisym_maps_on_space_sym_matrices}, that
		\begin{equation*}
			\Dt\norm{S}{L^p}
			= \Dt {\brac{ \int_{\T^n} \abs{S}^p }}^{\frac{1}{p}}
			= \frac{1}{p} {\brac{ \int_{\T^n} \abs{S}^p }}^{\frac{1}{p} - 1} \brac{ \int_{\T^n} p \abs{S}^{p-2} S:F - \int_{\T^n} p \abs{S}^{p-2} S:\brac{u\cdot\nabla}S }.
		\end{equation*}
		Now observe that on one hand
		\begin{equation*}
			-p\int_{\T^n} \abs{S}^{p-2} S : \brac{u\cdot\nabla} S
			= -\int_{\T^n} \brac{u\cdot\nabla} \abs{S}^p
			= \int_{\T^n} \brac{\nabla\cdot u} \abs{S}^p
		\end{equation*}
		whilst on the other hand, since $p' \brac{1-p} = p$,
		\begin{equation*}
			\int_{\T^n} \abs{S}^{p-2} S:F
			\leqslant \int_{\T^n} \abs{S}^{p-1} \abs{\sym \brac{F}}
			\leqslant {\brac{ \int_{\T^n} \abs{S}^p }}^{\frac{1}{p'}} {\brac{ \int_{\T^n} \abs{\sym\brac{F}}^p }}^{\frac{1}{p}}.
		\end{equation*}
		So finally we deduce that
		\begin{equation*}
			\Dt\norm{S}{L^p}
			\leqslant \frac{1}{p} {\brac{ \int_{\T^n} \abs{S}^p }}^{\frac{1}{p} - 1} \norm{\nabla\cdot u}{L^\infty} \brac{ \int_{\T^n} \abs{S}^p }
			+ {\brac{ \int_{\T^n} \abs{S}^p }}^{\frac{1}{p} - 1} {\brac{ \int_{\T^n} \abs{S}^p }}^{\frac{1}{p'}} \norm{\sym\brac{F}}{L^p}.
		\end{equation*}
		from which the claim would follow upon performing a Gronwall argument.

		To make this computation precise it suffices to use standard approximation techniques from the theory of $L^p$ estimates for transport equations.
		For example we may approximate $s \mapsto \abs{s}^p$ by non-negative $C^1$ functions in a monotone fashion and approximate $S_0$ and $F$ by continuously differentiable functions.
		The computation above then holds rigorously at the level of the approximation, and we may pass to the limit using standard tools from measure theory.
	\end{proof}

	With the $L^p$ estimates above in hand we derive $L^\infty$ bounds on both $K$ and $\nabla K$.
	These bounds are used to control low-order terms appearing later in this section when we seek to parlay the $L^p$ estimates above into $H^k$ estimates for $K$.

	\begin{lemma}[$L^\infty$ estimate for $K$]
	\label{lemma:L_infty_est_K}
		Suppose that $K$ solves \eqref{eq:pertub_sys_no_ten_pdt_K} for some given $u$ and $\theta$. Then it satisfies the estimate
		\begin{equation*}
			\norm{K\brac{t}}{L^\infty} \lesssim \norm{K\brac{0}}{L^\infty} + \int_0^t \normns{\bar{\theta} (s)}{L^\infty}ds.
		\end{equation*}
	\end{lemma}
	\begin{proof}
		Since $\sbrac{\Omega_{eq}, J_{eq}} = 0$ (c.f. \fref{Lemma}{lemma:block_form_Omega_J}) we may write \eqref{eq:pertub_sys_no_ten_pdt_K} as
		$
			\pdt K + u\cdot\nabla K = \sbrac{\Omega_{eq} + \Theta, K} + \sbrac{\Theta, J_{eq}}.
		$
		It then follows from \fref{Proposition}{prop:Lp_est_adv_rot_eq_diff_form} that, for any $1<p<\infty$,
		$
			\norm{K\brac{t}}{L^p} \leqslant \norm{K\brac{0}}{L^p} + \int_0^t \norm{ \sbrac{\Theta\brac{s}, J_{eq}} }{L^p} ds.
		$
		Note that \fref{Lemma}{lemma:block_form_Omega_J} tells us that
		\begin{equation*}
			\sbrac{\Theta, J_{eq}} = - (\nu-\lambda) \begin{pmatrix}
				0 & \bar{\theta}^\perp \\
				{\brac{ \bar{\theta}^\perp }}^T & 0
			\end{pmatrix}
		\end{equation*}
		from which we deduce that the Fr\"{o}benius norm of this commutator is $\vbrac{ \sbrac{ \Theta, J_{eq} } } = \sqrt{2(\nu-\lambda)} \abs{\bar{\theta}}$.
		Since $\norm{\,\cdot\,}{L^p} \leqslant \norm{\,\cdot\,}{L^\infty}$ we may conclude that
		\begin{equation*}
			\norm{K\brac{t}}{L^p}
			\leqslant \norm{K\brac{0}}{L^p} + \int_0^t \sqrt{2\brac{\nu-\lambda}} \norm{\bar{\theta}\brac{s}}{L^\infty} ds
			\lesssim \norm{K\brac{0}}{L^p} + \int_0^t \norm{\bar{\theta} (s)}{L^\infty} ds.
		\end{equation*}
		The claim holds upon taking $p\to\infty$.
	\end{proof}

	\begin{lemma}[$L^\infty$ estimate for $\nabla K$]
	\label{lemma:L_infty_est_nabla_K}
		Suppose that $K$ solves \eqref{eq:pertub_sys_no_ten_pdt_K} for some given $u$ and $\theta$.
		Then $\nabla K$ satisfies the estimate
		\begin{equation*}
			\norm{\nabla K\brac{t}}{L^\infty} \lesssim \exp\brac{\int_0^t \norm{\nabla u \brac{r}}{L^\infty} dr} \brac{
				\norm{\nabla K \brac{0}}{L^\infty} + \int_0^t \brac{1 + \normtyp{ K(s) }{L}{\infty} } \norm{\nabla\theta\brac{s}}{L^\infty} ds
			}.
		\end{equation*}
	\end{lemma}
	\begin{proof}
		Since $K$ solves \eqref{eq:pertub_sys_no_ten_pdt_K} we see that $ \partial_i K$ solves
		\begin{equation*}
			\brac{ \pdt + u\cdot\nabla - \sbrac{\Omega_{eq} + \Theta, \,\cdot\,} } \partial_i K =  \sbrac{\partial_i \Theta, J_{eq} + K} - \partial_i u \cdot\nabla K.
		\end{equation*}
		We note that the $L^p$ norm of the right-hand side can be estimated in the following way:
		\begin{equation*}
			\normtyp{(RHS)}{L}{p}
			\leqslant \normtyp{ \sbrac{\partial_i \Theta, J} }{L}{p} + \normtyp{ \partial_i u \cdot\nabla K}{L}{p}
			\lesssim \normtyp{ \nabla\theta }{L}{p} \brac{1 + \normtyp{ K }{L}{\infty} }  + \normtyp{ \nabla u }{L}{\infty} \normtyp{ \nabla K}{L}{p}.
		\end{equation*}
		\fref{Proposition}{prop:Lp_est_adv_rot_eq_diff_form} therefore tells us that
		\begin{equation*}
			\norm{\nabla K\brac{t}}{L^p}
			\lesssim \exp\brac{ \int_0^t \norm{\nabla u\brac{r}}{L^\infty} dr} \brac{
				\norm{\nabla K\brac{0}}{L^p} + \int_0^t \brac{1 + \normtyp{ K(s) }{L}{\infty} } \norm{\nabla\theta\brac{s}}{L^p} ds
			}
		\end{equation*}
		from which the result follows upon first recalling that $\norm{\,\cdot\,}{L^p} \leqslant \norm{\,\cdot\,}{L^\infty}$ and then taking $p~\to~\infty$.
	\end{proof}

	We now move towards estimates of $K$ and its time derivatives in $H^k$. We begin with estimating $K$.

	\begin{lemma}[$H^k$ estimate for $K$]
	\label{lemma:H_k_est_K}
		Suppose that $K$ solves \eqref{eq:pertub_sys_no_ten_pdt_K} for some given $u$ and $\theta$. Then, for any $k\in\N$, it satisfies the estimate
		\begin{align*}
			\normtyp{K(t)}{H}{k}
			\lesssim \exp\brac{ \int_0^t \normtyp{ \nabla u }{L}{\infty} + \normtyp{ \theta }{L}{\infty} }
			\left(
				\normtyp{K(0)}{H}{k}
				+ \int_0^t \brac{
					1 + \normtyp{ K }{L}{\infty} + \normtyp{ \nabla K }{L}{\infty} 
				} \normtyp{(u, \theta)}{H}{k}
			\right).
		\end{align*}
	\end{lemma}
	\begin{proof}
		Since $K$ solves \eqref{eq:pertub_sys_no_ten_pdt_K} we know that, for any multi-index $\alpha$ with length $\abs{\alpha} = k$, $ \partial^\alpha K $ solves
		\begin{equation*}
			\brac{ \pdt + u\cdot\nabla - \sbrac{\Omega_{eq} + \Theta, \,\cdot\,}} \partial^\alpha K
			= \sbrac{ \partial^\alpha \Theta, J_{eq}} + \sbrac{u\cdot\nabla, \partial^\alpha} K - \sbrac{ \sbrac{\Theta, \,\cdot\,}, \partial^\alpha } K.
		\end{equation*}
		Applying \fref{Lemmas}{lemma:comm_est_transp_op} and \ref{lemma:comm_est_transp_op} then tells us that the right-hand side may be estimate in the following way:
		\begin{align*}
			\normtyp{(RHS)}{L}{2} 
			&\leqslant \norm{ \sbrac{ \partial^\alpha \Theta, J_{eq} }}{L^2}
			+ \norm{ \sbrac{u\cdot\nabla, \partial^\alpha} K}{L^2}
			+ \norm{ \sbrac{\sbrac{\Theta, \,\cdot\,}, \partial^\alpha } K}{L^2}
			\\
			&\lesssim \normns{ \partial^\alpha \bar{\theta}}{L^2}
			+ \brac{ \norm{\nabla u}{L^\infty} + \norm{\theta}{L^\infty} } \norm{K}{H^k}
			+ \brac{ \norm{K}{L^\infty} + \norm{\nabla K}{L^\infty} } \brac{ \norm{u}{H^k} + \norm{\theta}{H^k} }.
		\end{align*}
		Summing over $\abs{\alpha}$ and appealing to \fref{Proposition}{prop:Lp_est_adv_rot_eq_diff_form} then yields the claim.
	\end{proof}

	Once $K$ is under control we can read off estimates on $\pdt K$ from \eqref{eq:pertub_sys_no_ten_pdt_K}. The resulting estimate is recorded below.

	\begin{lemma}[$H^k$ estimates for $\pdt K$]
	\label{lemma:H_k_est_pdt_K}
		Suppose that $K$ solves \eqref{eq:pertub_sys_no_ten_pdt_K} for some given $u$ and $\theta$. Then, for any $k\in\N$, $\pdt K$ satisfies the estimate
		\begin{equation*}
			\normtyp{\pdt K}{H}{k} \lesssim \normtyp{K}{H}{k}
			+ \brac{ \norm{u}{L^\infty} + \norm{\theta}{L^\infty} } \normtyp{K}{H}{k+1}
			+ \brac{ 1 + \norm{K}{L^\infty} + \norm{\nabla K}{L^\infty} } \brac{ \normtyp{u}{H}{k} + \normtyp{\theta}{H}{k} }.
		\end{equation*}
	\end{lemma}
	\begin{proof}
	This follows immediately from using the high-low estimates of \fref{Corollary}{cor:est_interactions_L_2_via_Gagliardo_Nirenberg}
	to estimate the quadratic terms in \eqref{eq:pertub_sys_no_ten_pdt_K}.
	\end{proof}

	We continue establishing estimates on $K$ and its time derivatives by taking a time derivative of \eqref{eq:pertub_sys_no_ten_pdt_K} and thus reading off an estimate for $\pdt^2 K$,
	which is recorded below.

	\begin{lemma}[$H^k$ estimates for $\pdt^2 K$]
	\label{lemma:H_k_est_pdt_2_K}
		Suppose that $K$ solves \eqref{eq:pertub_sys_no_ten_pdt_K} for some given $u$ and $\theta$. Then, for any $k\in\N$, $\pdt K$ satisfies the estimate
		\begin{align*}
			\norm{\pdt^2 K}{H^k} \lesssim \norm{\pdt K}{H^k}
			+ \brac{ 1 + \norm{K}{L^\infty} + \norm{\nabla K}{L^\infty} + \norm{\pdt K}{L^\infty} + \norm{\nabla\pdt K}{L^\infty} }
			\brac{ \norm{\brac{u, \theta}}{H^k} + \norm{\pdt\brac{u, \theta}}{H^k} }
			\\
			+ \brac{ \norm{\brac{u, \theta}}{L^\infty} + \norm{\pdt\brac{u, \theta}}{L^\infty} }
			\brac{ \norm{K}{H^{k+1}} + \norm{\pdt K}{H^{k+1}} }.
		\end{align*}
	\end{lemma}
	\begin{proof}
		As in \fref{Lemma}{lemma:H_k_est_pdt_K} this follows from the high-low estimates of \fref{Corollary}{cor:est_interactions_L_2_via_Gagliardo_Nirenberg} upon noticing that $\pdt^2 K$ solves
		\begin{equation*}
			\pdt^2 K = \sbrac{\Omega_{eq}, \pdt K} + \sbrac{\pdt\Theta, J_{eq}} + \sbrac{\pdt\Theta, K} + \sbrac{\Theta, \pdt K}
			- \pdt u \cdot \nabla K - u\cdot\nabla \pdt K. \qedhere
		\end{equation*}
	\end{proof}
	
	We conclude our sequence of estimates on $K$ and its temporal derivatives with an estimate on $K$ when an arbitrary number of temporal derivatives are applied.

	\begin{lemma}[$H^k$ estimates for $\pdt^j K$]
	\label{lemma:H_k_est_pdt_j_K}
		Suppose that $K$ solves \eqref{eq:pertub_sys_no_ten_pdt_K} for some given $u$ and $\theta$. Then, for any $k\in\N$ with $k > \frac{n}{2}$ and any $j \geqslant 1$,
		\begin{equation*}
			\normns{\pdt^j K}{H^k} \lesssim \normns{\pdt^{j-1} K}{H^k} + \normns{\pdt^{j-1}\theta}{H^k}
			+ \sum_{l=0}^{j-1} \brac{
				\normns{ \pdt^l u }{H^k}^2 + \normns{ \pdt^l \theta }{H^k}^2 + \normns{\pdt^l K}{H^{k+1}}^2
			}.
		\end{equation*}
	\end{lemma}
	\begin{proof}
		The proof is immediate upon noting that taking $j-1$ time derivatives of \eqref{eq:pertub_sys_no_ten_pdt_K} results in
		\begin{equation*}
			\pdt^j K = \sbrac{\Omega_{eq}, \pdt^{j-1} K} + \sbrac{\pdt^{j-1} \Theta, J_{eq}}
			+ \sum_{l=0}^{j-1} \brac{
				\sbrac{ \pdt^{\brac{j-1}-l} \Theta, \pdt^l K} - \brac{ \pdt^{\brac{j-1}-l} u \cdot\nabla } \pdt^l K
			}
		\end{equation*}
		and recalling that $H^k$ is a Banach algebra precisely when $k > \frac{n}{2}$.
	\end{proof}

	Having obtained estimates for $K$ and its time derivatives above, we may now synthesize the results of this section in \fref{Proposition}{prop:est_K}.
	Note that, as discussed in \fref{Section}{sec:discuss_ap}, this proposition is one of the four building blocks of the scheme of a priori estimates.
	Recall that the functionals $ \overline{\mathcal{E}}_M$, $ \mathcal{E}_M^{(K)}$, $ \mathcal{F}_M$, $ \overline{\mathcal{K}}_I $, and $ \overline{\mathcal{D}}_M$
	are defined in \eqref{eq:not_E_M_bar}, \eqref{eq:not_E_M_K}, \eqref{eq:not_E_M_and_F_M}, \eqref{eq:not_K_bar}, and \eqref{eq:not_D_M}, respectively.

	\begin{prop}[Advection-rotation estimates for $K$]
	\label{prop:est_K}
		Let $M\geqslant 3$ be an integer and suppose that, for some time horizon $T > 0$ and some universal constant $\overline{C} > 0$,
		\begin{equation}
		\label{eq:est_K_star}
			\sup_{1 \leqslant I \leqslant M} \sup_{0 \leqslant t \leqslant T} \overline{\mathcal{K}}_I (t) {\brac{1+t}}^{2M-2I}
			+ \overline{\mathcal{E}}_M (t) + \int_0^t \overline{\mathcal{D}}_M (s) ds
			=\vcentcolon C_0 \leqslant \overline{C} < \infty.
		\end{equation}
		Then, for every $0 \leqslant t \leqslant T$,
		\begin{align*}
			{\mathcal{E}_M^{(K)} (t)}^{1/2} \lesssim P_f
			\text{ and } 
			\mathcal{F}_M^{1/2} (t) \lesssim \mathcal{F}_M^{1/2} (0) + \brac{1 + P_e} \int_0^t \overline{\mathcal{D}}_M^{1/2} (s) ds + (1 + P_f) \overline{\mathcal{K}}_M^{1/2} (t)
		\end{align*}
		where the constants appearing in these two estimates depends on $\overline{C}$, and where
		\begin{equation*}
			P_e \vcentcolon= P\brac{ C_0^{1/2},\, {\mathcal{E}_M^{(K)} (0)}^{1/2} }
			\text{ and }
			P_f \vcentcolon= P\brac{ C_0^{1/2},\, { \mathcal{E}_M^{(K)}  (0)}^{1/2},\, {\mathcal{F}_M (0)}^{1/2} }
		\end{equation*}
		for $P$, which may differ in each instance,
		a polynomial with non-negative coefficients which vanishes at zero.
	\end{prop}
	\begin{proof}
		The strategy of the proof is as follows.
		The target estimates on $ \mathcal{E}_M^{(K)}$ and $ \mathcal{F}_M $ follow from putting together the advection-rotation estimates of \fref{Section}{sec:adv_rot_est_K}
		and appropriately leveraging the decay afforded to us by \eqref{eq:est_K_star}.
		The key idea is that the potential growth of terms controlled by $ \mathcal{F}_M $ may be offset by the decay of terms appearing in $ \overline{\mathcal{K}}_I $.

		We begin by recording, in Step 1, elementary estimates which are consequences of \eqref{eq:est_K_star} and can be used to control some of the time integrals appearing
		in the advection-rotation estimates of \fref{Section}{sec:adv_rot_est_K}.
		We then obtain $L^\infty$ bounds on $K$ and $\nabla K$ in Step 2 and deduce estimates on $K$, $\pdt K$, $\pdt^2 K$, and higher-order temporal derivatives in Steps 3--6.
		We conclude in Step 7 by recording how to perform the synthesis of Steps 1-6 and read off the desired estimates of $ \mathcal{E}_M $ and $ \mathcal{F}_M $.

		Before we begin the proof in earnest, we fix some notation. For $x_1,\, \dots\, x_n \geqslant 0$, $P(x_1,\, \dots,\, x_n)$ denotes a polynomial
		of $(x_1,\, \dots,\, x_n)$ which may change from line to line and has the following properties: it vanishes at zero and it has non-negative coefficients.
		In particular, we write
		\begin{equation*}
			P_e \vcentcolon= P\brac{ C_0^{1/2},\, {\mathcal{E}_M^{(K)} (0)}^{1/2} }
			\text{ and } 
			P_f \vcentcolon= P\brac{ C_0^{1/2},\, { \mathcal{E}_M^{(K)}  (0)}^{1/2} ,\, {\mathcal{F}_M (0)}^{1/2}}.
		\end{equation*}

		\paragraph{\textbf{Step 1: Preliminary estimates.}}
		We begin by recording some elementary estimates which are consequences of \eqref{eq:est_K_star},
		such as estimates on time integrals of the functionals $ \overline{\mathcal{K}}_I $.
		First, note that for any $1\leqslant I \leqslant M - 2$,
		\begin{equation*}
		\label{eq:est_K_sq_1}
			\int_0^t \normtyp{(u, \theta) (s)}{H}{2I} ds
			\lesssim \int_0^t \frac{C_0^{1/2}}{ {\brac{1+s}}^{M - I} } ds.
			\lesssim C_0^{1/2}.
		\end{equation*}
		By interpolation, we note that similar estimates also hold for $H^k$ norms of $u$ and $\theta$ when $k$ is odd.
		Indeed, observe first that for any odd $k$ satisfying $3 \leqslant k \leqslant 2M-1$,
		if we write $k = 2I+1$ for some $1 \leqslant I \leqslant M-1$ then we have the following bounds, pointwise-in-time,
		\begin{equation*}
			\normtyp{(u, \theta)(t)}{H}{k}
			\lesssim \normtyp{(u, \theta)(t)}{H}{k-1}^{1/2} \normtyp{(u, \theta)(t)}{H}{k+1}^{1/2}
			\lesssim \frac{ C_0^{1/4} }{ {\brac{1+t}}^{\frac{M-I}{2}} } \frac{ C_0^{1/4} }{ {\brac{1+t}}^{\frac{M-I-1}{2}} }
			= \frac{ C_0^{1/2} }{ {\brac{1+t}}^{M - \frac{k}{2}} }.
		\end{equation*}
		Therefore, if $k \leqslant 2M-3$ we may deduce the time-integrated bound
		\begin{equation}
		\label{eq:est_K_sq_2}
			\int_0^t \normtyp{(u, \theta) (s)}{H}{k} ds \lesssim C_0^{1/2}.
		\end{equation}
		Since the functionals $ \overline{\mathcal{K}}_I $ which appear in \eqref{eq:est_K_star} also involve temporal derivatives, we may proceed in the same way to deduce that,
		for any $2 \leqslant k \leqslant 2M-5$,
		\begin{equation}
		\label{eq:est_K_sq_3}
			\int_0^t \normtyp{\pdt(u,\theta) (s)}{H}{k} \lesssim C_0^{1/2}.
		\end{equation}
		The final preliminary estimate, before we being estimating $K$, has to do with exponential factors that arise in
		\fref{Lemmas}{lemma:L_infty_est_K} and \ref{lemma:L_infty_est_nabla_K}.
		In light of \eqref{eq:est_K_sq_1} and \eqref{eq:est_K_sq_2} we see that, for some constants $C > 0$ which may change from line to line
		\begin{equation}
		\label{eq:est_K_sq_4}
			\exp\brac{ \int_0^t \normtyp{ (\nabla u, \theta) (s) }{L}{\infty} ds }
			\leqslant \exp\brac{ C \int_0^t \normtyp{(u,\theta)}{H}{3} }
			\leqslant \exp \brac{ C_2 C_0^{1/2} }
			\lesssim 1
		\end{equation}
		where recall that, as in the statement of the proposition, the constants implied by the notation ``$\lesssim$'' may depend on $\overline{C}$.

		\paragraph{\textbf{Step 2: $L^\infty$ estimates on $K$ and $\nabla K$}}
		We are now ready to record the first estimates on $K$, which are $L^\infty$ estimates on $K$ and $\nabla K$ coming from
		\fref{Lemmas}{lemma:L_infty_est_K} and \ref{lemma:L_infty_est_nabla_K}.
		We deduce from \fref{Lemma}{lemma:L_infty_est_K}, the fact that $M\geqslant 3$, and \eqref{eq:est_K_sq_1}, that
		\begin{equation}
		\label{eq:est_K_a}
			\normtyp{K}{L}{\infty}
			\lesssim \normtyp{K(0)}{H}{2} + \int_0^t \normtyp{\theta}{H}{2}
			\lesssim {\mathcal{E}_M^{(K)} (0)}^{1/2} + C_0^{1/2}
			\lesssim P_e.
		\end{equation}
		Similarly, we deduce from \fref{Lemma}{lemma:L_infty_est_nabla_K}, the fact that $M\geqslant 3$, \eqref{eq:est_K_sq_2}, \eqref{eq:est_K_sq_4}, and \eqref{eq:est_K_a} that
		\begin{equation}
		\label{eq:est_K_b}
			\normtyp{\nabla K }{L}{\infty}
			\lesssim \normtyp{K(0)}{H}{3} + (1 + P_e) \int_0^t \normtyp{\theta}{H}{3} 
			\lesssim {\mathcal{E}_M^{(K)} (0)}^{1/2} + (1 + P_e) C_0^{1/2}
			\lesssim P_e.
		\end{equation}

		\paragraph{\textbf{Step 3: Estimating $K$.}}
		We are now ready to use \fref{Lemma}{lemma:H_k_est_K} to estimate $K$.
		Combining \fref{Lemma}{lemma:H_k_est_K} with \eqref{eq:est_K_sq_4}, \eqref{eq:est_K_a}, and \eqref{eq:est_K_b} tells us that
		\begin{equation*}
			\normtyp{K(t)}{H}{2M-3}
			\lesssim \normtyp{K(0)}{H}{2M-3} + (1 + P_e) \int_0^t \normtyp{(u, \theta)}{H}{2M-3}.
		\end{equation*}
		Using \eqref{eq:est_K_sq_2} allows to conclude that
		\begin{equation}
		\label{eq:est_K_c1}
			\normtyp{K}{H}{2M-3}
			\lesssim {\mathcal{E}_M^{(K)} (0)}^{1/2} + (1 + P_e) C_0^{1/2}
			\lesssim P_e.
		\end{equation}
		Combining \fref{Lemma}{lemma:H_k_est_K} with \eqref{eq:est_K_sq_4}, \eqref{eq:est_K_a}, and \eqref{eq:est_K_b} also tells us that
		\begin{equation}
		\label{eq:est_K_c2}
			\normtyp{K}{H}{2M+1}
			\lesssim \normtyp{K(0)}{H}{2M+1} + (1 + P_e) \int_0^t \normtyp{(u, \theta)}{H}{2M+1} 
			\lesssim { \mathcal{F}_M (0) }^{1/2} + (1 + P_e) \int_0^t \overline{\mathcal{D}}_M^{1/2}.
		\end{equation}

		\paragraph{\textbf{Step 4: Estimating $\pdt K$.}}
		We now estimate $\pdt K$, using \fref{Lemma}{lemma:H_k_est_pdt_K}.
		\fref{Lemma}{lemma:H_k_est_pdt_K}, combined with \eqref{eq:est_K_star}, \eqref{eq:est_K_a}, \eqref{eq:est_K_b}, and \eqref{eq:est_K_c1}, tells us that
		\begin{align}
			\normtyp{\pdt K}{H}{2M-3}
			\lesssim \normtyp{K}{H}{2M-3} + \normtyp{ (u, \theta) }{L}{\infty} \normtyp{K}{H}{2M-2} + (1 + P_e) \normtyp{(u,\theta)}{H}{2M-3} 
		\nonumber
		\\
			\lesssim P_e + \normtyp{ (u,\theta) }{L}{\infty} \normtyp{K}{H}{2M-2}.
		\label{eq:est_K_intermediate_1}
		\end{align}
		The trick now lies in controlling the term $\normtyp{ (u,\theta) }{L}{\infty} \normtyp{K}{H}{2M-3}$ by playing off the decay of $ \normtyp{ (u,\theta) }{L}{\infty}$
		against the (potential) growth of $ \normtyp{K}{H}{2M-2}$.
		Using \eqref{eq:est_K_star} and \eqref{eq:est_K_c2} we see that
		\begin{equation*}
			\normtyp{ (u, \theta) }{L}{\infty} \normtyp{K}{H}{2M-2}
			\lesssim \frac{ C_0^{1/2} }{ {\brac{1+t}}^{M-1} } \brac{ {\mathcal{F}_M (0)}^{1/2} + (1 + P_e) \int_0^t \overline{\mathcal{D}}_M^{1/2} }.
		\end{equation*}
		Note that by applying Cauchy-Schwarz to $\int \overline{\mathcal{D}}_M$, we see that, by virtue of \eqref{eq:est_K_star},
		\begin{equation}
		\label{eq:est_K_D}
			\int_0^t \overline{\mathcal{D}}_M^{1/2}
			= t \fint_0^t \overline{\mathcal{D}}_M^{1/2}
			\leqslant t {\brac{ \fint_0^t \overline{\mathcal{D}}_M }}^{1/2}
			= t^{1/2} {\brac{ \int_0^t \overline{\mathcal{D}}_M }}^{1/2}
			\leqslant C_0^{1/2} t^{1/2}.
		\end{equation}
		Therefore, since $M\geqslant 2$,
		\begin{equation}
		\label{eq:est_K_intermediate_2}
			\normtyp{ (u,\theta) }{L}{\infty} \normtyp{K}{H}{2M-2}
			\lesssim C_0^{1/2} \brac{ { \mathcal{F}_M (0) }^{1/2} + (1 + P_e) C_0^{1/2} } \frac{ 1+t^{1/2} }{ {\brac{1+t}}^{M-1} }
			\lesssim P_f \frac{ {\brac{1+t}}^{1/2} }{ {\brac{1+t}}^{M-1} }
			\lesssim P_f.
		\end{equation}
		So finally, putting \eqref{eq:est_K_intermediate_1} and \eqref{eq:est_K_intermediate_2} together, we see that
		\begin{equation}
		\label{eq:est_K_c3}
			\normtyp{\pdt K}{H}{2M-3} \lesssim P_f.
		\end{equation}

		We now seek to control $\pdt K$ in $H^{2M}$, i.e. through $ \mathcal{F}_M $.
		This is slightly easier than controlling $K$ in $H^{2M-3}$ as is done above since now we do not have to deal with ``decay-growth'' interactions.
		Combining \fref{Lemma}{lemma:H_k_est_pdt_K} with \eqref{eq:est_K_star}, \eqref{eq:est_K_a}, \eqref{eq:est_K_b}, and \eqref{eq:est_K_c2} shows that
		\begin{align}
			\normtyp{\pdt K}{H}{2M}
			\lesssim \normtyp{K}{H}{2M} + \normtyp{ (u,\theta) }{L}{\infty} \normtyp{K}{H}{2M+1} + (1 + P_e) \normtyp{(u,\theta)}{H}{2M} 
		\nonumber
		\\
			\lesssim \brac{1 + C_0^{1/2}} \brac{ {\mathcal{F}_M (0)}^{1/2} + (1 + P_e) \int_0^t \overline{\mathcal{D}}_M^{1/2} } + (1 + P_e) C_0^{1/2}
		\nonumber
		\\
			\lesssim \mathcal{F}_M^{1/2} (0) + (1 + P_e) \int_0^t \overline{\mathcal{D}}_M^{1/2} + (1 + P_e) \overline{\mathcal{K}}_M^{1/2} (t).
		\label{eq:est_K_c4}
		\end{align}

		\paragraph{\textbf{Step 5: Estimating $\pdt^2 K$.}}
		We now use \fref{Lemma}{lemma:H_k_est_pdt_2_K} to control $\pdt^2 K$.
		\fref{Lemma}{lemma:H_k_est_pdt_2_K} tells us that
		\begin{align}
			\normtyp{\pdt^2 K}{H}{2M-3}
			\lesssim \normtyp{\pdt K}{H}{2M-3}
			+ \brac{
				\normtyp{ (u,\theta) }{L}{\infty} + \normtyp{ \pdt (u,\theta) }{L}{\infty}
			} \brac{
				\normtyp{K}{H}{2M-2} + \normtyp{\pdt K}{H}{2M-2} 
			}
		\nonumber
		\\
			+ \brac{
				1 + \normtyp{ (K, \nabla K) }{L}{\infty} + \normtyp{ \pdt(K, \nabla K) }{L}{\infty}
			} \brac{
				\normtyp{(u,\theta)}{H}{2M-3} + \normtyp{\pdt(u,\theta)}{H}{2M-3}
			}.
		\label{eq:est_K_intermediate_3}
		\end{align}
		In particular \eqref{eq:est_K_c3} allows us to control $\pdt K$ and $\pdt\nabla K$ in $L^\infty$ since $M\geqslant 3$ and hence
		\begin{equation}
		\label{eq:est_K_A}
			\normtyp{ \pdt(K, \nabla K)}{L}{\infty}
			\lesssim \normtyp{\pdt K}{H}{3}
			\lesssim \normtyp{\pdt K}{H}{2M-3}
			\lesssim P_f.
		\end{equation}
		As in the estimate of $\pdt K$ in $H^{2M-3}$, the subtlety now lies in estimating the decay-growth interaction.
		In light of \eqref{eq:est_K_star}, \eqref{eq:est_K_c2}, \eqref{eq:est_K_D}, \eqref{eq:est_K_c4},
		and the fact that $M\geqslant 3$,
		\begin{align}
			\brac{
				\normtyp{ (u,\theta) }{L}{\infty} + \normtyp{ \pdt(u, \theta) }{L}{\infty} 
			}\brac{
				\normtyp{K}{H}{2M-2} + \normtyp{\pdt K}{H}{2M-2} 
			}
			\lesssim \frac{ C_0^{1/2} }{ {\brac{1+t}}^{M-2} } \brac{
				P_f + (1 + P_e) \int_0^t \overline{\mathcal{D}}_M^{1/2}
			}
		\nonumber
		\\
			\lesssim C_0^{1/2} (1 + P_f) C_0^{1/2} \frac{ 1+t^{1/2} }{ {\brac{1+t}}^{M-2} }
			\lesssim P_f \frac{ {\brac{1+t}}^{1/2} }{ {\brac{1+t}}^{M-2} }
			\lesssim P_f.
		\label{eq:est_K_B}
		\end{align}
		So finally, combining \eqref{eq:est_K_star}, \eqref{eq:est_K_a}, \eqref{eq:est_K_b}, \eqref{eq:est_K_c3}, \eqref{eq:est_K_A}, and \eqref{eq:est_K_B}
		tells us that
		\begin{equation}
		\label{eq:est_K_c5}
			\normtyp{\pdt^2 K}{H}{2M-3} \lesssim P_f.
		\end{equation}

		We now seek to control $\pdt^2 K$ in $H^{2M-2}$.
		We may put together \fref{Lemma}{lemma:H_k_est_pdt_2_K},
		\eqref{eq:est_K_star}, \eqref{eq:est_K_a}, \eqref{eq:est_K_b}, \eqref{eq:est_K_c2}, \eqref{eq:est_K_c4}, and \eqref{eq:est_K_A} to see that
		\begin{align}
			\normtyp{\pdt^2 K}{H}{2M-2}
			\lesssim \normtyp{\pdt K}{H}{2M-2}
			+ \brac{
				\normtyp{ (u,\theta) }{L}{\infty} + \normtyp{ \pdt(u,\theta)}{L}{\infty} 
			}\brac{
				\normtyp{K}{H}{2M-1} + \normtyp{\pdt K}{H}{2M-1} 
			}
		\nonumber
		\\
			+ \brac{
				1 + \normtyp{ (K, \nabla K) }{L}{\infty} + \normtyp{ \pdt(K, \nabla K)}{L}{\infty} 
			}\brac{
			\normtyp{(u,\theta)}{H}{2M-2} + \normtyp{\pdt(u, \theta)}{H}{2M-2} 
			}
		\nonumber
		\\
			\lesssim \brac{1 + C_0^{1/2}} \brac{ \normtyp{K}{H}{2M-1} + \normtyp{\pdt K}{H}{2M-1} } + (1 + P_f) \overline{\mathcal{K}}_M^{1/2}
			\lesssim \mathcal{F}_M^{1/2} (0) + (1 + P_e) \int_0^t \overline{\mathcal{D}}_M^{1/2} + (1 + P_f) \overline{\mathcal{K}}_M^{1/2}.
		\label{eq:est_K_c6}
		\end{align}

		\paragraph{\textbf{Step 6: Estimating $\pdt^j K$ for $j\geqslant 3$.}}
		We conclude this proof by obtaining control over $\pdt^j K$ when $j \geqslant 3$.
		We proceed by induction, relying on \fref{Lemma}{lemma:H_k_est_pdt_j_K} for both the base case and the induction step,
		and we will show that
		\begin{equation}
		\label{eq:est_K_c7}
			\normtyp{\pdt^j K}{H}{2M-2j+2} \lesssim P_f
			\text{ for every } 3 \leqslant j \leqslant M.
		\end{equation}
		Note that the hypotheses of \fref{Lemma}{lemma:H_k_est_pdt_j_K} are always satisfied here since $2M - 2j + 2 \geqslant 2 > \frac{3}{2}$ when $j \leqslant M$.
		We begin with the base case.
		By \fref{Lemma}{lemma:H_k_est_pdt_j_K}, \eqref{eq:est_K_star}, \eqref{eq:est_K_c1}, \eqref{eq:est_K_c3}, and \eqref{eq:est_K_c5} we obtain that
		\begin{align*}
			\normtyp{\pdt^3 K}{H}{2M-4}
			\lesssim \normtyp{\pdt^2 K}{H}{2M-4} + \normtyp{\pdt^2 \theta}{H}{2M-4} + \sum_{l=0}^2 \brac{ \normtyp{\pdt^l (u,\theta) }{H}{2M-4}^2 + \normtyp{\pdt^l (u,\theta)}{H}{2M-3}^2 }
		\\
			\lesssim P_f + C_0^{1/2} + \brac{ C_0 + P_f}
			\lesssim P_f.
		\end{align*}
		We may now proceed with the induction step.
		Suppose that there is some $3 \leqslant j < M$ such that
		\begin{equation}
		\label{eq:est_K_IH}
			\normtyp{\pdt^l K}{H}{2M-2l+2} \lesssim P_f
			\text{ for every } 3 \leqslant l \leqslant j.
		\end{equation}
		Then, by \fref{Lemma}{lemma:H_k_est_pdt_j_K}, \eqref{eq:est_K_star}, \eqref{eq:est_K_c1}, \eqref{eq:est_K_c3}, \eqref{eq:est_K_c5}, and \eqref{eq:est_K_IH} we see that
		\begin{align*}
			\normtyp{\pdt^{j+1} K}{H}{2M-2j}
			\lesssim \normtyp{\pdt^j K}{H}{2M-2j} + \normtyp{\pdt^j \theta}{H}{2M-2j} + \sum_{l=0}^j \brac{
				\normtyp{\pdt^l (u,\theta)}{H}{2M-2j}^2 + \normtyp{\pdt^l K}{H}{2M-2j+1}^2
			}
		\\
			\lesssim P_f + C_0^{1/2} + \normtyp{(u,\theta)}{P}{2M}^2 + P_f^2
			\lesssim P_f.
		\end{align*}
		This proves that the induction step holds, from which \eqref{eq:est_K_c7} follows.

		\paragraph{\textbf{Step 7: Synthesis.}}
		We combine \eqref{eq:est_K_c1}, \eqref{eq:est_K_c3}, \eqref{eq:est_K_c5}, and \eqref{eq:est_K_c7} to deduce the bound on $ \mathcal{E}_M^{(K)}$
		and we combine \eqref{eq:est_K_c2}, \eqref{eq:est_K_c4}, and \eqref{eq:est_K_c6} to deduce the bound on $ \mathcal{F}_M $.
	\end{proof}

\subsection{Energy-dissipation structure}
\label{sec:ED_structure}

	In this section we identify the energy-dissipation structure of the problem and record some related auxiliary results,
	such as the precise form of the interactions, a comparison result for the various versions of the energy,
	and a coercivity estimate for the dissipation.
	Since the dissipation $D$ will appear frequently throughout this section we recall that it is defined in \eqref{eq:not_dissip}.
	We begin with the energy-dissipation relation.

	\begin{prop}[The generic energy-dissipation relation]
	\label{prop:gen_pert_ED_rel}
		Let the stress tensors $T$ and $M$ be as defined in \eqref{eq:intro_stress_tensors_def} and
		suppose that $\brac{v, q, \theta, b}$ solves
		\begin{subnumcases}{}
			(\pdt + u\cdot\nabla) v = \brac{\nabla\cdot T}\brac{v, q, \theta} + f,
			\label{eq:gen_pert_ED_lin_mom}\\
			\nabla\cdot v = 0,
			\label{eq:gen_pert_ED_incompress}\\
			J (\pdt + u\cdot\nabla)\theta + \brac{\omega\times J} \theta + \tilde{\tau}^2 \tilde{b}^\perp + \theta \times J\omega_{eq}
				= 2\vc T\brac{v, q, \theta} + \brac{\nabla\cdot M}\brac{\theta} + g, \text{ and }
			\label{eq:gen_pert_ED_ang_mom}\\
			(\pdt + u\cdot\nabla)b = -\brac{\nu-\lambda} \bar{\theta}^\perp + \omega_3 b^\perp + h,
			\label{eq:gen_pert_ED_b}
		\end{subnumcases}
		where $\brac{u, \omega, J}$ are given and satisfy
		\begin{subnumcases}{}
			\nabla\cdot u \text{ and }
			\label{eq:gen_pert_ED_constraint_incompress}\\
			(\pdt + u\cdot\nabla) J = \sbrac{\Omega, J},
			\label{eq:gen_pert_ED_constraint_cons_microinertia}
		\end{subnumcases}
		and where $f$, $g$, and $h$ are given.
		Then the following energy-dissipation relation holds:
		\begin{align}
			\frac{\mathrm{d}}{\mathrm{d}t} \brac{
				\int_{\T^3} \frac{1}{2} \abs{v}^2
				+ \frac{1}{2} J\theta\cdot\theta
				+ \frac{\tilde{\tau}^2}{\nu-\lambda}  \frac{1}{2} \abs{b}^2
			}
			+ D(v, \theta)
			= \int_{\T^3} f\cdot v + g \cdot \theta + \frac{\tilde{\tau}^2}{\nu-\lambda} h \cdot b.
			\label{eq:gen_pert_ED_rel}
		\end{align}
	\end{prop}
	\begin{proof}
		We multiply by the unknowns and integrate by parts: since $u$ and $v$ are divergence-free,
		\begin{equation}
			\Dt \int_{\T^3} \frac{1}{2} \abs{v}^2
			= \int_{\T^3} (\pdt + u\cdot\nabla) v \cdot v
			= \int_{\T^3} \brac{\nabla\cdot T}\cdot v + \int_{\T^3} f\cdot v
			= - \int_{\T^3} T : \nabla v + \int_{\T^3} f\cdot v.
			\label{eq:gen_pert_ED_rel_proof_1}
		\end{equation}
		Similarly, using the incompressibility of $u$, \eqref{eq:gen_pert_ED_constraint_cons_microinertia}, \fref{Lemma}{lemma:identity_comm_A_S_and_sym_A_cross_S} we see that
		\begin{equation*}
			\frac{1}{2} (\pdt + u\cdot\nabla) J \theta\cdot\theta= \frac{1}{2} \sbrac{\Omega, J} \theta\cdot\theta = \brac{\omega\times J}\theta\cdot\theta,
		\end{equation*}	
		and from the fact that $\theta \times J\omega_{eq} \cdot \theta = 0$ we obtain that
		\begin{align}
			\Dt \int_{\T^3} \frac{1}{2} J\theta \cdot \theta
			= \int_{\T^3} \frac{1}{2} (\pdt + u\cdot\nabla) J \theta\cdot\theta + \int_{\T^3} J (\pdt + u\cdot\nabla) \theta\cdot\theta
			= \int_{\T^3} T:\Omega - M:\nabla\omega + g\cdot\theta + \tilde{\tau}^2 b \cdot \bar{\theta}^\perp.
			\label{eq:gen_pert_ED_rel_proof_2}
		\end{align}
		Finally we compute that
		\begin{equation}
			\Dt \int_{\T^3} \frac{1}{2} \abs{b}^2
			= \int_{\T^3} (\pdt + u\cdot\nabla) b \cdot b
			= - \int_{\T^3} \brac{\nu-\lambda} \bar{\theta}^\perp \cdot b + \int_{\T^3} h\cdot b.
			\label{eq:gen_pert_ED_rel_proof_3}
		\end{equation}
		To conclude it suffices to add \eqref{eq:gen_pert_ED_rel_proof_1}, \eqref{eq:gen_pert_ED_rel_proof_2}, and $\frac{\tilde{\tau}^2}{\nu-\lambda}$ \eqref{eq:gen_pert_ED_rel_proof_3}
		and observe that
		\begin{equation*}
			\int_{\mathbb{T}^3} T(v, q, \theta) : (\Theta - \nabla v) + M(\theta) : (\nabla \theta) = D(v, \theta).
		\end{equation*}
		This equation follows from the identities
		$
			(\vc M) \cdot v = \frac{1}{2} M : \ten v
			\text{ and }
			\Skew(\nabla v) = \frac{1}{2} \ten\nabla\times v
		$
		and the fact that $\R^{n\times n}$ may be orthogonally decomposed with respect to the Fr\"{o}benius inner product
		as $\R^{n\times n} \cong \R I \oplus \dev(n) \oplus \Skew(n)$ where $\dev(n)$ denotes set of trace-free symmetric $n$-by-$n$ matrices
		and $\Skew(n)$ denotes the set of antisymmetric $n$-by-$n$ matrices.

		To conclude we add \eqref{eq:gen_pert_ED_rel_proof_1}, \eqref{eq:gen_pert_ED_rel_proof_2}, and $\frac{\tilde{\tau}^2}{\nu-\lambda}$ \eqref{eq:gen_pert_ED_rel_proof_3},
		to obtain \eqref{eq:gen_pert_ED_rel}.
	\end{proof}

	Having established the precise form of the energy-dissipation relation we may now record the specific form of the interactions.
	\fref{Lemma}{lemma:record_form_interactions} is a necessary precursor to the interactions estimates of \fref{Section}{sec:close_at_low} and \ref{sec:close_at_high}.

	\begin{lemma}[Recording the form of the interactions]
	\label{lemma:record_form_interactions}
		If $(u, p, \theta, K)$ solve \eqref{eq:pertub_sys_no_ten_pdt_u}--\eqref{eq:pertub_sys_no_ten_pdt_K} then, for any multi-index $\alpha\in\N^{1+3}$,
		\begin{equation}
		\label{eq:record_int}
			\frac{\mathrm{d}}{\mathrm{d}t} \brac{
				\int_{\T^3} \frac{1}{2} \abs{ \partial^\alpha u }^2
				+ \frac{1}{2} J \partial^\alpha \theta \cdot \partial^\alpha \theta
				+ \frac{\tilde{\tau}^2}{\nu-\lambda} \frac{1}{2} \abs{ \partial^\alpha a}^2
			}
			+ D( \partial^\alpha u, \partial^\alpha \theta)
			= \mathcal{I}^\alpha
		\end{equation}
		where
		\begin{align*}
			\mathcal{I}^\alpha = 
			  \int_{\T^3} \sbrac{ u\cdot\nabla,		\partial^\alpha} u	\cdot \partial^\alpha u
			+ \int_{\T^3} \sbrac{ J\pdt,			\partial^\alpha} \theta \cdot \partial^\alpha \theta
			+ \int_{\T^3} \sbrac{ J\brac{u\cdot\nabla},	\partial^\alpha} \theta \cdot \partial^\alpha \theta
			+ \int_{\T^3} \sbrac{ \omega\times J,		\partial^\alpha} \theta \cdot \partial^\alpha \theta
		\\
			- \int_{\T^3} \sbrac{ J \omega_{eq} \times,	\partial^\alpha} \theta \cdot \partial^\alpha \theta
			+ \int_{\T^3} \sbrac{ u\cdot\nabla,		\partial^\alpha} a	\cdot \partial^\alpha a
			+ \int_{\T^3} \sbrac{ \omega_3 R,		\partial^\alpha} a	\cdot \partial^\alpha a
			+ \int_{\T^3} \partial^\alpha \brac{ \brac{\bar{K} - K_{33}I_2} \bar{\theta}^\perp } \cdot \partial^\alpha a
		\\
			=\vcentcolon \mathcal{I}_1^\alpha + \dots + \mathcal{I}_8^\alpha.
		\end{align*}
	\end{lemma}
	\begin{proof}
		The first order of business is to write \eqref{eq:pertub_sys_no_ten_pdt_u}--\eqref{eq:pertub_sys_no_ten_pdt_K} in the form of \fref{Proposition}{prop:gen_pert_ED_rel}.
		In order to do this we note that \eqref{eq:pertub_sys_no_ten_pdt_u}--\eqref{eq:pertub_sys_no_ten_pdt_K} can be written using the stress tensor $T$ and the couple stress tensor $M$
		as \eqref{eq:pertub_sys_ten_pdt_u}--\eqref{eq:pertub_sys_ten_pdt_K}.
		In light of \eqref{eq:pertub_sys_pdt_a} we therefore see that $(u, p, \theta, a)$ solve
		\begin{subnumcases}{}
			(\pdt + u\cdot\nabla) u = (\nabla\cdot T)(u, p, \theta),
			\label{eq:record_int_pdt_u}\\
			\nabla\cdot u = 0,
			\label{eq:record_int_div_free}\\
			J(\pdt + u\cdot\nabla)\theta + (\omega\times J)\theta + \tilde{\tau}^2 \tilde{b}^\perp + \theta\times J\omega_{eq}
				= 2\vc T(u, p, \theta) + (\nabla\cdot M)(\theta), \text{ and }
			\label{eq:record_int_pdt_theta}\\
			(\pdt + u\cdot\nabla) a = -(\nu-\lambda) \bar{\theta}^\perp + \omega_3 a^\perp + (\bar{K} - K_{33} I_2) \bar{\theta}^\perp
			\label{eq:record_int_pdt_a}
		\end{subnumcases}
		subject to $(\pdt + u\cdot\nabla)J = \sbrac{\Omega, J}$ where $J = J_{eq} + K$ and $\omega = \omega_{eq} + \theta$.
		Note in particular that the full precession $\omega \times J\omega$ is present in \eqref{eq:record_int_pdt_theta} since $\omega_{eq} \times J_{eq} \omega_{eq} = 0$
		and $\omega_{eq} \times K \omega_{eq} = \tilde{\tau}^2 \tilde{a}^\perp$ such that indeed
		$
			\omega \times J\theta + \tilde{\tau}^2 \tilde{b}^\perp + \theta\times J\omega_{eq}
			= \omega\times J\omega.
		$
		We may now apply derivatives to \eqref{eq:record_int_pdt_u}--\eqref{eq:record_int_pdt_a} and use \fref{Proposition}{prop:gen_pert_ED_rel}.
		This tells us
		\begin{equation}
		\label{eq:record_int_deriv}
			\frac{\mathrm{d}}{\mathrm{d}t} \brac{
				\int_{\T^3} \frac{1}{2} \abs{ \partial^\alpha u }^2
				+ \frac{1}{2} J \partial^\alpha \theta \cdot \partial^\alpha \theta
				+ \frac{\tilde{\tau}^2}{\nu-\lambda} \frac{1}{2} \abs{ \partial^\alpha a}^2
			}
			+ D( \partial^\alpha u, \partial^\alpha \theta)
			= \int_{\T^3} f^\alpha \cdot \partial^\alpha u + g^\alpha \cdot \partial^\alpha \theta + h^\alpha \cdot \partial^\alpha a
		\end{equation}
		for
		\begin{align*}
			f^\alpha = \sbrac{u\cdot\nabla, \partial^\alpha } u,\,
			g^\alpha = \sbrac{J\pdt, \partial^\alpha }\theta
				+ \sbrac{J(u\cdot\nabla), \partial^\alpha } \theta
				+ \sbrac{\omega\times J, \partial^\alpha }\theta
				- \sbrac{J\omega_{eq}\times, \partial^\alpha } \theta,\,
			\text{ and } 
		\\
			h^\alpha = \sbrac{u\cdot\nabla, \partial^\alpha} a
				+ \sbrac{\theta_3 R, \partial^\alpha } a
				+ \partial^\alpha \brac{ \brac{\bar{K} - K_{33}I_2}\bar{\theta}^\perp },
		\end{align*}
		where $R = e_2 \otimes e_1 - e_1 \otimes e_2 \in \R^{2\times 2}$ is the (counterclockwise) $\frac{\pi}{2}$ rotation in $\R^2$,
		such that \eqref{eq:record_int_deriv} is precisely \eqref{eq:record_int}, as desired.
	\end{proof}

	We now record two auxiliary results related to the energy-dissipation structure of the problem.
	First we record a precise comparison of various versions of the energy,
	and then we record the coercivity of the dissipation over $H^1$.
	Recall that $ \overline{\mathcal{E}}_M $ and $ \widetilde{\mathcal{E}}_M $ are defined in \eqref{eq:not_E_M_bar}.

	\begin{lemma}[Comparison of the different versions of the energy]
	\label{lemma:comp_version_en}
		There exist constants $c_E, C_E > 0$ such that for every time horizon $T>0$,
		if
		\begin{equation}
		\label{eq:comp_version_en_assumption}
			\sup_{0 \leqslant t < T} \normtyp{(u,\theta) (t)}{H}{3} + \normtyp{J (t)}{H}{3} + \normtyp{\pdt(u, \theta) (t)}{H}{2} + \normtyp{\pdt J (t)}{H}{2} < \infty
		\end{equation}
		then, for any non-negative integer $M$, $c_E \overline{\mathcal{E}}_M \leqslant \widetilde{\mathcal{E}}_M \leqslant C_E \overline{\mathcal{E}}_M $ on $\cobrac{0,T}$.
	\end{lemma}
	\begin{proof}
		It is crucial here to remember the global assumption according to which the spectrum of $J_0(x)$ is $\cbrac{\lambda, \lambda, \nu}$, where $\nu > \lambda > 0$, for every $x\in\T^3$.
		The key observation is then that we may combine the assumption \eqref{eq:comp_version_en_assumption} and \fref{Proposition}{prop:persist_spec_sols_adv_rot_eqtns} to deduce that
		for every $(t,x) \in \cobrac{0,T}\times\T^3$ the spectrum of $J(t,x)$ is $\cbrac{\lambda, \lambda, \nu}$.
		Therefore
		\begin{equation*}
			\lambda \int_{\T^3} \abs{\theta}^2
			\leqslant \int_{\T^3} J\theta\cdot\theta
			\leqslant \nu \int_{\T^3} \abs{\theta}^2
		\end{equation*}
		and hence the claim follows upon picking
		$
			c_E = \frac{1}{2} \min\brac{1, \lambda, \frac{\tilde{\tau}^2}{\nu-\lambda}}
			\text{ and } 
			C_E = \frac{1}{2} \max\brac{1, \nu, \frac{\tilde{\tau}^2}{\nu-\lambda}}. \qedhere
		$
	\end{proof}
	
	We now record the coercivity of the dissipation over $H^1$ in \fref{Lemma}{lemma:coercivity_dissip} below.
	Note that this lemma is copied from Lemma 4.9 of the companion paper \cite{rt_tice_amf_iut}, and so we omit the proof.
	\begin{lemma}[Coercivity of the dissipation]
	\label{lemma:coercivity_dissip}
		There exists a universal constant $C_D > 0$ such that for every $(u, \theta)\in H^1$ where $u$ has average zero, $D(u,\theta) \geqslant C_D \normtyp{(u, \theta)}{H}{1}$.
	\end{lemma}

\subsection{Closing the estimates at the low level}
\label{sec:close_at_low}

	We now turn our attention to the second of the four building blocks of the scheme of a priori estimates: closing the energy estimates at the low level.
	This is carried out in this section and culminates in \fref{Proposition}{prop:close_est_low_level}.
	In the remainder of this section we proceed as follows:
	first we derive auxiliary estimates of $a$ and use them to improve the low level energy and dissipation,
	then we estimate the low level interactions,
	and finally we record the $\theta$-coercivity central to the algebraic decay of the energy at the low level.
	We recall that the energy, dissipation, and interaction functionals at the low level, which will appear throughout this section, are defined in \eqref{eq:not_E_low}, \eqref{eq:not_D_low},
	and \eqref{eq:not_I}, respectively. The various versions of the high level energy are defined in \eqref{eq:not_E_M_bar} and \eqref{eq:not_E_M_and_F_M}.

	We begin with auxiliary estimates for $a$ intended to improve the low level energy and dissipation.
	The strategy is simple: we use the appearance of $a$ in the conservation of angular momentum \eqref{eq:pertub_sys_no_ten_pdt_theta} to control $a$ via the dissipation, then
	use the equation \eqref{eq:pertub_sys_pdt_a} for $\pdt a$ to bootstrap the control of $a$ to new or better control of $\pdt a$ in the dissipation and the energy respectively,
	and finally use the time-differentiated equation for $\pdt a$ in order to control $\pdt^2 a$ energetically if the additional assumption (2) holds.

	\begin{lemma}[Auxiliary estimate for $a$]
	\label{lemma:aux_est_a_low_level}
		Suppose that \eqref{eq:pertub_sys_no_ten_pdt_theta} holds. For any $k\in\N$ and any $s > \frac{3}{2}$ such that $s\geqslant k$, we have the estimate
		\begin{equation*}
			\normtyp{a}{H}{k} \lesssim \brac{
				1 + \normtyp{K}{H}{s} 
				+ \normtyp{(u,\theta)}{H}{s} 
				+ \normtyp{K}{H}{s} \normtyp{(u,\theta)}{H}{s} 
			} \brac{
				\normtyp{\theta}{H}{k+2}
				+ \normtyp{\pdt\theta}{H}{k} 
				+ \normtyp{u}{H}{k+1} 
			}.
		\end{equation*}
	\end{lemma}
	\begin{proof}
		This estimate follows from isolating $a$ in \eqref{eq:pertub_sys_no_ten_pdt_theta} and using \fref{Lemma}{lemma:prod_est} to estimate the nonlinearities.
		To isolate $a$ in \eqref{eq:pertub_sys_no_ten_pdt_theta} we use the facts that $\omega_{eq} \times J_{eq} \omega_{eq} = 0$
		and $\omega_{eq} \times K \omega_{eq} = \tilde{\tau}^2 \tilde{a}^\perp$ to rewrite the precession term as
		\begin{equation*}
			(\omega_{eq} + \theta) \times (J_{eq} + K) (\omega_{eq} + \theta)
			= (\omega_{eq} + \theta) \times (J_{eq} + K) \theta + \theta \times (J_{eq} + K) \omega_{eq} + \tilde{\tau}^2 \tilde{a}^\perp.
		\end{equation*}
		\eqref{eq:pertub_sys_no_ten_pdt_theta} may then be written as
		\begin{align*}
			\tilde{\tau}^2 \tilde{a}^\perp
			= -(J_{eq} + K)(\pdt\theta + u\cdot\nabla\theta) - (\omega_{eq} + \theta) \times (J_{eq} + K)\theta - \theta\times(J_{eq} + K) \omega_{eq}
		\\
		+ \kappa \nabla\times u - 2\kappa\theta + (\tilde{\alpha} - \tilde{\gamma}) \nabla(\nabla\cdot\theta) + \tilde{\gamma} \Delta\theta. &\qedhere
		\end{align*}
	\end{proof}

	We continue obtaining auxiliary estimates for $a$ by obtaining an estimate for its first two time derivatives.

	\begin{lemma}[Auxiliary estimate for $\pdt a$ and $\pdt^2 a$]
	\label{lemma:aux_est_pdt_a_low_level}
		Suppose that \eqref{eq:pertub_sys_pdt_a} holds. For any $k\in\N$ and any $s > \frac{3}{2}$ such that $s\geqslant k$, we have the estimates
		\begin{align*}
			\normtyp{\pdt a}{H}{k}
			&\lesssim \brac{1 + \normtyp{(u, \theta)}{H}{s} } \normtyp{a}{H}{k+1} + \brac{1 + \normtyp{K}{H}{s} } \normtyp{\theta}{H}{s}.
			\text{ and }\\
			\normns{\pdt^2 a}{H^k}
			&\lesssim
			\norm{\pdt\brac{u,\theta}}{H^s} \norm{a}{H^{k+1}} + \brac{1 + \norm{\brac{u, \theta}}{H^s}} \norm{\pdt a}{H^{k+1}}
			+ \norm{\pdt K}{H^s} \norm{\theta}{H^k} + \brac{1 + \norm{K}{H^s}} \norm{\pdt\theta}{H^k}.
		\end{align*}
	\end{lemma}
	\begin{proof}
		The first estimate follows as in \fref{Lemma}{lemma:aux_est_a_low_level} from isolating $\pdt a$ in \eqref{eq:pertub_sys_pdt_a}
		and using \fref{Lemma}{lemma:prod_est} to estimate the quadratic terms.
		The second estimate follows from differentiating \eqref{eq:pertub_sys_pdt_a} in time and then proceeding as in \fref{Lemma}{lemma:aux_est_a_low_level},
		namely isolating $\pdt^2 a$ and using \fref{Lemma}{lemma:prod_est}.
	\end{proof}

	With these auxiliary estimates for $a$ and its first two temporal derivatives in hand we may now improve the low level energy and dissipation.

	\begin{prop}[Improvement of the low level energy and dissipation]
	\label{prop:imp_low_level_en_ds}
		Let $T > 0$ be a time horizon.
		Consider the assumptions
		\begin{equation*}
			(1)\; \sup_{0\leqslant t < T} \normtyp{(u,\theta)(t)}{H}{3} + \normtyp{K(t)}{H}{3} \leqslant \overline{C} < \infty \text{ and }
			\;(2)\; \sup_{0\leqslant t < T} \normtyp{\pdt(u,\theta)(t)}{H}{2} + \normtyp{\pdt K (t)}{H}{2} \leqslant \overline{C} < \infty.
		\end{equation*}
		Then we have the following estimates, where the constants implicit in ``$\lesssim$'' may depend on $\overline{C}$.
		If (1)  holds then
		$
			\overline{\mathcal{E}}_\text{low} \gtrsim \normtyp{\pdt a}{H}{1}^2
		$
		and
		$
			\overline{\mathcal{D}}_\text{low} \gtrsim \normtyp{a}{H}{1}^2 + \normtyp{ \pdt a }{L}{2}^2.
		$
		If (1) and (2) hold then
		$
			\overline{\mathcal{E}}_\text{low} \gtrsim \normtyp{ \pdt^2 a }{L}{2}^2.
		$
	\end{prop}
	\begin{proof}
		(1) and \fref{Lemma}{lemma:aux_est_a_low_level} tell us that
		$\norm{a}{H^1}^2 \lesssim \norm{\brac{u,\theta}}{H^3}^2 + \norm{\pdt\theta}{H^1}^2 \lesssim \overline{\mathcal{D}}_\text{low}$.
		Then we may use (1), \fref{Lemma}{lemma:aux_est_pdt_a_low_level}, and the previous estimate to see that
		\begin{equation*}
			\norm{\pdt a}{L^2}^2 \lesssim \norm{a}{H^1}^2 + \norm{\theta}{L^2}^2 \lesssim \overline{\mathcal{D}}_\text{low} \text{ and }
			\norm{\pdt a}{H^1}^2 \lesssim \norm{a}{H^2}^2 + \norm{\theta}{H^1}^2 \lesssim \overline{\mathcal{E}}_\text{low}.
		\end{equation*}
		Finally, if both assumptions (1) and (2) hold then we may use \fref{Lemma}{lemma:aux_est_pdt_a_low_level} again to see that
		\begin{equation*}
			\norm{\pdt^2 a}{L^2}^2
			\lesssim \norm{a}{H^1}^2 + \norm{\pdt a}{H^1}^2 + \norm{\theta}{L^2}^2 + \norm{\pdt\theta}{L^2}^2
			\lesssim \overline{\mathcal{E}}_\text{low}. \qedhere
		\end{equation*}
	\end{proof}

	We now turn our attention to the low-level interactions and record their estimates here.
	Note that they may be estimated in a simpler way for the sole purpose of closing the energy estimates at the low level,
	but by doing the estimates slightly more carefully as done below we can also use them when we study the local well-posedness theory (in \fref{Section}{sec:lwp}).

	\begin{lemma}[Careful estimates of the low-level interactions]
	\label{lemma:careful_est_low_level_int}
		Recall that $\overline{\mathcal{I}}_\text{low}$ is defined in \eqref{eq:not_I} for $\mathcal{I}^\alpha$ as in \fref{Lemma}{lemma:record_form_interactions}.
		The following estimate holds.
		\begin{equation*}
			\abs{ \overline{\mathcal{I}}_\text{low} }
			\lesssim \brac{
				\normtyp{(u,\theta)}{P}{2} 
				+ \normtyp{a}{P}{3} 
				+ \brac{ 1 + \normtyp{(u,\theta)}{P}{2}  }
				\brac{ \norm{K}{H^3} + \norm{\pdt K}{L^\infty}}
			} \mathcal{D}_\text{low}.
		\end{equation*}
	\end{lemma}
	\begin{proof}
		Recall that, at the low level, $\abs{\alpha}_P \leqslant 2$.
		In particular if $\beta + \gamma = \alpha$ and $\beta > 0$ then $( \partial^\alpha , \partial^\beta , \partial^\gamma )$ corresponds to one of five possible cases:
		$
			(\partial_x^2, \partial_x^2, 0)
		$, $
			(\partial_x^2, \partial_x, \partial_x)
		$, $
			(\partial_x, \partial_x, 0)
		$, $
			(\pdt, \pdt, 0)
		$, and $
			(0, 0, 0),
		$
		where $\partial_x^k$ indicates a derivative $ \partial^\alpha $ for a purely spatial multi-index $\alpha\in\N^3$ of length $\abs{\alpha} = k$.
		Note that we have the bound
		$
			\normtyp{ K}{L}{\infty} + \normtyp{ \nabla K}{L}{\infty} + \normtyp{ \nabla^2 K}{L}{6} \lesssim \normtyp{K}{H}{3}
		$
		such that, for any $\abs{\alpha}_P \leqslant 2$,
		\begin{equation}
		\label{eq:careful_est_low_level_int_est_K_L6}
			\normtyp{ \partial^\alpha K}{L}{6} \lesssim \normtyp{K}{H}{3} + \normtyp{ \pdt K}{L}{\infty}.
		\end{equation}
		Recall from \fref{Lemma}{lemma:record_form_interactions} that $\overline{\mathcal{I}}_\text{low} = \sum_{\abs{\alpha}_P \leqslant 2} \sum_{i=1}^8 \mathcal{I}_i^\alpha$.
		In light of \eqref{eq:careful_est_low_level_int_est_K_L6} we may estimate $\mathcal{I}_1^\alpha - \mathcal{I}_5^\alpha$ and $\mathcal{I}_7^\alpha$ easily, obtaining:
		\begin{align*}
			\vbrac{\mathcal{I}_1^\alpha} &\lesssim
				\normtyp{u}{P}{2} \mathcal{D}_\text{low},\;
			\vbrac{\mathcal{I}_2^\alpha} \lesssim
				\brac{ \normtyp{K}{H}{3} + \normtyp{ \pdt K }{L}{\infty} } \mathcal{D}_\text{low},\;
			\vbrac{\mathcal{I}_3^\alpha} \lesssim
				\brac{ 1 + \normtyp{K}{H}{3} + \normtyp{ \pdt K }{L}{\infty} } \normtyp{u}{P}{2}  \mathcal{D}_\text{low},\\
			\vbrac{\mathcal{I}_4^\alpha} &\lesssim
				\brac{ 1 + \normtyp{\theta}{P}{2} } \brac{ \normtyp{K}{H}{3} + \normtyp{ \pdt K }{L}{\infty} } \mathcal{D}_\text{low}
				+ \brac{1 + \normtyp{K}{H}{3} } \normtyp{\theta}{P}{2} \mathcal{D}_\text{low},\\
			\vbrac{\mathcal{I}_5^\alpha} &\lesssim
				\brac{ \normtyp{K}{H}{3} + \normtyp{ \pdt K }{L}{\infty} } \mathcal{D}_\text{low}, \text{ and }
			\vbrac{\mathcal{I}_7^\alpha} \lesssim
				\normtyp{a}{P}{3} \mathcal{D}_\text{low}.
		\end{align*}
		The only two terms requiring particularly delicate care are $\mathcal{}I_6$ and $\mathcal{I}_8$, due to the presence of $\partial_x^2 a$.
		We provide the details on how to estimate these two interactions below.

		\vspace{0.5em}
		\paragraph{\textbf{Estimating $\mathcal{I}_6^\alpha$.}}
			Recall that
			\begin{equation*}
				\mathcal{I}_6^\alpha = - \sum_{\substack{ \beta + \gamma = \alpha \\ \beta > 0 }} \binom{\alpha}{\beta}
				\int_{\T^3} ( \partial^\beta u\cdot \nabla \partial^\gamma a) \cdot \partial^\alpha a.
			\end{equation*}
			The difficulty lies in $ \partial^\alpha = \partial_x^2$ since then two copies of $\partial_x^2 a$ may appear and we only control $a$ dissipatively in $H^1$.
			We thus split into two cases, emphasizing that only the first case is somewhat troublesome and requiring of particular care.
			In the first case we consider $\abs{\bar{\alpha}} = 2$ and $\abs{\bar{\beta}} = \abs{\bar{\gamma}} = 1$
			and proceed by interpolation:
			\begin{align*}
				\vbrac{ \int_{\T^3} (\partial_x u \cdot\nabla\partial_x a) \cdot \partial_x^2 a}
				\leqslant \normtyp{ \partial_x u}{L}{\infty} \normtyp{ \partial_x^2 a}{L}{2}^2
				\lesssim \normtyp{u}{H}{3} \normtyp{a}{H}{2}^2
				\lesssim \normtyp{u}{H}{3} {\brac{ \normtyp{a}{H}{1}^{1/2} \normtyp{a}{H}{3}^{1/2} }}^2
				\lesssim \normtyp{a}{H}{3} \mathcal{D}_\text{low}.
			\end{align*}
			In the second case we consider either $\abs{\bar{\alpha}} = \abs{\bar{\beta}} = 2$ or $\abs{\bar{\alpha}} \leqslant 1$.
			Either way, since $\beta > 0$ we deduce that $\gamma = 0$ and hence $\beta = \alpha$.
			The estimate is then immediate:
			\begin{equation*}
				\vbrac{\int_{\mathbb{T}^3} \brac{ \partial^\alpha u \cdot \nabla a} \cdot \partial^\alpha a}
				\leqslant \normtyp{ \partial^\alpha u}{L}{6} \normtyp{ \nabla a}{L}{2} \normtyp{ \partial^\alpha a}{L}{6} 
				\lesssim \normtyp{u}{P}{3} \normtyp{a}{H}{1} \normtyp{a}{P}{3} 
				\lesssim \normtyp{a}{P}{3} \mathcal{D}_\text{low}.
			\end{equation*}

		\paragraph{\textbf{Estimating $\mathcal{I}_8^\alpha$.}}
			Recall that
			\begin{equation*}
				\mathcal{I}_8^\alpha = - \sum_{\beta+\gamma = \alpha}
				\int_{\T^3} \partial^\beta ( \bar{K} - K_{33}I_2) \partial^\gamma \bar{\theta}^\perp \cdot \partial^\alpha a
				\;\sim\;
				\sum_{\beta+\gamma=\alpha} \int_{\T^3} (\partial^\beta K) \partial^\gamma \theta \cdot \partial^\alpha a.
			\end{equation*}
			where the left-hand side is the precise form of the interaction and the right-hand side is its schematic form which we will work with henceforth.
			The difficulty lies in $ \partial^\alpha = \partial_x^2 $ since we have no dissipative control over $K$ and only control $a$ dissipatively in $H^1$.
			We must therefore integrate by parts to reduce the term $\partial_x^2 a$ to $\partial_x a$.
			Now we split into two cases. As in the consideration of $\mathcal{I}_6^\alpha$ above, only the first case is somewhat troublesome.

			In the first case we consider $\abs{\bar{\alpha}} = 2$.
			We integrate by parts and estimate each term by hand.
			We write
			\begin{equation*}
				\sum_{\beta+\gamma = 2} \int_{\T^3} (\partial_x^\beta K ) \partial_x^\gamma \theta \cdot \partial_x^2 a
				= \sum_{\beta+\gamma = 2} \int_{\T^3} (\partial_x^{\beta+1}K) (\partial_x^\gamma \theta) \cdot \partial_x a
					+ \int_{\T^3} ( \partial_x^\beta K) (\partial_x^{\gamma+1} \theta) \cdot \partial^\alpha a
				=\vcentcolon \RN{1} + \RN{2}
			\end{equation*}
			where
			\begin{align*}
				\abs{\RN{1}}
				\leqslant \vbrac{ \int_{\T^3} (\partial_x K) (\partial_x^2 \theta) \partial_x a}
					+ \vbrac{ \int_{\T^3} (\partial_x^2 K) (\partial_x\theta) \cdot \partial_x a}
					+ \vbrac{ \int_{\T^3} (\partial_x^3 K) \theta\cdot \partial_x a}
			\\
				\leqslant \brac{
					\normtyp{ \partial_x K}{L}{\infty} \normtyp{ \partial_x^2 \theta}{L}{6} 
					+ \normtyp{ \partial_x^2 K}{L}{6} \normtyp{ \partial_x\theta}{L}{\infty} 
					+ \normtyp{ \partial_x^3 K}{L}{2} \normtyp{ \theta}{L}{\infty} 
				} \normtyp{ \partial_x a}{L}{2} 
			\\
				\lesssim \normtyp{K}{H}{3} \normtyp{\theta}{H}{3} \normtyp{a}{H}{1} 
				\lesssim \normtyp{K}{H}{3} \mathcal{D}_\text{low}.
			\end{align*}
			and
			\begin{equation*}
				\abs{\RN{1}}
				\leqslant \sum_{\beta+\gamma = 2} \normtypns{ \partial_x^\beta K}{L}{\infty} \normtyp{ \partial_x^{\gamma+1}\theta}{L}{2} \normtyp{ \partial_x a}{L}{2} 
				\lesssim \normtyp{K}{H}{3} \normtyp{\theta}{H}{3} \normtyp{a}{H}{1} 
				\lesssim \normtyp{K}{H}{3} \mathcal{D}_\text{low}.
			\end{equation*}

			In the second case we consider $\abs{\bar{\alpha}} \leqslant 1$.
			Note that, as was noted above when considering $\mathcal{I}_6^\alpha$,
			it then follows from the constraint $\beta > 0$ that $\gamma = 0$ and $\beta = \alpha$.
			The estimate is then immediate:
			\begin{align*}
				\abs{\RN{2}}
				\leqslant \normtyp{ \partial^\alpha K}{L}{2} \normtyp{ \theta}{L}{\infty} \normtyp{ \partial^\alpha a}{L}{2}
				\lesssim \brac{ \normtyp{K}{H}{3} + \normtyp{ \pdt K}{L}{\infty} } \normtyp{\theta}{H}{2} ( \normtyp{a}{H}{1} + \normtyp{ \pdt a}{L}{2} )
			\\
				\lesssim \brac{ \normtyp{K}{H}{3} + \normtyp{ \pdt K}{L}{\infty} } \mathcal{D}_\text{low}.&\qedhere
			\end{align*}
	\end{proof}

	In particular, for our purposes here it suffices to control the low-level interactions in the following way.

	\begin{cor}[Control of the low-level interactions]
	\label{cor:control_low_level_interactions}
		If $M\geqslant 3$ and $\mathcal{E}_M \leqslant 1$ then $\abs{ \overline{\mathcal{I}}_\text{low} } \lesssim \mathcal{E}_M^{1/2} \mathcal{D}_\text{low}$.
	\end{cor}
	\begin{proof}
		Recall that
		$
			\mathcal{E}_M \gtrsim \normtyp{(u, \theta, a)}{P}{2M}^2 + \normtyp{K}{H}{2M-3}^2 + \normtyp{\pdt K}{H}{2M-3}^2.
		$
		In particular, since $M\geqslant 3$ and $ \mathcal{E}_M \leqslant 1$ we may deduce the claim from \fref{Lemma}{lemma:careful_est_low_level_int}.
	\end{proof}

	We now turn our attention to the last piece needed to close the energy estimates at the low level, namely the $\theta$-coercivity estimate recorded below.
	In particular note that below $\theta\uparrow 1$ as $M\uparrow\infty$.

	\begin{lemma}[$\theta$-coercivity]
	\label{lemma:hypocoercivity}
		If $M \geqslant 2$ then $ \overline{\mathcal{E}}_\text{low} \lesssim \overline{\mathcal{E}}_M^{1-\theta} \mathcal{D}_\text{low}^\theta$ where $\theta = \frac{2M-2}{2M-1}$.
	\end{lemma}
	\begin{proof}
		Since the low-level dissipation controls every term in the low-level energy \emph{except} $\norm{a}{H^2}^2$, we rely on an interpolation estimate to control that term
		using the low-level dissipation and the \emph{high-level} energy. More precisely, recall that
		\begin{equation*}
			\left\{
			\begin{aligned}
				\overline{\mathcal{E}}_\text{low} &= \normtyp{\brac{u,\theta}}{H}{2}^2 + \normtyp{a}{H}{2}^2 + \normtyp{\pdt\brac{u,\theta}}{L}{2}^2 + \normtyp{ \pdt a}{L}{2}^2
				\text{ and }\\
				\mathcal{D}_\text{low} &= \normtyp{\brac{u,\theta}}{H}{3}^2 + \normtyp{a}{H}{1}^2 + \normtyp{\pdt\brac{u,\theta}}{H}{1}^2  + \normtyp{\pdt a}{H}{1}^2.
			\end{aligned}
			\right.
		\end{equation*}
		So let us write $ \overline{\mathcal{E}}_\text{low} - \norm{a}{H^2}^2 =\vcentcolon \mathcal{E}_\text{good}$. Then
		\begin{equation*}
			\mathcal{E}_\text{good} \lesssim \mathcal{D}_\text{low} \text{ and }
			\norm{a}{H^2}^2 \lesssim \norm{a}{H^1}^{2\theta} \norm{a}{H^{2M}}^{2(1-\theta)},
			\text{ where } \theta = \frac{2M-2}{2M-1}.
		\end{equation*}
		So finally, since $M\geqslant 2$ we note that $ \overline{\mathcal{E}}_M \gtrsim \mathcal{D}_\text{low}$ and hence we may conclude that, for $\theta$ as above,
		\begin{equation*}
			\overline{\mathcal{E}}_\text{low}
			= \mathcal{E}_\text{good} + \norm{a}{H^2}^2
			\lesssim \mathcal{D}_\text{low} + \mathcal{D}_\text{low}^\theta \overline{\mathcal{E}}_M^{1-\theta}
			\lesssim \mathcal{D}_\text{low}^\theta \overline{\mathcal{E}}_M^{1-\theta}.\qedhere
		\end{equation*}
	\end{proof}

	In light of the $\theta$-coercivity result above we now record a particular instance of the Bihari Lemma which applies to the low level energy.
	This is recorded here in order to streamline the proof of \fref{Proposition}{prop:close_est_low_level} below in which we close the energy estimates at the low level.

	\begin{lemma}
	\label{lemma:Bihari_hypocoercive_energy}
		Suppose that $y : \cobrac{0, \infty} \to \cobrac{0, \infty}$ is continuously differentiable such that
		$
			y' + Cy^{\frac{1}{\theta}} \alpha_0^{1-\frac{1}{\theta}} \leqslant 0
		$
		on $\cobrac{0, \infty}$ for some $\alpha_0, C > 0$ and $\theta \in \brac{0,1}$. Then
		$
			y(t) \leqslant 
			\alpha_0 {\brac{ {\big( \alpha_0 / y(0) \big)}^{1/\beta} + \tilde{C}t }}^{-\beta}
		$
		for $\beta \vcentcolon= {\brac{\frac{1}{\theta}-1}}\inv > 0$ and $\tilde{C} = C\brac{\frac{1}{\theta} - 1} > 0$.
		In particular note that $\beta \uparrow +\infty$ if $\theta\uparrow 1$.
	\end{lemma}
	\begin{proof}
		Integrating in time tells us that
		$
			y(t) + \int_0^t C {y(s)}^{\frac{1}{\theta}} \alpha_0^{1-\frac{1}{\theta}} ds \leqslant y(0).
		$
		We may thus apply the Bihari Lemma (\fref{Lemma}{lemma:Bihari}) with $f(x) = C x^{\frac{1}{\theta}} \alpha_0^{1-\frac{1}{\theta}}$.
		Using the same notation as in Bihari's Lemma we compute $F(x) = \frac{1}{\tilde{C}} {\brac{ \alpha_0 / x }}^{1/\beta}$
		and $F\inv(x) = \alpha_0 {(\tilde{C}x)}^{-\beta}$,
		from which the claim follows.
	\end{proof}

	We conclude this section with \fref{Proposition}{prop:close_est_low_level} below, which performs the synthesis of the results proved in this section in order
	to close the energy estimates at the low level.
	Recall that, as discussed in \fref{Section}{sec:discuss_ap} this is one of the four building blocks of the scheme of a priori estimates.

	\begin{prop}[Closing the energy estimates at the low level]
	\label{prop:close_est_low_level}
		Let $M \geqslant 3$ be an integer.
		There exist $0 < \delta_\text{low}, \delta_\text{low}^* \leqslant 1$ and $C_L > 0$ such that the following holds.
		For any time horizon $T > 0$ and any $0 < \delta \leqslant \delta_\text{low}$, if
		$
			\displaystyle\sup_{0 \leqslant t \leqslant T} \mathcal{E}_M (t) \leqslant \delta_\text{low}^*
		$
		and
		$
			\displaystyle\sup_{0 \leqslant t \leqslant t} \overline{\mathcal{E}}_M (t) \leqslant \delta
		$
		then
		$
			\displaystyle\sup_{0\leqslant t \leqslant T} \mathcal{E}_\text{low} (t) {\brac{1+t}}^{2M-2} \leqslant C_L \delta.
		$
	\end{prop}
	\begin{proof}
		The strategy of the proof is as follows.
		We combine the energy-dissipation relation, the control of the interactions, and the improvement of the dissipation to see that
		$
			\frac{\mathrm{d}}{\mathrm{d}t} \widetilde{\mathcal{E}}_\text{low} + \mathcal{D}_\text{low} \leqslant 0.
		$
		This differential inequality can then be coupled with the $\theta$-coercivity and the improvement of the low level energy to deduce the result.

		More precisely, recall that, by \fref{Lemma}{lemma:record_form_interactions} and \fref{Lemma}{lemma:coercivity_dissip},
		$
			\frac{\mathrm{d}}{\mathrm{d}t} \widetilde{\mathcal{E}}_\text{low} + \overline{\mathcal{D}}_\text{low} \lesssim \overline{\mathcal{I}}_\text{low}.
		$
		Since
		\begin{equation}
		\label{eq:main_a_priori_thm_eq1_clean}
			\sup_{0\leqslant t\leqslant T} \norm{(u,\theta)}{H^3}^2 + \norm{J}{H^3}^2 + \norm{\pdt(u,\theta)}{H^2}^2 + \norm{\pdt J}{H^2}^2
			\leqslant \sup_{0\leqslant t\leqslant T} \mathcal{E}_M (t) \leqslant 1
		\end{equation}
		it follows from \fref{Proposition}{prop:imp_low_level_en_ds} and \fref{Corollary}{cor:control_low_level_interactions} that
		$ \overline{\mathcal{D}}_\text{low} \gtrsim \mathcal{D}_\text{low}$ and $\abs{ \overline{\mathcal{I}}_\text{low} } \lesssim \mathcal{E}_M^{1/2} \mathcal{D}_\text{low}$.
		Therefore there exists $C_{ED} > 0$ such that
		$
			\frac{\mathrm{d}}{\mathrm{d}t} \widetilde{\mathcal{E}}_\text{low} + \mathcal{D}_\text{low} \lesssim C_{ED} \mathcal{E}_M^{1/2} \mathcal{D}_\text{low}.
		$
		In particular if $\delta_\text{low}^* > 0$ is chosen sufficiently small to ensure that $C_{ED} {\brac{\delta_\text{low}^*}}^{1/2} \leqslant \frac{1}{2} $ then
		$
			\frac{\mathrm{d}}{\mathrm{d}t} \widetilde{\mathcal{E}}_\text{low} + \frac{1}{2} \mathcal{D}_\text{low} \leqslant 0.
		$
		Now note that, as a consequence of \eqref{eq:main_a_priori_thm_eq1_clean}, \fref{Proposition}{prop:persist_spec_sols_adv_rot_eqtns} and \fref{Lemma}{lemma:comp_version_en}
		tell us that
		\begin{equation}
		\label{eq:main_a_priori_thm_eq2_clean}
			\frac{c_E}{2} \overline{\mathcal{E}}_\text{low} \leqslant \widetilde{\mathcal{E}}_\text{low} \leqslant \frac{C_E}{2} \overline{\mathcal{E}}_\text{low}.
		\end{equation}
		We may combine this with \fref{Lemma}{lemma:hypocoercivity} to deduce that, for $\theta = \frac{2M-2}{2M-1}$,
		$
			\widetilde{\mathcal{E}}_\text{low} \lesssim \overline{\mathcal{E}}_\text{low} \lesssim \overline{\mathcal{E}}_M^{1-\theta} \mathcal{D}_\text{low}^{\theta}
		$
		and hence there exists a constant $C>0$ such that
		$
			\frac{\mathrm{d}}{\mathrm{d}t} \widetilde{\mathcal{E}}_\text{low} + C \widetilde{\mathcal{E}}_\text{low}^{1/\theta} \delta^{1-\frac{1}{\theta}} \leqslant 0.
		$
		We deduce from \eqref{eq:main_a_priori_thm_eq1_clean} that $\theta$ and $J$ are sufficiently regular
		for $t \mapsto \widetilde{\mathcal{E}}_\text{low} (t)$ to be continuously differentiable.
		Applying \fref{Lemma}{lemma:Bihari_hypocoercive_energy} thus tells us that, for $0 \leqslant t \leqslant T$
		\begin{equation*}
			\widetilde{\mathcal{E}}_\text{low} (t) \lesssim \delta {\brac{ {\brac{ \frac{\delta}{ \widetilde{\mathcal{E}}_\text{low}(0) } }}^{1/\beta} + \tilde{C}t }}^{-\beta}
		\end{equation*}
		for some $\tilde{C} > 0$ and for $\beta = \frac{1}{\frac{1}{\theta} - 1} = 2M-2$.
		Using \eqref{eq:main_a_priori_thm_eq2_clean} once again we note that
		$
			\widetilde{\mathcal{E}}_\text{low}(0) \leqslant \frac{C_E}{2} \overline{\mathcal{E}}_\text{low} (0) \leqslant \frac{C_E \delta}{2}
		$
		and hence $\frac{\delta}{ \widetilde{\mathcal{E}}_\text{low} (0) } \geqslant \frac{2}{C_E}$ such that
		$ \widetilde{\mathcal{E}}_\text{low} \lesssim \delta {\brac{1+t}}^{-(2M-2)}$.
		To conclude this step, note that combining \eqref{eq:main_a_priori_thm_eq1_clean} and \fref{Proposition}{prop:imp_low_level_en_ds} tells us that
		$ \mathcal{E}_\text{low} \lesssim \overline{\mathcal{E}}_\text{low}$ and hence, in light of \eqref{eq:main_a_priori_thm_eq2_clean}, we deduce that
		$
			\mathcal{E}_\text{low} (t) {\brac{1+t}}^{2M-2} \lesssim \delta \text{ for } 0 \leqslant t \leqslant T.\qedhere
		$
	\end{proof}

\subsection{Closing the estimates at the high level}
\label{sec:close_at_high}

	In this section we consider the third of the four building blocks of the scheme of a priori estimates and close the energy estimates at the high level.
	This section is structured similarly, but not identically, to \fref{Section}{sec:close_at_low} above where we close the energy estimates at the \emph{low} level.
	The differences are due to the fact that at the high level the improvements to the dissipation and the estimates of the interactions
	only hold in a time-integrated sense, and not pointwise-in-time as was the case at the low level.
	This means that by contrast with the low level, where the auxiliary estimates relied on product estimates (c.f. \fref{Lemma}{lemma:prod_est})
	here at the high level the auxiliary estimates rely instead on high-low estimates (c.f. \fref{Corollary}{cor:est_interactions_L_2_via_Gagliardo_Nirenberg}).
	Recall that the functionals $ \mathcal{E}_M $ and $ \mathcal{F}_M$, $\overline{\mathcal{K}}_I$ and $ \mathcal{K}_\text{low}$, and $ \overline{\mathcal{D}}_M$, which
	will be used throughout this section, are defined in \eqref{eq:not_E_M_and_F_M}, \eqref{eq:not_K_bar}, and \eqref{eq:not_D_M}, respectively.

	We begin with auxiliary estimates for $a$ whose purpose is to allow the improvement of the high level dissipation.

	\begin{lemma}[Auxiliary estimate for $a$]
	\label{lemma:aux_est_a_high_level}
		Suppose that \eqref{eq:pertub_sys_no_ten_pdt_theta} holds.
		For any $k\in\N$ and any $s > \frac{3}{2}$ such that $s\geqslant k$, we have the estimate
		\begin{align*}
			\norm{a}{H^k}
			\lesssim \brac{
				\normtyp{K}{L}{\infty} + \normtyp{\theta}{L}{\infty} + \normtyp{\nabla\theta}{L}{\infty} + \normtyp{u}{H}{s} + \normtyp{\theta}{H}{s}
			} \normtyp{\brac{u,\theta}}{P}{k+2}
		\\
			+ \brac{
				\normtyp{u}{L}{\infty} + \normtyp{\theta}{L}{\infty} + \normtyp{\pdt\theta}{L}{\infty}
			} \normtyp{K}{H}{k}
		\end{align*}
	\end{lemma}
	\begin{proof}
		This estimate is similar to that of \fref{Lemma}{lemma:aux_est_a_low_level} since we begin by isolating $a$ in \eqref{eq:pertub_sys_no_ten_pdt_theta}.
		However we then use the high-low estimate of \fref{Corollary}{cor:est_interactions_L_2_via_Gagliardo_Nirenberg} instead of the product estimate of \fref{Lemma}{lemma:prod_est}
		to estimate the nonlinearities.
		Note that for cubic terms we perform high-low estimates in the following crude fashion:
		$
			\normtyp{fgh}{H}{k} \lesssim \normtyp{f}{H}{\infty} \normtyp{gh}{H}{k} + \normtyp{f}{H}{k} \normtyp{gh}{L}{\infty},
		$
		and then use \fref{Lemma}{lemma:prod_est} and the fact that $L^\infty$ is a Banach algebra.
	\end{proof}

	We continue our sequence of auxiliary estimates for $a$ with an estimate on its time derivative.

	\begin{lemma}[Auxiliary estimate for $\pdt a$]
	\label{lemma:aux_est_pdt_a_high_level}
		Suppose that \eqref{eq:pertub_sys_pdt_a} holds.
		For any $k\in\N$ and any $s > \frac{3}{2}$ such that $s\geqslant k$, we have the estimate
		\begin{align*}
			\norm{\pdt a}{H^k}
			\lesssim \brac{1 + \norm{u}{H^s} + \norm{\theta}{H^s}} \norm{a}{H^{k+1}}
			+ \brac{1 + \norm{K}{L^\infty}} \norm{\theta}{H^k}
			+ \norm{\theta}{L^\infty} \norm{K}{H^k}.
		\end{align*}
	\end{lemma}
	\begin{proof}
		This follows immediately from isolating $\pdt a$ in \eqref{eq:pertub_sys_pdt_a} and using \fref{Corollary}{cor:est_interactions_L_2_via_Gagliardo_Nirenberg}
		and \fref{Lemma}{lemma:prod_est}.
	\end{proof}

	We conclude our sequence of auxiliary estimates on $a$ with estimates on its higher-order temporal derivatives.

	\begin{lemma}[Auxiliary estimate for $\pdt^j a$]
	\label{lemma:aux_est_pdt_j_a_high_level}
		Suppose that \eqref{eq:pertub_sys_pdt_a} holds. For any $k\in\N$, any $s > \frac{3}{2}$, and any $j\geqslant 1$, if $s\geqslant k$ then
		\begin{equation*}
			\normns{\pdt^j a}{H^k} \lesssim
			\brac{
				1 + \sum_{l=0}^{j-1} \normns{\pdt^l \brac{u, \theta}}{H^s}
			}\brac{
				\sum_{l=0}^{j-1} \normns{\pdt^l a}{H^{k+1}}
			} + \brac{
				1 + \sum_{l=0}^{j-1} \normns{\pdt^l K}{H^s}
			}\brac{
				\sum_{l=0}^{j-1} \normns{\pdt^l \theta}{H^k}
			}.
		\end{equation*}
	\end{lemma}
	\begin{proof}
		This is immediate upon applying $j-1$ temporal derivatives to \eqref{eq:pertub_sys_pdt_a} and using \fref{Lemma}{lemma:prod_est} to estimate the nonlinearities.
	\end{proof}

	With these various auxiliary estimates on $a$ in hand we may now improve the high level dissipation.
	Recall that $ \mathcal{D}_M^a$ is defined in \eqref{eq:not_D_M}.

	\begin{prop}[Improvement of the dissipation at the high level]
	\label{prop:imp_hig_level_ds}
		If $M \geqslant 3$ and $ \mathcal{E}_M \leqslant 1$ then
		$
			\mathcal{D}_M^a \lesssim \overline{\mathcal{D}}_M^{1/2} + \normns{(u, \theta, \pdt\theta)}{L^\infty} \mathcal{F}_M^{1/2}.
		$
	\end{prop}
	\begin{proof}
		For simplicity, we will write $d \defeq \overline{\mathcal{D}}_M^{1/2} + \norm{(u, \theta, \pdt\theta)}{L^\infty} \mathcal{F}_M^{1/2}$
		for the right-hand side of the inequality we are after.
		Since $M\geqslant 3$, since $H^2\brac{\T^3} \hookrightarrow L^\infty\brac{\T^3}$, and since $\mathcal{E}_M \leqslant 1$,
		\fref{Lemma}{lemma:aux_est_a_high_level} tells us that
		$
			\normtyp{a}{H}{2M-1}
			\lesssim d
			$
		and consequently \fref{Lemma}{lemma:aux_est_pdt_a_high_level} says that
		$
			\normtyp{\pdt a}{H}{2M-2}
			\lesssim d.
		$
		To see that
		\begin{equation*}
			\sum_{j=2}^3 \normtypns{\pdt^j a}{H}{2M-j-1} + \sum_{j=4}^M \normtypns{\pdt^j a}{H}{2M-2j+3}
			\lesssim d
		\end{equation*}
		and thus conclude the proof it suffices to prove by induction that
		\begin{equation*}
			\normtypns{\pdt^j a}{H}{k\brac{j}} \lesssim d
			\text{ for }
			k\brac{j} = \begin{cases}
				2M-j-1	&\text{if } j = 2 \text{ or } j=3, \text{ and }\\
				2M-2j+3 &\text{if } j = 4,\, \dots,\, M.
			\end{cases}
		\end{equation*}
		This induction argument is immediate: the base cases $j=0$ and $j=1$ were taken care of above, and the induction step is precisely given by \fref{Lemma}{lemma:aux_est_pdt_j_a_high_level}.
	\end{proof}

	We now turn our attention towards the control of the high level interactions.
	First we record a technical lemma used to control derivatives of $K$.

	\begin{lemma}
	\label{lemma:control_K_in_L_4_few_spatial_derivatives}
		If $\alpha\in\N^{1+3}$ satisfies $\abs{\alpha}_P \leqslant 2M$ and $\abs{\bar{\alpha}}\leqslant 2M-4$ then $ \normtyp{ \partial^\alpha K}{L}{4} \lesssim \mathcal{E}_M^{1/2}$.
	\end{lemma}
	\begin{proof}
		We split the proof into two cases depending on the number of temporal derivatives hitting $K$.
		Suppose first that $\alpha_0 \leqslant 1$. Then the estimate is immediate:
		$
			\normtyp{ \partial_x^{\bar{\alpha}} K}{L}{4} \lesssim \normtyp{K}{H}{1+(2M-4)} \lesssim \mathcal{E}_M^{1/2}
		$
		and
		$
			\normtyp{ \pdt \partial_x^{\bar{\alpha}} K}{L}{4} \lesssim \normtyp{\pdt K}{H}{1+(2M-4)} \lesssim \mathcal{E}_M^{1/2} .
		$
		Suppose now that $\alpha_0 \geqslant 2$. The estimate in this case follows from the fact that 
		$\pdt^j K$ is controlled at parabolic order $2M+1$ when $j\geqslant 2$.
		Therefore, since $1 + \abs{\alpha}_P - 4\leqslant 2M-3$ we may deduce that
		$
			\normtyp{ \partial^\alpha K}{L}{4} \lesssim \normtyp{\pdt^2 K}{P}{1+\abs{\alpha}_P-4} \lesssim \mathcal{E}_M^{1/2} .
		$
	\end{proof}

	We may now state and prove the estimate of the high level interactions.
	Recall that $ \mathcal{D}_M $ and $ \overline{\mathcal{I}}_M$ are defined in \eqref{eq:not_D_M} and \eqref{eq:not_I}, respectively.

	\begin{prop}[Control of the high-order interactions]
	\label{prop:control_high_order_interactions}
		Suppose that $M\geqslant 3$ and that $ \mathcal{E}_M \leqslant 1 $.
		Then
		$
			\abs{ \overline{\mathcal{I}}_M } \lesssim \mathcal{E}_M^{1/2} \mathcal{D}_M + \mathcal{K}_\text{low}^{1/2} \mathcal{F}_M^{1/2} \mathcal{D}_M^{1/2}.
		$
	\end{prop}
	\begin{proof}
		Recall that the interactions are recorded in \fref{Lemma}{lemma:record_form_interactions}.
		There are three difficulties that manifest themselves here.
		\begin{enumerate}
			\item	$a$ appears when hit with a full count of $2M$ spatial derivatives. This is troublesome because the lack of dissipative control of $a$ in $H^{2M}$ is precisely
				why coercivity fails. To handle this it will be necessary to integrate by parts. This issue manifests itself in $\mathcal{I}_6$ and $\mathcal{I}_8$.
			\item	We have poorer spatial regularity control over $\pdt^j K$ when $j$ is small (i.e. $j=0,1$) than when it is large
				-- this is due to the mixed hyperbolic-parabolic nature of the problem.
				There is no particularly clever workaround here besides simply breaking up the estimates into cases depending on the number of temporal derivatives hitting $K$
				and performing the estimates in each case. This manifests itself in $\mathcal{I}_2 - \mathcal{I}_5$, and $\mathcal{I}_8$.
			\item	In all but one of the interactions where $ \mathcal{F}_M $ must be invoked, it's possible growth is counteracted by the presence of $ \overline{\mathcal{K}}_2$.
				However, in $\mathcal{I}_2$ this is not possible. Instead, we may only counteract the growth of $ \mathcal{F}_M $ by $\normns{\pdt^2 \theta}{L^2}$.
				Since two temporal derivatives of $\theta$ are \emph{not} controlled in the low level energy, this term does, at first pass, have any decay.
				Producing such decay will require the auxiliary estimate recorded in \fref{Lemma}{lemma:aux_est_pdt_2_theta} in \fref{Section}{sec:decay_int_norms} below.
		\end{enumerate}
		For the reader's sake, we briefly remark on which interaction terms are discussed in detail and why.

		The details of the estimates of $\mathcal{I}_1$ are provided. This is a simple interaction to estimate but we do use ``hands-on high-low estimates'' which form the basis
		for all the other estimates of the interactions, and thus warrants a detailed discussion of $\mathcal{I}_1$.
		In particular, $\mathcal{I}_7$ is handled in exactly the same way.

		The interactions $\mathcal{I}_2 - \mathcal{I}_5$ are handled in essentially the same way. We only discuss $\mathcal{I}_2$ in detail since it
		has the additional wrinkle of requiring us to invoke $\normns{\pdt^2 \theta}{L^2}$ to counterbalance $ \mathcal{F}_M $.

		The last two interactions we discuss in detail are $\mathcal{I}_6$ and $\mathcal{I}_8$. Those are the most difficult interactions to control
		since they both require us to integrate by parts to get around the appearance of $\nabla_x^{2M} a$. Moreover, temporal derivatives of $K$ appear in $\mathcal{I}_8$,
		which requires us to divide the estimates of that interaction into further subcases.

		We now estimate the difficult interactions one by one.
		\vspace{0.5em}
		\paragraph{\textbf{Estimating $\mathcal{I}_1$ and $\mathcal{I}_7$.}}
			Recall that
			\begin{equation*}
				\mathcal{I}_1 = \sum_{\abs{\alpha}_P \leqslant 2M} \int_{\T^3} \sbrac{u\cdot\nabla, \partial^\alpha} u \cdot \partial^\alpha u
				= - \sum_{\abs{\alpha}_P \leqslant 2M} \sum_{\substack{\beta+\gamma=\alpha\\\beta>0}}
				\binom{\alpha}{\beta} \int_{\T^3} \brac{ \partial^\beta u \cdot \nabla \partial^\gamma u } \cdot \partial^\alpha u.
			\end{equation*}
			Therefore
			\begin{equation*}
				\abs{\mathcal{I}_1} \lesssim
				\sum_{\substack{
					\abs{\beta+\gamma}_P \leqslant 2M
				\\
					\abs{\gamma}_P \leqslant 2M-3
				\\
				}}
				\underbrace{\vbrac{\int_{\T^3} \brac{ \partial^\beta u \cdot \nabla \partial^\gamma u } \cdot \partial^\alpha u}}_{\RN{1}}
				+ \sum_{\substack{
					\abs{\beta+\gamma}_P \leqslant 2M
				\\
					\abs{\beta}_P \leqslant 2M-2
				\\
					\abs{\gamma}_P \leqslant 2M-1
				}}
				\underbrace{\vbrac{\int_{\T^3} \brac{ \partial^\beta u \cdot \nabla \partial^\gamma u } \cdot \partial^\alpha u}}_{\RN{2}}
			\end{equation*}
			where
			\begin{align*}
				\RN{1}
				\leqslant \normns{ \partial^\beta u }{L^4} \norm{ \nabla \partial^\gamma u }{L^\infty} \normns{ \partial^{\beta + \gamma} u}{L^4}
				\lesssim \normns{\partial^\beta u}{H^1} \norm{\nabla \partial^\gamma u}{H^2} \normns{\partial^{\beta+\gamma} u}{H^1}
				\\
				\leqslant \normtyp{u}{P}{\abs{\beta}_P+1} \normtyp{u}{P}{\abs{\gamma}_P+3} \normtyp{u}{P}{\abs{\beta+\gamma}_P+1}
				\lesssim \mathcal{D}_M^{1/2} \mathcal{E}_M^{1/2} \mathcal{D}_M^{1/2}
			\end{align*}
			and
			\begin{align*}
				\RN{2}
				\leqslant \normns{\partial^\beta u}{L^\infty} \norm{\nabla\partial^\gamma u}{L^4} \normns{\partial^{\beta+\gamma}u}{L^4}
				\lesssim \normtyp{u}{P}{\abs{\beta}_P+2} \normtyp{u}{P}{\abs{\gamma}_P+2} \normtyp{u}{P}{\abs{\beta+\gamma}_P +1}
				\lesssim \mathcal{E}_M^{1/2} \mathcal{D}_M^{1/2} \mathcal{D}_M^{1/2}
			\end{align*}
			such that $\abs{\mathcal{I}}_1 \leqslant \mathcal{E}_M^{1/2} \mathcal{D}_M $.

			To control $\mathcal{I}_7$ we proceed in exactly the same way. Note that the presence of $a$ in $\mathcal{I}_7$ is harmless since there is only at most
			one copy of $\nabla_x^{2M} a$ which appears, and hence there is no need to integrate by parts here. Hands-on high-low estimates very similar to those discussed
			in detail above therefore tell us that
			$
				\abs{\mathcal{I}_7} \lesssim \mathcal{E}_M^{1/2} \mathcal{D}_M.
			$
		\vspace{0.5em}
		\paragraph{\textbf{Estimating $\mathcal{I}_2$, $\mathcal{I}_3$, $\mathcal{I}_4$, and $\mathcal{I}_5$.}}
			Recall that $\mathcal{I}_2$ is of particular importance since it is the only terms that requires the incorporation of $\normns{\pdt^2\theta}{L^2}$
			into the decaying functional $ \mathcal{K}_\text{low}$.
			We seek to estimate
			\begin{equation*}
				\mathcal{I}_2 = \sum_{\abs{\alpha}_P \leqslant 2M} \int_{\T^3} \sbrac{J\pdt, \partial^\alpha} \theta \cdot \partial^\alpha\theta
				= - \sum_{\substack{ \abs{\alpha}_P \leqslant 2M \\ \beta+\gamma=\alpha,\, \beta > 0 }}
				\binom{\alpha}{\beta} \int_{\T^3} ( \partial^\beta K) \brac{ \pdt \partial^\gamma \theta} \cdot \brac{ \partial^\alpha \theta}
			\end{equation*}
			where we have used the fact that $ \partial^\beta J = \partial^\beta K$, since $\beta > 0$ and $J$ and $K$ differ by a constant.
			There are two difficulties in handling this term and so we split $\mathcal{I}_2$ accordingly as $\mathcal{I}_2 = \RN{1} + \RN{2}$.
			\begin{enumerate}
				\item	When few temporal derivatives hit $K$ the only way we have of controlling a high number of spatial derivatives is through $ \mathcal{F}_M $.
					Terms concerned by this issue are grouped in $\RN{1}$.
				\item 	The ``better'' terms in $\RN{2}$ are estimated directly.
					However, due to the poorer spatial regularity of $K$ and $\pdt K$ relatively to $\pdt^j K$ for $j\geqslant 2$ we split the estimate of $\RN{2}$ into two pieces
					that are handled differently from one another.
			\end{enumerate}
			To be precise, we write
			\begin{equation*}
				- \mathcal{I}_2 = \sum_{\substack{ \abs{\alpha}_P \leqslant 2M \\ \beta+\gamma = \alpha ,\, \beta > 0  }}
				\binom{\alpha}{\beta} \int_{\T^3} ( \partial^\beta K) \brac{ \pdt \partial^\gamma \theta} \cdot \brac{ \partial^\alpha \theta}
				= \sum_{\substack{ \cdots \\ \abs{\bar{\beta}} \geqslant 2M-3 }} \cdots + \sum_{\substack{ \cdots \\ \abs{\bar{\beta}} \leqslant 2M-4 }} \cdots
				=\vcentcolon \RN{1} + \RN{2}.
			\end{equation*}
			Note that the condition $\abs{\bar{\beta}} \geqslant 2M-3$ in the term $\RN{1}$ is coupled with the usual condition $\abs{\beta}_P \leqslant 2M$,
			and thus requires that $\beta_0 = 0, 1$.
			In other words: only $K$ and $\pdt K$ appear in $\RN{1}$.

			First we estimate $\RN{1}$.
			Two competing factors are at play here.
			\begin{enumerate}
				\item	$ \partial^\beta K$, for $\abs{\bar{\beta}} \geqslant 2M-3$, \emph{must} be controlled in $ \mathcal{F}_M $
					and hence $\pdt \partial^\gamma \theta$ \emph{must} be controlled via decaying factors.
				\item	Since only $\theta$ and $\pdt\theta$ are controlled by $ \mathcal{E}_\text{low}$,
					the decay of $\pdt^2 \theta$ is obtained through \fref{Lemma}{lemma:aux_est_pdt_2_theta}, which only yields control of $\pdt^2 \theta$ in $L^2$.
					We thus have fairly poor control of $\pdt^2 \theta$ through decaying factors.
			\end{enumerate}
			To carefully address this we split $\RN{1}$ into two:
			\begin{equation*}
				\RN{1} = \sum_{\substack{ \cdots \\ \abs{\bar{\beta}} \geqslant 2M-2 }} \cdots + \sum_{\substack{ \cdots \\ \abs{\bar{\beta}} = 2M-3 }} \cdots
				=\vcentcolon \RN{1}_1 + \RN{1}_2.
			\end{equation*}
			To estimate $\RN{1}_1$ we note that $\abs{\gamma}_P = \abs{\alpha}_P - \abs{\beta}_P \leqslant \abs{\alpha}_P - \abs{\bar{\beta}} \leqslant 2M - (2M-2) = 2$.
			Therefore
			\begin{align*}
				\abs{\RN{1}_1}
				&\lesssim \sum_{\cdots} \normtypns{ \partial^\beta K}{L}{4} \normtyp{ \pdt \partial^\gamma \theta}{L}{2} \normtyp{ \partial^\alpha \theta}{L}{4} 
			\\
				&\lesssim \max\brac{ \normtyp{K}{H}{2M+1}, \normtyp{\pdt K}{H}{2M-1} } \max \brac{ \normtyp{\pdt\theta}{H}{2}, \normtyp{ \pdt^2 \theta}{L}{2} }
					\normtyp{\theta}{P}{2M+1} 
			\\
				&\lesssim \mathcal{F}_M^{1/2} \brac{ \overline{\mathcal{K}}_2^{1/2} + \norm{\pdt^2 \theta}{L^2}} \mathcal{D}_M^{1/2}
				\lesssim \mathcal{F}_M^{1/2} \mathcal{K}_\text{low}^{1/2} \mathcal{D}_M^{1/2}.
			\end{align*}
			To estimate $\RN{1}_2$ we note that now $\abs{\gamma}_P \leqslant 3$ such that, since $M\geqslant 3$,
			\begin{align*}
				\abs{\RN{1}_2} \lesssim \sum_{\cdots} \normtypns{ \partial^\beta K}{L}{2} \normtyp{ \pdt \partial^\gamma \theta}{L}{4} \normtyp{ \partial^\alpha \theta}{L}{4} 
				\lesssim \max\brac{ \normtyp{K}{H}{2M-3}, \normtyp{\pdt K}{H}{2M-3} } \normtyp{\theta}{P}{6}  \normtyp{\theta}{P}{2M+1} 
				\lesssim \mathcal{E}_M^{1/2} \mathcal{D}_M.
			\end{align*}
			
			Second we estimate $\RN{2}$. Due to the poorer spatial regularity of $K$ and $\pdt K$ relative to $\pdt^j K$ when $j\geqslant 2$ we split $\RN{2}$ into two:
			\begin{equation*}
				\RN{2} = \sum_{\substack{ \abs{\alpha}_P \leqslant 2M ,\, \abs{\bar{\beta}} \leqslant 2M-4 \\ \beta+\gamma = \alpha ,\, \beta > 0 }}
				\binom{\alpha}{\beta} \int_{\T^3}  ( \partial^\beta K) \brac{  \partial_t \partial^\gamma  \theta} \cdot \brac{ \partial^\alpha \theta}
				= \sum_{\substack{ \cdots \\ \beta_0 \leqslant 1 }} + \sum_{\substack{ \cdots \\ \beta_0 \geqslant 2 }}
				=\vcentcolon \RN{2}_1 + \RN{2}_2.
			\end{equation*}
			Estimating $\RN{2}_1$ is immediate upon recalling that $\abs{\gamma}_P \leqslant 2M-1$ (since $\beta > 0$)
			and using \fref{Lemma}{lemma:control_K_in_L_4_few_spatial_derivatives}:
			\begin{align*}
				\abs{\RN{2}_1}
				\lesssim \sum_{\cdots} \normtypns{ \partial^\beta K}{L}{4} \normtyp{\pdt \partial^\gamma  \theta}{L}{2} \normtyp{ \partial^\alpha \theta}{L}{4} 
				\lesssim \mathcal{E}_M^{1/2} \normtyp{\theta}{P}{2+(2M-1)} \normtyp{\theta}{P}{2M+1} 
				\lesssim \mathcal{E}_M^{1/2} \mathcal{D}_M^{1/2} \mathcal{D}_M^{1/2}.
			\end{align*}
			Estimating $\RN{2}_2$ relies on the crucial observation that $ \normtyp{\pdt^2 K}{P}{2M-3} \lesssim \mathcal{E}_M^{1/2}$.
			The estimate for $\RN{2}_2$ is then immediate:
			\begin{equation*}
				\abs{\RN{2}_2}
				\lesssim \sum_{\cdots} \normtypns{ \partial^\beta K}{L}{4} \normtyp{ \partial_t \partial^\gamma \theta}{L}{2} \normtyp{ \partial^\alpha \theta}{L}{4} 
				\lesssim \normtyp{\pdt^2 K}{P}{1-4+\abs{\beta}_P}  \normtyp{\theta}{P}{2M+1} \normtyp{\theta}{P}{2M+1} 
				\lesssim \mathcal{E}_M^{1/2} \mathcal{D}_M^{1/2} \mathcal{D}_M^{1/2}.
			\end{equation*}
			Putting it all together tells us that
			$
				\abs{\mathcal{I}_2} \lesssim \mathcal{E}_M^{1/2} \mathcal{D}_M + \mathcal{F}_M^{1/2} \mathcal{K}_\text{low}^{1/2} \mathcal{D}_M^{1/2}.
			$

			We may proceed in a similar fashion to estimate $\mathcal{I}_3$, $\mathcal{I}_4$, and $\mathcal{I}_5$, splitting the interactions terms into cases depending on the
			number of temporal derivatives hitting $K$ and using \fref{Lemma}{lemma:control_K_in_L_4_few_spatial_derivatives} where appropriate.
			Proceeding in this fashion we obtain that
			$
				\abs{\mathcal{I}_3} + \abs{\mathcal{I}_4} + \abs{\mathcal{I}_5}
				\lesssim \mathcal{F}_M^{1/2} \overline{\mathcal{K}}_2^{1/2} \mathcal{D}_M^{1/2}
				+ \mathcal{E}_M^{1/2} \mathcal{D}_M.
			$
		\vspace{0.5em}
		\paragraph{\textbf{Estimating $\mathcal{I}_6$.}}
			We seek to estimate
			\begin{equation*}
				\mathcal{I}_6 = \sum_{\abs{\alpha}_P \leqslant 2M} \int_{\T^3} \sbrac{u\cdot\nabla, \partial^\alpha} a \cdot \partial^\alpha a
				= - \sum_{\substack{\beta+\gamma = \alpha\\\beta>0}} \binom{\alpha}{\beta} \int_{\mathbb{T}^3}
				( \partial^\beta u \cdot \nabla \partial^\gamma a) \cdot \partial^\alpha a.
			\end{equation*}
			However, recall that the failure of coercivity manifests itself precisely in the poor dissipative control over $a$.
			We thus treat the case where the multi-index $\alpha$ is purely spatial separately.
			This interaction is particularly troublesome when the derivatives are purely spatial \emph{and} $\abs{\gamma} = 2M-1$
			(note that $\abs{\gamma} = 2M$ is impossible since the conditions $\beta + \gamma = \alpha$ and $\beta > 0$ impose that $\gamma < \alpha$).
			In that case the interaction takes the (schematic) form
			\begin{equation}
			\label{eq:trickiest_term_I_6}
				\int_{\T^3} \brac{ \partial_x u} \brac{\partial_x^{2M} a} \brac{ \partial_x^{2M} a},
			\end{equation}
			where $\partial_x^k$ indicates a derivative $ \partial^\alpha $ for a purely spatial multi-index $\alpha\in\N^3$ of length $\abs{\alpha} = k$.
			This is out of reach of an estimate of the form $\mathcal{E}_M^{1/2} \mathcal{D}_M$ since we only control $a$ in $H^{2M-1}$ dissipatively.
			We thus treat this specific interaction \eqref{eq:trickiest_term_I_6} as a \emph{subcase} of the case of purely spatial derivatives.

			To summarize: (recall that $\alpha = \brac{\alpha_0, \bar{\alpha}} \in \N^{1+3}$)
			\begin{equation*}
				- \mathcal{I}_6 = \sum_{\abs{\alpha}_P \leqslant 2M} \sum_{\substack{ \beta+\gamma=\alpha \\ \beta>0 }}
				\binom{\alpha}{\beta} \int_{\T^3} ( \partial^\beta u \cdot \nabla \partial^\gamma a) \cdot \partial^\alpha a
				= \sum_{\substack{ \abs{\alpha}_P \leqslant 2M \\ \alpha_0 = 0}} \sum_{\cdots} \cdots
				+ \sum_{\substack{ \abs{\alpha}_P \leqslant 2M \\ \alpha_0 \geqslant 1 }} \sum_{\cdots} \cdots
				\eqdef \RN{1} + \RN{2}
			\end{equation*}
			where $\RN{1}$ corresponds to purely spatial derivatives and $\RN{2}$ corresponds to the remaining terms.
			We break up $\RN{1}$ further: (where now $\beta,\gamma\in\N^3$ and \emph{not} $\N^{1+3}$ as above)
			\begin{align*}
				\RN{1} &= \sum_{\substack{ \abs{\beta} + \abs{\gamma} \leqslant 2M \\ \beta>0  }}
				\binom{\beta+\gamma}{\beta} \int_{\T^3} ( \partial^\beta u \cdot \nabla \partial^\gamma a) \cdot \partial^{\beta+\gamma} a
				\\
				&= \sum_{\substack{ \abs{\beta}+\abs{\gamma} \leqslant 2M \\ \abs{\beta} = 1 }}
				\binom{\beta+\gamma}{\beta} \int_{\T^3} ( \partial^\beta u \cdot \nabla \partial^\gamma a) \cdot \partial^{\beta+\gamma} a
				+ \sum_{\substack{ \abs{\beta}+\abs{\gamma} \leqslant 2M \\ \abs{\beta} > 1 }}
				\binom{\beta+\gamma}{\beta} \int_{\T^3} ( \partial^\beta u \cdot \nabla \partial^\gamma a) \cdot \partial^{\beta+\gamma} a
				\eqdef \RN{1}_1 + \RN{1}_2,
			\end{align*}
			where $\RN{1}_1$ consists of the most troublesome term.
			We also break up $\RN{2}$ further:
			\begin{equation*}
				\RN{2}
				= \sum_{\substack{ \abs{\alpha}_P \leqslant 2M \\ \alpha_0 \geqslant 1 }} \sum_{\substack{ \beta + \gamma = \alpha \\ \beta > 0 }}
				\binom{\alpha}{\beta} \int_{\T^3} ( \partial^\beta u \cdot \nabla \partial^\gamma a) \cdot \partial^\alpha a
				= \sum_{\cdots} \sum_{\substack{ \cdots \\ \gamma_0 = 0 }} \cdots
				+ \sum_{\cdots} \sum_{\substack{ \cdots \\ \gamma_0 \geqslant 1 }} \cdots
				\eqdef \RN{2}_1 + \RN{2}_2.
			\end{equation*}
			Again, due to the poorer dissipative control of $a$ compared with $\pdt a$ and higher-order temporal derivatives of $a$,
			we must estimate $\RN{2}_1$ carefully. Note that by ``poorer control'' we mean that we have control of spatial derivatives at a lower parabolic count.
			To be very clear: we control $\norm{a}{H^{2M-1}}$ and $\norm{\pdt a}{H^{2M-2}}$ dissipatively,
			which means that we control $a$ at a parabolic count of $2M-1$ and $\pdt a$ at a parabolic count of $(2M-2) + 2 = 2M$.

			\paragraph{\textbf{Estimating $\RN{1}_1$.}}
			The key is to integrate by parts, at the cost of having to invoke $\mathcal{F}$, which is possibly growing in time, to control $\nabla^{2M+1}a$.
			Then, where for every $\gamma$ we pick $i$ such that $\gamma \geqslant e_i$,
			\begin{align*}
				\RN{1}_1 &= \sum_{\substack{ \abs{\beta}=1 \\ \abs{\gamma} \leqslant 2M-1 }}
				\binom{\beta+\gamma}{\beta}\int_{\T^3} ( \partial^\beta u \cdot \nabla \partial^\gamma a) \cdot \partial^{\beta+\gamma} a
				\\
				 &= - \sum_{\substack{ \abs{\beta} = 1 \\ \abs{\gamma} \leqslant 2M-1 }} \binom{\beta+\gamma}{\gamma} \brac{
					 \int_{\T^3} ( \partial^{\beta+e_i} u \cdot \nabla \partial^{\gamma-e_i} a) \cdot \partial^{\beta+\gamma}a
					 + \int_{\T^3} ( \partial^\beta u \cdot \nabla \partial^{\gamma-e_i} a) \cdot \partial^{\beta+\gamma+e_i}a
				}.
			\end{align*}
			Therefore
			\begin{align*}
				\abs{\RN{1}_1}
				\lesssim \sum_{\abs{\gamma}\leqslant 2M-1}
				\normtypns{\nabla^2 u}{L}{\infty} \normtypns{\nabla \partial^{\gamma-e_i}a}{L}{2} \normtyp{\nabla \partial^{\gamma} a}{L}{2} 
				+ \normtyp{\nabla u}{L}{\infty} \normtypns{\nabla \partial^{\gamma-e_i}a}{L}{2} \normtypns{\nabla \partial^{\gamma+e_i} a}{L}{2}
				\\
				\lesssim \mathcal{D}_2^{1/2} \mathcal{D}_M^{1/2} \mathcal{E}_M^{1/2} + \overline{\mathcal{K}}_2^{1/2} \mathcal{D}_M^{1/2} \mathcal{F}_M^{1/2}.
			\end{align*}

			\paragraph{\textbf{Estimating $\RN{1}_2$, $\RN{2}_1$, and $\RN{2}_2$.}}
			These terms are all estimated using standard ``hands-on high-low estimates'', which yield
			$
				\abs{\RN{1}_2} + \abs{\RN{2}_1} + \abs{\RN{2}_2} \lesssim \mathcal{E}_M^{1/2} \mathcal{D}_M
			$.
			We have thus shown that
			$
				\abs{\mathcal{I}_6} \lesssim \mathcal{E}_M^{1/2} \mathcal{D}_M + \mathcal{F}_M^{1/2} \overline{\mathcal{K}}_2^{1/2} \mathcal{D}_M^{1/2}
			$.
		\vspace{0.5em}
		\paragraph{\textbf{Estimating $\mathcal{I}_8$.}}
			This is the trickiest interaction to estimate since it involves both $2M$ spatial derivatives of $a$ and temporal derivatives of $K$.
			Recall that
			\begin{equation*}
				\mathcal{I}_8 = \sum_{\substack{ \abs{\alpha}_P \leqslant 2M \\ \beta + \gamma = \alpha }}
				\binom{\alpha}{\beta} \int_{\T^3} \partial^\beta \brac{ \brac{ \bar{K} - K_{33}I_2} \bar{\theta}^\perp} \cdot \partial^\alpha a
				\,\sim\,
				\binom{\alpha}{\beta} \int_{\T^3} ( \partial^\beta K) \brac{ \partial^\gamma \bar{\theta}} \cdot \brac{ \partial^\alpha a}
				\eqdef \sum_{\substack{ \abs{\alpha}_P \leqslant 2M \\ \beta+\gamma = \alpha }} \mathcal{I}_8^{\alpha, \beta, \gamma}
			\end{equation*}
			where the left-hand side is the precise form of the interaction and the right-hand side is its schematic form which we will now estimate.
			Note that, by contrast with \emph{all} the other interactions, this one does \emph{not} come from a commutator,
			and therefore there are no restrictions on $\beta$ and $\gamma$ besides
			the fact $\beta + \gamma = \alpha$ (i.e. there is no restriction $\beta > 0$, or equivalently $\gamma < \alpha$).

			To control $\mathcal{I}_8$ we will break it up into several pieces, as summarized pictorially in \fref{Figure}{fig:est_I_8_break_up_clean} below.

			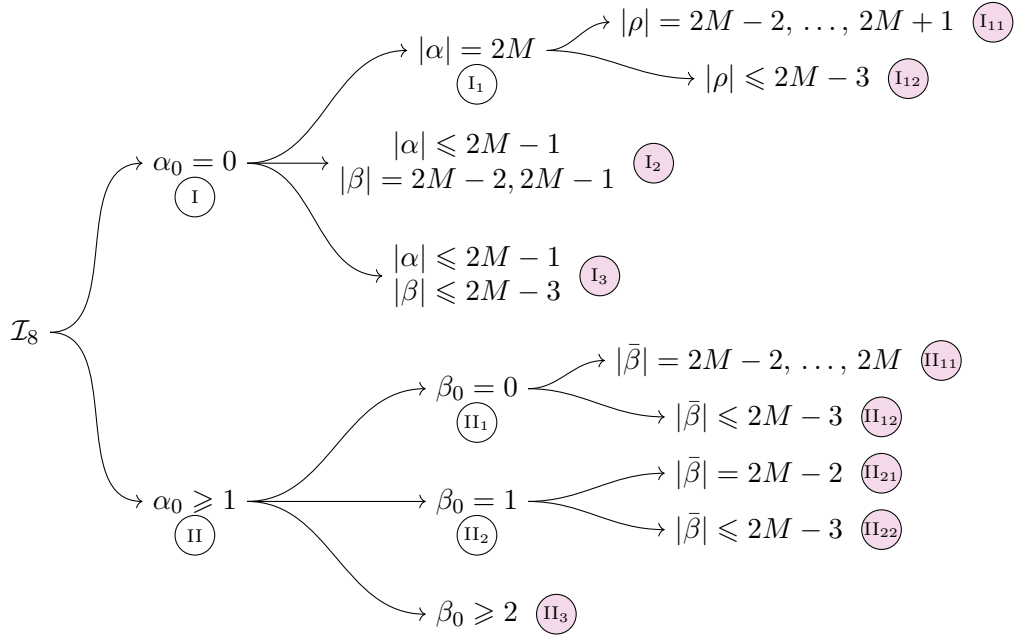
\begin{figure}
				\centering
				\begin{tikzpicture}[scale=0.75]
					% First column
					\node (I8) at (0,0) {$\mathcal{I}_8$};

					% Second column
					\node (I) at (3, 3) {$\alpha_0 = 0$};
					\draw        (3, 3-0.6) circle(10pt);
					\node at     (3, 3-0.6) {\tiny$\RN{1}$};

					\node (II) at (3, -3) {$\alpha_0 \geqslant 1$};
					\draw         (3, -3-0.6) circle(10pt);
					\node at      (3, -3-0.6) {\tiny$\RN{2}$};

					% Arrows from first to second column
					\draw[->] (I8.east) to[out=0, in=180] (I.west);
					\draw[->] (I8.east) to[out=0, in=180] (II.west);

					% Third column
					\node			(I1) at  (8,  5) {$\abs{\alpha} = 2M$};
					\draw				 (8,  5-0.6) circle(10pt);
					\node			     at  (8,  5-0.6) {\tiny$\RN{1}_1$};

					\node[align=center]	(I2) at  (8,  3) {$\abs{\alpha} \leqslant 2M-1$ \\ $\abs{\beta} = 2M-2, 2M-1$};
					\draw				 (I2.east)+(0.5, 0) circle(10pt) [fill=Thistle, fill opacity=0.3, text opacity=1] node {\tiny$\RN{1}_2$};

					\node[align=center]	(I3) at  (8,  1) {$\abs{\alpha} \leqslant 2M-1$ \\ $\abs{\beta} \leqslant 2M-3$};
					\draw				 (I3.east)+(0.5, 0) circle(10pt) [fill=Thistle, fill opacity=0.3, text opacity=1] node {\tiny$\RN{1}_3$};

					\node			(II1) at (8, -1) {$\beta_0 = 0$};
					\draw				 (8, -1-0.6) circle(10pt);
					\node			     at  (8, -1-0.6) {\tiny$\RN{2}_1$};

					\node			(II2) at (8, -3) {$\beta_0 = 1$};
					\draw				 (8, -3-0.6) circle(10pt);
					\node			     at  (8, -3-0.6) {\tiny$\RN{2}_2$};

					\node			(II3) at (8, -5) {$\beta_0 \geqslant 2$};
					\draw				 (II3.east)+(0.5, 0) circle(10pt) [fill=Thistle, fill opacity=0.3, text opacity=1] node {\tiny$\RN{2}_3$};

					% Arrows from second to third column
					\draw[->] (I.east)  to[out=0, in=180] (I1.west);
					\draw[->] (I.east)  to                (I2.west);
					\draw[->] (I.east)  to[out=0, in=180] (I3.west);
					\draw[->] (II.east) to[out=0, in=180] (II1.west);
					\draw[->] (II.east) to                (II2.west);
					\draw[->] (II.east) to[out=0, in=180] (II3.west);

					% Fourth column
					\node (I11) at (13.5, 5.5) {$\abs{\rho} = 2M-2,\, \dots,\, 2M+1$};
					\node (I12) at (13.5, 4.5) {$\abs{\rho} \leqslant 2M-3$};
					\node (II11) at (13, -0.5) {$\abs{\bar{\beta}} = 2M-2,\, \dots,\, 2M$};
					\node (II12) at (13, -1.5) {$\abs{\bar{\beta}} \leqslant 2M-3$};
					\node (II21) at (13, -2.5) {$\abs{\bar{\beta}} = 2M-2$};
					\node (II22) at (13, -3.5) {$\abs{\bar{\beta}} \leqslant 2M-3$};

					% Arrows from third to fourth column
					\draw[->] (I1.east) to[out=0, in=180] (I11.west);
					\draw	  (I11.east)+(0.5, 0) circle(10pt) [fill=Thistle, fill opacity=0.3, text opacity=1] node {\tiny$\RN{1}_{11}$};

					\draw[->] (I1.east) to[out=0, in=180] (I12.west);
					\draw	  (I12.east)+(0.5, 0) circle(10pt) [fill=Thistle, fill opacity=0.3, text opacity=1] node {\tiny$\RN{1}_{12}$};

					\draw[->] (II1.east) to[out=0, in=180] (II11.west);
					\draw	  (II11.east)+(0.5, 0) circle(10pt) [fill=Thistle, fill opacity=0.3, text opacity=1] node {\tiny$\RN{2}_{11}$};

					\draw[->] (II1.east) to[out=0, in=180] (II12.west);
					\draw	  (II12.east)+(0.5, 0) circle(10pt) [fill=Thistle, fill opacity=0.3, text opacity=1] node {\tiny$\RN{2}_{12}$};

					\draw[->] (II2.east) to[out=0, in=180] (II21.west);
					\draw	  (II21.east)+(0.5, 0) circle(10pt) [fill=Thistle, fill opacity=0.3, text opacity=1] node {\tiny$\RN{2}_{21}$};

					\draw[->] (II2.east) to[out=0, in=180] (II22.west);
					\draw	  (II22.east)+(0.5, 0) circle(10pt) [fill=Thistle, fill opacity=0.3, text opacity=1] node {\tiny$\RN{2}_{22}$};

				\end{tikzpicture}
				\caption{How the terms in the interaction $\mathcal{I}_8$ are broken up in order to be estimated.}
				\label{fig:est_I_8_break_up_clean}
			\end{figure}

			More precisely, we now detail how we go about breaking up $\mathcal{I}_8$.
			First we separate terms for which the derivatives of $a$ are purely spatial:
			\begin{equation*}
				\mathcal{I}_8
				= \sum_{\substack{ \abs{\alpha}_P \leqslant 2M \\ \beta+\gamma = \alpha }} \mathcal{I}_8^{\alpha, \beta, \gamma}
				= \sum_{\substack{ \cdots \\ \alpha_0 = 0 }} \mathcal{I}_8^{\alpha, \beta, \gamma}
				+ \sum_{\substack{ \cdots \\ \alpha_0 \geqslant 1 }} \mathcal{I}_8^{\alpha, \beta, \gamma}
				\eqdef \RN{1} + \RN{2}.
			\end{equation*}
			We now split $\RN{1}$ to account for two factors: the lack of dissipative control of $\partial_x^{2M} a$
			and the poorer (i.e. through $\mathcal{F}_M$ and not $\mathcal{E}_M$) control of $K$
			when many spatial derivatives are applied to it.
			Recall that $\partial_x^k$ indicates a derivative $ \partial^\alpha $ for a purely spatial multi-index $\alpha\in\N^3$ of length $\abs{\alpha} = k$.
			\begin{equation*}
				\RN{1} = \sum_{\substack{ \alpha_0 = 0 \\ \abs{\alpha} \leqslant 2M \\ \beta+\gamma = \alpha }} \mathcal{I}_8^{\alpha, \beta, \gamma}
				= \sum_{\substack{ \cdots \\ \abs{\alpha} = 2M }} \mathcal{I}_8^{\alpha, \beta, \gamma}
				+ \sum_{\substack{ \cdots \\ \abs{\alpha} \leqslant 2M-1 \\ \abs{\beta} = 2M-2, 2M-1 }} \mathcal{I}_8^{\alpha, \beta, \gamma}
				+ \sum_{\substack{ \cdots \\ \abs{\alpha} \leqslant 2M-1 \\ \abs{\beta} \leqslant 2M-3 }}
				\eqdef \RN{1}_1 + \RN{1}_2 + \RN{1}_3.
			\end{equation*}
			Estimating $\RN{1}_1$ is the trickiest part of estimating $\mathcal{I}_8$ since we \emph{must} control $K$ via the energy and do \emph{not}
			have control of $\partial_x^{2M} a$ via the dissipation.
			To get around this issue we integrate by parts: (below $i$ is an index dependent on $\alpha$ chosen such that $\alpha_i \geqslant 0$,
			which may always be done since $\alpha \neq 0$)
			\begin{align*}
				\RN{1}_1 
				= - \sum_{\substack{ \alpha_0 = 0 \\ \abs{\alpha} = 2M \\ \beta+\gamma = \alpha }} \binom{\alpha}{\beta} \brac{
					\int_{\T^3} ( \partial^{\beta + e_i} K ) ( \partial^\gamma \bar{\theta} ) \cdot ( \partial^{\alpha-e_i} a)
					+ \int_{\T^3} (\partial^\beta K) (\partial^{\gamma+e_i} \bar{\theta}) \cdot ( \partial^{\alpha-e_i} a)
				}
			\end{align*}
			which allows us to split $\RN{1}_1$ in the following way
			\begin{equation*}
				\abs{\RN{1}_1} \lesssim \vbrac{
					\sum_{\substack{ \abs{\pi} = 2M-1 \\ \abs{\rho+\sigma} = 2M+1 }}
					\int_{\mathbb{T}^3} \brac{\partial^\rho K} \brac{\partial^\sigma \bar{\theta}} \cdot \brac{\partial^\pi a}
				}
				\leqslant \vbrac{ \sum_{\substack{ \cdots \\ \abs{\rho} = 2M-2,\, \dots,\, 2M+1 }} \cdots }
				+ \vbrac{ \sum_{\substack{ \cdots \\ \abs{\rho} \leqslant 2M-3 }} \cdots }
				\eqdef \abs{\RN{1}_{11}} + \abs{\RN{1}_{12}}.
			\end{equation*}
			where $\pi$, $\rho$, and $\sigma$ are \emph{spatial} multi-indices, i.e. they belong to $\N^3$ and not $\N^{1+3}$.
			We now direct our attention to $\RN{2}$ (which is easier to handle than $\RN{1}$ since more temporal derivatives are involved).
			The key in the splitting here is that things get easier as more temporal derivatives of $K$ are involved.
			\begin{equation*}
				\RN{2} = \sum_{\substack{ \alpha_0 \geqslant 1 \\ \abs{\alpha}_P \leqslant 2M \\ \beta+\gamma = \alpha }} \mathcal{I}_8^{\alpha, \beta, \gamma}
				= \sum_{\substack{ \cdots \\ \beta_0 = 0 }} \mathcal{I}_8^{\alpha, \beta, \gamma}
				+ \sum_{\substack{ \cdots \\ \beta_0 = 1 }} \mathcal{I}_8^{\alpha, \beta, \gamma}
				+ \sum_{\substack{ \cdots \\ \beta_0 \geqslant 2 }} \mathcal{I}_8^{\alpha, \beta, \gamma}
				\eqdef \RN{2}_1 + \RN{2}_2 + \RN{2}_3.
			\end{equation*}
			We finally split $\RN{2}_1$ and $\RN{2}_2$ further depending on the number of spatial derivatives hitting $K$
			(since this determines whether we estimate the factor involving $K$
			using $\mathcal{E}_M$ or $\mathcal{F}_M$).
			\begin{equation*}
				\RN{2}_1 = \sum_{\substack{ \alpha_0 \geqslant 1 \\ \beta_0 = 0 \\ \abs{\alpha}_P \leqslant 2M \\ \beta + \gamma = \alpha }} \mathcal{I}_8^{\alpha, \beta, \gamma}
				= \sum_{\substack{ \cdots \\ \abs{\bar{\beta}} = 2M-2,\, \cdots,\, 2M }} \mathcal{I}_8^{\alpha, \beta, \gamma}
				+ \sum_{\substack{ \cdots \\ \abs{\bar{\beta}} \leqslant 2M-3 }} \mathcal{I}_8^{\alpha, \beta, \gamma}
				\eqdef \RN{2}_{11} + \RN{2}_{12}
			\end{equation*}
			and
			\begin{equation*}
				\RN{2}_2 = \sum_{\substack{ \alpha_0 \geqslant 1 \\ \beta_0 = 1 \\ \abs{\alpha}_P \leqslant 2M \\ \beta + \gamma = \alpha }} \mathcal{I}_8^{\alpha, \beta, \gamma}
				= \sum_{\substack{ \cdots \\ \abs{\bar{\beta}} = 2M-2}} \mathcal{I}_8^{\alpha, \beta, \gamma}
				+ \sum_{\substack{ \cdots \\ \abs{\bar{\beta}} \leqslant 2M-3 }} \mathcal{I}_8^{\alpha, \beta, \gamma}
				\eqdef \RN{2}_{21} + \RN{2}_{22}.
			\end{equation*}

			Having carefully split $\mathcal{I}_8$ into appropriate pieces, we now proceed to estimating each of these pieces.
			Note that due this extensive sub-division of $\mathcal{I}_8$ into various pieces, most terms can be handled with similar techniques.
			The table below summarizes how each term is handled.
			\begin{center}
			\begin{tabular}{|c|ccccccccc|}
				\hline
							&$\RN{1}_{11}$	&$\RN{1}_{12}$	&$\RN{1}_{2}$	&$\RN{1}_{3}$	&$\RN{2}_{11}$	&$\RN{2}_{12}$	&$\RN{2}_{21}$	&$\RN{2}_{22}$	&$\RN{3}$\\
				\hline
				Direct			&		&		&X		&X		&		&		&X		&X		&\\
				Hands-on high-low	&		&		&		&		&		&X		&		&		&\\
				\fref{Corollary}{cor:est_interactions_L_2_via_Gagliardo_Nirenberg}
							&X		&X		&		&		&		&		&		&		&\\
				Special consideration	&		&		&		&		&X		&		&		&		&X
				\\\hline
			\end{tabular}
			\end{center}
			\paragraph{\textbf{Direct estimates.}} The terms $\RN{1}_3$, $\RN{2}_{21}$, and $\RN{2}_{22}$ can be estimated directly:
			\begin{align*}
				\abs{\RN{1}_3}
					&\lesssim \sum_{\cdots} \normtypns{ \partial^\beta K}{L}{2} \normtypns{ \partial^\gamma \bar{\theta}}{L}{\infty} \normtyp{ \partial^\alpha a}{L}{2}
					\lesssim \normtyp{K}{H}{2M-3} \normtypns{\bar{\theta}}{H}{2M+1} \normtyp{a}{H}{2M-1}
					\lesssim \mathcal{E}_M^{1/2} \mathcal{D}_M^{1/2} \mathcal{D}_M^{1/2},
			\\
				\abs{\RN{2}_{21}}
					&\lesssim \sum_{\cdots} \normtypns{\pdt\partial_x^{\bar{\beta}}K}{L}{2} \normtyp{\theta}{L}{\infty} \normtyp{ \partial^\alpha a}{L}{2}
					\lesssim \normtyp{\pdt K}{H}{2M-2} \normtyp{\theta}{H}{2} \normtyp{\pdt a}{P}{2M-2}
					\lesssim \mathcal{F}_M^{1/2} \overline{\mathcal{K}}_1^{1/2} \mathcal{D}_M^{1/2}, \text{ and } 
			\\
				\abs{\RN{2}_{22}}
					&\lesssim \sum_{\cdots} \normtypns{ \partial^\beta K}{L}{2} \normtypns{ \partial^\gamma \bar{\theta}}{L}{\infty} \normtyp{ \partial^\alpha a}{L}{2} 
					\lesssim \normtyp{\pdt K}{H}{2M-3} \normtypns{\bar{\theta}}{P}{2M} \normtyp{\pdt a}{P}{2M-2} 
					\lesssim \mathcal{E}_M^{1/2} \mathcal{D}_M^{1/2} \mathcal{D}_M^{1/2}.
			\end{align*}
			Similarly, $\RN{1}_2$ can be split into precisely three terms which can all be estimated directly:
			\begin{align*}
				\RN{1}_2 = \sum_{\substack{ \alpha_0 = 0 ,\, \alpha=\beta \\ \abs{\alpha}=2M-1 }}
				\binom{\alpha}{\beta} \int_{\T^3} (\partial^\beta K) \bar{\theta} \cdot \brac{ \partial^\alpha a}
				+ \sum_{\substack{ \alpha_0 = 0 ,\, \abs{\alpha}=2M-1 \\ \abs{\beta} = 2M-2 ,\, \gamma = \alpha - \beta,\, \abs{\gamma}_P = 1 }}
				\binom{\alpha}{\beta}\int_{\T^3} (\partial^\beta K) \brac{ \partial^\gamma\bar{\theta}} \cdot \brac{ \partial^\alpha a}
			\\
				+ \sum_{\substack{ \alpha_0 = 0 ,\, \alpha=\beta \\ \abs{\alpha}=2M-2}}
				\binom{\alpha}{\beta} \int_{\T^3} (\partial^\beta K) \bar{\theta} \cdot \brac{ \partial^\alpha a}
			\end{align*}
			such that
			\begin{align*}
				\abs{\RN{1}_2} &\lesssim
				\norm{\nabla^{2M-1}K}{L^\infty} \norm{\bar{\theta}}{L^2} \norm{\nabla^{2M-1}a}{L^2}
				+ \norm{\nabla^{2M-2}K}{L^\infty} \norm{\bar{\nabla \theta}}{L^2} \norm{\nabla^{2M-1}a}{L^2}
				\\
				&\quad+ \norm{\nabla^{2M-2}K}{L^\infty} \norm{\bar{\theta}}{L^2} \norm{\nabla^{2M-2}a}{L^2}
				\lesssim \normtyp{K}{H}{2M+1} \normtyp{\bar{\theta}}{H}{1} \normtyp{a}{H}{2M-1}
				\lesssim \mathcal{F}_M^{1/2} \overline{\mathcal{K}}_1^{1/2} \mathcal{D}_M^{1/2}.
			\end{align*}

			\paragraph{\textbf{Hands-on high-low estimates.}} $\RN{2}_{12}$ can be estimated using ``hands-on high-low estimates''.
			We may split $\RN{2}_{12}$ as
			\begin{equation*}
				\RN{2}_{12} = \sum_{\substack{ \alpha_0 \geqslant 1 ,\, \beta_0 = 0 \beta + \gamma = \alpha \\ \abs{\alpha}_P \leqslant 2M ,\, \abs{\bar{\beta}} \leqslant 2M-3}}
				\binom{\alpha}{\beta}\int_{\T^3} (\partial^\beta K) \brac{ \partial^\gamma\bar{\theta}} \cdot \brac{ \partial^\alpha a}.
				\leqslant
					\sum_{\substack{ \cdots \\ \abs{\bar{\beta}} \leqslant 2M-5 }} \underbrace{\abs{\cdots}}_{(1)}
					+ \sum_{\substack{ \cdots \\ \abs{\gamma}_P \leqslant 2M-1 }} \underbrace{\abs{\cdots}}_{(2)}
			\end{equation*}
			where
			\begin{equation*}
				(1) + (2)
				\leqslant \normtypns{\partial^{\bar{\beta}} K}{L}{\infty} \normtyp{ \partial^\gamma \bar{\theta}}{L}{2} \normtyp{ \partial^\alpha a}{L}{2}
					+ \normtypns{\partial^{\bar{\beta}} K}{L}{2} \normtyp{ \partial^\gamma \bar{\theta}}{L}{\infty} \normtyp{ \partial^\alpha a}{L}{2}
				\lesssim \mathcal{E}_M^{1/2} \mathcal{D}_M.
			\end{equation*}

			\paragraph{\textbf{Using \fref{Corollary}{cor:est_interactions_L_2_via_Gagliardo_Nirenberg}.}} The terms $\RN{1}_{11}$ and $\RN{1}_{12}$ can be estimated
			using \fref{Corollary}{cor:est_interactions_L_2_via_Gagliardo_Nirenberg}, which provides a way to use the Gagliardo-Nirenberg inequality to obtain
			bounds on products of derivatives in $L^2$.
			Indeed:
			\begin{align*}
				\abs{\RN{1}_{11}}
				&\lesssim \sum_{\substack{ p = 2M-1 \\ r = 2M-2,\, \cdots,\, 2M+1 \\ r+s = 2M+1 }}
					C(p, r) \int_{\T^3} \abs{\nabla^r K} \abs{\nabla^s \bar{\theta}} \abs{\nabla^p a}
				\lesssim \sum_{\cdots} \norm{\;\abs{\nabla^{r-(2M-2)}\nabla^{2M-2}K}\;\abs{\nabla^s\bar{\theta}}\;}{L^2} \norm{\nabla^p a}{L^2}
				\\
				&\lesssim \sum_{\substack{ r = 2M-2,\, \dots,\, 2M+1 \\ r+s = 2M+1 }} \brac{
					\norm{\nabla^{2M-2}K}{L^\infty} \norm{\bar{\theta}}{H^{r+s-(2M-2)}}
					+ \norm{\nabla^{2M-2}K}{H^{r+s-(2M-2)}} \norm{\bar{\theta}}{L^\infty}
				} \norm{a}{H^{2M-1}}
				\\
				&\lesssim \brac{
					\normtyp{K}{H}{2M} \normtyp{\bar{\theta}}{H}{3} + \normtyp{K}{H}{2M+1} \normtyp{\bar{\theta}}{H}{2} 
				} \norm{a}{H^{2M-1}}
				\lesssim \normtyp{K}{H}{2M+1} \normtyp{\bar{\theta}}{H}{3} \normtyp{a}{H}{2M-1} 
			\end{align*}
			such that $\abs{\RN{1}_{11}} \lesssim \mathcal{F}_M^{1/2} \overline{\mathcal{K}}_2^{1/2} \mathcal{D}_M^{1/2}$
			and, similarly,
			\begin{align*}
				\abs{\RN{1}_{12}}
				\leqslant \sum_{\substack{ p = 2M-1 \\ r+s = 2M+1 \\ r \leqslant 2M-3 }}
					C(p, r) \int_{\T^3} \abs{\nabla^r K}\abs{\nabla^s \bar{\theta}} \abs{\nabla^p a}
				\lesssim \sum_{\cdots} \norm{\;\abs{\nabla^r K}\;\abs{\nabla^{s-4}\nabla^4 \bar{\theta}}\;}{L^2} \norm{\nabla^p a}{L^2}
				\lesssim \mathcal{E}_M^{1/2} \mathcal{D}_M.
			\end{align*}

			\paragraph{\textbf{Special consideration.}}
			Finally we estimate $\RN{2}_{11}$ and $\RN{2}_3$.
			Recall that
			\begin{equation*}
				\RN{2}_{11} = \sum_{\substack{ \alpha_0 \geqslant 1,\, \beta_0 = 0,\, \abs{\alpha}_P \leqslant 2M \\ \abs{\bar{\beta}} = 2M-2,\, \dots,\, 2M,\, \beta+\gamma = \alpha }}
				\binom{\alpha}{\beta} \int_{\T^3} ( \partial^\beta K) \brac{ \partial^\gamma \bar{\theta} } \cdot \brac{ \partial^\alpha \alpha}.
			\end{equation*}
			There are two important observations to make here.
			\begin{itemize}
				\item	Since $\alpha_0 \geqslant 1$ and $\abs{\alpha}_P \leqslant 2M$ we control $ \partial^\alpha a $ in $ \mathcal{D}_M $ in the following way:
					\begin{equation*}
						\norm{ \partial^\alpha a}{L^2} \lesssim \norm{\pdt a}{P^{2M-2}} \lesssim \mathcal{D}_M^{1/2}.
					\end{equation*}
				\item 	We must use $ \mathcal{F}_M $ to control $ \partial^\beta K$ when $\beta_0 = 0$ and $\abs{\bar{\beta}} \geqslant 2M-2$
					so we therefore ask as little regularity as possible of $\bar{\theta}$ (to invoke $\mathcal{E}_I$ for the smallest possible $I$).
					We thus split the estimate depending on whether or not we can control $ \partial^\beta K$ in $L^\infty$ via $\mathcal{F}_M$.
			\end{itemize}
			We obtain:
			\begin{align*}
				\abs{\RN{2}_{11}}
				\lesssim \sum_{\substack{ \cdots \\ \abs{\bar{\beta}} = 2M-2, 2M-1 }} \cdots + \sum_{\substack{ \cdots \\ \abs{\bar{\beta}} = 2M }} \cdots
				\lesssim \sum_{\cdots} \normtypns{ \partial^\beta K}{L}{\infty} \normtypns{ \partial^\gamma\bar{\theta}}{L}{2} \normtyp{ \partial^\alpha a}{L}{2}
				+ \sum_{\cdots} \normtypns{ \partial^\beta K}{L}{4} \normtypns{\bar{\theta}}{L}{4} \normtyp{ \partial^\alpha a}{L}{2}
				\\
				\lesssim \normtyp{K}{H}{2M+1} \normtypns{\bar{\theta}}{P}{2} \normtyp{\pdt a}{P}{2M-2}
				+ \normtyp{K}{H}{2M+1} \normtypns{\bar{\theta}}{H}{1} \normtyp{\pdt a}{P}{2M-2}
				\lesssim \mathcal{F}_M^{1/2} \overline{\mathcal{K}}_1^{1/2} \mathcal{D}_M^{1/2}.
			\end{align*}
			
			Now we estimate $\RN{2}_3$. Recall that
			\begin{equation*}
				\RN{2}_3 = \sum_{\substack{ \alpha_0 \geqslant 1,\, \abs{\alpha}_P \leqslant 2M \\ \beta_0 \geqslant 2,\, \beta+\gamma = \alpha }}
				\binom{\alpha}{\beta} \int_{\T^3} ( \partial^\beta K) \brac{ \partial^\gamma\bar{\theta}} \cdot \brac{ \partial^\alpha a}.
			\end{equation*}
			The key observation is the following: when $\beta_0 \geqslant 2$ we control $\pdt^{\beta_0} K$ at parabolic order $2M+1$.
			Consequently: if $\abs{\beta}_P \leqslant 2M-1$ then we control $ \partial^\beta K$ in $L^\infty$ via $\mathcal{E}_M$
			since
			\begin{equation*}
				\normns{ \partial^\beta K}{L^\infty} \lesssim \normtypns{\pdt^2 K}{P}{2+\abs{\beta}_P-4} \lesssim \normtypns{\pdt^2 K}{P}{2M-3} \lesssim \mathcal{E}_M^{1/2}.
			\end{equation*}
			We may then estimate $\RN{2}_3$ with the usual ``hands-on high-low'' estimates. We note that, since $\abs{\alpha}_P\leqslant 2M$ and $\beta + \gamma = \alpha$
			it follows from \fref{Corollary}{cor:est_interactions_L_2_via_Gagliardo_Nirenberg} that, as long as $M\geqslant 1$,
			either $\abs{\beta}_P \leqslant 2M-1$ or $\abs{\gamma}_P \leqslant 2M-1$.
			Therefore
			\begin{equation*}
				\abs{\RN{2}_3}
				\lesssim \sum_{\substack{ \cdots \\ \abs{\beta}_P \leqslant 2M-1 }} \underbrace{\abs{\cdots}}_{(1)}
				+ \sum_{\substack{ \cdots \\ \abs{\gamma}_P \leqslant 2M-1 }} \underbrace{\abs{\cdots}}_{(2)}
			\end{equation*}
			where
			\begin{equation*}
				(1) + (2) \leqslant \normtypns{ \partial^\beta K}{L}{\infty} \normtypns{ \partial^\gamma \bar{\theta}}{L}{2} \normtyp{ \partial^\alpha a}{L}{2}
					+ \normtypns{ \partial^\beta K}{L}{2} \normtypns{ \partial^\gamma \bar{\theta}}{L}{\infty} \normtyp{ \partial^\alpha a}{L}{2}
					\lesssim \mathcal{E}_M^{1/2} \mathcal{D}_M.
			\end{equation*}

			Putting all these estimates together we see that have obtained that
			$
				\abs{\mathcal{I}_8} \lesssim \mathcal{E}_M^{1/2} \mathcal{D}_M + \mathcal{F}_M^{1/2} \overline{\mathcal{K}}_2^{1/2} \mathcal{D}_M^{1/2}.
			$
	\end{proof}

	As mentioned at the beginning of this section, at the high level both the improvement to the dissipation and the control of the interaction only allows us
	to close the energy estimates in a \emph{time-integrated} fashion. This is because the closure of the estimates relies crucially on playing the potential growth
	of $ \mathcal{F}_M $ against the decay of intermediate norms $ \overline{\mathcal{K}}_I $ and $\mathcal{K}_\text{low}$.
	The next two results record precisely this balancing act between growth and decay.
	First we consider the growth-decay interactions arising from the improvement of the dissipation.

	\begin{lemma}
	\label{lemma:control_growth_decay_interactions_imp_dissipation}
		Suppose that $M \geqslant 3$ and that for some time horizon $T>0$
		\begin{equation}
		\label{eq:lemma_est_growth_decay_imp_D_second_pass_assumption}
			\sup_{1\leqslant I \leqslant M} \sup_{0\leqslant t\leqslant T} \overline{\mathcal{K}}_I (t) {\brac{1+t}}^{2M-2I} =\vcentcolon C_0 < \infty,
		\end{equation}
		and, for every $0 \leqslant t \leqslant T$,
		\begin{equation*}
			\mathcal{F}_M^{1/2}  (t) \lesssim \alpha_0^{1/2} + \beta_0^{1/2} \int_0^t \overline{\mathcal{D}}_M^{1/2} (s) ds + \gamma_0^{1/2} \overline{\mathcal{K}}_M^{1/2} (t),
		\end{equation*}
		for some $\alpha_0, \beta_0, \gamma_0 > 0$.
		Then, for every $0\leqslant t \leqslant T$,
		\begin{equation*}
			\int_0^t \norm{\brac{u,\theta,\pdt\theta}}{L^\infty}^2 \mathcal{F}_M
			\lesssim \alpha_0 C_0 + (\beta_0 + \gamma_0) C_0 \int_0^t \overline{\mathcal{D}}_M (s) ds.
		\end{equation*}
	\end{lemma}
	\begin{proof}
		First we note that by interpolating \eqref{eq:lemma_est_growth_decay_imp_D_second_pass_assumption} we may obtain decay estimates for fractional Sobolev norms
		-- this is very similar to what was done by interpolation in Step 1 of \fref{Proposition}{prop:est_K}.
		Indeed, for any $s\in\R$ which satisfies $2\leqslant s \leqslant 2M$ we may pick $\sigma = \frac{2M-s}{2M-2}$ and deduce that
		\begin{equation*}
			\normtyp{\brac{u,\theta}}{H}{s}^2 \lesssim \normtyp{\brac{u,\theta}}{H}{2}^{2\sigma} \normtyp{\brac{u,\theta}}{H}{2M}^{2(1-\theta)}
			\leqslant C_0^\sigma {\brac{1+t}}^{(2M-2)\sigma} C_0^{1-\sigma}
			= C_0 {\brac{1+t}}^{2M-s}
		\end{equation*}
		and, similarly, $\norm{\pdt\brac{u,\theta}}{H^s}^2 \lesssim C_0 {\brac{1+t}}^{2M-2-s}$.
		Crucially: $s = \frac{7}{4}$ satisfies both $s > \frac{3}{2}$, such that $H^{\frac{7}{4}}$ embeds continuously into $L^\infty$, and $2M-2-s > 2$ (since $M\geqslant 3$),
		such that the resulting decaying bound is integrable.
		The term $\norm{\brac{u,\theta,\pdt\theta}}{L^\infty}^2$ may then be shown to decay fast enough, in the space $H^{7/4}$, to justify the estimate.

		Using Cauchy-Schwarz on $\int \overline{\mathcal{D}}_M^{1/2}$ thus tells us that
		\begin{align*}
			&\int_0^t \normtyp{ (u, \theta, \pdt\theta) }{L}{\infty}^2 \mathcal{F}_M
			\lesssim \int_0^t \brac{
				\frac{C_0}{ {\brac{1+s}}^{2M-2-\frac{7}{4}} }
			}\brac{
				\alpha_0 + \beta_0 {\brac{
					\int_0^s \overline{\mathcal{D}}_M^{1/2} (r) dr
				}}^2 ds
			}
			+ \gamma_0 C_0 \int_0^t \overline{\mathcal{D}}_M
		\\
			&\lesssim \alpha_0 C_0 + \beta_0 C_0 \int_0^t \frac{s}{ {\brac{1+s}}^{2M-\frac{15}{4} }} \brac{ \int_0^s \overline{\mathcal{D}}_M (r) dr } ds
			+ \gamma_0 C_0 \int_0^t \overline{\mathcal{D}}_M
			\lesssim \alpha_0 C_0 + (\beta_0 + \gamma_0) C_0 \int_0^t \overline{\mathcal{D}}_M (s) ds.
		\end{align*}
	\end{proof}

	Now we consider the growth-decay interactions arising from the control of the high level interactions.

	\begin{lemma}
	\label{lemma:control_growth_decay_interactions_est_interactions}
		Suppose that $M\geqslant 4$, that
		\begin{equation*}
			\mathcal{F}_M^{1/2} (t) \lesssim \alpha_0^{1/2} + \beta_0^{1/2} \int_0^t \overline{\mathcal{D}}_M^{1/2} (s)ds + \gamma_0^{1/2} \overline{\mathcal{K}}_M^{1/2} (t),
		\end{equation*}
		for some $\alpha_0, \beta_0, \gamma_0 > 0$,
		and that
		$
			\mathcal{K}_\text{low} (t) \lesssim \min\brac{ \overline{\mathcal{D}}_2 (t),\, C_0 {\brac{1+t}}^{-(2M-4)} }.
		$
		Then
		\begin{equation*}
			\int_0^t \mathcal{K}_\text{low}^{1/2} (s) \mathcal{F}_M^{1/2} (s) \mathcal{D}_M^{1/2} (s) ds
			\lesssim \brac{ \alpha_0^{1/2} + \beta_0^{1/2} C_0^{1/2} + \gamma_0^{1/2} C_0^{1/2} } \int_0^t \mathcal{D}_M (s) ds.
		\end{equation*}
	\end{lemma}
	\begin{proof}
		By applying Cauchy-Schwarz to $\int \mathcal{D}_M^{1/2}$ we deduce that
		\begin{align*}
			&\int_0^t \mathcal{K}_\text{low}^{1/2} (s) \mathcal{F}_M^{1/2} (s) \mathcal{D}_M^{1/2} (s) ds
		\\
			&\lesssim \alpha_0^{1/2} \int_0^t \overline{\mathcal{D}}_2^{1/2} \mathcal{D}_M^{1/2}
			+ \beta_0^{1/2} \int_0^t \frac{ C_0^{1/2} }{ {\brac{1+s}}^{M-2} } \brac{ \int_0^s \overline{\mathcal{D}}_M^{1/2} (r) dr } \mathcal{D}_M^{1/2} (s) ds
			+ \gamma_0^{1/2} C_0^{1/2} \int_0^t \mathcal{D}_M
		\\
			&\lesssim (\alpha_0 + \gamma_0^{1/2} C_0^{1/2}) \int_0^t \mathcal{D}_M
			+ \beta_0^{1/2} C_0^{1/2} \underbrace{
				\int_0^t \frac{ \mathcal{D}_M^{1/2} (s) }{ {\brac{1+s}}^{M-2} } s^{1/2} {\brac{ \int_0^t \overline{\mathcal{D}}_M (r) dr }}^{1/2} ds
			}_{(\star)}.
		\end{align*}
		Employing the Cauchy-Schwarz inequality again and noting that $2M - 5 > 1$ (since $M \geqslant 4$) we see that
		\begin{align*}
			(\star)
			\lesssim \brac{
				\int_0^t \frac{ \mathcal{D}_M^{1/2} (s) }{ {\brac{1+s}}^{M-\frac{5}{2}} } ds
			} {\brac{
				\int_0^t \overline{\mathcal{D}}_M (s) ds
			}}^{1/2}
			\leqslant {\brac{
				\int_0^t \frac{1}{ {\brac{1+s}}^{2M-5} } ds
			}}^{1/2} \brac{
				\int_0^t \mathcal{D}_M (s) ds
			}
			\lesssim \int_0^t \mathcal{D}_M (s) ds.
		\end{align*}
	\end{proof}

	We conclude this section with the third of the four building blocks of the scheme of a priori estimates and close the interactions at the high level.
	This is done in \fref{Proposition}{prop:close_est_high_level} which synthesizes the various results of this section.
	In particular, recall that $ \overline{\mathcal{E}}_M $, which appears in \fref{Proposition}{prop:close_est_high_level} below, is defined in \eqref{eq:not_E_M_bar}.

	\begin{prop}[Closing the energy estimates at the high level]
	\label{prop:close_est_high_level}
		Let $M \geqslant 4$ be an integer. There exist $\eta_M > 0$, $0 < \delta_M \leqslant 1$, and $C_H > 0$ such that the following holds.
		For any time horizon $T > 0$, any $0 < \eta \leqslant \eta_M$, any $0 < \delta \leqslant \delta_M$, and any $C>0$, if
		\begin{subnumcases}{}
			\brac{ \mathcal{E}_M + \mathcal{F}_M } (0) \leqslant \eta
			\label{eq:close_high_est_assum_1}\\
			\sup_{0\leqslant t\leqslant T} \sup_{1\leqslant I\leqslant M} \overline{\mathcal{K}}_I (t) {\brac{1+t}}^{2M-2I} + \mathcal{K}_\text{low} (t) {\brac{1+t}}^{2M-4} \leqslant \delta,
			\label{eq:close_high_est_assum_2}\\
			\sup_{0\leqslant t\leqslant T} \mathcal{E}_M (t) \leqslant \delta, \text{ and } 
			\label{eq:close_high_est_assum_3}\\
			\mathcal{F}_M (t) \leqslant C \brac{
				\mathcal{F}_M (0)
				+ {\brac{ \int_0^t \overline{\mathcal{D}}_M^{1/2} (s) ds }}^2
				+ \overline{\mathcal{K}}_M (t)
			} \text{ for all } 0\leqslant t\leqslant T
			\label{eq:close_high_est_assum_4}
		\end{subnumcases}
		then
		\begin{equation}
		\label{eq:close_high_est_conclusion}
			\sup_{0\leqslant t\leqslant T} \overline{\mathcal{E}}_M (t) + \int_0^t \mathcal{D}_M (s) ds
			\leqslant C_H \brac{ \mathcal{E}_M + \mathcal{F}_M }(0).
		\end{equation}
	\end{prop}
	\begin{proof}
		The basic idea of the proof is that we want to go from the energy-dissipation of the problem,
		namely
		$
			\frac{\mathrm{d}}{\mathrm{d}t} \widetilde{\mathcal{E}}_M + \overline{\mathcal{D}}_M \lesssim \overline{\mathcal{I}}_M
		$
		to the more useful energy-dissipation relation
		$
			\frac{\mathrm{d}}{\mathrm{d}t} \widetilde{\mathcal{E}}_M + C \mathcal{D}_M \leqslant 0
		$
		where $C >0$ is a universal constant and where the non-negativity of the improved dissipation $ \mathcal{D}_M $ ensures the boundedness of the energy $ \mathcal{E}_M $.
		This is done by controlling the interactions and improving the dissipation.
		However, both of these steps, which are performed precisely in \fref{Propositions}{prop:control_high_order_interactions} and \ref{prop:imp_hig_level_ds} respectively,
		are delicate and lead to the appearance of terms that must be controlled in a time-integrated fashion
		-- this control is recorded in \fref{Lemmas}{lemma:control_growth_decay_interactions_imp_dissipation} and \ref{lemma:control_growth_decay_interactions_est_interactions}.

		First we note that, in light of \eqref{eq:close_high_est_assum_2} and \eqref{eq:close_high_est_assum_4}, 
		\fref{Lemmas}{lemma:control_growth_decay_interactions_imp_dissipation} and  \ref{lemma:control_growth_decay_interactions_est_interactions}
		tell us respectively that, since $\delta_M \leqslant 1$,
		\begin{equation}
		\label{eq:main_a_priori_thm_eq_step_4_2_clean}
			\int_0^t \normtyp{ (u, \theta, \pdt\theta) }{L}{\infty}^2 \mathcal{F}_M
			\lesssim \mathcal{F}_M (0) + \delta \int_0^t \overline{\mathcal{D}}_M
		\end{equation}
		and
		\begin{equation}
		\label{eq:main_a_priori_thm_eq_step_4_1_clean}
			\int_0^t \mathcal{K}_\text{low}^{1/2} \mathcal{F}_M^{1/2} \mathcal{D}_M^{1/2}
			\lesssim \brac{ \delta^{1/2} + \mathcal{F}_M^{1/2} (0) } \int_0^t \mathcal{D}_M.
		\end{equation}
		We may now proceed with the energy estimates.
		\fref{Lemma}{lemma:record_form_interactions} and \fref{Lemma}{lemma:coercivity_dissip} tell us that
		\begin{equation}
		\label{eq:main_a_priori_thm_eq_step_4_3_clean}
			\widetilde{\mathcal{E}}_M (t) + \int_0^t \overline{\mathcal{D}}_M (s) ds \lesssim \widetilde{\mathcal{E}}_M (0) + \int_0^t \overline{\mathcal{I}}_M (s) ds.
		\end{equation}
		Combining the fact that
		\begin{equation*}
			\sup_{0\leqslant t\leqslant T} \norm{(u,\theta)}{H^3}^2 + \norm{J}{H^3}^2 + \norm{\pdt(u,\theta)}{H^2}^2 + \norm{\pdt J}{H^2}^2
			\leqslant \sup_{0\leqslant t\leqslant T} \mathcal{E}_M (t) \leqslant 1,
		\end{equation*}
		\fref{Proposition}{prop:persist_spec_sols_adv_rot_eqtns}, and \fref{Lemma}{lemma:comp_version_en} we obtain that
		\begin{equation}
		\label{eq:main_a_priori_thm_eq_step_4_4_clean}
			\widetilde{\mathcal{E}}_M \asymp \overline{\mathcal{E}}_M.
		\end{equation}
		We may now use \eqref{eq:main_a_priori_thm_eq_step_4_4_clean} and \fref{Proposition}{prop:imp_hig_level_ds} first,
		then use \eqref{eq:main_a_priori_thm_eq_step_4_3_clean}, \fref{Proposition}{prop:control_high_order_interactions}, and \eqref{eq:close_high_est_assum_3} to see that
		\begin{align*}
			\overline{\mathcal{E}}_M (t) + \int_0^t \mathcal{D}_M
			\lesssim \widetilde{\mathcal{E}}_M (t) + \int_0^t \overline{\mathcal{D}}_M + \int_0^t \normtyp{ (u, \theta, \pdt\theta) }{L}{\infty}^2 \mathcal{F}_M
		\\
		\lesssim \widetilde{\mathcal{E}}_M (0) + \delta^{1/2} \int_0^t \mathcal{D}_M + \int_0^t \mathcal{K}_\text{low}^{1/2} \mathcal{F}_M^{1/2} \mathcal{D}_M^{1/2}
				+ \int_0^t \normtyp{ (u, \theta, \pdt\theta) }{L}{\infty}^2 \mathcal{F}_M.
		\end{align*}
		Combining this with \eqref{eq:main_a_priori_thm_eq_step_4_2_clean}, \eqref{eq:main_a_priori_thm_eq_step_4_1_clean}, and \eqref{eq:main_a_priori_thm_eq_step_4_4_clean} allows to deduce that
		there exists $C_s > 0$ such that
		\begin{equation*}
			\overline{\mathcal{E}}_M (t) + \int_0^t \mathcal{D}_M
			\leqslant C_s \brac{ \mathcal{E}_M + \mathcal{F}_M  } (0) + C_s \brac{ \delta^{1/2} + \mathcal{F}_M^{1/2} (0) } \int_0^t \mathcal{D}_M.
		\end{equation*}
		In particular, if $\eta_M, \delta_M > 0$ are chosen sufficiently small to ensure that $C_s \brac{\delta_M^{1/2} + \eta_M^{1/2}} \leqslant \frac{1}{2}$ then we may deduce
		\eqref{eq:close_high_est_conclusion}.
	\end{proof}

\subsection{Decay of intermediate norms}
\label{sec:decay_int_norms}

	In this section we consider the last of the four building blocks of our scheme of a priori estimates and proceed with the interpolation argument required to obtain the 
	decay of intermediate norms provided that both the low level and high level energies are controlled.
	This is supplemented by an auxiliary estimate for $\pdt^2 \theta$ whose purpose is to improve $\overline{\mathcal{K}}_2$ in order
	to control the term involving $\pdt^2 \theta$ which appears when controlling the high order interactions -- recall that this is discussed in detail in \fref{Section}{sec:discuss_ap}.
	Note that the functionals $ \overline{\mathcal{E}}_M$, $ \overline{\mathcal{E}}_\text{low}$, $ \mathcal{E}_M$, and $ \overline{\mathcal{K}}_I$ and $ \mathcal{K}_\text{low}$,
	which will be used throughout this section,
	are defined in \eqref{eq:not_E_M_bar}, \eqref{eq:not_E_low}, \eqref{eq:not_E_M_and_F_M}, and \eqref{eq:not_K_bar}, respectively.
	We begin with the interpolation argument.

	\begin{prop}[Decay of intermediate norms]
	\label{prop:decay_intermediate_norms}
		Suppose that there exists a time horizon $T>0$, an integer $M \geqslant 2$, and a constant $C_0 > 0$ such that
		\begin{equation}
		\label{eq:decay_intermediate_norms_assumption}
			\sup_{0 \leqslant t \leqslant T} \overline{\mathcal{E}}_\text{low} (t) {\brac{1+t}}^{2M-2} + \overline{\mathcal{E}}_M (t) \leqslant C_0.
		\end{equation}
		Then there exists a constant $C_I > 0$ which depends on $M$ and is universal otherwise such that we may estimate
		$
			\displaystyle\sup_{1 \leqslant I \leqslant M} \displaystyle\sup_{0 \leqslant t \leqslant T} \overline{\mathcal{K}}_I (t) {\brac{1+t}}^{2M-2I} \leqslant C_I C_0
		$.
	\end{prop}
	\begin{proof}
		This estimate on intermediate norms follows from the interpolation of $H^s$ spaces which says that, if $f\in H^l \cap H^h$, then $f\in H^i$ for any $l\leqslant i \leqslant h$
		with the interpolation estimate
		$
			\normtyp{f}{H}{i} \lesssim \normtyp{f}{H}{l}^\theta \normtyp{f}{H}{h}^{1-\theta}
		$
		where
		$
			\theta = \frac{h-i}{h-l}.
		$
		Therefore, for $\theta = \frac{M-I}{M-1}$,
		\begin{align*}
			\overline{\mathcal{K}}_I
			&\lesssim \normtyp{(u, \theta, a)}{H}{2}^{2\theta} \normtyp{(u, \theta, a)}{H}{2M}^{2(1-\theta)}
				+ \normtyp{\pdt (u, \theta, a)}{L}{2}^{2\theta} \normtyp{\pdt (u, \theta, a)}{H}{2M-2}^{2(1-\theta)}
		\\
			&\leqslant 2 {\brac{ \normtyp{(u, \theta, a)}{H}{2}^2 + \normtyp{\pdt (u, \theta, a)}{L}{2}^2 }}^\theta {\brac{ \normtyp{(u, \theta, a)}{H}{2M}^2
			+\normtyp{\pdt (u, \theta, a)}{H}{2M-2}^2 }}^{1-\theta}
		\\
			&\lesssim \overline{\mathcal{E}}_\text{low}^\theta \overline{\mathcal{E}}_M^{1-\theta}
			\lesssim {\brac{\frac{C_0}{ {\brac{1+t}}^{2M-2} }}}^\theta C_0^{1-\theta}
			= \frac{C_0}{ {\brac{1+t}}^{2M-2I} }.&&\qedhere
		\end{align*}
	\end{proof}

	We now record an auxiliary estimate for $\pdt^2\theta$ which will be used to deduce the decay of $\pdt^2 \theta$ when $\overline{\mathcal{K}}_2$ decays.

	\begin{lemma}[Auxiliary estimate for $\pdt^2 \theta$]
	\label{lemma:aux_est_pdt_2_theta}
		Suppose that \eqref{eq:pertub_sys_no_ten_pdt_theta} holds. Then, for $J = J_{eq} + K$,
		\begin{align*}
			\norm{J\pdt^2 \theta}{L^2}
			&\lesssim \norm{\pdt a}{L^2} + \norm{\brac{u, \theta}}{P^2}
			+ \brac{1 + \norm{K}{L^\infty} + \norm{\pdt K}{L^\infty}} \norm{\theta}{P^1}
		\\
			&\quad+ \brac{1 + \norm{K}{L^\infty} + \norm{\pdt K}{L^\infty}} \brac{1 + \norm{\brac{u, \theta}}{L^\infty}} \norm{\brac{u,\theta}}{P^2}.
		\end{align*}
	\end{lemma}
	\begin{proof}
		This estimate follows immediately from differentiating \eqref{eq:pertub_sys_no_ten_pdt_theta} in time.
	\end{proof}

	We now improve the control afforded to us by $\overline{\mathcal{K}}_2$ so as to also control the term involving $\pdt^2\theta$ which appears
	when controlling the high-level interactions.

	\begin{cor}[Improvement of $ \overline{\mathcal{K}}_2 $]
	\label{cor:imp_K_bar_2}
		For any time horizon $T > 0$, if
		\begin{equation}
		\label{eq:imp_K_bar_2_assumption}
			\sup_{0\leqslant t<T} \normtyp{(u, \theta)(t)}{H}{3} + \normtyp{J(t)}{H}{3} + \normtyp{\pdt(u, \theta)}{H}{2} + \normtyp{\pdt J}{H}{2} < \infty
		\end{equation}
		and $ \overline{\mathcal{E}}_3 \leqslant 1$ on $\cobrac{0,T}$ then
		$
			\normtypns{ \pdt\theta }{L}{2} \lesssim \brac{ 1 + \mathcal{E}_3^{1/2} } \brac{ \normtyp{ \pdt a }{L}{2} + \overline{\mathcal{K}}_2^{1/2} }
		$
		holds in $\cobrac{0,T}$, where the constant implicit in ``$\lesssim$'' is independent of the time horizon $T$.
	\end{cor}
	\begin{proof}
		It is crucial to recall there the global assumption that the spectrum of $J_0(x)$ is equal to $\cbrac{\lambda, \lambda, \nu}$, where $\nu > \lambda > 0$, for every $x\in\T^3$.
		The key observation is then that the assumption \eqref{eq:imp_K_bar_2_assumption} combines with \fref{Proposition}{prop:persist_spec_sols_adv_rot_eqtns} to tell us that
		$\norm{\pdt^2 \theta}{L^2} \leqslant \lambda\inv \normtyp{ J\pdt^2 \theta}{L}{2}$. The result then follows from \fref{Lemma}{lemma:aux_est_pdt_2_theta}.
	\end{proof}

	To conclude this section we record the decay of $\mathcal{K}_\text{low}$, which is the improved version of $\overline{\mathcal{K}}_2$ which also controls $\pdt^2 \theta$.
	
	\begin{cor}[Decay of $ \mathcal{K}_\text{low} $]
	\label{cor:decay_K_low}
		Suppose that there exists a time horizon $T>0$, an integer $M\geqslant 2$, and a constant $C_0 > 0$ such that
		\begin{equation*}
			\sup_{0\leqslant t\leqslant T} \overline{\mathcal{E}}_\text{low}(t) {\brac{1+t}}^{2M-2} + \overline{\mathcal{E}}_M (t) \leqslant C_0 \leqslant 1
			\text{ and } \sup_{0\leqslant t\leqslant T} \norm{J(t)}{H^3} + \norm{\pdt J(t)}{H^2} < \infty.
		\end{equation*}
		Then
		$
			\sup_{0\leqslant t \leqslant T} \mathcal{K}_\text{low} (t) {\brac{1+t}}^{2M-4} \leqslant \widetilde{C}_I C_0.
		$
		for some constant $\widetilde{C}_I > 0$ which depends only on $M$, and is universal otherwise.
	\end{cor}
	\begin{proof}
		This follows directly from combining \fref{Proposition}{prop:decay_intermediate_norms} and \fref{Corollary}{cor:imp_K_bar_2}.
	\end{proof}

\subsection{Synthesis}
\label{sec:ap_synthesis}

	In this section we put together all four building blocks of our scheme of a priori estimates that we have constructed in \fref{Sections}{sec:adv_rot_est_K}--\ref{sec:decay_int_norms}.
	This allows us to state and prove our main ``a priori estimates'' result in \fref{Theorem}{thm:a_priori} below.
	Recall that the various energy and dissipation functionals encountered in the statement and the proof of the theorem below are
	defined in \eqref{eq:not_E_M_bar}--\eqref{eq:not_D_M}.

	\begin{thm}[A priori estimates]
	\label{thm:a_priori}
		Let $M\geqslant 4$.
		There exist $\eta_{ap}, \delta_{ap}, C_{ap} > 0$ depending only on $M$ such that if $\brac{u, p, \theta, K}$
		is a solution of \eqref{eq:pertub_sys_no_ten_pdt_u}--\eqref{eq:pertub_sys_no_ten_pdt_K} on the time interval $\sbrac{0,T}$, for any $T>0$, which satisfies the smallness conditions
		\begin{align}
			\brac{\mathcal{E}_M + \mathcal{F}_M} (0) =\vcentcolon \eta_0 \leqslant \eta_{ap} \leqslant 1 \text{ and }
			\label{eq:a_priori_smallness_intial}\\
			\sup_{0\leqslant t \leqslant T} \mathcal{E}_M (t) + \int_0^T \overline{\mathcal{D}}_M \leqslant \delta_{ap} \leqslant 1,
			\label{eq:a_priori_smallness_time_interval}
		\end{align}
		then the following estimates hold
		\begin{align*}
			\sup_{0\leqslant t\leqslant T}
				\mathcal{E}_\text{low} (t) {\brac{1+t}}^{2M-2}
				+ \overline{\mathcal{E}}_M (t) + \mathcal{E}_M^a (t)
				+ \frac{ \mathcal{F}_M (t) }{1+t}
			+ \int_0^T \mathcal{D}_M
			\leqslant C_{ap} (\overline{\mathcal{E}}_M + \mathcal{F}_M) (0)
		\\
			\text{ and } 
			\sup_{0\leqslant t\leqslant T} \mathcal{E}_M^{(K)} (t) \leqslant C_{ap} ( \mathcal{E}_M + \mathcal{F}_M ) (0).
		\end{align*}
	\end{thm}
	\begin{proof}
		We take two passes at the estimates in this proof.
		During the first pass we obtain \emph{unstructured} estimates, meaning that the estimates are in term of the smallness parameter and not the initial conditions.
		During the second pass we obtain \emph{structured} estimates, meaning that the estimates are in terms of the initial conditions.

		Each of these two passes relies on the four key results we have proved in \fref{Section}{sec:a_prioris}, namely
		\fref{Proposition}{prop:est_K} where we record the advection-rotation estimates for $K$,
		\fref{Proposition}{prop:close_est_low_level} where we close the energy estimates at the low level,
		\fref{Proposition}{prop:close_est_high_level} where we close the energy estimates at the high level, and
		\fref{Proposition}{prop:decay_intermediate_norms} and \fref{Corollary}{cor:decay_K_low} where we obtain the decay of the intermediate norms.

		Before beginning the proof in earnest we record the smallness conditions which $\delta_{ap}$ and $\eta_{ap}$ must satisfy:
		\begin{enumerate}
			\item	\label{it:smallness_condition_1}
				$\delta_{ap} \leqslant \max\brac{ \delta_\text{low}, \delta_\text{low}^* }$ for $\delta_\text{low}$ and $\delta_\text{low}^*$
				as in \fref{Proposition}{prop:close_est_low_level},
			\item 	\label{it:smallness_condition_2}
				$\brac{1 + C_L} \delta_{ap} \leqslant 1$ for $C_L$ as in \fref{Proposition}{prop:close_est_low_level},
			\item 	\label{it:smallness_condition_3}
				$\brac{1 + \brac{C_I + \widetilde{C}_I} \brac{1 + C_L} } \delta_{ap} \leqslant 1$ for $C_I$ and $\widetilde{C}_I$
				as in \fref{Proposition}{prop:decay_intermediate_norms} and \fref{Corollary}{cor:decay_K_low}, respectively,
			\item 	\label{it:smallness_condition_4}
				$\eta_{ap} \leqslant \eta_M$ for $\eta_M$ as in \fref{Proposition}{prop:close_est_high_level},
			\item 	\label{it:smallness_condition_5}
				$\brac{C_I + \widetilde{C}_I}\brac{1 + C_L}\delta_{ap} \leqslant \delta_M$ for $\delta_M$ as in \fref{Proposition}{prop:close_est_high_level},
			\item	\label{it:smallness_condition_6}
				$\delta_{ap} \leqslant \delta_M$,
			\item 	\label{it:smallness_condition_7}
				$C_H \eta_{ap} \leqslant \delta_\text{low}$ for $C_H$ as in \fref{Proposition}{prop:close_est_high_level},
			\item 	\label{it:smallness_condition_8}
				$\brac{1 + C_L} C_H \eta_{ap} \leqslant 1$, and
			\item 	\label{it:smallness_condition_9}
				$\brac{1 + \brac{C_I + \widetilde{C}_I}\brac{1 + C_L}}C_H \eta_{ap} \leqslant 1$.
		\end{enumerate}

		To be very clear about the structure of the proof we break it up into seven steps.
		Note that our scheme of a priori estimates is also summarized diagrammatically in \fref{Figure}{fig:a_priori_thm_summary}.

		\begin{figure}
		\begin{center}
		\begin{tikzpicture}
			% Define the styles
			\tikzstyle{block} = [rectangle, draw, fill=RoyalBlue!20, text width = 3.5cm, text centered, minimum height = 1cm];
			\tikzstyle{arrow} = [draw, ->, thick];
			% Boxes
			\node [block]				(smallAssum)	at (0,  4)	{Smallness\\ assumption};
			\node [block]				(decayLowEn)	at (0,  2)	{Decay of the\\ low-level energy};
			\node [block]				(decayIntNorm)	at (0,  0)	{Decay of the\\ intermediate norms};
			\node [block]				(contF)		at (0, -2)	{Control of $ \mathcal{F}_M $};
			\node [block] 				(bddHighEn)	at (0, -4)	{Boundedness of\\ the high-level energy};
			\node [block] 				(contE)		at (4.5, -2)	{Control of $ \mathcal{E}_M^{(K)} $};
			% Arrows
			\draw [arrow, dashed, color=Thistle]	(smallAssum)		-- node [left]					{1L}	(decayLowEn);
			\draw [arrow, dashed, color=Thistle]	(decayLowEn.230)	-- node [left]					{2D}	(decayIntNorm.132);
			\draw [arrow, color=ForestGreen]	(decayLowEn.310)	-- node [right]					{6D}	(decayIntNorm.49);
			\draw [arrow, color=ForestGreen]	(decayIntNorm)		-- node [right]					{3A}	(contF);
			\draw [arrow, color=ForestGreen]	(contF)			-- node [right]					{4H}	(bddHighEn);
			\draw [arrow, color=ForestGreen]	(bddHighEn.west) 	-- + (-1.5, 0) |- node[left, yshift=-3cm]	{5L}	(decayLowEn.west);
			\draw [arrow, color=ForestGreen]	(decayIntNorm)		-| node [right, yshift=-0.7cm]			{7A}	(contE.155);
			% Legend
			\node [rectangle, draw, fill=Black!10, text width = 8.3cm, minimum height = 1cm] at (8.5, 2) {
				\textbf{Legend}
				\begin{tabular}{cl}
					\textcolor{Thistle}{$\dashrightarrow$} & Unstructured estimates\\
					\textcolor{ForestGreen}{$\longrightarrow$} & Structured estimates\\
					$n$ \dots	& Step $n$, $n = 1,\,\dots,\, 7$\\
					\dots A 	& Advection-rotation estimates for $K$\\
					\dots L 	& Closing the low level energy estimates\\
					\dots H 	& Closing the high level energy estimates\\
					\dots D 	& Decay of intermediate norms
				\end{tabular}
				For example: \textcolor{ForestGreen}{$\xrightarrow{3A}$} indicates that we use the advection-rotation estimates for $K$ in Step 3 to obtain a structured estimate.
			};
		\end{tikzpicture}
		\end{center}
		\caption{The strategy of the theorem for the main a priori estimates, namely \fref{Theorem}{thm:a_priori}.}
		\label{fig:a_priori_thm_summary}
		\end{figure}
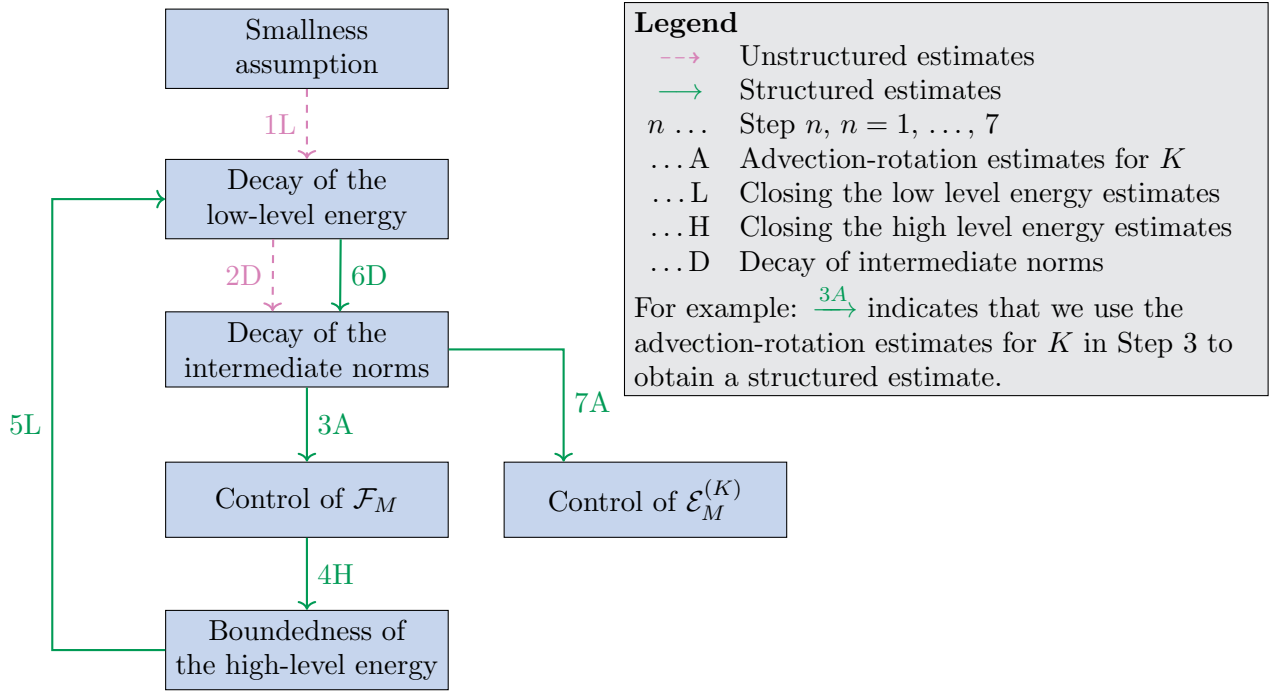

		\paragraph{\textbf{Step 1.}}
		In the first step we close the energy estimates at the low level to deduce the unstructured decay of the low level energy.
		We deduce from \eqref{eq:a_priori_smallness_time_interval}, the smallness condition \eqref{it:smallness_condition_1}, and \fref{Proposition}{prop:close_est_low_level} that
		\begin{equation}
		\label{eq:a_priori_1}
			\sup_{0\leqslant t\leqslant T} \mathcal{E}_\text{low} (t) {\brac{1+t}}^{2M-2} \leqslant C_L \delta_{ap}.
		\end{equation}
		
		\paragraph{\textbf{Step 2.}}
		In the second step we obtain the unstructured decay of the intermediate norms.
		Observe that
		\begin{equation*}
			\sup_{0\leqslant t\leqslant T} \normtyp{J(t)}{H}{3}^2 + \normtyp{\pdt J(t)}{H}{2}^2
			\leqslant \sup_{0\leqslant t\leqslant T} \mathcal{E}_M (t)
			< \infty.
		\end{equation*}
		Combining this with the smallness condition \eqref{it:smallness_condition_2} and \eqref{eq:a_priori_1},
		\fref{Proposition}{prop:decay_intermediate_norms} and \fref{Corollary}{cor:decay_K_low} tell us that
		\begin{equation}
		\label{eq:a_priori_2}
			\sup_{0\leqslant t\leqslant T} \sup_{1\leqslant I\leqslant M} \overline{\mathcal{K}}_I (t) {\brac{1+t}}^{2M-2I} + \mathcal{K}_\text{low} (t) {\brac{1+t}}^{2M-4}
			\leqslant \brac{C_I + \widetilde{C}_I} \brac{1 + C_L} \delta_{ap}.
		\end{equation}

		\paragraph{\textbf{Step 3.}}
		In the third step we obtain our first structured estimate, using the advection-rotation estimates for $K$ to wrest control over $\mathcal{F}_M$.
		We obtain from \eqref{eq:a_priori_smallness_intial}, \eqref{eq:a_priori_smallness_time_interval}, the smallness condition \eqref{it:smallness_condition_3},
		\eqref{eq:a_priori_2}, and \fref{Proposition}{prop:est_K}
		that, for all $0\leqslant t\leqslant T$,
		\begin{equation}
		\label{eq:a_priori_3}
			\mathcal{F}_M (t) \lesssim \mathcal{F}_M (0) + {\brac{ \int_0^t \overline{\mathcal{D}}_M^{1/2} (s) ds }}^2 + \overline{\mathcal{K}}_M (t).
		\end{equation}

		\paragraph{\textbf{Step 4.}}
		In the fourth step we close the energy estimates at the high level to obtain the structured boundedness of the high level energy
		(and the time-integrated control over the high level dissipation).
		By virtue of \eqref{eq:a_priori_smallness_intial}, \eqref{eq:a_priori_smallness_time_interval},
		the smallness conditions \eqref{it:smallness_condition_4}--\eqref{it:smallness_condition_6},
		\eqref{eq:a_priori_2}, and \eqref{eq:a_priori_3}, we may apply \fref{Proposition}{prop:close_est_high_level}, which tells us that
		\begin{equation}
		\label{eq:a_priori_4}
			\sup_{0\leqslant t\leqslant T} \overline{\mathcal{E}}_M (t) + \int_0^t \mathcal{D}_M (s) ds \leqslant C_H \brac{ \mathcal{E}_M + \mathcal{F}_M }(0).
		\end{equation}

		\paragraph{\textbf{Step 5.}}
		In the fifth step we continue our second pass by obtaining structured versions of previously unstructured estimates.
		We close the energy estimates at the low level to deduce the structured decay of the low level energy.
		The smallness condition \eqref{it:smallness_condition_7}, \eqref{eq:a_priori_4}, and \fref{Proposition}{prop:close_est_low_level} show us that
		\begin{equation}
		\label{eq:a_priori_5}
			\sup_{0\leqslant t\leqslant T} \mathcal{E}_\text{low} (t) {\brac{1+t}}^{2M-2}
			\leqslant C_L C_H \brac{ \mathcal{E}_M + \mathcal{F}_M } (0).
		\end{equation}

		\paragraph{\textbf{Step 6.}}
		In the sixth step we revisit Step 2 and obtain the structured decay of intermediate norms.
		\fref{Proposition}{prop:decay_intermediate_norms} and \fref{Corollary}{cor:decay_K_low} tell us, in light of \eqref{eq:a_priori_smallness_intial},
		the smallness condition \eqref{it:smallness_condition_8}, \eqref{eq:a_priori_4}, and \eqref{eq:a_priori_5}, that
		\begin{equation}
		\label{eq:a_priori_6}
			\sup_{0\leqslant t\leqslant T} \sup_{1\leqslant I\leqslant M} \overline{\mathcal{K}}_I (t) {\brac{1+t}}^{2M-2I} + \mathcal{K}_\text{low} (t) {\brac{1+t}}^{2M-4}
			\leqslant \brac{C_I + \widetilde{C}_I} \brac{1 + C_L} C_H \brac{ \mathcal{E}_M + \mathcal{F}_M }(0).
		\end{equation}

		\paragraph{\textbf{Step 7.}}
		In the seventh step we conclude the proof by using the advection-rotation estimates to get the energetic terms involving $K$, i.e. $ \mathcal{E}_M^{(K)} $, under control.
		We deduce from \eqref{eq:a_priori_smallness_intial}, the smallness condition \eqref{it:smallness_condition_9},
		\eqref{eq:a_priori_4}, \eqref{eq:a_priori_6}, and \fref{Proposition}{prop:est_K} that
		$
			\displaystyle\sup_{0\leqslant t\leqslant T} \mathcal{E}_M^{(K)} (t) \lesssim \brac{ \mathcal{E}_M + \mathcal{F}_M }(0),
		$
		which concludes the proof.
	\end{proof}

%----------------------------------------------------------------------------------------------------
%	LOCAL WELL-POSEDNESS
%----------------------------------------------------------------------------------------------------

\section{Local well-posedness}
\label{sec:lwp}

	In this section we build a local well-posedness theory sufficient to prove the existence of solutions in the spaces where our a priori estimates apply.
	We employ Galerkin scheme to construct a sequence of approximate solutions of \eqref{eq:pertub_sys_no_ten_pdt_u}--\eqref{eq:pertub_sys_no_ten_pdt_K} and
	this section is structured as follows. First we formulate appropriate approximate problems,
	then in \fref{Section}{sec:inv_T_K} we treat in detail the matter of inverting the operator $J_{eq} + P_n \circ K$ which appears in the
	approximate problems (where $P_n$ is a projection onto the subspaces where the approximate solutions live),
	we obtain various estimates on our sequence of solutions in \fref{Section}{sec:est_approx_prob},
	and finally we produce local solutions via our Galerkin scheme in \fref{Section}{sec:Galerkin}.

	Before writing down the approximate system we will solve, we must introduce the spaces in which we will solve it.
	We take $V$ to be the subspace of $L^2$ defined as
	\begin{equation*}
		V \vcentcolon= \cbrac{
			Z = (u, \theta, K) \in L^2 \brac{ \T^3;\, \R^3 \times \R^3 \times \sym(3) }
		\;:\;
			\nabla\cdot u = 0
			\text{ and } 
			\fint_{\T^3} u = 0
		},
	\end{equation*}
	we define, for any $\delta > 0$, $\sigma = \sigma (\delta) > 0$ such that for any $K \in L^2 \brac{ \T^3;\, \sym(3) }$,
	\begin{equation}
	\label{eq:def_sigma}
		\text{if } \normtyp{K}{H}{3} < \sigma(\delta)
		\text{ then } \normtyp{ K }{L}{\infty} < \frac{\lambda}{2}
		\text{ and } \normtyp{ \nabla K}{L}{\infty} < \delta,
	\end{equation}
	and we define
	\begin{equation}
	\label{eq:def_U}
		\U(\sigma) \vcentcolon= \cbrac{
			Z = (u, \theta, K) \in V : \normtyp{K}{H}{3} < \sigma
		}.
	\end{equation}
	Recall that $\lambda > 0$ is the smallest eigenvalue of $J_0$ (and hence of $J$), since our global assumption is that $J_0 (x) $ has spectrum
	$\cbrac{\lambda, \lambda, \nu}$ for every $x\in\T^3$, where $\nu > \lambda > 0$.
	To define the function spaces where the approximate solutions will live, we first define $P_n$ to be the projection onto the Fourier modes with wavenumber at most $n$,
	i.e. $ {\brac{ P_n f }}^\wedge (k) = \mathbbm{1}\brac{ \abs{k} \leqslant n} \hat{f}(k)$. This allows us to define $\mathcal{P}_n : V \to V$ via
	\begin{equation*}
		\mathcal{P}_n \vcentcolon= P_n \oplus P_n \oplus P_{2n}
		\text{ and } 
		V_n \defeq \mathcal{P}_n V
		= \cbrac{
			Z = (u, \theta, K) \in V : u, \theta \in \im P_n \text{ and } K \in \im P_{2n}
		}
	\end{equation*}
	as well as
	\begin{equation*}
		\U_n (\sigma) \vcentcolon= \cbrac{
			Z_n = (u_n, \theta_n, K_n) \in V_n : \normtyp{K_n}{H}{3} < \sigma
		} = \mathcal{P}_n \U(\sigma).
	\end{equation*}
	Note that since the projection is performed in a symmetric fashion, i.e. its symbol $\hat{\mathbb{P}}_n (k) = \mathbbm{1}\brac{\abs{k}\leqslant n}$ is even,
	it maps real-valued spaces to real-valued spaces.
	We will produce solutions $(u_n, \theta_n, K_n) \in \U_n(\sigma)$ to the following approximate system:
	\begin{subnumcases}{}
		\pdt u_n - (\nabla\cdot T)(u_n, \theta_n) = - P_n \mathbb{P}_L (u_n\cdot\nabla u_n),\\
		\nabla\cdot u_n = 0,\\
		(J_{eq} + P_n \circ K_n) \pdt\theta_n + P_n \brac{ (J_{eq} + K_n)(u_n\cdot\nabla)\theta_n } + P_n \brac{ (\omega_{eq} + \theta_n) \times (J_{eq} + K_n) \theta_n }
			\nonumber\\\qquad
			+ \tilde{\tau}^2 \tilde{a_n}^\perp + \theta \times J_{eq} \omega_{eq} - 2 \vc T(u_n, \theta_n) - (\nabla\cdot M)(\theta_n) = - P_n (\theta_n\times K_n\omega_{eq}), \text{ and }\\
		\pdt K_n - \sbrac{\Theta_n, J_{eq}} - \sbrac{\Omega_{eq}, K_n} = - P_{2n} (u_n\cdot\nabla K_n) + P_{2n} \brac{ \sbrac{\Theta_n, K_n} },
		\label{eq:approx_sys_pdt_K_n}
	\end{subnumcases}
	where $(P_n \circ K_n) v \vcentcolon= P_n ( K_n v) $ for any $v\in L^2\brac{\T^3;\, \R^3}$.

	When writing down the associated energy estimate further below we will need to distinguish between the variables that are viewed
	as \emph{unknowns} and those that are viewed as enforcing the \emph{constraints}. We will thus recapitulate the system above in the following form
	(in particular in order to fix notation regarding a compact way to write down the system above)
	\begin{subnumcases}{}
		\pdt v_n - (\nabla\cdot T)(v_n, \phi_n) = f_1,\\
		(J_{eq} + P_n \circ K_n) \pdt \phi_n + P_n \brac{ (J_{eq} + K_n) (u_n\cdot\nabla) \phi_n } + P_n \brac{ (\omega_{eq} + \theta_n) \times (J_{eq} + K_n) \phi_n }
			\nonumber\\\qquad
			+ \tilde{\tau}^2 \tilde{b}_n^\perp + \phi_n \times J_{eq} \omega_{eq} - 2 \vc T(v_n, \phi_n) - (\nabla\cdot M)(\phi_n) = f_2, \text{ and }\\
		\pdt H_n - \sbrac{\Phi_n, J_{eq}} - \sbrac{\Omega_{eq}, H_n} = F_3
	\end{subnumcases}
	subject to the constraints
	\begin{subnumcases}{}
		\nabla\cdot u_n = 0 \text{ and }
		\label{eq:compact_notation_div_free_constraint}\\
		\pdt K_n + P_{2n} (u_n\cdot\nabla K_n) = P_{2n} \brac{\sbrac{\Omega_{eq} + \Theta_n, J_{eq} + K_n}}.
		\label{eq:compact_notation_div_advection_rotation_K}
	\end{subnumcases}
	Here we view $(v_n, \phi_n, H_n)$ as the unknowns, where $b_n = ({(H_n)}_{12}, {(H_n)}_{13})$ and $\Phi_n = \ten\phi_n$,
	and $(u_n, \theta_n, K_n)$ as the variables enforcing the constraints.
	In particular, for $Z_n = (v_n, \phi_n, H_n)$, $ W_n = (u_n, \theta_n, K_n)$, and for $\mathcal{F} = \brac{f_1, f_2, F_3}$,
	we may rewrite this form of the system as
	$
		\widetilde{T}_n(K_n) \pdt Z_n - \Leb_{W_n, n} Z_n = \mathcal{F}
	$
	subject to the constraints \eqref{eq:compact_notation_div_free_constraint}--\eqref{eq:compact_notation_div_advection_rotation_K},
	where
	$
		\widetilde{T}_n (K_n) \vcentcolon= I_3 \oplus (J_{eq} + P_n \circ K_n ) \oplus I_{3\times 3}
	$
	and the operator $\Leb_{W_n, n}$ is given by
	$
		\Leb_{W_n, n} Z_n = \brac{
			- (\nabla\cdot T)(v_n, \phi_n),\,
			(\star),\,
			- \sbrac{\Phi_n, J_{eq}} - \sbrac{\Omega_{eq}, H_n}
		},
	$
	where
	\begin{align*}
		(\star) = P_n \brac{ (J_{eq} + K_n) (u_n\cdot\nabla) \phi_n } + P_n \brac{ (\omega_{eq} + \theta_n) \times (J_{eq} + K_n) \phi_n }
			+ \tilde{\tau}^2 \tilde{b}_n^\perp + \phi_n\times J_{eq}\omega_{eq}
	\\
			- 2\vc T(v_n,\phi_n) - (\nabla\cdot M)(\phi_n)
	\end{align*}
	Since $\widetilde{T}_n$ only has one non-trivial block we will write $ \widetilde{T}_n (K_n) = I_3 \oplus T_n (K_n) \oplus I_{3\times 3} $ where we define
	$
		T_n (K_n) \vcentcolon= J_{eq} + P_n \circ K_n.
	$
	This allows us to rewrite the approximate problem in the following form:
	\begin{equation}
	\label{eq:ref_approx_sys}
		\widetilde{T}_n(K_n) \pdt Z_n - \Leb_{Z_n, n} Z_n = N_n (Z_n)
		\text{ subject to }
		\nabla\cdot u_n = 0,
	\end{equation}
	where
	$
		N_n (Z_n) = \brac{
			- P_n \mathbb{P}_L (u_n\cdot\nabla u_n),\,
			- P_n (\theta_n \times K_n\omega_{eq}),\,
			- P_{2n} (u_n\cdot\nabla K_n) + P_{2n} \brac{ \sbrac{\Theta_n, K_n} }
		}.
	$

	Note that in some situations it is helpful to decompose the linear operator $\Leb_{W_n,n}$ into its part that has constant coefficient and the remainder.
	More precisely, we write
	\begin{equation}
	\label{eq:split_Leb_Z_n}
		\Leb_{W_n, n} = \Leb_0 + \overline{\Leb}_{W_n, n}
	\end{equation}
	where
	$
		\overline{\Leb}_{W_n, n} Z_n = \brac{
			0,\,
			P_n \brac{ (J_{eq} + K_n) (u_n\cdot\nabla) \phi_n } + P_n \brac{ (\omega_{eq} + \theta_n) \times (J_{eq} + K_n) \phi_n } - \omega_{eq} \times J_{eq} \theta_n,\,
			0
		}
	$ and
	\begin{equation*}
		\Leb_0 Z_n = \begin{pmatrix}
			- (\nabla\cdot T)(v_n, \phi_n)\\
			\tilde{\tau}^2 \tilde{b}_n^\perp + \phi_n\times J_{eq}\omega_{eq} + \omega_{eq} \times J_{eq} \theta_n - 2\vc T(v_n,\phi_n) - (\nabla\cdot M)(\phi_n)\\
			- \sbrac{\Phi_n, J_{eq}} - \sbrac{\Omega_{eq}, H_n}
		\end{pmatrix}.
	\end{equation*}
	
\subsection{Inverting $T(K)$}
\label{sec:inv_T_K}

	In this section we deal carefully with the inversion of $T(K) = J_{eq} + P \circ K$, the smoothness of its inverse, and we obtain $H^k$-to-$H^k$ bounds
	on the inverse. Note that in this section we will work in the generic framework where $P$ is an $L^2$-orthogonal projection onto a finite-dimensional subspace of $L^2$
	which is not necessarily $V_n$ (and so $P$ is not necessarily $P_n$).
	We begin by establishing the invertibility of $T(K)$.

	\begin{lemma}[Invertibility of $T(K)$]
	\label{lemma:invertibility_T_K}
		Let $V \subseteq L^2 (\T^3, \R^3)$ be a finite-dimensional subspace and let $P$ denote the $L^2$-orthogonal projection onto $V$.
		Let $K\in L^\infty (\T^3, \R^{3\times 3})$ be almost everywhere symmetric and satisfy $\norm{K}{\infty} < \frac{\lambda}{2}$.
		Recall that $\lambda$ is the repeated eigenvalue of the microinertia, as stated in the global assumptions of \fref{Definition}{def:global_assumptions}.
		Then
		$
			T(K) \vcentcolon= J_{eq} + P \circ K,
		$
		where $(P \circ K) v \vcentcolon= P(Kv)$ for every $v \in L^2 (\T^3, \R^3)$, is, with respect to the $L^2$ inner product, a self-adjoint invertible operator on $V$.
		Moreover we have the bound
		$
			\normns{{T(K)}\inv}{\Leb(V,V)} \leqslant \frac{2}{\lambda}.
		$
	\end{lemma}
	\begin{proof}
		The self-adjointness of $T(K)$ follows from the symmetry of $K$.
		Indeed, for every $\theta,\phi \in V$,
		\begin{equation*}
			{( T(K)\theta, \phi )}_{L^2}
			= {( (J_{eq} + P\circ K)\theta, \phi )}_{L^2}
			= {( (J_{eq} + K)\theta, \phi )}){L^2}
			= {( \theta, (J_{eq} + K) \phi )}_{L^2}
			= {( \theta, T(K)\phi )}_{L^2}.
		\end{equation*}
		The invertibility of $T(K)$ follows from the almost-everywhere invertibility of $J_{eq} + K$.
		Indeed, note that since $T(K)$ is a self-adjoint operator it suffices to study the quadratic form that it generates in order to determine its spectrum.
		So we note that for every $\theta\in V$,
		\begin{equation*}
			{( T(K)\theta, \theta )}_{L^2}
			= {( (J_{eq} + P\circ K) \theta, \theta )}_{L^2}
			= {( (J_{eq} + K)\theta, \theta )}_{L^2}
			> \frac{\lambda}{2} \norm{\theta}{L^2}^2,
		\end{equation*}
		and hence $\lambda_\text{min} \brac{ T(K) } \geqslant \frac{\lambda}{2}$.
		In particular we deduce that $T(K)$ is an invertible operator from $V$ to itself, and we have the bound
		$
			\normns{{T(K)}\inv}{\Leb(V,V)} \leqslant \frac{2}{\lambda}.\qedhere
		$
	\end{proof}

	Now that we know that ${T(K)}\inv$ is well-defined we verify that its dependence on $K$ is smooth.

	\begin{lemma}[Smoothness of ${T(K)}\inv$]
	\label{lemma:smooth_T_K_inv}
		Let $V \subseteq L^2 (\T^3, \R^3)$ and $W \subseteq L^2 (\T^3, \sym(\R^{3\times 3}))$ be finite-dimensional subspaces,
		let $P$ denote the $L^2$-orthogonal projection onto $V$, let
		$
			\U \vcentcolon= \cbrac{K\in W : \norm{K}{\infty} < \frac{\lambda}{2}}
		$
		be the open $L^\infty$-ball of radius $\frac{\lambda}{2}$ in $W$,
		and let, for any $K\in\U$,
		$
			T(K) \vcentcolon= J_{eq} + P \circ K
		$
		where $(P\circ K)v \vcentcolon= P(Kv)$ for every $v\in L^2 (\T^3, \R^3)$.
		Then the map $\Phi : \U \to \Leb(V,V)$ defined by
		$
			\Phi(K) \vcentcolon= {T(K)}\inv
		$
		is smooth.
	\end{lemma}
	\begin{proof}
		The crucial observation here is that $\Phi$ may be written as the composition of $T : \U \to \Leb(V, V)$ and $\text{inv} : \mathcal{G}\Leb (V) \to \mathcal{G}\Leb (V)$,
		where
		$
			\mathcal{G}\Leb (V) \vcentcolon= \cbrac{ L \in \Leb(V, V) : L \text{ is invertible }}
		$
		and $\text{inv}(L) \vcentcolon= L\inv$ for every $L\in \mathcal{G}\Leb(V)$.
		Note that it is precisely \fref{Lemma}{lemma:invertibility_T_K} which tells us that $T(\U) \subseteq \mathcal{G}\Leb (V)$ such that
		$\Phi = \text{inv} \circ T$ is indeed well-defined.
		All that remains to show is that both $T$ and $\text{inv}$ are smooth.
		The smoothness of $\text{inv}$ is a well-known fact -- see for example \cite{marsden}.
		To see that $T$ is smooth note that, for every $K,H\in\U$, $T(K) - T(H) = P \circ (K-H)$.
		We deduce that $T$ is affine, and hence smooth.
	\end{proof}

	We now turn our attention towards the obtention of $H^k$-to-$H^k$ estimates on ${T(K)}\inv$.
	In order to do so we first define the operator $M$ which will be useful when deriving formulae for the derivatives of ${T(K)}\inv$.

	\begin{definition}
	\label{def:good_operator_M}
		Let $V \subseteq L^2 (\T^3, \R^3)$ be a finite-dimensional subspace and let $P$ denote the $L^2$-orthogonal projection onto $V$.
		For any $K\in L^\infty (\T^3, \R^{3\times 3})$ let
		$
			T(K) \vcentcolon= J_{eq} + P \circ K
		$
		where $(P \circ K) v \vcentcolon= P (Kv)$ for any $v$ in $L^2(\T^3, \R^3)$.
		For any multi-indices $\alpha_1,\, \dots,\, \alpha_m \in \N^{3\times 3}$ with $k \vcentcolon= \max \abs{\alpha_i}$
		and any $K \in W^{k,\infty} \brac{\T^3;\, \R^{3\times 3}}$ for which $T(K)$ is invertible we define
		\begin{equation*}
			M (\alpha_1,\, \dots,\, \alpha_m) \vcentcolon= 
			{T(K)}^{-1} ( P \circ \partial^{\alpha_1} K ) {T(K)}^{-1} ( P \circ \partial^{\alpha_2} K ) {T(K)}^{-1} \cdots {T(K)}^{-1} ( P \circ \partial^{\alpha_m} K ) {T(K)}^{-1}.
		\end{equation*}
	\end{definition}

	With the operator $M$ in hand we may write down useful formulae for derivatives of ${T(K)}\inv$.

	\begin{lemma}[Formula for the derivatives of ${T(K)}\inv$]
	\label{lemma:formula_deriv_T_K_inv}
		Let $\U$ and $T$ be as in \fref{Lemma}{lemma:smooth_T_K_inv} and let $M$ be as in \fref{Definition}{def:good_operator_M}.
		For any multi-index $\alpha\in \N^{3\times 3}$ and any $K\in W^{\abs{\alpha}, \infty} (T^3, \R^{3\times 3})$ we have the identity
		\begin{equation*}
			\partial^\alpha \brac{ {T(K)}^{-1} }
			= \sum_{k=1}^{\abs{\alpha}} {(-1)}^k \sum_{\beta_1 + \cdots + \beta_k = \alpha} M( \beta_1, \cdots, \beta_k) (K) .
		\end{equation*}
	\end{lemma}
	\begin{proof}
		The fundamental observations are that taking a single derivative of the maps $K \mapsto T(K)$ and $K \mapsto {T(K)}\inv$ yields
		\begin{equation*}
			\partial_i ( T(K)) = P \circ \partial_i K
			\text{ and } 
			\partial_i \brac{ T{(K)}\inv } = - {T(K)}\inv ( P \circ \partial_i K ) {T(K)}\inv
		\end{equation*}
		(see \fref{Lemma}{lemma:smooth_T_K_inv} for analogous computations).
		Using these two identities we may deduce an identity for derivatives of $M$:
		\begin{align*}
			\partial_i \brac{ 
				M(\alpha_1,\, \alpha_2,\, \cdots,\, \alpha_m) (K)
			} =
			M(\alpha_1 + e_i,\, \alpha_2,\, \cdots,\, \alpha_m) (K)
			+ \cdots
			+ M(\alpha_1,\, \alpha_2,\, \cdots,\, \alpha_m + e_i) (K)
		\\
			- M(e_i,\, \alpha_1,\, \alpha_2,\, \cdots,\, \alpha_m) (K)
			- M(\alpha_1,\, e_i,\, \alpha_2,\, \cdots,\, \alpha_m) (K)
			- \cdots
			- M(\alpha_1,\, \alpha_2,\, \cdots,\, \alpha_m,\, e_i) (K).
		\end{align*}
		The result then follows by induction.
	\end{proof}

	In light of these formulae for derivatives of ${T(K)}\inv$ we may now conclude this section and obtain $H^k$-to-$H^k$ bounds on ${T(K)}\inv$.

	\begin{lemma}[$H^k$ bounds on ${T(K)}\inv$]
	\label{lemma:H_k_bounds_T_K_inv}
		Let $\U$ and $T$ be as in \fref{Lemma}{lemma:smooth_T_K_inv}.
		For every $k\geqslant 2$ and every $K\in\U\cap H^{k+2}$,
		$
			\normns{ {T(K)}\inv }{ \Leb(H^k, H^k) }
			\lesssim \norm{K}{H^k} + \norm{K}{H^{k+2}}^k.
		$
	\end{lemma}
	\begin{proof}
		As a starting point, combining \fref{Lemma}{lemma:invertibility_T_K} and \fref{Lemma}{lemma:formula_deriv_T_K_inv} tells us that,
		for any multi-indices $\beta_1,\, \dots,\, \beta_m$ and for $M$ as in \fref{Definition}{def:good_operator_M},
		if we write $\alpha \vcentcolon= \beta_1 + \dots + \beta_m$ and $k = \abs{\alpha}$ then
		\begin{equation*}
			\norm{ M(\beta_1,\, \dots,\, \beta_m)(K) }{\Leb(L^2, L^2)}
			\leqslant \norm{ {T(K)}^{-1} }{\Leb(L^2, L^2)}^{k+1} \prod_{i=1}^m \norm{ \partial^{\beta_i} K }{L^\infty}
			\leqslant {\left(\frac{2}{\lambda}\right)}^{k+1} \norm{\nabla K}{W^{k-1,\infty}}^k
			\lesssim \norm{K}{H^{k+2}}^k
		\end{equation*}
		since, when $n=3$, $H^2 (\T^3) \hookrightarrow L^\infty (\T^3)$.

		We may now combine this inequality with \fref{Lemma}{lemma:formula_deriv_T_K_inv} to obtain $L^2$-to-$L^2$ bounds on
		$ \partial^\beta \brac{ {T(K)}^{-1} }$: for any multi-index $\beta\in\N^3$ with $\abs{\beta} = l$
		\begin{align*}
			\norm{ \partial^\beta \brac{ {T(K)}^{-1} }}{\Leb(L^2, L^2)}
			\leqslant \sum_{i=1}^l \sum_{\gamma_1 + \dots + \gamma_i = \beta}
				\norm{ M(\gamma_1,\, \dots,\, \gamma_i) (K) }{\Leb(L^2, L^2)}
			\lesssim \norm{K}{H^{l+2}}^l.
		\end{align*}
		We may now finally obtain $H^k$-to-$H^k$ bounds on $ {T(K)}^{-1} $.
		For any $v\in H^k (\T^3, \R^3)$
		\begin{align*}
			\norm{ {T(K)}^{-1} v}{H^k}^2
			= \sum_{\abs{\alpha} \leqslant k} \norm{ \partial^\alpha \brac{ {T(K)}^{-1} v } }{L^2}^2
			\leqslant \sum_{\abs{\alpha}\leqslant k} \sum_{\beta + \delta}
				\norm{ \partial^\beta \brac{ {T(K)}^{-1} } }{\Leb(L^2, L^2)}^2 \norm{ \partial^\delta v}{L^2}^2
		\\
			\lesssim \sum_{\abs{\alpha} \leqslant k} \sum_{\beta + \delta = \alpha}
				\normtyp{K}{H}{\abs{\beta}+2}^{2\abs{\beta}} \norm{v}{H^{\abs{\delta}}}^2
			\lesssim \brac{
				\norm{K}{H^2}^2 + \norm{K}{H^{k+2}}^{2k}
			} \norm{v}{H^k}^2,
		\end{align*}
		where note that the last inequality follows by interpolation.
	\end{proof}

\subsection{Estimates for the approximate problem}
\label{sec:est_approx_prob}

	In this section we obtain two types of estimates: a priori estimates on the sequence of approximate solutions
	and estimates of the initial energy (which involves temporal derivatives) in terms of purely spatial norms.

	Note that by contrast with the main scheme of a priori estimates built in \fref{Section}{sec:a_prioris},
	the a priori estimates here are almost exclusively centered around energy estimates
	(some advection-rotation estimates for $K_n$ are present, but play an auxiliary role).
	This is because we are working locally in time and therefore can get away with ``sloppier'' estimates,
	in the sense that the nonlinear interactions need not be estimated in structured ways (e.g. as $\abs{\mathcal{I}} \lesssim \sqrt{\mathcal{E}} \mathcal{D}$)
	such that cruder estimates (e.g. $\abs{\mathcal{I}} \lesssim \mathcal{E}^{3/2}$) suffice.

	This section is structured as follows.
	First we record some projected variants of the advection-rotations estimates, then we proceed with the energy estimates,
	and finally we turn our attention to estimates of the initial energy in terms of purely spatial norms.

	As a precursor to $H^k$ estimates for the projected advection-rotation operator appearing in \eqref{eq:approx_sys_pdt_K_n} we first obtain an $L^2$ estimate.
	Note that \eqref{eq:adv_rot_eq_involving_Pn_L_2} in the statement of \fref{Lemma}{lemma:L2_est_proj_adv_rot_eq}
	is equivalent to \eqref{eq:approx_sys_pdt_K_n} for an appropriate definition of $F$.
	It is written in this slightly different form since it makes it clear which operator produces good $L^2$ estimates,
	and hence when operator must be kept on the left-hand side when taking derivatives and performing $H^k$ estimates.

	\begin{lemma}[$L^2$ estimates for projected advection-rotation equations]
	\label{lemma:L2_est_proj_adv_rot_eq}
		Let $K_n \in L^2 \brac{ \T^3;\, \sym(3) } \cap \im P_n$, let $u$ be divergence-free,
		and let $\Theta$ be antisymmetric. If $K_n$ solves
		\begin{equation}
		\label{eq:adv_rot_eq_involving_Pn_L_2}
			P_n \circ (\pdt + u \cdot\nabla - \sbrac{\Omega_{eq} + \Theta, \,\cdot\,}) K_n = F
		\end{equation}
		then
		$
			\frac{\mathrm{d}}{\mathrm{d}t} \norm{K_n}{L^2} \leqslant \norm{F}{L^2}.
		$
	\end{lemma}
	\begin{proof}
		The key observation is that
		\begin{equation}
		\label{eq:time_deriv_L2_norm}
			\frac{\mathrm{d}}{\mathrm{d}t} \norm{K_n}{L^2}
			= \frac{\mathrm{d}}{\mathrm{d}t} {\brac{ \int_{\T^3} \abs{K_n}^2 }}^{1/2}
			= \frac{1}{2} {\brac{ \int_{\T^3} \abs{K_n}^2 }}^{-1/2} \int_{\T^3} 2 K_n : \pdt K_n
			= \frac{
				{( K_n, \pdt K_n )}_{L^2}
			}{
				\norm{K_n}{L^2}
			}
		\end{equation}
		where we may bound ${(K_n, \pdt K_n)}_{L^2}$ via a simple energy estimate on \eqref{eq:adv_rot_eq_involving_Pn_L_2}.
		Indeed it follows from \eqref{eq:adv_rot_eq_involving_Pn_L_2}, the incompressibility of $u$, and \fref{Lemma}{lemma:commut_are_antisym_maps_on_space_sym_matrices} that
		\begin{equation}
		\label{eq:adv_rot_eq_involving_Pn_intermediate}
			\int_{\T^3} \pdt K_n : K_n
			= \int_{\T^3} \brac{ \pdt + u \cdot \nabla - \sbrac{\Omega_{eq} + \Theta,\,\cdot\,} } K_n : K_n
			= \int_{\T^3} F : K_n.
		\end{equation}
		Putting \eqref{eq:time_deriv_L2_norm} and \eqref{eq:adv_rot_eq_involving_Pn_intermediate} together with the Cauchy-Schwarz inequality allows us to conclude.
	\end{proof}

	With this $L^2$ estimate in hand we may now derive $H^k$ estimates for $K_n$.

	\begin{lemma}[$H^k$ estimates for projected advection-rotation equations]
	\label{lemma:H_k_est_proj_adv_rot_eq}
		Let $K_n \in L^2 \brac{ \T^3;\, \sym(3) } \cap \im P_n$, let $u$ be divergence-free,
		and let $\Theta$ be antisymmetric. If $K_n$ solves
		$
			P_n \circ (\pdt + u \cdot\nabla - \sbrac{\Omega_{eq} + \Theta, \,\cdot\,}) K_n = 0
		$
		and satisfies $\norm{K_n}{\infty}, \norm{\nabla K_n}{\infty} \lesssim 1$,
		then, for every $k\geqslant 0$,
		\begin{equation*}
			\norm{K_n (t)}{H^{k}} \lesssim \exp\brac{ \int_0^t \norm{ (u, \theta) (s) }{H^3} ds} \brac{ \norm{K_n(0)}{H^{k}} + \int_0^t \norm{ (u, \theta)(s) }{H^{k}} ds}.
		\end{equation*}
	\end{lemma}
	\begin{proof}
		Since $ P_n$ commutes with $ \partial^\alpha $ and since $ \norm{ P_n }{\Leb(L^2, L^2)} \leqslant 1$ we may deduce that
		\begin{equation*}
			\norm{ \sbrac{ P_n \circ (u \cdot \nabla), \partial^\alpha } K_n }{L^2}
			= \norm{ \brac{ P_n \circ \sbrac{u \cdot \nabla, \partial^\alpha }} K_n }{L^2}
			\leqslant \norm{ \sbrac{u \cdot \nabla, \partial^\alpha } K_n }{L^2}
		\end{equation*}
		and similarly
		\begin{equation*}
			\norm{ \sbrac{ P_n \circ \sbrac{\Theta, \,\cdot\,}, \partial} K_n}{L^2} \leqslant \norm{ \sbrac{\sbrac{\Theta,\,\cdot\,}, \partial^\alpha } K_n}{L^2}.
		\end{equation*}

		With these two commutator inequalities and \fref{Lemma}{lemma:L2_est_proj_adv_rot_eq} in hand
		we may proceed as in \fref{Lemma}{lemma:H_k_est_K} to deduce the claim, keeping in mind that $\norm{K_n}{L^\infty}, \norm{\nabla K_n}{L^\infty} \lesssim 1$.
	\end{proof}

	We now turn our attention to the energy-dissipation structure of the approximate problem.
	We begin by defining appropriate versions of the energy.

	\begin{definition}[Versions of the local energies]
	\label{def:versions_local_energy}
		For $Z = (u, \theta, K)$ we define $E_{K,\, \text{loc}}$, $ \widetilde{\mathcal{E}}_{M,\text{loc}} $, and $ \overline{\mathcal{E}}_{M,\text{loc}} $ as follows:
		\begin{equation}
		\label{eq:def_local_energy}
			E_{K,\, \text{loc}} (u, \theta, K) \vcentcolon=
			\frac{1}{2} \int_{\T^3} \abs{u}^2
			+ \frac{1}{2} \int_{\T^3} (J_{eq} + K) \theta\cdot\theta
			+ \frac{1}{2} \frac{\tilde{\tau}^2}{\nu-\lambda} \int_{\T^3} \abs{K}^2
		\end{equation}
		while
		\begin{equation}
		\label{eq:def_summed_local_energy}
			\widetilde{\mathcal{E}}_{M,\text{loc}} (Z) \vcentcolon= \sum_{\abs{\alpha}_P \leqslant 2M} E_{K,\, \text{loc}} ( \partial^\alpha Z)
			\text{ and }
			\overline{\mathcal{E}}_{M,\text{loc}} (Z) \vcentcolon= \sum_{\abs{\alpha}_P \leqslant 2M} \norm{ \partial^\alpha Z}{L^2}^2.
		\end{equation}
	\end{definition}

	We now record a precise comparison of various versions of the energy.
	We emphasize here that \fref{Lemma}{lemma:exact_comparison_versions_E} below differs from \fref{Lemma}{lemma:comp_version_en}
	since the former is a consequence of \emph{smallness} whereas the latter is as consequence of \emph{regularity}.

	\begin{lemma}[Comparisons of the different verions of the local energies]
	\label{lemma:exact_comparison_versions_E}
		Let $ \overline{\mathcal{E}}_{M,\text{loc}} $ and $ \widetilde{\mathcal{E}}_{M,\text{loc}} $ be defined as in \fref{Definition}{def:versions_local_energy}.
		There exist constants $\tilde{c}_E, \widetilde{C}_E > 0$ such that if $ \normtyp{ K }{L}{\infty} < \frac{\lambda}{2}$
		then we have the estimate
		$
			\tilde{c}_E \overline{\mathcal{E}}_{M,\text{loc}}
			\leqslant \widetilde{\mathcal{E}}_{M,\text{loc}}
			\leqslant \widetilde{C}_E \overline{\mathcal{E}}_{M,\text{loc}}.
		$
	\end{lemma}
	\begin{proof}
		The key observation is that since the spectrum of $J_{eq}$ is $\cbrac{\lambda, \lambda, \nu}$, if $ \normtyp{ K }{L}{\infty} < \frac{\lambda}{2}$
		then the spectrum of $J_{eq} + K$ is contained in $\brac{ \lambda/2, \nu + \lambda/2}$.
		The claim then follows as in \fref{Lemma}{lemma:comp_version_en}.
	\end{proof}

	We now record an elementary lemma which is crucial in deriving the energy-dissipation relation associated with the approximate system.
	Indeed, \fref{Lemma}{lemma:dual_charac_proj_op} below is precisely what justifies approximating $K$ with twice as many Fourier modes as the other variables.

	\begin{lemma}[Finite Fourier mode cut-off of products]
	\label{lemma:dual_charac_proj_op}
		For any $M\in L^2 \brac{\T^3;\, \R^{3\times 3}}$ and any $L^2$ vector field $v\in\im P_n$, if $P_{2n} M = 0$ then $P_n (Mv) = 0$.
	\end{lemma}
	\begin{proof}
		Suppose that $P_{2n} M = 0$ and that $v \in V_n$.
		Then
		$
			{\brac{ Mv }}^\wedge (k)
			= \sum_{\abs{l} > 2n} \hat{M} (l) \hat{v} (k-l).
		$
		In particular, if $\abs{k}\leqslant n$ and $\abs{l} > 2n$ then
		$
			\abs{k - l}
			> n
		$
		such that, since $v\in V_n$, $\hat{v}_j (k - l) = 0$.
		This shows that ${\brac{ Mv }}^\wedge (k) = 0$ for any $\abs{k} \leqslant n$,
		i.e. indeed $P_n( Mv ) = 0$.
	\end{proof}

	We are now equipped to state and prove the energy-dissipation relation associated with the approximate system.
	In particular, as discussed in more detail in \fref{Section}{sec:discuss_lwp},
	note that in the approximate system considered below in \fref{Proposition}{prop:gen_ED_approx_system} we use regular time derivatives for the unknowns $v$ and $H$
	and an advective time derivative for $\phi$.
	We could have used advective derivatives for $v$ and $H$, but this formulation makes it more clear which nonlinear structure is optional and which is not.
	In particular, as discussed in \fref{Section}{sec:discuss_lwp}, the nonlinear structure in the equation governing the dynamics of $\phi$ is essential in
	order to obtain a good energy-dissipation relation.

	\begin{prop}[The generic energy-dissipation relation associated with the approximate system]
	\label{prop:gen_ED_approx_system}
		Suppose that the unknowns $(v, \phi, H)$ and $b$, where $b = (H_{12}, H_{13})$ and $\Phi = \ten \phi$
		and the constraint variables $(u, \theta, K)$ satisfy
		\begin{equation*}
			\left\{
			\begin{aligned}
				&\pdt v - (\nabla\cdot T)(v, \phi) = f_1,\\
				&(J_{eq} + P_n \circ K) \pdt \phi + P_n \brac{ (J_{eq} + K) (u\cdot\nabla) \phi } + P_n \brac{ (\omega_{eq} + \theta) \times (J_{eq} + K) \phi }
					\\&\qquad
					+ \tilde{\tau}^2 \tilde{b}^\perp + \phi \times J_{eq} \omega_{eq} - 2 \vc T(v, \phi) - (\nabla\cdot M)(\phi) = f_2, \text{ and }\\
				&\pdt H - \sbrac{\Phi, J_{eq}} - \sbrac{\Omega_{eq}, H} = F_3,
			\end{aligned}
			\right.
		\end{equation*}
		subject to the constraints
		\begin{subnumcases}{}
			\nabla\cdot u = 0 \text{ and }
			\label{eq:gen_ED_approx_sys_constraint_div_free}\\
			\pdt K + P_{2n} (u\cdot\nabla K) = P_{2n} \brac{\sbrac{\Omega_{eq} + \Theta, J_{eq} + K}},
			\label{eq:gen_ED_approx_sys_constraint_advection_rotation_K}
		\end{subnumcases}
		where $(v, \phi, H) \in V_n$ and $K \in \im P_{2n}$.
		Then the following energy-dissipation identity holds:
		\begin{align*}
			\frac{\mathrm{d}}{\mathrm{d}t}\brac{
				\int_{\T^3} \frac{1}{2} \abs{v}^2
				+ \int_{\T^3} \frac{1}{2} (J_{eq} + K) \phi \cdot \phi
				+ \int_{\T^3} \frac{1}{2} \frac{\tilde{\tau}^2}{\nu-\lambda} \abs{H}^2
			}
			+ D(v,\phi)
			= \int_{\T^3} f_1 \cdot v
			+ \int_{\T^3} f_2 \cdot \phi
			+ \int_{\T^3} \check{\tau} F_3 : H,
		\end{align*}
		where $\check{\tau} = \frac{1}{2} \frac{\tilde{\tau}^2}{\nu-\lambda}$.
		Recall that the dissipation $D$ is defined in \eqref{eq:not_dissip}.
	\end{prop}
	\begin{proof}
		We begin by computing the time derivative of the kinetic energy due to $u$:
		\begin{equation*}
			\frac{\mathrm{d}}{\mathrm{d}t} \int_{\T^3} \frac{1}{2} \abs{v}^2
			= \int_{\T^3} \brac{\pdt v} \cdot v
			= \int_{\T^3} (\nabla\cdot T)(v, \phi) \cdot v + \int_{\T^3} f_1 \cdot v
			= - \int_{\T^3} T(v, \phi) : \nabla v + \int_{\T^3} f_1 \cdot v.
		\end{equation*}
		We now compute the time derivative of the kinetic energy due to $\omega$.
		In light of \eqref{eq:gen_ED_approx_sys_constraint_div_free} we see that
		\begin{equation*}
			\frac{\mathrm{d}}{\mathrm{d}t} \int_{\T^3} \frac{1}{2} (J_{eq} + K) \phi \cdot \phi
			= \int_{\T^3} \frac{1}{2} \brac{ (\pdt + u\cdot\nabla) K} \phi \cdot \phi
			+ \int_{\T^3} (J_{eq} + K) (\pdt + u\cdot\nabla) \phi \cdot \phi
			=\vcentcolon \RN{1} + \RN{2}
		\end{equation*}
		where we may combine \eqref{eq:gen_ED_approx_sys_constraint_advection_rotation_K} and \fref{Lemma}{lemma:dual_charac_proj_op},
		and use \fref{Lemma}{lemma:identity_comm_A_S_and_sym_A_cross_S} to see that
		\begin{align*}
			\RN{1}
			= \int_{\T^3} \frac{1}{2} P_n \brac{ \brac{ \brac{ \pdt + u\cdot\nabla } K } \phi } \cdot \phi
			= \int_{\T^3} \frac{1}{2} P_n \brac{ \sbrac{ \Omega_{eq} + \Theta, J_{eq} + K } \phi } \cdot \phi
			= \int_{\T^3} \frac{1}{2} \sbrac{\Omega_{eq} + \Theta, J_{eq} + K } \phi \cdot \phi
		\\
			= \int_{\T^3} (\omega_{eq} + \theta) \times (J_{eq} + K) \phi \cdot \phi
			= \int_{\T^3} P_n \brac{ (\omega_{eq} + \theta) \times (J_{eq} + K) \phi } \cdot \phi
		\end{align*}
		and where we may compute directly that
		\begin{equation*}
			\RN{2}
			= \int_{\T^3} P_n \brac{ (J_{eq} + K) (\pdt + u\cdot\nabla) \phi } \cdot\phi
			= \int_{\T^3} (J_{eq} + P_n \circ K) \pdt \phi + P_n \brac{ (J_{eq} + K) (u\cdot\nabla) \phi } \phi.
		\end{equation*}
		Adding $\RN{1}$ and $\RN{2}$ together therefore tells us that
		\begin{align*}
			\frac{\mathrm{d}}{\mathrm{d}t} \int_{\T^3} \frac{1}{2} (J_{eq} + K) \phi \cdot \phi
			= - \int_{\T^3} \tilde{\tau}^2 \tilde{b}^\perp \cdot \phi
				+ \int_{\T^3} 2\vc T(v, \phi) \cdot \phi
				+ \int_{\T^3} (\nabla\cdot M)(\phi) \cdot \phi
				+ \int_{\T^3} f_2 \cdot \phi
		\\
			= \int_{\T^3} \tilde{\tau}^2 b \cdot \bar{\phi}^\perp
				+ \int_{\T^3} T(v, \phi) : \Phi
				- \int_{\T^3} M(\phi) : \nabla\phi
				+ \int_{\T^3} f_2 \cdot \phi.
		\end{align*}
		Finally we compute the energetic contribution from $H$. As a preliminary, note that we can deduce from \fref{Lemma}{lemma:block_form_Omega_J}
		that
		\begin{equation*}
			\sbrac{\Phi, J_{eq}} = - (\nu - \lambda) \begin{pmatrix}
				0 & \bar{\phi}^\perp\\
				{\brac{ \bar{\phi}^\perp }}^T & 0
			\end{pmatrix}.
		\end{equation*}
		We may therefore compute that, in light of the equation above and \fref{Lemma}{lemma:commut_are_antisym_maps_on_space_sym_matrices},
		\begin{align*}
			\frac{\mathrm{d}}{\mathrm{d}t} \int_{\T^3} \frac{1}{2} \abs{H}^2
			= \int_{\T^3} (\pdt H) : H
			= \int_{\T^3} \sbrac{\Phi, J_{eq}} : H
				+ \int_{\T^3} \sbrac{\Omega_{eq}, H} : H
				+ \int_{\T^3} F_3 : H
		\\
			= - (\nu - \lambda) \int_{\T^3} \bar{\phi}^\perp \cdot b
				+ \int_{\T^3} F_3 : H.
		\end{align*}
		To conclude we multiply this last identity by $\frac{\tilde{\tau}^2}{\nu-\lambda}$
		and add it to the identities obtained above for the evolution of the different components of the kinetic energy.
		Upon noting that 
		\begin{equation*}
			\int_{\mathbb{T}^3} T(v, \phi) : (\nabla v - \Phi) + \int_{\mathbb{T}^3} M(\phi) : \nabla\phi
			= D(v, \phi)
		\end{equation*}
		(c.f. \fref{Proposition}{prop:gen_pert_ED_rel} for details) we deduce the claim.
	\end{proof}

	With the energy-dissipation in hand we may now tackle the nonlinear interactions.
	We begin by recording the precise form of the interactions.
	Recall that $E_{K, loc}$ and $D$ are defined in \eqref{eq:def_local_energy} and \eqref{eq:not_dissip}, respectively.

	\begin{lemma}[Recording the form of the local interactions]
	\label{lemma:record_form_local_interactions}
		Let $Z = (u, \theta, K) \in V_n$ solve \eqref{eq:ref_approx_sys}.
		Then for every multi-index $\alpha \in \N^{1+3}$ we have that
		$
			\frac{\mathrm{d}}{\mathrm{d}t} E_{K,\,\text{loc}} ( \partial^\alpha Z ) + D ( \partial^\alpha u, \partial^\alpha \theta ) = \mathcal{N}^\alpha
		$
		where
		\begin{align*}
			\mathcal{N}^\alpha =
			- \int_{\T^3} \partial^\alpha ( u\cdot\nabla u ) \cdot \partial^\alpha u
			- \int_{\T^3} \partial^\alpha ( \theta \times K \omega_{eq} ) \cdot \partial^\alpha \theta
			- \check{\tau} \int_{\T^3} \partial^\alpha ( u\cdot\nabla K ) : \partial^\alpha K
			+ \check{\tau} \int_{\T^3} \partial^\alpha ( \sbrac{\Theta, K} ) : \partial^\alpha K
		\\
			+ \int_{\T^3} \sbrac{K \pdt, \partial^\alpha } \theta \cdot \partial^\alpha \theta
			- \int_{\T^3} \sbrac{ (J_{eq} + K) (u\cdot\nabla), \partial^\alpha } \theta \cdot \partial^\alpha \theta
			- \int_{\T^3} \sbrac{ (\omega_{eq} + \theta) \times (J_{eq} + K), \partial^\alpha } \theta \cdot \partial^\alpha \theta.
		\end{align*}
		for $\check{\tau} \vcentcolon= \frac{\tilde{\tau}^2}{\nu-\lambda}$.
	\end{lemma}
	\begin{proof}
		If $Z$ solves \eqref{eq:ref_approx_sys} then, for any multi-index $\alpha$, $ \partial^\alpha Z$ solves
		\begin{equation*}
			\widetilde{T}_n (K) \pdt \partial^\alpha Z - \Leb_{Z, n} \partial^\alpha Z
			= \partial^\alpha \brac{ N_n (Z) }
			+ \sbrac{ \widetilde{T}_n (K) \pdt, \partial^\alpha } Z
			- \sbrac{ \Leb_{Z, n}, \partial^\alpha } Z
			=\vcentcolon F_n^\alpha
		\end{equation*}
		subject to
		\begin{equation*}
			\nabla\cdot u = 0
			\text{ and } 
			\pdt K + P_{2n} (u\cdot\nabla K) = P_{2n} ( \sbrac{ \Omega_{eq} + \Theta, J_{eq} + K })
		\end{equation*}
		and hence \fref{Proposition}{prop:gen_ED_approx_system} tells us that
		\begin{equation*}
			\frac{\mathrm{d}}{\mathrm{d}t} E_{K,\, \text{loc}} ( \partial^\alpha Z) + D( \partial^\alpha u, \partial^\alpha \theta)
			= \int_{\T^3} F_n^\alpha \cdot \mathfrak{C} \partial^\alpha Z
			=\vcentcolon \mathcal{N}^\alpha.
		\end{equation*}
		where $\mathfrak{C} = I_3 \oplus I_3 \oplus \check{\tau} I_{3\times 3}$.
		To compute $\mathcal{N}^\alpha$ is suffices to use the fact that $P_n$ and $ \partial^\alpha $ commute,
		to recall that $\widetilde{T}_n (K) = I_3 \oplus (J_{eq} + P_n \circ K) \oplus I_{3\times 3}$,
		and to split $\Leb_{Z,n}$ into its part with constant coefficients and the remainder as is done in \eqref{eq:split_Leb_Z_n}.
	\end{proof}

	Having recorded the precise form of the interactions we estimate them.

	\begin{lemma}[Estimates of the local interactions]
	\label{lemma:est_local_interactions}
		Let $M\geqslant 4$ be an integer and let
		$\mathcal{N} \vcentcolon= \sum_{\abs{\alpha}_P \leqslant 2M} \mathcal{N}^\alpha$ for $\mathcal{N}^\alpha$
		as in \fref{Lemma}{lemma:record_form_local_interactions}.
		The following estimate holds:
		\begin{equation*}
			\abs{\mathcal{N}}
			\lesssim \normtyp{ \nabla K }{L}{\infty} \normtyp{(u,\theta)}{P}{2M+1}^2
				+ \normtyp{(u, \theta, K)}{P}{2M}^3
				+ \normtyp{(u, \theta, K)}{P}{2M}^4.
		\end{equation*}
	\end{lemma}
	\begin{proof}
		Let us write the terms in $\mathcal{N}$ as $\mathcal{N}_1^\alpha,\,\cdots,\,\mathcal{N}_7^\alpha$, following the indexing of \fref{Lemma}{lemma:record_form_local_interactions}, such that
		\begin{equation}
			\mathcal{N} = \sum_{\abs{\alpha}_P \leqslant 2M} \mathcal{N}_1^\alpha + \cdots + \mathcal{N}_7^\alpha.
		\end{equation}
		We will estimate each of these seven terms in turn.
		First, however, we note that the interaction term
		$
			\mathcal{N}_5^\alpha = - \int_{\T^3} \sbrac{ K\pdt, \partial^\alpha } \theta \cdot \partial^\alpha \theta
		$
		bears a particular importance in this estimate. Indeed, due to the temporal derivative appearing in the commutator
		we must invoke a parabolic count of $2M+1$ derivatives acting on $\theta$, which gives rise to the term
		$ \normtyp{ \nabla K }{L}{\infty} \normtyp{(u,\theta,K)}{P}{2M+1}^2$ in the estimate.
		Most notably, $\mathcal{N}_5$ is the only interaction which requires us to invoke a parabolic count of $2M+1$ derivatives.

		We now go through the estimates of the interactions one by one -- although due to the great similarity in estimating many terms we will only provide details for a few of the interactions.
		\paragraph{\textbf{Estimating $\mathcal{N}_1$.}}
			By applying the Leibniz rule we see that
			\begin{equation*}
				-\mathcal{N}_1^\alpha = \sum_{\substack{ \beta+\gamma = \alpha \\ \abs{\alpha}_P \leqslant 2M }}
				\binom{\alpha}{\beta} \int_{\T^3} ( \partial^\beta u\cdot\nabla \partial^\gamma u) \cdot \partial^\alpha u
			\end{equation*}
			where we have used the fact that $u$ is divergence-free to deduce that
			\begin{equation*}
				\int_{\T^3} \brac{u\cdot\nabla \partial^\alpha u} \cdot \partial^\alpha u
				= - \frac{1}{2} \int_{\T^3} (\nabla\cdot u) \abs{ \partial^\alpha u}^2 = 0.
			\end{equation*}
			To estimate $\mathcal{N}_1^\alpha$ it then suffices to perform a ``hands-on high-low'' estimate.
			Since $M \geqslant 3$ we note that $(2M-2) + (2M-3) > 2M$ and hence either
			$\abs{\beta}_P \leqslant 2M-2$ or $\abs{\gamma}{P} \geqslant 2M-3$, so we may estimate
			\begin{equation*}
				\abs{\mathcal{N}_1^\alpha}
				\lesssim \sum_{\substack{\cdots \\ \abs{\beta}_P \leqslant 2M-2}}
					\normtypns{ \partial^\beta u}{L}{\infty} \normtyp{ \nabla \partial^\gamma u}{L}{2} \normtyp{ \partial^\alpha u}{L}{2}
					+ \sum_{\substack{ \cdots \\ \abs{\gamma}_P \leqslant 2M-3 }}
					\normtypns{ \partial^\beta u}{L}{2} \normtyp{ \nabla \partial^\gamma u}{L}{\infty} \normtyp{ \partial^\alpha u}{L}{2}
				\lesssim \normtyp{u}{P}{2M}^3.
			\end{equation*}
		\paragraph{\textbf{Estimating $\mathcal{N}_2$, $\mathcal{N}_3$, and $\mathcal{N}_4$.}}
			We proceed as we did for $\mathcal{N}_1$ and obtain that
			\begin{equation*}
				\abs{\mathcal{N}_2^\alpha} \lesssim \normtyp{K}{P}{2M} \normtyp{\theta}{P}{2M}^2,\,
				\abs{\mathcal{N}_3^\alpha} \lesssim \normtyp{K}{P}{2M}^2 \normtyp{u}{P}{2M} \text{ and } 
				\abs{\mathcal{N}_4^\alpha} \lesssim \normtyp{K}{P}{2M}^2 \normtyp{\theta}{P}{2M}.
			\end{equation*}
		\paragraph{\textbf{Estimating $\mathcal{N}_5$.}}
			We split $\mathcal{N}_5^\alpha$ into two pieces:
			\begin{equation*}
				\mathcal{N}_5^\alpha = \sum_{\substack{ \beta+\gamma = \alpha \\ \beta > 0 }}
				\int_{\T^3} ( \partial^\beta K) \brac{ \pdt \partial^\gamma \theta} \cdot \brac{ \partial^\alpha \theta}
				= \sum_{\substack{ \cdots \\ \abs{\beta}=1 }} \cdots + \sum_{\substack{ \cdots \\ \abs{\beta}=2 }} =\vcentcolon \RN{1} + \RN{2},
			\end{equation*}
			where $\RN{1}$ is the \emph{only} term in $\mathcal{N}$ that requires the use of $ \normtyp{ \nabla K }{L}{\infty} $
			since it unavoidably contains a parabolic count of derivatives of $2M+1$.
			Estimating $\RN{1}$ is immediate:
			\begin{equation*}
				\abs{\RN{1}} = \vbrac{ \sum_{i=1}^3 \int_{\T^3} \brac{ \partial_i K } \brac{\pdt^{\alpha-e_i} \theta} \cdot \brac{ \partial^\alpha \theta} }
				\lesssim \sum_i \norm{\nabla K}{L^\infty} \norm{\pdt \partial^{\alpha-e_i}\theta}{L^2} \norm{ \partial^\alpha \theta}{L^4}
				\lesssim \normtyp{ \nabla K }{L}{\infty}  \normtyp{\theta}{P}{2M+1}^2.
			\end{equation*}
			Estimating $\RN{2}$ can be done via ``hands-on high-low'' estimates very similar to those employed to control $\mathcal{N}_1^\alpha$.
			Since $M\geqslant 4$ we see that $(2M-2)+(2M-4) > 2M$ and hence either $\abs{\beta}_P \leqslant 2M-2$ or $\abs{\gamma}_P \leqslant 2M-4$, such that
			\begin{align*}
				\abs{\RN{2}}
				\leqslant \sum_{\substack{ \cdots \\ \abs{\beta}_P \leqslant 2M-2 }}
					\normtypns{ \partial^\beta K}{L}{\infty} \normtyp{ \pdt \partial^\gamma \theta}{L}{2} \normtyp{ \partial^\alpha \theta}{L}{2}
					+ \!\!\sum_{\substack{ \cdots \\ \abs{\gamma}_P \leqslant 2M-4 }}\!\!
					\normtypns{ \partial^\beta K}{L}{2} \normtyp{ \pdt \partial^\gamma \theta}{L}{\infty} \normtyp{ \partial^\alpha \theta}{L}{2} 
				\lesssim \normtyp{K}{P}{2M} \normtyp{\theta}{P}{2M}^2.
			\end{align*}
		\paragraph{\textbf{Estimating $\mathcal{N}_6$.}}
			We split $\mathcal{N}_6$ into two pieces:
			\begin{align*}
				\mathcal{N}_6^\alpha
				= \int_{\T^3} \sbrac{ J_{eq} (u\cdot\nabla), \partial^\alpha } \theta \cdot \partial^\alpha \theta
					+ \int_{\T^3} \sbrac{ K(u\cdot\nabla), \partial^\alpha } \theta \cdot \partial^\alpha \theta
			\\
				= \sum_{\substack{ \beta + \gamma = \alpha \\ \beta > 0 }}
					\int_{\T^3} J_{eq} ( \partial^\beta u \cdot\nabla) \partial^\gamma \theta \cdot \partial^\alpha \theta
				+ \sum_{\substack{ \beta + \gamma + \delta = \alpha \\ \beta + \gamma > 0 }}
					\int_{\T^3} ( \partial^\beta K) \brac{ \partial^\gamma u \cdot \nabla} \partial^\delta \theta \cdot \partial^\alpha \theta
				=\vcentcolon \RN{1} + \RN{2}.
			\end{align*}
			To control $\RN{1}$ we proceed as we did for $\mathcal{N}_1$ and obtain that $\abs{\RN{1}} \lesssim \normtyp{(u, \theta)}{P}{2M}^3$.
			To control $\RN{2}$ we proceed in a similar fashion, namely with ``hands-on high-low'' estimates.
			Since the interaction is quartic we will rely on \fref{Lemma}{lemma:from_ineq_of_sums_to_overlap_ineq_3_terms}
			in order to ensure that there are always at least two factors that have a sufficiently low derivative count.

			More precisely: the key observation is that, since $\delta < \alpha$, all factors in $\RN{2}$ are controlled in $L^2$ via $P^{2M}$.
			To control $\RN{2}$ it therefore suffices to ensure that two of the four factors are controlled in $L^\infty$ through $P^{2M}$.
			This occurs when
			\begin{equation*}
				(1)\; \abs{\beta} \leqslant \abs{\alpha} - 2 \text{ for } \partial^\beta K,\,
				(2)\; \abs{\gamma} \leqslant \abs{\alpha} - 2 \text{ for } \partial^\gamma u,\, \text{ and } 
				(3)\; \abs{\delta} \leqslant \abs{\alpha} - 3 \text{ for } \nabla \partial^\delta \theta.
			\end{equation*}
			Crucially, since $M\geqslant 3$ and hence $(2M-2) + (2M-3) > 2M$, \fref{Lemma}{lemma:from_ineq_of_sums_to_overlap_ineq_3_terms} tells us
			that at least two out of (1), (2), or (3) hold. We may then deduce that $\abs{\RN{2}} \lesssim \normtyp{K}{P}{2M} \normtyp{(u,\theta)}{P}{2M}^3$.
			For example if (1) and (2) hold then we estimate the interaction as follows:
			\begin{equation*}
				\vbrac{ \int_{\T^3} ( \partial^\beta K) \brac{ \partial^\gamma u \cdot \nabla} \partial^\delta \theta \cdot \partial^\alpha \theta }
				\lesssim
					\normtypns{ \partial^\beta K}{L}{\infty} \normtyp{ \partial^\gamma u}{L}{\infty}
					\normtypns{ \nabla \partial^\delta \theta}{L}{2} \normtyp{ \partial^\alpha \theta}{L}{2}
				\lesssim \normtyp{K}{P}{2M} \normtyp{(u,\theta)}{P}{2M}^3.
			\end{equation*}
		\paragraph{\textbf{Estimating $\mathcal{N}_7$.}}
			We proceed similarly to how we handled $\mathcal{N}_6$.
			We begin by splitting $\mathcal{N}_7$ into four pieces:
			\begin{align*}
				\mathcal{N}_7^\alpha
				= \int_{\T^3} \sbrac{\omega_{eq} \times J_{eq},	\partial^\alpha } \theta \cdot \partial^\alpha \theta
				+ \int_{\T^3} \sbrac{\omega_{eq} \times K,	\partial^\alpha } \theta \cdot \partial^\alpha \theta
				+ \int_{\T^3} \sbrac{\theta	 \times J_{eq},	\partial^\alpha } \theta \cdot \partial^\alpha \theta
			\\
				+ \int_{\T^3} \sbrac{\theta	 \times K,	\partial^\alpha } \theta \cdot \partial^\alpha \theta
				=\vcentcolon \RN{1} + \RN{2} + \RN{3} + \RN{4}
			\end{align*}
			where note that $\RN{1} = 0$ since $\sbrac{\omega_{eq} \times J_{eq}, \partial^\alpha } = 0$.
			To estimate $\RN{2}$ and $\RN{3}$ we proceed as we did for $\mathcal{N}_1$ and obtain that
			\begin{equation*}
				\abs{\RN{2}} \lesssim \normtyp{K}{P}{2M} \normtyp{\theta}{P}{2M}^2
				\text{ and } 
				\abs{\RN{3}} \lesssim \normtyp{\theta}{P}{2M}^3.
			\end{equation*}
			Finally, to estimate $\RN{4}$ we proceed as did for $\RN{2}$ of $\mathcal{N}_6$, namely using \fref{Lemma}{lemma:from_ineq_of_sums_to_overlap_ineq_3_terms}
			to split up the terms in a fashion amenable to ``hands-on high-low estimates, and obtain
			that $\abs{\RN{4}} \lesssim \normtyp{K}{P}{2M} \normtyp{\theta}{P}{2M}^3 $.
	\end{proof}
	
	Once the nonlinear interactions are controlled we may deduce the a priori energy estimates recorded in \fref{Lemma}{lemma:local_a_priori_estimates} below.
	Recall that $ \widetilde{\mathcal{E}}_{M, loc}$ and $ \overline{\mathcal{D}}_M $ are defined in \eqref{eq:def_summed_local_energy} and \eqref{eq:not_D_M}, respectively.

	\begin{lemma}[Local a priori energy estimates]
	\label{lemma:local_a_priori_estimates}
		Suppose that $Z_n = (u_n, \theta_n, K_n) \in V_n$ solves \eqref{eq:ref_approx_sys}
		and satisfies $ \normtyp{ K_n }{L}{\infty} <\frac{\lambda}{2}$.
		For any integer $M\geqslant 4$ there exists $\delta_{ap}^{loc} > 0$ such that if $ \normtyp{ \nabla K_n }{L}{\infty} < \delta^{loc}_{ap}$ then
		\begin{equation}
		\label{eq:local_a_priori_est}
			\frac{\mathrm{d}}{\mathrm{d}t} \widetilde{\mathcal{E}}_{M,\text{loc}} (Z_n) + \frac{1}{2} \overline{\mathcal{D}}_M (u_n, \theta_n)
			\leqslant C_G \brac{ \widetilde{\mathcal{E}}_{M,\text{loc}}^{\,3/2} (Z_n) + \widetilde{\mathcal{E}}_{M,\text{loc}}^{\,2} (Z_n) }.
		\end{equation}
	\end{lemma}
	\begin{proof}
		The energy estimate of \fref{Proposition}{prop:gen_ED_approx_system} combined with \fref{Lemma}{lemma:coercivity_dissip} tells us that,
		for $\mathcal{N}$ as in \fref{Lemma}{lemma:est_local_interactions},
		$
			\frac{\mathrm{d}}{\mathrm{d}t} \widetilde{\mathcal{E}}_{M,\text{loc}} (Z_n) + \overline{\mathcal{D}}_M (u_n, \theta_n) \lesssim \mathcal{N}.
		$
		We may combine with with the estimate of $\mathcal{N}$ of \fref{Lemma}{lemma:est_local_interactions}
		and with \fref{Lemma}{lemma:exact_comparison_versions_E}, since $ \normtyp{ K_n }{L}{\infty} <\frac{\lambda}{2}$, to deduce that
		there exists $C_{ap}^{loc} > 0$ such that
		\begin{equation*}
			\frac{\mathrm{d}}{\mathrm{d}t} \widetilde{\mathcal{E}}_{M,\text{loc}} + \overline{\mathcal{D}}_M (u_n, \theta_n)
			\leqslant C_{ap}^{loc} \normtyp{\nabla K_n}{L}{\infty} \overline{\mathcal{D}}_M (u_n, \theta_n)
				+ \widetilde{\mathcal{E}}_{M,\text{loc}}^{\,3/2} (Z_n) + \widetilde{\mathcal{E}}_{M,\text{loc}}^{\,2} (Z_n).
		\end{equation*}
		So finally, if we pick $\delta^{loc}_{ap} > 0$ sufficiently small to ensure that $C_{ap}^{loc} \delta^{loc}_{ap} \leqslant \frac{1}{2} $, then we may conclude that
		there exists $C_G > 0$ such that \eqref{eq:local_a_priori_est} holds.
	\end{proof}

	To produce uniform bounds on the approximate solutions from the a priori estimates of \fref{Lemma}{lemma:local_a_priori_estimates}
	it suffices to couple it with a nonlinear Gronwall-type argument. 

	\begin{lemma}[Bihari argument]
	\label{lemma:nonlinear_Gronwall}
		Suppose that, for some $T>0$, $e, d : \cobrac{0,T} \to \cobrac{0, \infty}$ are continuous and satisfy, for some $\alpha_0 > 0$ and $C > 0$,
		\begin{align*}
			e'(t) + d(t) \leqslant Cf(e(t))
			\text{ for every } 0 \leqslant t < T
			\text{ and } e(0) \leqslant \alpha_0,
		\end{align*}
		where $f(x) \vcentcolon=  x^{3/2} + x^2$ for every $x \geqslant 0$. Then, for every $0 \leqslant t < \min\brac{T, \frac{2F(\alpha_0)}{C}}$,
		\begin{equation*}
			e(t) \leqslant F\inv \brac{ F(\alpha_0) - \frac{C}{2}t} \text{ and }
			\int_0^t d(s) ds \leqslant \alpha_0 + C t \brac{f \circ F}\inv \brac{ F(\alpha_0) - \frac{C}{2}t}
		\end{equation*}
		where $F(x) \defeq \frac{1}{\sqrt{x}} - \log\brac{1 + \frac{1}{\sqrt{x}}}$ for every $x>0$.
	\end{lemma}
	\begin{proof}
		Bounding $e$ follows from a standard nonlinear Gronwall argument (see for example \cite{boyer_fabrie}).
		The bound on $d$ follows from integrating the differential inequality in time and the monotonicity of $f$.
	\end{proof}

	We may now state the first of the two main results of this section, obtaining uniform bounds on the approximate solutions.
	Recall that $ \widetilde{\mathcal{E}}_{M, loc}$ and $ \overline{\mathcal{D}}_M $ are defined in \eqref{eq:def_summed_local_energy} and \eqref{eq:not_D_M}, respectively.

	\begin{cor}[Uniform a priori bounds on approximate solutions]
	\label{cor:unif_a_priori_bounds_approx_sol}
		Let $M \geqslant 4$ be an integer, let $T>0$ be some time horizon, and
		let ${( Z_n )}_{n\in\N}$ be a sequence of solutions $Z_n = (u_n, \theta_n, K_n) \in V_n$ such that, for every $n\in\N$, $Z_n$ solves \eqref{eq:ref_approx_sys}
		and satisfies
		$\norm{K_n}{\infty} < \frac{\lambda}{2}$ and $\norm{\nabla K_n}{\infty} < \delta^{loc}_{ap}$ for $\delta^{loc}_{ap}$ as in \fref{Lemma}{lemma:local_a_priori_estimates}
		on the time interval $\cobrac{0,T}$.
		Then, for every $n\in\N$ and every $\alpha_0 > 0$, if $ \widetilde{\mathcal{E}}_{M,\text{loc}} (Z_n (0)) \leqslant \alpha_0$ it follows that,
		for every $0 \leqslant t < \min\brac{T, \frac{2F(\alpha_0)}{C}}$,
		\begin{equation*}
			\widetilde{\mathcal{E}}_{M,\text{loc}} (Z_n (t)) \leqslant F\inv \brac{ F(\alpha_0) - \frac{C_Gt}{2}}
			\text{ and }
			\int_0^t \overline{\mathcal{D}}_M (u_n (s), \theta_n (s)) ds \leqslant \alpha_0 + C_G t \brac{f \circ F\inv} \brac{ F(\alpha_0) - \frac{C_Gt}{2}},
		\end{equation*}
		where $C_G > 0$ is as in \fref{Lemma}{lemma:local_a_priori_estimates}, $f(x) \vcentcolon= x^{3/2} + x^2$,
		and $F(x) \vcentcolon= \frac{1}{\sqrt{x}} - \log\brac{ 1 + \frac{1}{\sqrt{x}}}$.
	\end{cor}
	\begin{proof}
		This follows immediately from combining \fref{Lemma}{lemma:local_a_priori_estimates} and \fref{Lemma}{lemma:nonlinear_Gronwall}.
	\end{proof}

	We now turn our attention towards the second of the two main result of this section,
	namely controlling the initial energy (which involves temporal derivatives) exclusively in terms of spatial norms.
	In order to do so we first record the following estimates of the nonlinearities.

	\begin{lemma}[Estimates of the nonlinearities for the approximate problem]
	\label{lemma:nonlinear_est_approx_prob}
		Let $n, j, k, M \in \N$, where $2\leqslant j \leqslant M$, and let
		$
			Z = (u, \theta, K) \in L^2 \brac{\T^3;\, \R^3 \times \R^3 \times \sym(3) }.
		$
		The following estimates hold:
		\begin{align*}
			&(1)\; \normtyp{ \Leb_{Z,n} Z }{H}{k} \lesssim \normtyp{Z}{H}{k+2} + \normtyp{Z}{H}{k+1}^2,
			&&\hspace{-4em}(2)\; \normtyp{N_n(Z)}{H}{k} \lesssim \normtyp{Z}{H}{k+1}^2,
		\\
			&(3)\; \normtyp{\sbrac{K\pdt, \pdt^{j-1}}}{H}{2M-2j} \lesssim \norm{Z}{P^{2M}_{j-1}}^2,
			&&\hspace{-4em}(4)\; \normtyp{ \pdt^{j-1} ( N_n (Z) )}{H}{2M-2j} \lesssim \norm{Z}{P^{2M}_{j-1}}^2,
		\\
			&(5)\; \normtyp{
				     \sbrac{ (J_{eq} + K) (u\cdot\nabla), \pdt^{j-1} } \theta
				}{H}{2M-2j}
				\lesssim \norm{Z}{P^{2M}_{j-1}}^2 + \norm{Z}{P^{2M}_{j-1}}^3, \text{ and }
		\\
			&(6)\; \normtyp{
				      \sbrac{ (\omega_{eq} + \theta) \times (J_{eq} + K), \pdt^{j-1} }
				}{H}{2M-2j}
				\lesssim \norm{Z}{P^{2M}_{j-1}}^2 + \norm{Z}{P^{2M}_{j-1}}^3.
		\end{align*}
	\end{lemma}
	\begin{proof}
		These estimates rely mostly on the fact that, for $s > \frac{3}{2}$, $H^s \brac{ \T^3 }$ is a Banach algebra.

		To obtain (1) we proceed as in the beginning of \fref{Section}{sec:lwp} and split $\Leb_{Z, n}$ into its part with constant coefficients and the remainder,
		writing $\Leb_{Z, n} = \Leb_0 + \overline{\Leb}_{Z, n}$.
		In particular, the estimate $ \normtyp{\Leb_0 Z}{H}{k} \lesssim \normtyp{Z}{H}{k+2}$ is immediate.
		The estimate $ \normtyp{\overline{\Leb}_{Z,n} Z}{H}{k} \lesssim \normtyp{Z}{H}{k+1} + \normtyp{Z}{H}{k+1}^2$ follows from the fact that
		$H^{k+2}$ is a Banach algebra and from \fref{Lemma}{lemma:prod_est}.
		Obtaining (2) follows in the same way.

		Obtaining (3) -- (6) follows a similar procedure, and we thus only provide the details for (3).
		Observe that if $a \leqslant j-1$ and $b\leqslant j-2$ then $2M-2a, 2M-2b-2 \geqslant 2M-2j+2$.
		Crucially, since $2M-2j+2 \geqslant 2$ we know that $H^{2M-2j+2}$ is a Banach algebra, and hence
		\begin{align*}
			\normtypns{ \pdt^a K \pdt^{b+1} \theta }{H}{2M-2j}
			\lesssim \normtyp{\pdt^a K}{H}{2M-2j+2} \normtypns{\pdt^{b+1} \theta}{H}{2M-2j+2} 
			\lesssim \normtyp{\pdt^a K}{H}{2M-2j+2} \normtypns{\pdt^{b+1} \theta}{H}{2M-2b+2}
			\lesssim \norm{Z}{P^{2M}_{j-1}}^2.
		\end{align*}
		So finally
		\begin{align*}
			\normtyp{
				\sbrac{ K\pdt, \pdt^{j-1} } \theta
			}{H}{2M-2j}
			\lesssim \sum_{\substack{ a+b = j-1 \\ b < j-1 }} \normtypns{
				\pdt^a K \pdt^{b+1} \theta
			}{H}{2M-2j}
			\lesssim \norm{Z}{P^{2M}_{j-1}}^2.&\qedhere
		\end{align*}
	\end{proof}

	We may now conclude this section with the second of the two main result of this section
	and bound the initial energy in terms of purely spatial norms.

	\begin{lemma}[Bounds on the initial energy in terms of purely spatial norms]
	\label{lemma:bounds_on_P_2M_by_H_2M}
		Let $M\geqslant 0$ be an integer and $n\in\N$. Then there exist constants $C_{IC}, C_M > 0$ such that
		if $Z = (u, \theta, K)$ solves
		\begin{equation}
		\label{eq:bounds_on_P_2M_by_H_2M}
			\widetilde{T_n} (K) \pdt Z = \Leb_{Z, n} Z + N_n (Z)
		\end{equation}
		and satisfies $ \normtyp{ K }{L}{\infty} < \frac{\lambda}{2}$ then $Z(t)$ satisfies, for every $t$ for which it is defined,
		\begin{equation*}
			\normtyp{Z(t)}{P}{2M} \leqslant C_{IC} \brac{ \normtyp{Z(t)}{H}{2M} + \normtyp{Z(t)}{H}{2M}^{C_M} }.
		\end{equation*}
		In particular this holds when $t=0$.
	\end{lemma}
	\begin{proof}
		Suppose that $Z$ solves \eqref{eq:bounds_on_P_2M_by_H_2M}.
		Applying $j-1$ temporal derivatives then tells us that
		\begin{equation*}
			\widetilde{T}_n (K) \pdt (\pdt^{j-1} Z)
			= \Leb_{Z, n} (\pdt^{j-1} Z)
			+ \sbrac{ \widetilde{T}_n(K) \pdt, \pdt^{j-1} } Z
			- \sbrac{\Leb_{Z, n}, \pdt^{j-1}} Z
			+ \pdt^{j-1} (N_n(Z))
			=\vcentcolon F^j (Z)
		\end{equation*}
		where
		$
			\sbrac{ \widetilde(T)_n (K) \pdt, \pdt^{j-1} }
			= 0_3 \oplus P_n \circ \sbrac{ K\pdt, \pdt^{j-1} } \oplus 0_{3\times 3}
		$
		and
		\begin{align*}
			\sbrac{ \Leb_{Z,n}, \pdt^{j-1} }
			= 0_3 \oplus P_n \circ \brac{
				\sbrac{ (J_{eq} + K) (u\cdot\nabla), \pdt^{j-1} }
				+ \sbrac{ (\omega_{eq} + \theta) \times (J_{eq} + K), \pdt^{j-1} }
			} \oplus 0_{3 \times 3}.
		\end{align*}
		Therefore
		\begin{align*}
			F^j (Z)
			= \Leb_{Z, n} \brac{\pdt^{j-1} Z}
			+ P_n \brac{
				\sbrac{K_n \pdt, \pdt^{j-1}} \theta
			}
			- P_n \brac{
				\sbrac{ (J_{eq} + K)(u\cdot\nabla), \pdt^{j-1} } \theta
			}
		\\
			- P_n \brac{
				\sbrac{ (\omega_{eq} + \theta) \times (J_{eq} + K), \pdt^{j-1} } \theta
			}
			+ \pdt^{j-1} \brac{N_n(Z)}.
		\end{align*}
		such that, by \fref{Lemma}{lemma:nonlinear_est_approx_prob},
		\begin{align}
			\normtyp{F^j}{H}{2M-2j}
			\lesssim \normtypns{\pdt^{j-1}Z}{H}{2M-2j+2} + \normtypns{\pdt^{j-1}Z}{H}{2M-2j+1}^2
			+ \norm{Z}{P^{2M}_{j-1}}^2 + \norm{Z}{P^{2M}_{j-1}}^3
			\lesssim \norm{Z}{P^{2M}_{j-1}} + \norm{Z}{P^{2M}_{j-1}}^3.
		\label{eq:C_IC_est_2}
		\end{align}
		In particular we see that $\pdt^j Z = {\tilde{T}(K)}\inv F^j (Z)$.
		We now break into two cases, depending on whether $j \leqslant M-1$ or $j=M$.
		For $1\leqslant j \leqslant M-1$ we have that $2M-2j \geqslant 2$, so combining \fref{Lemma}{lemma:H_k_bounds_T_K_inv}
		with \eqref{eq:C_IC_est_2} tells us that
		\begin{equation}
		\label{eq:C_IC_est_3}
			\normtypns{\pdt^j Z}{H}{2M-2j}
			\lesssim \brac{ \normtyp{K}{H}{2M-2j} + \normtyp{K}{H}{2M-2j+2}^{2M-2j} } \normtyp{F^j}{H}{2M-2j}
			\lesssim \norm{Z}{P^{2M}_{j-1}}^2 + \norm{Z}{P^{2M}_{j-1}}^{2M-2j+3}.
		\end{equation}
		For $j=M$ we may not apply \fref{Lemma}{lemma:H_k_bounds_T_K_inv} since $2M-2j = 0 < 2$ but we do not need to since in that case we are only after an $L^2$ bound,
		which \fref{Lemma}{lemma:invertibility_T_K} readily provides.
		Indeed, using \fref{Lemma}{lemma:invertibility_T_K} and \eqref{eq:C_IC_est_2} we see that
		\begin{equation}
		\label{eq:C_IC_est_4}
			\norm{\pdt^M Z}{L^2}
			\leqslant \frac{2}{\lambda} \norm{ F^{M-1} (Z) }{L^2}
			\lesssim \norm{Z}{P^{2M}_{M-1}} + \norm{Z}{P^{2M}_{M-1}}^3.
		\end{equation}
		Crucially, combining \eqref{eq:C_IC_est_3} and \eqref{eq:C_IC_est_4} tells us that, for $1\leqslant j \leqslant M$,
		\begin{equation*}
			\norm{Z}{P^{2M}_{j}}
			\asymp \norm{Z}{P^{2M}_{j-1}} + \normtypns{\pdt^j Z}{H}{2M-2j}
			\lesssim \norm{Z}{P^{2M}_{j-1}} + \norm{Z}{P^{2M}_{j-1}}^{2M-2j+3},
		\end{equation*}
		from which the claim follows by induction.
	\end{proof}

\subsection{The Galerkin scheme}
\label{sec:Galerkin}

	In this section we put together the Galerkin scheme that will produce solutions to \eqref{eq:pertub_sys_no_ten_pdt_u}--\eqref{eq:pertub_sys_ten_pdt_K} locally-in-time.
	We proceed in a standard manner, first producing local approximate solutions, then obtaining uniform estimates on the approximates sufficient to 
	obtain a uniform lower bound on the time of existence and to pass to the limit by compactness.
	To conclude we pass to the limit and reconstruct the pressure.
	We begin by producing local approximate solutions.

	\begin{prop}[Producing local approximate solutions]
	\label{prop:produce_local_approx_sols}
		Let $\delta^{loc}_{ap} > 0$ be as in \fref{Lemma}{lemma:local_a_priori_estimates},
		pick some $0 < \sigma < \sigma \brac{ \delta^{loc}_{ap} }$ as in \eqref{eq:def_sigma}, and let $\U(\sigma)$ be defined as in \eqref{eq:def_U}.
		For every $Z_0 = (u_0, \theta_0, K_0) \in \U(\sigma)$ and every $n\in\N$ there exists a maximal time of existence $T_n > 0$ and a unique solution
		$
			Z_n = (u_n, \theta_n ,K_n)$
		in
		$
			C^\infty \brac{ \cobrac{0, T_n};\, \U_n(\sigma)}
		$
		of
		\begin{subnumcases}{}
			\widetilde{T}_n (K_n) \pdt Z_n = \Leb_{Z_n, n} Z_n + N_n (Z_n) \text{ and }
			\label{eq:prod_local_approx_sol_1}\\
			Z_n (0) = \mathcal{P}_n Z_0.
			\label{eq:prod_local_approx_sol_2}
		\end{subnumcases}
		Moreover we have the following blow-up criterion: for any $T > 0$, if $\sup\limits_{0\leqslant t \leqslant T} \norm{K_n (t)}{H^3} < \sigma$ then $T_n \geqslant T$.
	\end{prop}
	\begin{proof}
		The key is to write the system \eqref{eq:prod_local_approx_sol_1}--\eqref{eq:prod_local_approx_sol_2} as a finite-dimensional ODE in the standard form $\dot{x}(t) = f(x(t))$.
		Observe that by choice of $\sigma > 0$ and definition of $\U(\sigma)$, it follows from \fref{Lemma}{lemma:invertibility_T_K} that $\widetilde{T}_n (K_n)$ is invertible
		for any $Z_n \in \U_n(\sigma)$.
		The system \eqref{eq:prod_local_approx_sol_1}--\eqref{eq:prod_local_approx_sol_2} is thus equivalent to
		\begin{equation*}
			\pdt Z_n = {\widetilde{T}_n (K_n)}\inv \brac{
				\Leb_{Z_n, n} Z_n + N_n (Z_n)
			} =\vcentcolon F_n (Z_n) \text{ and }
			Z_n (0) = \mathbb{P}_n Z_0.
		\end{equation*}
		Since $Z_n \mapsto \Leb_{Z_n, n} Z_n + N_n(Z_n)$ is, up to the appearances of the projections $ P_n $ and $ \mathbb{P}_L $,
		a polynomial in $(Z_n, \nabla Z_n, \nabla^2 Z_n)$, it follows from \fref{Lemma}{lemma:smooth_T_K_inv} and the equivalence of $H^s (\T^3)$ norms ($s\geqslant 0$) on $V_n$
		that $Z_n \mapsto F_n (Z_n)$ is a smooth map from $\U_n(\sigma)$ to $V_n$.
		Note that deducing that the image of $F_n$ lies in $V_n$ comes from the fact that the Leray projection $\mathbb{P}_L$ enforces the divergence-free condition and
		preserves the average of the velocity of $u$ since $\hat{P}_L (0) = I$.
		By standard well-posedness theory for finite-dimensional ODEs we may now deduce the result,
		noting that the blow-up criterion follows from the definition of $\U(\sigma)$.
	\end{proof}

	We may now put together the local a priori estimates of \fref{Corollary}{cor:unif_a_priori_bounds_approx_sol}
	and the a priori projected advection-rotation estimates for $K_n$ of \fref{Lemma}{lemma:H_k_est_proj_adv_rot_eq}
	in order to deduce uniform bounds on the approximate solutions.

	\begin{prop}[Uniform bounds on approximate solutions and their intervals of existence]
	\label{prop:unif_bounds_time_of_exist_and_approx_sols}
		Let $\sigma > 0$ and $\U(\sigma)$ be as in \fref{Proposition}{prop:produce_local_approx_sols},
		let $C_K > 0$ be the constant implicit in the result of \fref{Lemma}{lemma:H_k_est_proj_adv_rot_eq} when $k=3$,
		let $M\geqslant 4$ be an integer,
		let $Z_0 = (u_0, \theta_0, K_0) \in \U(\sigma)$ with
		$
			\norm{Z_0}{H^{2M}}, \norm{K_0}{H^{2M+1}} < \infty
		$
		and
		\begin{equation}
		\label{eq:unif_T_n_assumption_K_0}
			\normtyp{K_0}{H}{3} < \sigma_* \vcentcolon= \frac{\sigma}{2 C_K},
		\end{equation}
		and let ${(Z_n)}_{n\in\N}$ be the sequence of approximate solutions obtained in \fref{Proposition}{prop:produce_local_approx_sols},
		with corresponding maximal times of existence ${(T_n)}_{n\in\N}$.
		There exists $0 < T_\text{lwp} \leqslant 1$ and there exist
		$
			\rho_e, \rho_d : (0, \infty) \to (0,\infty)
		$
		and
		$
			\rho_f : {(0, \infty)}^2 \to (0,\infty)
		$
		which are continuous, strictly increasing in each of their arguments, and asymptotically vanishing at zero
		such that $T_n \geqslant T_\text{lwp}$ for all $n\in\N$ and
		\begin{subnumcases}{}
			\sup_{n\in\N} \sup_{0\leqslant j\leqslant M} \normns{\pdt^j Z_n}{L^\infty \brac{ \sbrac{0, T_\text{lwp}},\, H^{2M-2j}}} \leqslant \rho_e \brac{ \norm{Z_0}{H^{2M}} },
			\label{eq:unif_T_n_bound_e}\\
			\sup_{n\in\N} \sup_{0\leqslant j\leqslant M} \normns{(u_n, \theta_n)}{L^2 \brac{ \sbrac{0, T_\text{lwp}},\, H^{2M+1-2j}}} \leqslant \rho_d \brac{ \norm{Z_0}{H^{2M}}}, \text{and}
			\label{eq:unif_T_n_bound_d}\\
			\sup_{n\in\N} \norm{K_n}{L^\infty \brac{ \sbrac{0, T_\text{lwp}},\, H^{2M+1}}} \leqslant \rho_f \brac{ \norm{Z_0}{H^{2M}},\, \norm{K_0}{H^{2M+1}} }.
			\label{eq:unif_T_n_bound_f}
		\end{subnumcases}
		Moreover $T_\text{lwp} = \phi \brac{ \normtyp{Z_0}{H}{2M} }$, where $\phi$ is non-increasing.
	\end{prop}
	\begin{proof}
		More precisely, let us define
		\begin{equation}
		\label{eq:unif_T_n_def_sigma_0}
			\sigma_0 \vcentcolon= 2 \widetilde{C}_E C_{IC}^2 \brac{ \norm{Z_0}{H^{2M}}^2 + \norm{Z_0}{H^{2M}}^{2C_M} }
		\end{equation}
		where $\widetilde{C}_E, C_{IC}, C_M > 0$ are as in \fref{Lemma}{lemma:exact_comparison_versions_E} and \fref{Lemma}{lemma:bounds_on_P_2M_by_H_2M}.
		We note that,
		\begin{itemize}
			\item	by definition of $\U(\sigma)$ (and of $\U_n(\sigma))$, for every $n\in\N$,
				$\norm{K_n}{H^3} < \sigma$ on $\cobrac{0, T_n}$, where $\sigma > 0$ is as in \fref{Proposition}{prop:produce_local_approx_sols}, and that,
			\item 	by \fref{Lemma}{lemma:exact_comparison_versions_E} and \fref{Lemma}{lemma:bounds_on_P_2M_by_H_2M},
				$ \widetilde{\mathcal{E}}_{M,\text{loc}} (Z_n (0)) \leqslant \sigma_0$.
		\end{itemize}
		We may thus use \fref{Corollary}{cor:unif_a_priori_bounds_approx_sol} to deduce that, for all $n\in\N$ and all $0 \leqslant t < \min\brac{T_n, \frac{2F(\sigma_0)}{C_G}}$,
		\begin{equation*}
			\widetilde{\mathcal{E}}_{M,\text{loc}} (Z_n(t)) \leqslant F\inv \brac{ F(\sigma_0) - \frac{C_G t}{2} } \text{ and }
			\int_0^t \overline{\mathcal{D}}_M (u_n, \theta_n) (s) ds \leqslant \sigma_0 + C_G t \brac{f \circ F\inv} \brac{ F(\sigma_0) - \frac{C_G t}{2} },
		\end{equation*}
		where recall that $ \widetilde{\mathcal{E}}_{M, loc}$ and $ \overline{\mathcal{D}}_M $ are defined in \eqref{eq:def_summed_local_energy} and \eqref{eq:not_D_M}, respectively.
		In particular if we pick $t = \frac{1}{2} \brac{ \frac{2F(\sigma_0)}{C_G} } =\vcentcolon \frac{T_G}{2}$ then
		we have that $ F\inv \brac{F(\sigma_0) - \frac{C_G t}{2}} = F\inv \brac{\frac{1}{2} F(\sigma_0)}$,
		and hence, for every $n\in\N$ and every $0 \leqslant t \leqslant T_n \land \frac{T_G}{2} \land 1$,
		\begin{equation}
		\label{eq:unif_est_sol_rho_e}
			\norm{Z_n(t)}{P^{2M}}^2
			= \overline{\mathcal{E}}_{M,\text{loc}} (Z_n(t))
			\leqslant \frac{1}{\tilde{c}_E} \widetilde{\mathcal{E}}_{M,\text{loc}} (Z_n(t))
			\leqslant \frac{1}{\tilde{c}_E} F\inv\brac{ \frac{1}{2} F(\sigma_0)}
			=\vcentcolon \rho_e^2 \brac{ \norm{Z_0}{H^{2M}} }
		\end{equation}
		and, since $\fint_{\T^3} u = 0$, it follows from \fref{Lemma}{lemma:coercivity_dissip} that, for every $n\in\N$,
		\begin{align}
			\int_0^{T_n \land \frac{T_G}{2} \land 1} \norm{(u_n, \theta_n)(s)}{P^{2M+1}}^2 ds
			\leqslant C_D \int_0^{T_n \land \frac{T_G}{2} \land 1} \overline{\mathcal{D}}_M (s) ds
			\leqslant C_D \brac{ \sigma_0 + C_G \brac{f \circ F\inv}\brac{ \frac{1}{2} F(\sigma_0)}}
		\nonumber
		\\
			=\vcentcolon \rho_d^2 \brac{\norm{Z_0}{H^{2M}}}.
		\label{eq:unif_est_sol_rho_d}
		\end{align}

		We may now appeal to the estimates for $\norm{K_n}{H^3}$ from \fref{Lemma}{lemma:H_k_est_proj_adv_rot_eq} to obtain a lower bound on $T_n$ which is uniform in $n$.
		Since $\norm{K_n(0)}{H^3} < \sigma$ and since $M\geqslant 2$, we know from \fref{Lemma}{lemma:H_k_est_proj_adv_rot_eq} and \eqref{eq:unif_T_n_assumption_K_0} that,
		for any $n\in\N$ and any $0 \leqslant t \leqslant T_n \land \frac{T_G}{2} \land 1$,
		\begin{equation*}
			\sup_{0\leqslant s \leqslant t} \norm{K_n(s)}{H^3}
			\leqslant C_K e^{t \rho_e \brac{\norm{Z_0}{H^{2M}}}} \brac{\frac{\sigma}{2 C_K} + t \rho_e \brac{\norm{Z_0}{H^{2M}}}}
			=\vcentcolon \omega\brac{ t \rho_e\brac{ \normtyp{Z_0}{H}{2M} }}
		\end{equation*}
		where $\omega$ depends on $\sigma$.
		Crucially, observe that $\omega(0) = \frac{\sigma}{2}$ and that $\omega$ is strictly increasing,
		so we see that for $T_\text{small} \vcentcolon= \frac{\omega\inv\brac{2\sigma/3}}{\rho_e\brac{ \normtyp{Z_0}{H}{2M} }}$
		and $\tilde{T}_n \vcentcolon= T_n \land \frac{T_G}{2} \land T_\text{small} \land 1$,
		\begin{equation*}
			\sup_{0\leqslant t\leqslant \tilde{T}_n} \norm{K_n(t)}{H^3}
			\leqslant \omega\brac{ T_\text{small} \,\rho_e\brac{ \normtyp{Z_0}{H}{2M} } }
			\leqslant \frac{2\sigma}{3}.
		\end{equation*}
		Therefore, by the blow-up criterion: $T_n \geqslant \frac{T_G}{2} \land T_\text{small} \land 1 =\vcentcolon T_\text{lwp}$ for every $n\in\N$.

		Note that
		\begin{equation*}
			T_\text{lwp}
			= \frac{T_G}{2} \land T_\text{small} \land 1
			= \frac{F (\sigma_0)}{C_G} \land \frac{\omega\inv\brac{2\sigma/3}}{\rho_e\brac{ \normtyp{Z_0}{H}{2M} }} \land 1.
		\end{equation*}
		In light of \eqref{eq:unif_T_n_def_sigma_0} and the facts that $F$ and $\rho_e$ are strictly decreasing and strictly increasing, respectively,
		we deduce that $T_\text{lwp}$ is non-increasing with respect to $ \normtyp{Z_0}{H}{2M}$, as desired.

		Finally we record the estimates on $K$ obtained in \fref{Lemma}{lemma:H_k_est_proj_adv_rot_eq}.
		It follows from the energy-dissipation estimates \eqref{eq:unif_est_sol_rho_e} and \eqref{eq:unif_est_sol_rho_d} that
		\begin{equation*}
			\sup_{0\leqslant t\leqslant T_\text{lwp}} \norm{K_n (t)}{H^{2M+1}}
			\leqslant C_K' e^{\rho_e T_\text{lwp}} \brac{ \norm{K_0}{H^{2M+1}} + \rho_d T_\text{lwp} }
			=\vcentcolon \rho_f\brac{\norm{Z_0}{H^{2M}} + \norm{K_0}{H^{2M+1}} }.\qedhere
		\end{equation*}
	\end{proof}

	With these uniform bounds in hand we may move towards passing to the limit.
	First we record the following technical lemma which is essential in allowing us to pass to the limit.

	\begin{lemma}
	\label{lemma:conv_orthog_complement_P_n_in_L2_Hs}
		Let $s\geqslant 0$ and $f\in L^2\brac{\cobrac{0, T}; H^s\brac{\T^n}}$ for $T>0$.
		Then
		$\norm{ ( P_n - I )f}{L^2 H^s} \to 0 $ as $n\to \infty$.
	\end{lemma}
	\begin{proof}
		This follows immediately from Tonelli's Theorem and the monotone convergence Theorem.
	\end{proof}

	We may now pass to the limit by compactness.

	\begin{prop}[Compactness and passage to the limit]
	\label{prop:compactness_and_pass_to_limit}
		Let $\U(\sigma)$ be as in \fref{Proposition}{prop:produce_local_approx_sols},
		let $M\geqslant 4$ be an integer,
		and let $Z_0 = (u_0, \theta_0, K_0) \in \U(\sigma)$ with
		$
			\norm{Z_0}{H^{2M}}, \norm{K_0}{H^{2M+1}} < \infty
		$
		and $ \normtyp{K_0}{H}{3} < \sigma_*$ for $\sigma_* > 0$ as in \fref{Proposition}{prop:unif_bounds_time_of_exist_and_approx_sols}.
		There exist $0 < T_\text{lwp} \leqslant 1$ and $Z = (u, \theta, K) \in C^2 \brac{ \sbrac{0, T_\text{lwp}} \times \T^3}$
		such that $Z(t, \,\cdot\,) \in \U(\sigma)$ for all $0 \leqslant t \leqslant T_\text{lwp}$ and $Z$ solves
		\eqref{eq:pertub_sys_no_ten_div}--\eqref{eq:pertub_sys_ten_pdt_K} and \eqref{eq:full_sys_u_Leray_proj}.
		Moreover $Z$ satisfies the estimates
		\begin{align*}
			\sup_{0\leqslant j \leqslant M} \normns{ \pdt^j Z}{L^\infty H^{2M-2j}} \leqslant \rho_e \brac{ \normtyp{Z_0}{H}{2M} },\;
			\sup_{0\leqslant j \leqslant M} \normns{ \pdt^j (u, \theta)}{L^2 H^{2M-2j+1}} \leqslant \rho_d \brac{ \normtyp{Z_0}{H}{2M} },
			\\
			\text{ and }
			\norm{K}{L^\infty H^{2M+1}} \leqslant \rho_f \brac{ \normtyp{Z_0}{H}{2M}, \normtyp{K_0}{H}{2M+1} }
		\end{align*}
		for $\rho_e$, $\rho_d$, and $\rho_f$ as in \fref{Proposition}{prop:unif_bounds_time_of_exist_and_approx_sols}.
	\end{prop}
	\begin{proof}
		Let ${( Z_n )}_{n\in\N}$ and ${(T_n)}_{n\in\N}$ denote the approximate solutions and their times of existence as obtained in \fref{Proposition}{prop:produce_local_approx_sols}.
		Note that, as per \fref{Proposition}{prop:unif_bounds_time_of_exist_and_approx_sols}, we know that $T_n \geqslant T_\text{lwp} > 0$
		for some $T_\text{lwp}$ which is independent of $n$.
		We also know from \fref{Proposition}{prop:unif_bounds_time_of_exist_and_approx_sols} that \eqref{eq:unif_T_n_bound_e}--\eqref{eq:unif_T_n_bound_f} hold.
		It then follows from Banach-Alaoglu (i.e. weak-* compactness) that, up to a subsequence which we do not relabel,
		\begin{equation}
		\label{eq:compactness_weak_limits}
			\left\{
			\begin{aligned}
				\pdt^j Z_n &\overset{\ast}{\rightharpoonup} \pdt^j Z \text{ in } L^\infty H^{2M-2j} \text{ for every } 0\leqslant j \leqslant M,\\
				\pdt^j (u_n, \theta_n) &\rightharpoonup \pdt^j (u, \theta) \text{ in } L^2 H^{2M-2j+1} \text{ for every } 0\leqslant j \leqslant M, \text{ and }\\
				K_n &\overset{\ast}{\rightharpoonup} K \text{ in } L^\infty H^{2M+1}
			\end{aligned}
			\right.
		\end{equation}
		for some $Z = (u, \theta, K) $ which satisfies, by weak and weak-* lower semi-continuity of the respective norms,
		\begin{align*}
			\sup_{0\leqslant j \leqslant M} \normns{ \pdt^j Z}{L^\infty H^{2M-2j}} \leqslant \norm{Z}{L^\infty P^{2M}} \leqslant \sqrt{\rho_e},\;
			\sup_{0\leqslant j \leqslant M} \normns{ \pdt^j (u, \theta)}{L^2 H^{2M-2j+1}} \leqslant \norm{(u, \theta)}{L^2 P^{2M+1}} \leqslant \sqrt{\rho_d},
			\\
			\text{ and }
			\norm{K}{L^\infty H^{2M+1}} \leqslant \sqrt{\rho_f}.
		\end{align*}

		All that remains is passing to the limit, for which we omit the details since this is done with a standard application of the Aubin-Lions-Simon Compactness Theorem
		(see for example \cite{boyer_fabrie}) in combination with \fref{Lemma}{lemma:conv_orthog_complement_P_n_in_L2_Hs} and the fact that $H^s$ is a Banach algebra when $s>\frac{3}{2}$.
		In particular we can pass to the limit in the nonlinearities uniformly on $\sbrac{0, T_\text{lwp}} \times \T^3$ such that the following limits hold in $C^0_{t,x}$:
		\begin{align*}
			P_n \brac{K_n \pdt\theta_n}						\to K\pdt \theta,\,
			P_n \brac{ (J_{eq} + K_n) (u_n \cdot \nabla) \theta_n }			\to (J_{eq} + K)(u\cdot\nabla)\theta,\\
			N_n (Z_n)								\to N(Z),
			\text{ and } P_n \brac{ (\omega_{eq} + \theta_n) \times (J_{eq} + K_n) \theta_n}	\to (\omega_{eq} + \theta) \times (J_{eq} + K) \theta.&\qedhere
		\end{align*}
	\end{proof}
	The last step of our Galerkin scheme is to reconstruct the pressure and the initial condition, which we do below.
	\begin{cor}[Reconstructing the pressure and the initial condition]
	\label{cor:reconstruct_pressure_and_IC}
		Under the assumptions of \fref{Proposition}{prop:compactness_and_pass_to_limit} we know moreover that
		$Z(0) = Z_0$ pointwise and that there exist $p\in L^2 H^{2M+1} \cap L^{\infty} P^{2M}_{M-1}$ such that
		$Z$ and $p$ solve \eqref{eq:pertub_sys_no_ten_pdt_u}--\eqref{eq:pertub_sys_ten_pdt_K}.
	\end{cor}
	\begin{proof}
		Recovering the initial condition is trivial. Since $Z_n (0) \vcentcolon= P_n Z_0$ with $Z_0 \in H^2$ it follows directly from
		the weak convergence recorded in \fref{Proposition}{prop:compactness_and_pass_to_limit} that $Z_n \to Z$ in $C^0 H^2$,
		and hence \fref{Lemma}{lemma:conv_orthog_complement_P_n_in_L2_Hs} tells us that
		\begin{equation*}
			\norm{Z(0) - Z_0}{C^0_x}
			\lesssim \norm{Z(0) - Z_n (0)}{C^0_x} + \norm{Z_n(0) - Z_0}{C^0_x}
			\lesssim \norm{Z - Z_n}{C^0 H^2} + \norm{ (P_n - I) Z_0 }{H^2}
			\to 0 \text{ as } n\to\infty.
		\end{equation*}

		We now reconstruct the pressure. We have split
		\begin{equation*}
			\pdt u + u\cdot\nabla u = (\mu + \kappa/2) \Delta u - \kappa\nabla\times\omega - \nabla p
		\end{equation*}
		subject to $\nabla\cdot u = 0$ into two parts, namely
		\begin{equation*}
			\pdt u + \mathbb{P}_L (u\cdot\nabla u) = (\mu + \kappa/2) u - \kappa\nabla\times\omega \text{ and }
			(I - \mathbb{P}_L) (u\cdot\nabla u) = -\nabla p
		\end{equation*}
		where $\mathbb{P}_L$ is the Leray projector, i.e. the $L^2$-orthogonal projection onto divergence-free vector fields given by $\hat{P}_L (k) = I - \frac{k\otimes k}{\abs{k}^2}$
		for every $k\in\Z^3$, where $\frac{k\otimes k}{\abs{k}^2} \vert_{k = 0} \vcentcolon= 0$.
		Then $I - \mathbb{P}_L = \nabla\Delta\inv\nabla\cdot$ and hence we may define $p \vcentcolon= -\Delta\inv \nabla\cdot (u\cdot\nabla u)$.
		In particular, for $s > \frac{3}{2}$ we have the estimate
		$
			\norm{p}{H^{s+1}}
			\lesssim \norm{u}{H^s} \norm{u}{H^{s+1}}
		$
		from which we deduce that, using standard ``hands-on high-low estimates'',
		\begin{equation*}
			\norm{p}{L^2 P^{2M+1}} \lesssim \norm{u}{L^\infty P^{2M}} \norm{u}{L^2 P^{2M+1}} \text{ and }
			\norm{p}{L^\infty P^{2M}_{M-1}} \lesssim \norm{u}{L^\infty P^{2M}}^2.\qedhere
		\end{equation*}
	\end{proof}

	In conclusion, we have proved in this section the following local well-posedness result.

	\begin{thm}[Local well-posedness]
	\label{thm:lwp}
		Let $M \geqslant 4$ be an integer, let $\delta^{loc}_{ap} > 0$ be as in \fref{Lemma}{lemma:local_a_priori_estimates},
		let $C_K > 0$ be the constant implicit in the result of \fref{Lemma}{lemma:H_k_est_proj_adv_rot_eq} when $k=3$,
		let $0 < \sigma < \sigma \brac{ \delta_{loc}^{ap} }$,
		and let $\sigma_* \vcentcolon= \frac{\sigma}{2C_K}$.
		Let
		$
			Z_0 = (u_0, \theta_0, K_0) \in L^2 \brac{ \T^3;\, \R^3 \times \R^3 \times \sym(3) }
		$
		such that
		\begin{equation*}
			\nabla\cdot u_0 = 0,\,
			\fint_{\T^3} u_0 = 0,\,
			\normtyp{K_0}{H}{3} < \sigma_*, \text{ and } 
			\normtyp{Z_0}{H}{2M},\, \normtyp{K_0}{H}{2M+1} < \infty.
		\end{equation*}
		There exist $0 < T_\text{lwp} \leqslant 1$,
		\begin{equation*}
			Z = (u, \theta, K) \in C^2 \brac{ \sbrac{0, T_\text{lwp}} \times \T^3;\, \R^3 \times \R^3 \times \sym(3) },
			\text{ and } 
			p \in C^2 \brac{ \sbrac{0, T_\text{lwp}} \times \T^3;\, \R}
		\end{equation*}
		such that $Z$ and $p$ form the unique strong solution of \eqref{eq:pertub_sys_no_ten_pdt_u}--\eqref{eq:pertub_sys_ten_pdt_K}.
		Moreover, $T_\text{lwp} = \phi\brac{ \normtyp{Z_0}{H}{2M} }$ for some non-increasing function $\phi$, and for every $0 \leqslant t \leqslant T_\text{lwp}$, the solution satisfies
		\begin{equation*}
			\nabla\cdot u(t, \,\cdot\,) = 0,\,
			\fint_{\T^3} u(t, \,\cdot\,) = 0, \text{ and } 
			\normtyp{K(t, \,\cdot\,)}{H}{3} < \sigma
		\end{equation*}
		as well as the estimates
		\begin{align*}
			&\sup_{0 \leqslant j \leqslant M} \normns{ \pdt^j Z }{ L^\infty H^{2M-2j} }
				+ \sup_{0 \leqslant j \leqslant M-1} \normns{ \pdt^j p}{ L^\infty H^{2M-2j} }
				\leqslant \rho_e \brac{ \normtyp{Z_0}{H}{2M} },\\
			&\sup_{0 \leqslant j \leqslant M} \normns{ \pdt^j (u, \theta) }{ L^2 H^{2M-2j+1} } + \normns{ \pdt^j p}{ L^2 H^{2M-2j+1} }
				\leqslant \rho_d \brac{ \normtyp{Z_0}{H}{2M} }, \text{ and }\\
			&\norm{K}{ L^\infty H^{2M+1} } \leqslant \rho_f \brac{ \normtyp{Z_0}{H}{2M},\, \normtyp{K_0}{H}{2M+1} }
		\end{align*}
		where $\rho_e, \rho_d : (0, \infty) \to (0, \infty)$ and $\rho_f : {(0, \infty)}^2 \to (0, \infty)$
		are continuous, strictly increasing in each of their arguments, and asymptotically vanishing at zero.
	\end{thm}
	\begin{proof}
		This follows from combining the various results of this section.
		Producing local approximate solutions is done using \fref{Proposition}{prop:produce_local_approx_sols}.
		We then employ \fref{Proposition}{prop:unif_bounds_time_of_exist_and_approx_sols} to obtain uniform bounds on the times of existence and the approximate solutions,
		which allows us to pass to the limit using \fref{Proposition}{prop:compactness_and_pass_to_limit}.
		Finally we reconstruct the pressure and the initial condition using \fref{Corollary}{cor:reconstruct_pressure_and_IC}.
		Note that the uniqueness follows from Theorem A.5 of \cite{rt_tice_amf_iut}, which is recorded below in \fref{Theorem}{thm:uniqueness}
		for the reader's convenience.
	\end{proof}

	\begin{thm}[Uniqueness]
	\label{thm:uniqueness}
		Suppose that $\brac{u_1, p_1, \omega_1, J_1}$ and $\brac{u_2, p_2, \omega_2, J_2}$ are strong solutions of \eqref{eq:full_sys_u}--\eqref{eq:full_sys_J}.
		on some common time interval $\brac{0,T}$ such that they agree at time $t=0$.
		If $J_1$ is uniformly positive-definite, $p_i, \pdt\brac{u_i, \omega_i, J_i} \in L_T^2 L^2$, $\brac{u_i, \omega_i, J_i}, \nabla\brac{u_i, \omega_i, J_i} \in L_T^\infty L^\infty$,
		and $\pdt J_1, \pdt\omega_2 \in L_T^\infty L^\infty$, then these solutions coincide on $\brac{0,T}$.
		Note that here $L^p_T L^q$ denotes the space $L^p \brac{ \cobrac{0,T};\, L^q\brac{\T^3} }$.
	\end{thm}

%----------------------------------------------------------------------------------------------------
%	CONTINUATION ARGUMENT
%----------------------------------------------------------------------------------------------------

\section{Continuation argument}
\label{sec:cont}

In this section we derive the estimates necessary to ``glue'' the a priori estimates of \fref{Section}{sec:a_prioris} and the local well-posedness theory of \fref{Section}{sec:lwp}.
We begin with ``reduced energy estimates'' in \fref{Section}{sec:red_en_est} (whose purposed is detailed in \fref{Section}{sec:discuss_cont}).
We recall that while the a priori estimates of \fref{Section}{sec:a_prioris} rely on the smallness of the solution,
the estimates here rely on the smallness of the time interval on which they hold.

Once we have these reduced energy estimates in hand we obtain supplementary estimates on \fref{Section}{sec:cont_supp_est}
before recording a continuation argument in \fref{Section}{sec:cont_synthesis}.
In some sense this continuation argument is the technical implementation of what was heuristically described as ``gluing'' the a priori estimates and the local well-posedness together.

\subsection{Local-in-time reduced energy estimates}
\label{sec:red_en_est}

	In this section we derive the local-in-time reduced energy estimates.
	We will follow a procedure familiar from \fref{Sections}{sec:a_prioris} and \ref{sec:est_approx_prob}:
	we first introduce appropriate notation, then record the relevant energy-dissipation relation and the precise form of the nonlinear interactions that arise,
	and finally we estimate these nonlinear interactions and close the reduced energy estimates.

	Let us introduce compact notation that will be used throughout this section when developing the local-in-time reduced energy estimates.
	Considering the functions 	$Y	= (v, \phi, b)		: \cobrac{0, T} \times \T^3 \to \R^3 \times \R^3 \times \R^2$,
					$W	= (u, \theta, K)	: \cobrac{0, T} \times \T^3 \to \R^3 \times \R^3 \times \sym(3)$,
	and 				$\mathcal{F}= (f_1, f_2, f_3)	: \cobrac{0, T} \times \T^3 \to \R^3 \times \R^3 \times \R^2$
	we will write the system
	\begin{equation*}
		\left\{
		\begin{aligned}
			&\pdt v - (\nabla\cdot T)(v, \phi) = f_1,\\
			&(J_{eq} + K)\pdt\phi + (J_{eq} + K) (u\cdot\nabla)\phi + (\omega_{eq} + \theta) \times (J_{eq} + K) \phi
				\\&\qquad
				+ \tilde{\tau}^2 \tilde{b}^\perp + \phi \times J_{eq} \omega_{eq} - 2 \vc T(v, \phi) + (\nabla\cdot M)(\phi) = f_2, \text{ and }\\
			&\pdt b - \tilde{\tau} b^\perp + (\nu - \lambda) \bar{\phi}^\perp = f_3
		\end{aligned}
		\right.
	\end{equation*}
	in more compact form as
	$
		\overline{T}(K) \pdt Y - \Leb_W Y = \mathcal{F}
	$
	where $\overline{T}(K) = I_3 \oplus ( J_{eq} + K ) \oplus I_2$ and the operator $\Leb_W$ is given by
	$
		\Leb_W Y = \brac{
			- (\nabla\cdot T)(v, \phi),\,
			(\star),\,
			- \tilde{\tau} b^\perp + (\nu-\lambda) \bar{\phi}^\perp
		}
	$
	for
	\begin{equation*}
		(\star) = (J_{eq} + K)(u\cdot\nabla) \phi + (\omega_{eq} + \theta) \times (J_{eq} + K) \phi + \tilde{\tau}^2 \tilde{b}^\perp + \phi \times J_{eq} \omega_{eq}
			- 2 \vc T(v, \phi) + (\nabla\cdot M)(\phi)
	\end{equation*}
	We also define the associated energy, namely
	\begin{equation}
	\label{eq:def_energies_continuation_not_summed}
		\mathbb{E} (Y; K) \vcentcolon=
			\frac{1}{2} \int_{\T^3} \abs{v}^2
			+ \frac{1}{2} \int_{\T^3} (J_{eq} + K) \phi \cdot \phi
			+ \frac{1}{2} \frac{\tilde{\tau}^2}{\nu-\lambda} \int_{\T^3} \abs{b}^2
	\end{equation}
	and its counterpart summed up to a $2M$ count of parabolic derivatives, i.e.
	\begin{equation}
	\label{eq:def_energies_continuation}
		\widetilde{\mathscr{E}}_{M, K} (Y) \vcentcolon= \sum_{\abs{\alpha}_P \leqslant 2M} \mathbb{E} ( \partial^\alpha Y; K ).
	\end{equation}

	We now introduce notation used to write the full system in terms of the system introduced above.
	So let us define, for $p : \cobrac{0, T} \times \T^3 \to \R$, $\Lambda(p) \vcentcolon= (-\nabla p, 0, 0)$ and, for $Z = (u, \theta, K)$,
	\begin{equation}
	\label{eq:notation_local_in_time_red_est_N}
		N(Z) = \brac{
			- u\cdot\nabla u,\,
			- \theta\times K\omega_{eq},\,
			- u\cdot\nabla a + \theta_3 a^\perp + (\bar{K} - K_{33} I_2) \bar{\theta}^\perp
		}.
	\end{equation}
	We may then write the full system \eqref{eq:pertub_sys_no_ten_pdt_u}--\eqref{eq:pertub_sys_no_ten_pdt_K} as
	\begin{equation}
		\overline{T}(K) \pdt Y - \Leb_Z Y = N(Z) + \Lambda (p)
	\end{equation}
	subject to
	\begin{equation*}
		\nabla\cdot u = 0 \text{ and }
		\pdt K + u\cdot\nabla K = \sbrac{\Omega_{eq} + \Theta, J_{eq} + K}.
	\end{equation*}
	Note that the form of $N_3(Z)$ in \eqref{eq:notation_local_in_time_red_est_N}
	comes from \fref{Lemma}{lemma:block_form_Omega_J} since, for $S = \sbrac{\Omega_{eq} + \Theta, J_{eq} + K}$,
	\begin{equation*}
		\brac{S_{12}, S_{13}} = - (\nu-\lambda) \bar{\theta}^\perp + \tilde{\tau} a^\perp + (\bar{K} - K_{33} I_2) \bar{\theta}^\perp + \theta_3 a^\perp.
	\end{equation*}

	We now record a result akin to \fref{Lemma}{lemma:exact_comparison_versions_E} where we compare precisely various versions of the local energy.

	\begin{lemma}[Comparisons of the different versions of the reduced energies]
	\label{lemma:comp_versions_reduced_energy}
		Let $ \widetilde{\mathscr{E}}_{M, K} $ be defined as in \eqref{eq:def_energies_continuation}.
		There exist constants $\tilde{c}_E, \widetilde{C}_E > 0$ such that if $ \normtyp{ K }{L}{\infty} < \frac{\lambda}{2}$ then
		$
			\tilde{c}_E \overline{\mathcal{E}}_M \leqslant \widetilde{\mathscr{E}}_{M, K} \leqslant \widetilde{C}_E \overline{\mathcal{E}}_M.
		$
	\end{lemma}
	\begin{proof}
		This follows by choosing $\tilde{c}_E$ and $\widetilde{C}_E$ exactly as in \fref{Lemma}{lemma:exact_comparison_versions_E}.
	\end{proof}

	We now turn our attention to the energy-dissipation relation, which we record below.

	\begin{lemma}[Generic energy-dissipation relation for the local-in-time reduced energy estimates]
	\label{lemma:gen_ED_rel_local_time_reduced_energy_est}
		Suppose that $Y = (v, \phi, b)$, $W = (u, \theta, K)$, and $p$ satisfy
		$
			\overline{T}(K) \pdt Y - \Leb_W Y = \mathcal{F} + \Lambda(p),
		$
		where $\Lambda(p) =\vcentcolon (-\nabla p, 0, 0)$, subject to
		$
			\nabla\cdot u = \nabla\cdot v = 0
		$
		and
		$
			\pdt K + (u\cdot\nabla)K = \sbrac{\Omega_{eq} + \Theta, J_{eq} + K}.
		$
		Then
		\begin{equation}
		\label{eq:gen_ED_local_time_reduced_energy_est}
			\frac{\mathrm{d}}{\mathrm{d}t} \mathbb{E} (Y; K) + D(u, \theta) = \int_{\T^3} \overline{\mathfrak{C}} \mathcal{F} \cdot Y
		\end{equation}
		where $\overline{\mathfrak{C}} \vcentcolon= I_3 \oplus I_3 \oplus \frac{\tilde{\tau}^2}{\nu-\lambda} I_2$
		and $D$ is the usual dissipation, as given in \eqref{eq:dissip_first_apparition}.
	\end{lemma}
	\begin{proof}
		This energy estimate is obtained in the same way as the energy estimate of \fref{Proposition}{prop:gen_pert_ED_rel}.
	\end{proof}

	With the energy-dissipation relation in hand we may identify the precise forms of the nonlinear interactions in \fref{Lemma}{lemma:record_form_interactions_local_time_reduced_energy_est} below.
	Recall that the energy $\mathbb{E}$ is defined in \eqref{eq:def_energies_continuation_not_summed}.

	\begin{lemma}[Recording the form of the interactions for the local-in-time reduced energy estimate]
	\label{lemma:record_form_interactions_local_time_reduced_energy_est}
		Suppose that $Z = (u, \theta, K)$, where $a = \brac{K_{12}, K_{12}}$, and $p$ solve
		\eqref{eq:pertub_sys_no_ten_pdt_u}--\eqref{eq:pertub_sys_no_ten_pdt_K}.
		Then we have that for every multi-index $\alpha\in\N^{1+3}$
		\begin{equation*}
			\frac{\mathrm{d}}{\mathrm{d}t} \mathbb{E} ( \partial^\alpha u, \partial^\alpha \theta, \partial^\alpha a ; K )
			+ D ( \partial^\alpha u, \partial^\alpha \theta ) = \overline{\mathcal{N}}^\alpha
		\end{equation*}
		where, for $\check{\tau} \vcentcolon= \frac{\tilde{\tau}^2}{\nu-\lambda}$,
		\begin{align}
			\overline{\mathcal{N}}^\alpha
			= \int_{\T^3} \sbrac{K\pdt, \partial^\alpha } \theta \cdot \partial^\alpha \theta
			- \int_{\T^3} \sbrac{ (J_{eq} + K) (u\cdot\nabla), \partial^\alpha } \theta \cdot \partial^\alpha \theta
			- \int_{\T^3} \sbrac{ (\omega_{eq} + \theta) \times (J_{eq} + K), \partial^\alpha } \theta \cdot \partial^\alpha \theta
		\nonumber
		\\
			- \int_{\T^3} \partial^\alpha (u\cdot\nabla u) \cdot \partial^\alpha u
			- \int_{\T^3} \partial^\alpha (\theta \times K \omega_{eq}) \cdot \partial^\alpha \theta
			- \check{\tau} \int_{\T^3} \partial^\alpha (u\cdot\nabla a) \cdot \partial^\alpha a
		\nonumber
		\\
			+\, \check{\tau} \int_{\T^3} \partial^\alpha (\theta_3 a^\perp) \cdot \partial^\alpha a
			+ \check{\tau} \int_{\T^3} \partial^\alpha \brac{ (\bar{K} - K_{33} I_2) \bar{\theta}^\perp } \cdot \partial^\alpha a.
		\label{eq:record_form_interactions_local_time_reduced_energy_est_conclusion}
		\end{align}
	\end{lemma}
	\begin{proof}
		In order to streamline the proof let us write $Y = (u, \theta, a)$.
		Applying a derivative $ \partial^\alpha $ to \eqref{eq:pertub_sys_no_ten_pdt_u}--\eqref{eq:pertub_sys_no_ten_pdt_K} shows that $ \partial^\alpha Y$ solves
		\begin{equation}
		\label{eq:record_form_interactions_local_time_reduced_energy_est_3}
			\overline{T}(K) \pdt \partial^\alpha Y
			- \Leb_Z \partial^\alpha Y
			= \sbrac{ \overline{T}(K) \pdt, \partial^\alpha } Y
			- \sbrac{ \Leb_Z, \partial^\alpha } Y
			+ \partial^\alpha (N(Z)) + \Lambda(p) =\vcentcolon F^\alpha + \Lambda( \partial^\alpha p)
		\end{equation}
		subject to \eqref{eq:pertub_sys_no_ten_div} and \eqref{eq:pertub_sys_no_ten_pdt_K}.
		We may thus apply \fref{Lemma}{lemma:gen_ED_rel_local_time_reduced_energy_est} to deduce that
		\begin{equation*}
			\frac{\mathrm{d}}{\mathrm{d}t} \mathbb{E} ( \partial^\alpha Y; K ) + D( \partial^\alpha u, \partial^\alpha \theta)
			= \int_{\T^3} \overline{\mathfrak{C}} F^\alpha \cdot \partial^\alpha Y
			=\vcentcolon \overline{\mathcal{N}}^\alpha
		\end{equation*}
		where $\overline{\mathfrak{C}} = I_3 \oplus I_3 \oplus \check{\tau} I_2$ is as in \fref{Lemma}{lemma:gen_ED_rel_local_time_reduced_energy_est}
		and where
		\begin{equation*}
			\overline{\mathcal{N}}^\alpha
			= \underbrace{\int_{\T^3} \overline{\mathfrak{C}} \sbrac{ \overline{T}(K) \pdt, \partial^\alpha } Y \cdot \partial^\alpha Y}_{ \overline{\mathcal{N}}^\alpha_{\RN{1}} }
			- \underbrace{\int_{\T^3} \overline{\mathfrak{C}} \sbrac{ \Leb_Z, \partial^\alpha } Y \cdot \partial^\alpha Y}_{ - \overline{\mathcal{N}}^\alpha_{\RN{2}} }
			+ \underbrace{\int_{\T^3} \overline{\mathfrak{C}} \partial^\alpha \brac{ N(Z) } \cdot \partial^\alpha Y}_{ \overline{\mathcal{N}}^\alpha_{\RN{3}} }.
		\end{equation*}
		It now suffices to compute $ \overline{\mathcal{N}}^\alpha_{\RN{1}}$, $ \overline{\mathcal{N}}^\alpha_{\RN{2}}$, and $ \overline{\mathcal{N}}^\alpha_{\RN{3}}$.
		Since $\sbrac{\overline{T}(K) \pdt, \partial^\alpha } = 0_3 \oplus \sbrac{K\pdt, \partial^\alpha} \oplus 0_2$
		we see immediately that
		\begin{equation}
		\label{eq:record_form_interactions_local_time_reduced_energy_est_a1}
			\overline{\mathcal{N}}^\alpha_{\RN{1}} = \int_{\T^3} \sbrac{ K\pdt, \partial^\alpha } \theta \cdot \partial^\alpha \theta.
		\end{equation}
		Now note that
		\begin{equation*}
			\sbrac{\Leb_Z, \partial^\alpha } = 0_3 \oplus \brac{
				\sbrac{ J_{eq} + K) (u\cdot\nabla), \partial^\alpha }
				+ \sbrac{(\omega_{eq} + \theta) \times (J_{eq} + K), \partial^\alpha }
			} \oplus 0_2
		\end{equation*}
		and hence
		\begin{equation}
		\label{eq:record_form_interactions_local_time_reduced_energy_est_a2}
			\overline{\mathcal{N}}^\alpha_{\RN{2}}
			= - \int_{\T^3} \sbrac{ (J_{eq} + K)(u\cdot\nabla), \partial^\alpha } \theta \cdot \partial^\alpha \theta
			- \int_{\T^3} \sbrac{ (\omega_{eq} + \theta) \times (J_{eq} + K), \partial^\alpha } \theta \cdot \partial^\alpha \theta.
		\end{equation}
		Finally, it follows immediately from the form of $N$ that
		\begin{align}
			\overline{\mathcal{N}}^\alpha_{\RN{3}}
			= - \int_{\T^3} \partial^\alpha (u\cdot\nabla a) \cdot \partial^\alpha u
			- \int_{\T^3} \partial^\alpha (\theta\times K\omega_{eq} ) \cdot \partial^\alpha \theta
			- \check{\tau} \int_{\T^3} \partial^\alpha (u\cdot\nabla a) \cdot \partial^\alpha a
		\nonumber
		\\
			+ \check{\tau} \int_{\T^3} \partial^\alpha (\theta_3 a^\perp) \cdot \partial^\alpha a
			+ \check{\tau} \int_{\T^3} \partial^\alpha \brac{ (\bar{K} - K_{33} I_2) \bar{\theta}^\perp } \cdot \partial^\alpha a.
		\label{eq:record_form_interactions_local_time_reduced_energy_est_a3}
		\end{align}

		\noindent
		To conclude we sum \eqref{eq:record_form_interactions_local_time_reduced_energy_est_a1}, \eqref{eq:record_form_interactions_local_time_reduced_energy_est_a2},
		and \eqref{eq:record_form_interactions_local_time_reduced_energy_est_a3} and obtain \eqref{eq:record_form_interactions_local_time_reduced_energy_est_conclusion}.
	\end{proof}

	We now estimate these interactions.

	\begin{lemma}[Estimating the interactions for the local-in-time reduced energy estimates]
	\label{lemma:est_interact_local_in_time_reduced_energy_est}
		Let $M\geqslant 4$ be an integer and let
		$
			\overline{\mathcal{N}} = \sum_{ \abs{\alpha}_P \leqslant 2M} \overline{\mathcal{N}}^\alpha
		$
		for $\overline{\mathcal{N}}^\alpha$ as in \fref{Lemma}{lemma:record_form_interactions_local_time_reduced_energy_est}.
		The following estimate holds.
		For $Y = \brac{u, \theta, a}$,
		\begin{equation*}
			\vbrac{ \overline{\mathcal{N}} }
			\lesssim \normtyp{ \nabla K }{L}{\infty}  \normtyp{(u, \theta)}{P}{2M+1}^2
			+ \normtyp{K}{P}{2M} \brac{ \normtyp{Y}{P}{2M}^2 + \normtyp{Y}{P}{2M}^3 }
			+ \normtyp{Y}{P}{2M}^3.
		\end{equation*}
	\end{lemma}
	\begin{proof}
		Let us write the terms in $ \overline{\mathcal{N}} $ in order as $ \overline{\mathcal{N}}^\alpha_1,\, \dots,\, \overline{\mathcal{N}}^\alpha_8$.
		These interactions are either identical or very similar to the interactions $\mathcal{N}_i$, $i = 1,\, \dots,\, 7$,  estimated in \fref{Lemma}{lemma:est_local_interactions}.
		We will thus provide very few details here and instead point to the relevant portions of the proof of \fref{Lemma}{lemma:est_local_interactions}.
		
		\paragraph{\textbf{Identical interactions.}}
		Some terms in $ \overline{\mathcal{N}} $ here are identical to terms in $\mathcal{N}$ in \fref{Lemma}{lemma:est_local_interactions}.
		The correspondence between these terms, and the ensuing estimates, are recorded below.
		\begin{align*}
			\overline{\mathcal{N}}_1 = \mathcal{N}_5,\;
			&\vbrac{ \overline{\mathcal{N}}_1 } \lesssim
			\normtyp{ \nabla K }{L}{\infty} \normtyp{\theta}{P}{2M+1}^2 + \normtyp{K}{P}{2M} \normtyp{\theta}{P}{2M}^2,\\
			\overline{\mathcal{N}}_2 = \mathcal{N}_6,\;
			&\vbrac{ \overline{\mathcal{N}}_2 } \lesssim
			\brac{ 1 + \normtyp{K}{P}{2M} } \normtyp{(u, \theta)}{P}{2M}^3,\\
			\overline{\mathcal{N}}_3 = \mathcal{N}_7,\;
			&\vbrac{ \overline{\mathcal{N}}_3 } \lesssim
			\normtyp{K}{P}{2M} \brac{ \normtyp{\theta}{P}{2M}^2 + \normtyp{\theta}{P}{2M}^3 } + \normtyp{\theta}{P}{2M}^3,\\
			\overline{\mathcal{N}}_4 = \mathcal{N}_1,\;
			&\vbrac{ \overline{\mathcal{N}}_4 } \lesssim
			\normtyp{u}{P}{2M}^3, \text{ and }
			\overline{\mathcal{N}}_5 = \mathcal{N}_2,\;
			\vbrac{ \overline{\mathcal{N}}_5 } \lesssim
			\normtyp{K}{P}{2M} \normtyp{\theta}{P}{2M}^2.
		\end{align*}

		\paragraph{\textbf{Similar interactions.}}
		The terms $ \overline{\mathcal{N}}_6 - \overline{\mathcal{N}}_8$ are similar to $\mathcal{N}_1$ in \fref{Lemma}{lemma:est_local_interactions} and so,
		proceeding similarly yields
		\begin{align*}
			\vbrac{ \overline{\mathcal{N}}_6 } \lesssim \normtyp{u}{P}{2M} \normtyp{a}{P}{2M}^2,\,
			\vbrac{ \overline{\mathcal{N}}_7 } \lesssim \normtyp{\theta}{P}{2M} \normtyp{a}{P}{2M}^2, \text{ and }
			\vbrac{ \overline{\mathcal{N}}_8 } \lesssim \normtyp{K}{P}{2M} \normtyp{\theta}{P}{2M} \normtyp{a}{P}{2M}.&\qedhere
		\end{align*}
	\end{proof}

	We may now combine the energy-dissipation relation of \fref{Lemma}{lemma:gen_ED_rel_local_time_reduced_energy_est}
	and the interactions estimates of \fref{Lemma}{lemma:est_interact_local_in_time_reduced_energy_est} in order to derive a preliminary form of the reduced energy estimates.
	Recall that $\widetilde{\mathscr{E}}_{M,K}$ and $ \overline{\mathcal{D}}_M$ are defined in \eqref{eq:def_energies_continuation} and \eqref{eq:not_D_M}, respectively.

	\begin{lemma}[Reduced a priori estimate]
	\label{lemma:reduced_a_priori_est}
		There exist $\delta_{r}^{loc} > 0$ and $C_G > 0$ such that if $K$ satisfies $ \normtyp{ K }{L}{\infty} < \frac{\lambda}{2}$ and $ \normtyp{ \nabla K }{L}{\infty} < \delta_r^{loc}$
		and $Y = (u, \theta, a)$, $Z = (u, \theta, K)$, where $a = (K_{12}, K_{13})$, and $p$ solve \eqref{eq:pertub_sys_no_ten_pdt_u}--\eqref{eq:pertub_sys_no_ten_pdt_K}
		then
		\begin{equation*}
			\frac{\mathrm{d}}{\mathrm{d}t} \widetilde{\mathscr{E}}_{M, K} (Y) + \frac{1}{2}\, \overline{\mathcal{D}}_M (u, \theta)
			\leqslant C_G \brac{1 + \normtyp{K}{P}{2M}} g\brac{ \widetilde{\mathscr{E}}_{M, K} (Y) },
		\end{equation*}
		where $g(x) = x + x^{3/2}$ for every $x\geqslant 0$.
	\end{lemma}
	\begin{proof}
		\fref{Lemma}{lemma:record_form_interactions_local_time_reduced_energy_est} tells us that, for any multi-index $\alpha\in\N^{1+3}$,
		$
			\frac{\mathrm{d}}{\mathrm{d}t} \mathbb{E} ( \partial^\alpha Y; K ) + D ( \partial^\alpha u, \partial^\alpha \theta ) = \overline{\mathcal{N}}^\alpha.
		$
		We may thus sum over $\abs{\alpha}_P \leqslant 2M$ and use \fref{Lemma}{lemma:coercivity_dissip},
		\fref{Lemma}{lemma:comp_versions_reduced_energy}, and \fref{Lemma}{lemma:est_interact_local_in_time_reduced_energy_est} to deduce that, for some $C_1 > 0$,
		\begin{equation*}
			\frac{\mathrm{d}}{\mathrm{d}t} \widetilde{\mathscr{E}}_{M, K} (Y) + \overline{\mathcal{D}}_M (u, \theta)
			\leqslant C_1 \normtyp{ \nabla K }{L}{\infty} \overline{\mathcal{D}}_M (u, \theta)
			+ C_1 \brac{ 1 + \normtyp{K}{P}{2M} } g\brac{ \widetilde{\mathscr{E}}_{M, K} (Y) }.
		\end{equation*}
		In particular, if we pick $\delta_r^{loc} > 0$ sufficiently small to ensure that $C_1 \delta_r^{loc} \leqslant \frac{1}{2} $ then we may deduce the result.
	\end{proof}

	The last tool required to derive the reduced energy estimates is a nonlinear Gronwall-type argument, which is recorded in \fref{Lemma}{lemma:local_in_time_nonlinear_Gronwall_argument} below.

	\begin{lemma}[The local-in-time Bihari argument]
	\label{lemma:local_in_time_nonlinear_Gronwall_argument}
		Suppose that, for some $T > 0$, $e, d : \cobrac{0, T} \to \cobrac{0, \infty}$ are continuously differentiable and satisfy, for some $C > 0$,
		$
			e'(t) + d(t) \leqslant C g(e(t))
			\text{ for every } 0 < t < T,
		$
		where $g(x) =\vcentcolon x + x^{3/2}$ for every $x \geqslant 0$.
		Suppose moreover that there are some $0 \leqslant t_1 < t_2 < T$ and $\alpha_1 > 0$ such that
		$
			e(t_1) \leqslant \alpha_1
			\text{ and }
			t_2 - t_1 \leqslant \min\brac{1, G\brac{\alpha_1}}/C,
		$
		where $G(x) \vcentcolon= \log\brac{1 + \frac{1}{\sqrt{x}} }$ for every $x \geqslant 0$.
		Then, for any $t_1 \leqslant t \leqslant t_2$,
		\begin{align*}
			e(t) \leqslant G\inv \brac{ \frac{G(\alpha_1)}{2} }
			\text{ and }
			d(t) \leqslant \alpha_1 + \brac{g \circ G\inv} \brac{ \frac{G(\alpha_1)}{2} }.
		\end{align*}
	\end{lemma}
	\begin{proof}
		As for \fref{Lemma}{lemma:nonlinear_Gronwall}, this follows from a nonlinear Gronwall argument (see for example \cite{boyer_fabrie}).
	\end{proof}

	We now have in hand all the pieces necessary to prove the local-in-time reduced energy estimates.
	In particular, recall that both $ \overline{\mathcal{E}}_M $ and $ \overline{\mathcal{D}}_M$ (which are defined in \fref{Section}{sec:notation})
	are functionals which depend only on $u$, $\theta$, and $a$. This is precisely why this is called a \emph{reduced} energy estimate.
	Note that the definitions of $ \overline{\mathcal{E}}_M$ and $ \overline{\mathcal{D}}_M$ may be found in \eqref{eq:not_E_M_bar} and \eqref{eq:not_D_M}, respectively.

	\begin{prop}[Local-in-time reduced energy estimate]
	\label{prop:local_in_time_reduced_energy_estimate}
		Let $\delta_r^{loc} > 0$ be as in \fref{Lemma}{lemma:reduced_a_priori_est}.
		There is a non-increasing and continuous function $\phi_r : (0, \infty) \to (0, \infty)$ and
		a strictly decreasing and continuous function $\rho_r : (0, \infty) \to (0, \infty)$ which vanishes asymptotically at zero
		such that, for any $T > 0$ and any $Y = (u, \theta, a)$, $Z = (u, \theta, K)$, where $a = \brac{K_{12}, K_{13}}$, and $p$ satisfying
		\begin{equation}
		\label{eq:local_in_time_red_en_est_K_ass}
			\sup_{0\leqslant t < T} \normtyp{ K(t) }{L}{\infty} < \frac{\lambda}{2}
			\text{ and } 
			\sup_{0\leqslant t < T} \normtyp{ \nabla K (t) }{L}{\infty} < \delta^{loc}_r
		\end{equation}
		and solving \eqref{eq:pertub_sys_no_ten_pdt_u}--\eqref{eq:pertub_sys_no_ten_pdt_K},
		if $0 \leqslant t_1 < t_2 < T$ satisfy
		\begin{equation}
		\label{eq:local_in_time_red_en_est_Delta_t_ass}
			t_2 - t_1 \leqslant 
				\frac{
					\phi_r\brac{ \normtyp{Y(t_1)}{P}{2M} }
				}{
					1 + \sup_{0\leqslant t \leqslant t_2} \normtyp{K(t)}{P}{2M} 
				}
		\end{equation}
		then the following estimate holds on $\sbrac{t_1, t_2}$:
		\begin{equation*}
			\sup_{t_1 \leqslant t \leqslant t_2} \overline{\mathcal{E}}_M (t) + \int_{t_1}^{t_2} \overline{\mathcal{D}}_M (s) ds
			\leqslant \rho_r \brac{ \normtyp{Y(t_1)}{P}{2M} }.
		\end{equation*}
		Recall that $ \overline{\mathcal{E}}_M $ and $ \overline{\mathcal{D}}_M$ are defined in \fref{Section}{sec:notation}.
	\end{prop}
	\begin{proof}
		Let us begin by defining
		$
			C_K \vcentcolon= \sup_{0 \leqslant t \leqslant t_2} \normtyp{K(t)}{P}{2M} 
		$
		and noting that, by virtue of \eqref{eq:local_in_time_red_en_est_K_ass}, \fref{Lemma}{lemma:comp_versions_reduced_energy} tells us that
		$
			\widetilde{\mathscr{E}}_{M, K} (Y(t_1)) \leqslant C_E \normtyp{Y(t_1)}{P}{2M}^2 =\vcentcolon \alpha_1.
		$
		Now \fref{Lemma}{lemma:reduced_a_priori_est} tells us that
		\begin{equation*}
			\frac{\mathrm{d}}{\mathrm{d}t} \widetilde{\mathscr{E}}_{M, K} (Y) + \overline{\mathcal{D}}_M (u, \theta)
			\leqslant C_G (1 + C_K) g \brac{ \alpha_1}
		\end{equation*}
		for $g(x) \vcentcolon= x + x^{3/2}$.
		Therefore, for $G(x) \vcentcolon= \log\brac{1 + \frac{1}{\sqrt{x}} }$ as in \fref{Lemma}{lemma:local_in_time_nonlinear_Gronwall_argument} and
		\begin{equation*}
			\phi_r (\alpha) \vcentcolon= \frac{\min\brac{1, G\brac{C_E \alpha}}}{C_G}
			\text{ for every } \alpha > 0
		\end{equation*}
		\eqref{eq:local_in_time_red_en_est_Delta_t_ass} tells us that we may apply \fref{Lemma}{lemma:local_in_time_nonlinear_Gronwall_argument}.
		Combining \fref{Lemma}{lemma:local_in_time_nonlinear_Gronwall_argument} with \fref{Lemma}{lemma:comp_versions_reduced_energy}, we deduce that
		\begin{align*}
			\sup_{t_1 \leqslant t \leqslant t_2} \overline{\mathcal{E}}_M (t) + \int_{t_1}^{t_2} \overline{\mathcal{D}}_M (s) ds
			\leqslant \frac{1}{c_e} G\inv \brac{ \frac{ G\brac{\alpha_1}}{2} }
			+ \alpha_1 + \brac{g\circ G\inv} \brac{ \frac{ G\brac{\alpha_1}}{2} }
			=\vcentcolon \rho_r \brac{ \normtyp{Y(t_1)}{P}{2M} }.&\qedhere
		\end{align*}
	\end{proof}

\subsection{Supplementary estimates}
\label{sec:cont_supp_est}

	In this section we record supplementary estimates that are required to parlay the reduced energy estimates obtained in \fref{Section}{sec:red_en_est} above
	into a continuation argument (recorded in \fref{Section}{sec:cont_synthesis}) capable of gluing together the main a priori estimates of \fref{Section}{sec:a_prioris}
	and the local well-posedness theory of \fref{Section}{sec:lwp}.

	Many of the results in this section are variants of results obtained in \fref{Section}{sec:a_prioris} which no longer rely on any smallness assumption on the solution.
	Correspondingly, the bounds obtained are often polynomial (whereas they were linear when a smallness assumption was made).
	In particular, we will employ the functionals $ \overline{\mathcal{E}}_M$ and $ \overline{\mathcal{E}}_\text{low}$ several times, so we recall that their
	definitions may be found in \eqref{eq:not_E_M_bar} and \eqref{eq:not_E_low}, respectively.
	We begin by recording a result comparing two versions of the energy,
	where recall that $ \widetilde{\mathcal{E}}_M $ is also defined in \eqref{eq:not_E_M_bar}.

	\begin{lemma}[Comparisons of different versions of the energies under a smallness condition]
	\label{lemma:comp_versions_energy_suff_small_K_L_infty}
		There exists constants $\tilde{c}_E, \widetilde{C}_E > 0$ such that if $ \normtyp{ K }{L}{\infty} < \frac{\lambda}{2}$ then
		$
			\tilde{c}_E \overline{\mathcal{E}}_M \leqslant \widetilde{\mathcal{E}}_M \leqslant \widetilde{C}_E \overline{\mathcal{E}}_M.
		$
	\end{lemma}
	\begin{proof}
		This follows by choosing $\tilde{c}_E$ and $\widetilde{C}_E$ exactly as in \fref{Lemma}{lemma:exact_comparison_versions_E}.
	\end{proof}

	We now record an auxiliary $L^\infty$ estimate for $\pdt K$ which is necessary in order to control the low level interactions.

	\begin{lemma}[$L^\infty$ estimate for $\pdt K$]
	\label{lemma:L_infty_control_pdt_K}
		If $K$ solves \eqref{eq:pertub_sys_no_ten_pdt_K} then
		$
			\normtyp{ \pdt K }{L}{\infty} \lesssim \normtyp{\theta}{H}{2} + \brac{1 + \normtyp{(u,\theta)}{H}{2} } \normtyp{K}{H}{3}.
		$
	\end{lemma}
	\begin{proof}
		This follows from \eqref{eq:pertub_sys_no_ten_pdt_K}, the $L^\infty$ being a Banach algebra, and the embedding $H^2 \hookrightarrow L^\infty$.
	\end{proof}

	With \fref{Lemma}{lemma:L_infty_control_pdt_K} in hand we may record the following reformulation of the control of the low level interactions
	obtained in \fref{Lemma}{lemma:careful_est_low_level_int}.
	We recall that $ \mathcal{D}_\text{low}$ and $ \overline{\mathcal{I}}_\text{low}$ are defined in \eqref{eq:not_D_low} and \eqref{eq:not_I}, respectively.
	\begin{cor}[Careful estimates of the low-level interactions]
	\label{cor:careful_est_low_level_interactions}
		There is a polynomial $P$ with non-negative coefficients and which vanishes at zero such that
		$
			\vbrac{\overline{\mathcal{I}}_\text{low}} \leqslant P\brac{ \normtyp{Y}{P}{3},\, \normtyp{K}{H}{3} } \mathcal{D}_\text{low}.
		$
		In particular, if $\normtyp{Y}{P}{3} \leqslant 1$ and $\normtyp{K}{H}{3} \leqslant 1$ then
		$\vbrac{ \overline{\mathcal{I}}_\text{low} } \lesssim \brac{ \normtyp{Y}{P}{3} + \normtyp{K}{H}{3} } \mathcal{D}_\text{low}$.
	\end{cor}
	\begin{proof}
		This follows immediately from combining \fref{Lemma}{lemma:careful_est_low_level_int} and \fref{Lemma}{lemma:L_infty_control_pdt_K}.
	\end{proof}

	We now turn our attention to a result similar to \fref{Proposition}{prop:decay_intermediate_norms} where we obtain the decay of intermediate norms.
	The difference here is that the smallness assumption on $ \mathcal{E}_M $ (present in \fref{Proposition}{prop:decay_intermediate_norms}) is replaced by a smallness assumption on
	$ \normtyp{K}{H}{3} $, since the latter is guaranteed to be small due to the space in which our local well-posedness theory produces solutions.

	\begin{prop}[Decay of low-level energy provided smallness of the reduced high-level energy]
	\label{prop:decay_low_E_given_small_reduced_high_E}
		Let $M\geqslant 3$ be an integer.
		There exist $\delta_I > 0$ and $C_I > 0$ such that for every $T > 0$, if
		\begin{equation}
		\label{eq:decay_low_E_given_small_red_high_E}
			\sup_{0 \leqslant t < T} \normtyp{K(t)}{H}{3} \leqslant \delta_I
			\text{ and } 
			\sup_{0\leqslant t < T} \overline{\mathcal{E}}_M (t) =\vcentcolon \delta_0 \leqslant \frac{1}{2}
		\end{equation}
		then
		\begin{equation}
		\label{eq:decay_low_E_given_small_red_high_E_conclusion}
			\sup_{0 \leqslant t < T} \overline{\mathcal{E}}_\text{low} (t) {\brac{1+t}}^{2M-2} \leqslant C_I \delta_0.
		\end{equation}
	\end{prop}
	\begin{proof}
		The proof of this result employs the same strategy as the proof of \fref{Proposition}{prop:close_est_low_level} where we close the energy estimates at the low level
		so we omit the details and only discuss how the proof of \fref{Proposition}{prop:close_est_low_level} must be modified to apply here.
		There are two key differences:
		(1) the low level interactions are controlled by \fref{Corollary}{cor:careful_est_low_level_interactions}
			(and not \fref{Corollary}{cor:control_low_level_interactions}) because here we must clearly identify how $K$
			appears in the low level interactions and
		(2) the different versions of the energy are compared using \fref{Lemma}{lemma:comp_versions_energy_suff_small_K_L_infty}
			(instead of \fref{Proposition}{prop:persist_spec_sols_adv_rot_eqtns} and \fref{Lemma}{lemma:comp_version_en})
			since here we use the smallness of $K$, instead of the regularity of the solutions, to ensure the positive-definiteness of $J = J_{eq} + K$.
		There is also a minor difference to take into account: there is no need here to improve the energy,
		so by contrast with \fref{Proposition}{prop:close_est_low_level} we do not need to appeal to \fref{Proposition}{prop:imp_low_level_en_ds}.
	\end{proof}

	We conclude this section with auxiliary estimates for $K$ which are a consequence of the advection-rotation estimates proved in \fref{Section}{sec:adv_rot_est_K}.
	\fref{Proposition}{prop:global_in_time_aux_est_K} is therefore similar to \fref{Proposition}{prop:est_K} which performed the synthesis
	of the advection-rotation estimates proved in \fref{Section}{sec:adv_rot_est_K}.
	The key difference here is that there are no smallness assumptions being made, and as a result
	the bounds in both the hypotheses and the conclusion of \fref{Proposition}{prop:global_in_time_aux_est_K}
	below are in terms of nonlinear functions of a smallness parameter.
	We note that the various energy and dissipation functionals used below are
	defined in \eqref{eq:not_E_low}--\eqref{eq:not_E_M_and_F_M} and \eqref{eq:not_D_M}.

	\begin{prop}[Auxiliary estimates for $K$]
	\label{prop:global_in_time_aux_est_K}
		Let $M\geqslant 3$ be an integer.
		If there is some time horizon $T>0$ such that
		\begin{equation}
		\label{eq:global_in_time_aux_est_K_ap_ass}
			\sup_{0 \leqslant t < T} \overline{\mathcal{E}}_\text{low} (t) {\brac{1+t}}^{2M-2} + \overline{\mathcal{E}}_M (t) + \int_0^T \overline{\mathcal{D}}_M (s) ds
			\leqslant \rho_{ap} \brac{ \delta_0 } \leqslant 1
		\end{equation}
		and
		\begin{equation}
		\label{eq:global_in_time_aux_est_K_IC_ass}
			\brac{ \mathcal{E}_M + \mathcal{F}_M } (0) \leqslant \rho_0 \brac{\delta_0} \leqslant 1
		\end{equation}
		for some $\delta_0 \geqslant 0$ and some $\rho_{ap}, \rho_0 : (0, \infty) \to (0, \infty)$ which are strictly increasing and vanish asymptotically at zero,
		then there exists $\rho_K : (0, \infty) \to (0, \infty)$ which is strictly increasing and vanishes asymptotically at zero such that
		\begin{equation}
		\label{eq:global_in_time_aux_est_K_conclusion}
		\sup_{0\leqslant t < T} \mathcal{E}_M^{(K)} (t) + \frac{ \mathcal{F}_M (t) }{1+t} \leqslant \rho_K \brac{\delta_0}
		\end{equation}
		where $\rho_K$ depends $\rho_{ap}$, and $\rho_0$.
		Moreover, if $\rho_{ap}$ and $\rho_0$ are continuous then so is $\rho_K$.
	\end{prop}
	\begin{proof}
		In light of \eqref{eq:global_in_time_aux_est_K_ap_ass}, \fref{Proposition}{prop:decay_intermediate_norms} tells us that
		\begin{equation}
		\label{eq:global_in_time_aux_est_K_1}
			\sup_{0 \leqslant t < T} \sup_{1 \leqslant I \leqslant M} \overline{\mathcal{K}}_I (t) {\brac{1+t}}^{2M-2} \lesssim \rho_{ap} \brac{\delta_0}.
		\end{equation}
		Combining \eqref{eq:global_in_time_aux_est_K_ap_ass}, \eqref{eq:global_in_time_aux_est_K_IC_ass}, and \eqref{eq:global_in_time_aux_est_K_1},
		we use \fref{Proposition}{prop:est_K} to deduce that there exist $C_1, C_2 > 0$ such that, for every $0\leqslant t < T$,
		\begin{equation}
		\label{eq:global_in_time_aux_est_K_2}
		\mathcal{E}_M^{(K)} (t) \leqslant C_1 \brac{ \rho_{ap} \brac{\delta_0} + \rho_0 \brac{\delta_0} }
		\end{equation}
		and, also using Cauchy-Schwarz to deal with $\int \overline{\mathcal{D}}_M^{1/2}$,
		\begin{equation}
		\label{eq:global_in_time_aux_est_K_3}
			\mathcal{F}_M (t) \leqslant C_2 \brac{ \mathcal{F}_M (0) + {\brac{ \int_0^t \overline{\mathcal{D}}_M^{1/2} (s) ds }}^2 + \overline{\mathcal{K}}_M (t) }
			\leqslant C_2 \brac{ (1+t) \rho_0 \brac{\delta_0} + \rho_{ap} \brac{\delta_0} }.
		\end{equation}
		Combining \eqref{eq:global_in_time_aux_est_K_2} and \eqref{eq:global_in_time_aux_est_K_3} yields \eqref{eq:global_in_time_aux_est_K_conclusion}.
	\end{proof}

\subsection{Synthesis}
\label{sec:cont_synthesis}

	In this section we record the continuation argument which allows us to glue together the local well-posedness theory and our scheme of a priori estimates.
	Recall that $ \mathcal{E}_\text{low}$, $ \mathcal{E}_M $ and $ \mathcal{F}_M$, and $ \mathcal{D}_M$
	are defined in \eqref{eq:not_E_low}, \eqref{eq:not_E_M_and_F_M}, and \eqref{eq:not_D_M}, respectively.

	\begin{thm}[Continuation argument]
	\label{thm:cont}
		Let $M\geqslant 4$ be an integer.
		There exists $\eta_{cont} > 0$ such that the following holds.
		For any finite time horizon $T > 0$, if we have a solution of \eqref{eq:pertub_sys_no_ten_pdt_u}--\eqref{eq:pertub_sys_no_ten_pdt_K}
		on $\sbrac{0, T}$ whose initial condition satisfies
		\begin{equation}
		\label{eq:cont_IC}
			\brac{ \mathcal{E}_M + \mathcal{F}_M } (0) =\vcentcolon \eta_0 \leqslant \eta_{cont}
		\end{equation}
		and which lives in the small energy regime, i.e.
		\begin{equation}
		\label{eq:cont_small_energy_regime}
			\sup_{0\leqslant t \leqslant T} \mathcal{E}_\text{low} (t) {\brac{1+t}}^{2M-2} + \mathcal{E}_M (t) + \frac{ \mathcal{F}_M (t) }{1 + t} + \int_0^T \mathcal{D}_M (s) ds
			\leqslant C_{ap} \eta_0
		\end{equation}
		for $C_{ap} > 0$ as in \fref{Theorem}{thm:a_priori},
		then there exists a timescale $\tau > 0$ such that the solution can be uniquely continued on $\sbrac{0, T+\tau}$ where it continues to live in the small energy regime,
		i.e. \eqref{eq:cont_small_energy_regime} holds with $T$ replaced by $T + \tau$.
	\end{thm}
	\begin{proof}
		\textbf{Step 1.}
		First we define the smallness parameter $\eta > 0$.
		We begin by picking $\sigma > 0$ small enough to satisfy
		\begin{enumerate}
			\item	\label{it:cont_small_cond_1}
				$\sigma \leqslant \sigma\brac{ \delta^{loc}_{ap} }$ for $\delta^{loc}_{ap}$ as in \fref{Lemma}{lemma:local_a_priori_estimates}
				and $\sigma = \sigma (\delta)$ as in the beginning of \fref{Section}{sec:lwp}.
			\item	\label{it:cont_small_cond_2}
				$\sigma \leqslant \sigma\brac{ \delta^{loc}_{r} }$ for $\delta^{loc}_{r}$ as in \fref{Lemma}{lemma:reduced_a_priori_est}
				and $\sigma = \sigma (\delta)$ as in the beginning of \fref{Section}{sec:lwp}.
			\item	\label{it:cont_small_cond_3}
				$\sigma \leqslant \delta_I$ for $\delta_I$ as in \fref{Proposition}{prop:global_in_time_aux_est_K}.
		\end{enumerate}
		We then pick $\eta > 0$ sufficiently small to satisfy
		\begin{enumerate}
			\setcounter{enumi}{3}
			\item	\label{it:cont_small_cond_4}
				$\eta \leqslant \frac{\sigma_*^2}{C_{ap}}$ for $\sigma_* = \frac{\sigma}{2C_K}$ as in \fref{Theorem}{thm:lwp},
			\item	\label{it:cont_small_cond_5}
				$\max\brac{ C_{ap} \eta, \rho_r \brac{ \sqrt{C_{ap} \eta} }} \leqslant \frac{1}{2} $ for $\rho_r$
				as in \fref{Proposition}{prop:local_in_time_reduced_energy_estimate},
			\item	\label{it:cont_small_cond_6}
				$\eta \leqslant 1$,
			\item	\label{it:cont_small_cond_7}
				$\rho_r\brac{\sqrt{ C_{ap} \eta}} + \rho_K\brac{\eta} \leqslant \frac{\delta_{ap}}{2}$ for $\delta_{ap}$ as in \fref{Theorem}{thm:a_priori}
				and $\rho_K$ as in \fref{Proposition}{prop:global_in_time_aux_est_K}, where $\rho_K$ depends in $\rho_0$ and $\rho_{ap}$ given by
				$\rho_0 \vcentcolon= \id$ and $\rho_{ap} (x) \vcentcolon= \max\brac{ C_{ap} x, \rho_r \brac{ \sqrt{C_{ap} x }}}$,
			\item 	\label{it:cont_small_cond_8}
				$C_{ap} \eta \leqslant \frac{\delta_{ap}}{2}$, and
			\item	\label{it:cont_small_cond_9}
				$\eta \leqslant \eta_{ap}$ for $\eta_{ap}$ as in \fref{Theorem}{thm:a_priori}.
		\end{enumerate}
		In particular, note that choosing the parameter $\eta$ in this way enforces the following implication: since $M\geqslant 3$,
		if \eqref{eq:cont_small_energy_regime} holds then
		\begin{equation}
		\label{eq:cont_1}
			\sup_{0\leqslant t\leqslant T} \normtyp{K(t)}{H}{3}^2
			\leqslant \sup_{0\leqslant t\leqslant T} \mathcal{E}_M^{(K)} (t)
			\leqslant \sigma_*^2,
		\end{equation}
		i.e. $ \normtyp{K(t)}{H}{3} < \sigma$ when the solution lives in the small energy regime.

		\paragraph{\textbf{Step 2.}}
		We now identify the timescale $\tau$ on which we may both (1) continue the solution and (2) obtain estimates on the continued solution.
		Correspondingly, there are two constraints on how large the timescale $\tau$ may be:
		(1) the first constraint comes from the local well-posedness theory of \fref{Theorem}{thm:lwp} and
		(2) the second constraint comes from the local-in-time reduced energy estimate of \fref{Proposition}{prop:local_in_time_reduced_energy_estimate}.

		So let us define, for $\phi$ as in \fref{Theorem}{thm:lwp} and $\phi_r$ as in \fref{Proposition}{prop:local_in_time_reduced_energy_estimate},
		\begin{align*}
			\tau_\text{lwp} \vcentcolon= \phi\brac{ \sqrt{C_{ap} (2 + T) \eta_0 }},\,
			\tau_r \vcentcolon= \frac{
				\phi_r \brac{\sqrt{ C_{ap} \eta_0 }}
			}{
				1 + \sqrt{C_{ap} {(2 + T)} \eta_0}
			}, \text{ and } 
			\tau \vcentcolon= \frac{1}{3} \min\brac{ \tau_\text{lwp}, \tau_r }.
		\end{align*}
		In particular note that $2\tau < \min\brac{ \tau_\text{lwp}, \tau_r}$.
		Note also that, for every $0 \leqslant t \leqslant T$, by virtue of \eqref{eq:cont_small_energy_regime},
		if we write $Z=(u,\theta, K)$ and $Y=(u,\theta, a)$ then
		\begin{equation}
		\label{eq:cont_2}
			\phi\brac{ \normtyp{Z(t)}{H}{2M} }
			\geqslant \phi\brac{ {\brac{ \mathcal{E}_M + \mathcal{F}_M }}^{1/2} (t) }
			\geqslant \tau_\text{lwp}
		\end{equation}
		and
		\begin{equation}
		\label{eq:cont_3}
			\frac{
				\phi_r \brac{ \normtyp{Y(t)}{P}{2M} }
			}{
				1 + \sup_{0\leqslant s\leqslant T} \normtyp{K(s)}{P}{2M} 
			}
			\geqslant \frac{
				\phi_r \brac{\sqrt{ C_{ap} \eta_0 }}
			}{
				1 + \sup_{0\leqslant s \leqslant T} {\brac{ \mathcal{E}_M + \mathcal{F}_M }}^{1/2} (s)
			}
			\geqslant \tau_r.
		\end{equation}

		\paragraph{\textbf{Step 3.}}
		Having identified the appropriate timescale $\tau$ we may now turn the crank of the local well-posedness theory.
		Feeding \eqref{eq:cont_1} and \eqref{eq:cont_2} into \fref{Theorem}{thm:lwp} using the initial condition $Z(T - \tau)$,
		where $Z = (u, \theta, K)$ as usual, we see that the solution may be uniquely continued to $\sbrac{0, T+\tau}$ where,
		in light of \eqref{eq:cont_1}, it satisfies
		\begin{equation}
		\label{eq:cont_4}
			\sup_{T - \tau \leqslant t \leqslant T + \tau} \normtyp{K(t)}{H}{3} < \sigma.
		\end{equation}

		\paragraph{\textbf{Step 4.}}
		We conclude by performing estimates on the solution on the time interval $\sbrac{T-\tau, T+\tau}$ that ensure that the solution remains in the
		small energy regime of \eqref{eq:cont_small_energy_regime}.
		In light of \eqref{eq:cont_3} and \eqref{eq:cont_4} we may apply the local-in-time reduced energy estimate
		of \fref{Proposition}{prop:local_in_time_reduced_energy_estimate} on the interval $\sbrac{T - \tau, T+\tau}$ to obtain that
		\begin{equation}
		\label{eq:cont_5}
			\sup_{T-\tau \leqslant t \leqslant T + \tau} \overline{\mathcal{E}}_M (t) + \int_{T-\tau}^{T+\tau} \overline{\mathcal{D}}_M (s) ds
			\leqslant \rho_r \brac{\sqrt{ C_{ap} \eta_0}},
		\end{equation}
		for $\rho_r$ as in \fref{Proposition}{prop:local_in_time_reduced_energy_estimate}.
		
		We now string together \eqref{eq:cont_5}, \fref{Proposition}{prop:decay_low_E_given_small_reduced_high_E}, and \fref{Proposition}{prop:global_in_time_aux_est_K},
		which tells us that, in light of the smallness conditions \eqref{it:cont_small_cond_3}, \eqref{it:cont_small_cond_4}, and \eqref{it:cont_small_cond_5},
		\begin{equation*}
			\sup_{T-\tau \leqslant t \leqslant T+\tau} \overline{\mathcal{E}}_\text{low} (t) {\brac{1+t}}^{2M-2}
			\leqslant C_I \max \brac{ C_{ap} \eta_0, \rho_r\brac{ \sqrt{ C_{ap} \eta_0} }}
		\end{equation*}
		where $C_I$ is as in \fref{Proposition}{prop:decay_low_E_given_small_reduced_high_E}, and hence,
		in light of the smallness condition \eqref{it:cont_small_cond_6},
		\begin{equation}
		\label{eq:cont_6}
		\sup_{T - \tau \leqslant t \leqslant T + \tau} \mathcal{E}_M^{(K)} (t) + \frac{ \mathcal{F}_M (t) }{1+t} \leqslant \rho_K \brac{\eta_0},
		\end{equation}
		for $\rho_K$ defined as in the smallness condition \eqref{it:cont_small_cond_7}.

		So finally, putting together \eqref{eq:cont_5} and \eqref{eq:cont_6} tells us that, in light of the smallness condition \eqref{it:cont_small_cond_7},
		\begin{equation}
		\label{eq:cont_7}
			\sup_{T-\tau \leqslant t \leqslant T+\tau} \mathcal{E}_M (t) + \int_{T-\tau}^{T+\tau} \overline{\mathcal{D}}_M (s) ds \leqslant \frac{\delta_{ap}}{2}.
		\end{equation}
		In light of the smallness condition \eqref{it:cont_small_cond_8} we may therefore combine \eqref{eq:cont_small_energy_regime} and \eqref{eq:cont_7} to see that
		\begin{equation}
		\label{eq:cont_8}
			\sup_{0 \leqslant t \leqslant T+\tau} \mathcal{E}_M (t) + \int_0^{T+\tau} \overline{\mathcal{D}}_M (s) ds \leqslant \delta_{ap}.
		\end{equation}
		To conclude we feed \eqref{eq:cont_8} into the a priori estimates of \fref{Theorem}{thm:a_priori},
		which is legal in light of the smallness condition \eqref{it:cont_small_cond_9}, and obtain that
		\begin{equation*}
			\sup_{0\leqslant t\leqslant T+\tau}
			\mathcal{E}_\text{low} (t) {\brac{1+t}}^{2M-2} + \mathcal{E}_M (t) + \frac{ \mathcal{F}_M (t) }{1+t} + \int_0^{T+\tau} \mathcal{D}_M (s) ds
			\leqslant C_{ap}\eta_0.\qedhere
		\end{equation*}
	\end{proof}

%----------------------------------------------------------------------------------------------------
%	GLOBAL WELL-POSEDNESS AND DECAY
%----------------------------------------------------------------------------------------------------

\section{Global well-posedness and decay}
\label{sec:gwp}

	In this section we put together the a priori estimates of \fref{Section}{sec:a_prioris}, the local well-posedness of \fref{Section}{sec:lwp},
	and the continuation argument of \fref{Section}{sec:cont} in order to obtain the main result of this paper,
	namely global well-posedness and decay about equilibrium.
	This is supplemented by a quantitative rigidity result which allows us to deduce decay of $K$.

	In order to prove the main result, there are two auxiliary results that we need in addition of the results proved in \fref{Sections}{sec:a_prioris}--\ref{sec:cont}.
	The first one is the first part of \fref{Lemma}{lemma:local_in_time_aux_est_K}, which accounts for the mismatch between the energies used for the local well-posedness and the a priori estimates.
	This result ensures that, close to time $t=0$, the local solution lives in the smallness regime to which the a priori estimates apply.
	The second one is \fref{Proposition}{prop:control_IC__energy_by_purely_spatial_deriv} which allows us to control the initial energy,
	involving time derivatives, in terms of purely spatial norms.
	Note that this is reminiscent of \fref{Lemma}{lemma:bounds_on_P_2M_by_H_2M} from the local well-posedness theory,
	which fulfilled a similar purpose for solutions of the approximate systems.
	In particular, the first part of \fref{Lemma}{lemma:local_in_time_aux_est_K} and \fref{Lemmas}{lemma:H_k_bounds_inverse_J_eq_plus_K}--\ref{lemma:control_P_2M_by_H_2M} only serve the purpose of
	leading up to \fref{Proposition}{prop:control_IC__energy_by_purely_spatial_deriv}.
	
	We begin with \fref{Lemma}{lemma:local_in_time_aux_est_K} below.
	Note that \fref{Lemma}{lemma:H_k_est_pdt_j_K} forms of the crux of the argument in the first part of \fref{Lemma}{lemma:local_in_time_aux_est_K},
	as it does for similar estimates in \fref{Section}{sec:a_prioris}.
	The difference here (by contrast with estimates recorded in \fref{Section}{sec:a_prioris}) is that we do not make any smallness assumptions.
	Note that in the first part of \fref{Lemma}{lemma:local_in_time_aux_est_K} we only control (parts of) $ \mathcal{E}_M^{(K)} $
	whereas in the second part we control $ \mathcal{F}_M $ as well.

	\begin{lemma}[Auxiliary estimates for $K$]
	\label{lemma:local_in_time_aux_est_K}
		Let $M \geqslant 2$ be an integer. There exists a constant $C_K > 0$ such that
		if $K$ solves \eqref{eq:pertub_sys_no_ten_pdt_K} then, for $Z=(u, \theta, K)$ we have the estimates
		\begin{align*}
			\normtyp{\pdt^2 K}{H}{2M-3}^2 + \sum_{j=3}^M \normtypns{\pdt^j K}{H}{2M-2j+2}^2
			&\leqslant C_K \brac{ \normtyp{Z}{P}{2M}^2 + \normtyp{Z}{P}{2M}^{\brac{2^{M+1}}}}
			\text{ and }\\
			\normtyp{K}{H}{2M+1}^2 + \sum_{j=1}^M \normtypns{\pdt^j K}{H}{2M-2j+2}^2
			&\leqslant C_K \brac{ \normtyp{Z}{P}{2M}^2 + \normtyp{K}{H}{2M+1}^2 + {\brac{ \normtyp{Z}{P}{2M}^2 + \normtyp{K}{H}{2M+1}^2 }}^{\brac{2^M}}}.
		\end{align*}
		Note that the summation on the left-hand side of the second inequality can be written more compactly as
		$
			\norm{K}{P^{2M+2}_{1, M}},
		$
		which comes in handy in the sequel.
	\end{lemma}
	\begin{proof}
		We begin with the first inequality.
		For any $j\geqslant 1$, \fref{Lemma}{lemma:H_k_est_pdt_j_K} tells us that
		\begin{equation}
		\label{eq:local_in_time_aux_est_K}
			\normtypns{\pdt^j K}{H}{2M-2j+2} \lesssim \normtyp{Z}{P}{2M} + \normtyp{Z}{P}{2M}^2 + \normtypns{\pdt^{j-1} K}{H}{2M-2j+3}^2. 
		\end{equation}
		We immediately deduce that, since $2M-3 = (2M - 2\cdot 2 + 2) -1$,
		\begin{equation*}
			\normtyp{\pdt^2 K}{H}{2M-3} \lesssim \normtyp{Z}{P}{2M-1} + \normtyp{Z}{P}{2M-1}^2 + \normtyp{\pdt K}{H}{2M-2}^2
			\lesssim \normtyp{Z}{P}{2M} + \normtyp{Z}{P}{2M}^4
		\end{equation*}
		and
		\begin{equation*}
			\normtyp{\pdt^3 K}{H}{2M-4} \lesssim \normtyp{Z}{P}{2M} + \normtyp{Z}{P}{2M}^2 + \normtyp{\pdt^2 K}{H}{2M-3}^2
			\lesssim \normtyp{Z}{P}{2M} + \normtyp{Z}{P}{2M}^8.
		\end{equation*}
		To conclude we proceed by induction.
		Suppose that, for some $3\leqslant j \leqslant M-1$,
		\begin{equation*}
			\normtypns{\pdt^j K}{H}{2M-2j+2} \lesssim \normtyp{Z}{P}{2M} + \normtyp{Z}{P}{2M}^{\brac{2^j}}.
		\end{equation*}
		Then, by \eqref{eq:local_in_time_aux_est_K},
		\begin{equation*}
			\normtypns{\pdt^{j+1}K}{H}{2M-2j} \lesssim \normtyp{Z}{P}{2M} + \normtyp{Z}{P}{2M}^2 + \normtypns{\pdt^j K}{H}{2M-2j+1}^2
			\lesssim \normtyp{Z}{P}{2M} + \normtyp{Z}{P}{2M}^{\brac{2^{j+1}}},
		\end{equation*}
		and so the claim follows by induction.

		We now prove the second inequality.
		The key observation is that, since $H^s$ is a Banach algebra when $s > \frac{3}{2}$, we may immediately deduce from \eqref{eq:pertub_sys_no_ten_pdt_K} that
		\begin{equation*}
			\normtypns{\pdt^j K}{H}{2M-2j+2} \lesssim \norm{Z}{P^{2M}_{j-1}} + \norm{K}{P^{2M+1}_{j-1}} + {\brac{ \norm{Z}{P^{2M}_{j-1}} + \norm{K}{P^{2M+1}_{j-1}} }}^2
			\text{ for every } 1\leqslant j \leqslant M.
		\end{equation*}
		For simplicity, we will write $p_a^b (x) \vcentcolon= x^a + x^b$. The inequality above may thus be written as
		\begin{equation}
		\label{eq:aux_est_K_intermediate}
			\normtypns{\pdt^j K}{H}{2M-2j+2} \lesssim p_1^2 \brac{ \norm{Z}{P^{2M}_{j-1}} + \norm{K}{P^{2M+1}_{j-1}} }
			\text{ for every } 1\leqslant j \leqslant M.
		\end{equation}
		This notation is particularly useful due to its behaviour under composition.
		Indeed, we see immediately that $p_a^b \circ p_c^d \lesssim p_{ac}^{bd}$.
		The result now follows from iterating \eqref{eq:aux_est_K_intermediate} and we may induct on $1\leqslant j\leqslant M$ to show that
		$
			\normtyp{K}{H}{2M+1} + \norm{K}{P^{2M+2}_{1,j}} \lesssim p_1^{2^j} \brac{ \normtyp{Z}{P}{2M} + \normtyp{K}{H}{2M+1} },
		$
		which proves the claim.
	\end{proof}

	We now turn our attention towards the control of the initial energy in terms of purely spatial norms.
	In order to do so, we first record $H^k$ bounds on the inverse of $J_{eq} + K$ reminiscent of \fref{Lemma}{lemma:H_k_bounds_T_K_inv}.
	However, such bounds are easier to obtain here since we do not have to deal with any projections, as was the case in \fref{Lemma}{lemma:H_k_bounds_T_K_inv}.
	Note that in the lemma below we consider $K : \T^3 \to \sym(3)$ (i.e. there is no dependence in time).
	When $K$ is satisfies \eqref{eq:pertub_sys_no_ten_pdt_u}--\eqref{eq:pertub_sys_no_ten_pdt_K} (and is hence time-dependent), this means that the lemma below applies pointwise in time.

	\begin{lemma}[$H^k$ bounds on $ {\brac{J_{eq} + K}}\inv $]
	\label{lemma:H_k_bounds_inverse_J_eq_plus_K}
		Suppose that $J_{eq} = \diag\brac{\lambda, \lambda, \nu}$ for $\nu > \lambda > 0$
		and that $K : \T^3 \to \sym(3)$ satisfies $ \normtyp{ K }{L}{\infty} < \frac{\lambda}{2} $.
		Then $J_{eq} + K$ is pointwise invertible and, for every integer $k \geqslant 1$,
		\begin{equation*}
			(1)\; \normns{ {\brac{ J_{eq} + K }}\inv }{\Leb\brac{L^2;\, L^2}} \leqslant \frac{2}{\lambda}
			\text{ and }
			(2)\; \normns{ {\brac{ J_{eq} + K }}\inv }{\Leb\brac{H^k;\, H^k}} \lesssim \normtyp{K}{H}{k+2} + \normtyp{K}{H}{k+2}^k.
		\end{equation*}
	\end{lemma}
	\begin{proof}
		The strategy here is the same as that which was used when studying the inverse of $J_{eq} + P_n \circ K$ in
		\fref{Lemmas}{lemma:invertibility_T_K}, \ref{lemma:formula_deriv_T_K_inv}, and \ref{lemma:H_k_bounds_T_K_inv}
		of \fref{Section}{sec:inv_T_K} when developing the local well-posedness theory.

		First we show that $J_{eq} + K$ is pointwise invertible. This follows from the fact that the quadratic form that it generates is pointwise positive-definite
		and indeed, for every $x\in\T^3$ and every $w\in\R^3$,
		\begin{equation}
		\label{eq:H_k_bounds_inverse_J_eq_plus_K_3}
			(J_{eq} + K)(x) w \cdot w
			= J_{eq} w \cdot w + K(x) w \cdot w
			> \lambda \abs{w}^2 - \normtyp{ K }{L}{\infty} \abs{w}^2
			> \frac{\lambda}{2} \abs{w}^2.
		\end{equation}
		Moreover we may immediately deduce from \eqref{eq:H_k_bounds_inverse_J_eq_plus_K_3} that
		$
			\normtypns{ {\brac{J_{eq} + K}}\inv }{L}{\infty} \leqslant 2 / \lambda
		$
		from which item (1) follows.

		Now we establish formulae for derivatives of $ {\brac{J_{eq} + K}}\inv $ reminiscent of the formulae of \fref{Lemma}{lemma:formula_deriv_T_K_inv}.
		Note that
		$
			\partial_i {\brac{J_{eq} + K}}\inv = - {\brac{J_{eq} + K}}\inv \brac{\partial_i K} {\brac{J_{eq} + K}}\inv,
		$
		and hence for any multi-index $\alpha\in\N^3$,
		\begin{equation*}
			\partial^\alpha \brac{ J_{eq} + K } = \sum_{l=1}^{\abs{\alpha}} {(-1})^l \sum_{\alpha_1 + \cdots + \alpha_l = \alpha}
			\mathbb{M} \brac{\alpha_1,\, \dots,\, \alpha_l} (K)
		\end{equation*}
		where
		$
			\mathbb{M} \brac{\alpha_1,\, \alpha_2,\, \dots,\, \alpha_l} (K)
			\vcentcolon= {\brac{J_{eq} + K}}\inv \brac{ \partial^{\alpha_1} K } {\brac{J_{eq} + K}}\inv \brac{\partial^{\alpha_2} K}
			\dots \brac{ \partial^{\alpha_l} K} {\brac{J_{eq} + K}}\inv.
		$

		The crux of the argument now lies in obtaining $L^2$-to-$L^2$ bounds on the operators $\mathbb{M}$.
		In light of the $L^\infty$ bound on ${\brac{ J_{eq} + K }}\inv$ we may proceed as in \fref{Lemma}{lemma:H_k_bounds_T_K_inv} and estimate, for any $v\in L^2$
		and for $k \vcentcolon= \max \abs{\alpha_i}$,
		$
			\norm{
				\mathbb{M}\brac{\alpha_1,\, \dots,\, \alpha_l}(K)
			}{
				\Leb\brac{L^2;\,L^2}
			} \lesssim \normtyp{K}{H}{k+2}^l.
		$

		We may now conclude the proof and obtain item (2) by proceeding once again as in \fref{Lemma}{lemma:H_k_bounds_T_K_inv}.
		For any $k \geqslant 2$ and any $v\in H^k$, the $L^2$-to-$L^2$ bounds on $\mathbb{M}$ tells us that
		\begin{align*}
			\normtypns{ {\brac{J_{eq} + K}}\inv v}{H}{k}
			\lesssim \brac{ \normtyp{K}{H}{k+2} + \normtyp{K}{H}{k+2}^k } \normtyp{v}{H}{k},
		\end{align*}
		from which item (2) follows.
	\end{proof}

	We continue our progress towards \fref{Proposition}{prop:control_IC__energy_by_purely_spatial_deriv} and record elementary estimates on the nonlinearities of the problem.

	\begin{lemma}[Auxiliary estimates of the nonlinearity]
	\label{lemma:aux_est_nonlinearity}
		Let $(\star)$ denote any of the nonlinear terms in \eqref{eq:pertub_sys_no_ten_pdt_theta}, \eqref{eq:pertub_sys_no_ten_pdt_K}, or \eqref{eq:full_sys_u_Leray_proj}, except $K\pdt\theta$.
		Writing $Z = (u, \theta, K)$ we have that, for every $j,k\in\N$ with $j\geqslant 1$,
		\begin{equation*}
			(1)\; \norm{\sbrac{ K\pdt, \pdt^{j-1} } \theta}{H^k} \lesssim \norm{Z}{P^{k+2j}_{j-1}}^2
			\text{ and }\;
			(2)\; \normns{\pdt^{j-1} (\star)}{H^k} \lesssim \norm{Z}{P^{k+2j}_{j-1}} + \norm{Z}{P^{k+2j}_{j-1}}^3.
		\end{equation*}
	\end{lemma}
	\begin{proof}
		These estimates rely on the fact that $H^s$ is a Banach algebra when $s > \frac{3}{2}$ and on the product estimates of \fref{Lemma}{lemma:prod_est}.
		We omit the details -- see \fref{Lemma}{lemma:nonlinear_est_approx_prob} for very similar estimates.
	\end{proof}

	The last result we need in order to prove \fref{Proposition}{prop:control_IC__energy_by_purely_spatial_deriv} is reminiscent of \fref{Lemma}{lemma:bounds_on_P_2M_by_H_2M}.
	In \fref{Lemma}{lemma:control_P_2M_by_H_2M} below we show that the parabolic norm of $Z$ can be controlled by a purely spatial norm.

	\begin{lemma}[Bounds on the parabolic norm by purely spatial norms]
	\label{lemma:control_P_2M_by_H_2M}
		Let $M\geqslant 1$ be an integer. There exists a constant $C_M > 0$ such that, for any time horizon $T > 0$, the following holds.
		If $Z = (u, \theta, K)$ solves \eqref{eq:pertub_sys_no_ten_div}--\eqref{eq:pertub_sys_no_ten_pdt_K} and \eqref{eq:full_sys_u_Leray_proj}
		on $\sbrac{0, T}$ and satisfies $\displaystyle\sup_{0\leqslant t \leqslant T} \normtyp{K(t)}{L}{\infty} \leqslant \frac{\lambda}{2}$, then
		for every $0\leqslant t \leqslant T$ we have that
		$
			\normtyp{Z(t)}{P}{2M} \lesssim \normtyp{Z(t)}{H}{2M} + \normtyp{Z(t)}{H}{2M}^{C_M}.
		$
		In particular this holds when $t=0$.
	\end{lemma}
	\begin{proof}
		We proceed as in \fref{Lemma}{lemma:bounds_on_P_2M_by_H_2M}.
		We apply $\pdt^{j-1}$ to \eqref{eq:pertub_sys_no_ten_div}--\eqref{eq:pertub_sys_no_ten_pdt_K} and \eqref{eq:full_sys_u_Leray_proj},
		invert $J_{eq} + K$ (which is allowed as per \fref{Lemma}{lemma:H_k_bounds_inverse_J_eq_plus_K}),
		and deduce the following.
		On one hand, for $1 \leqslant j \leqslant M-1$, we may use \fref{Lemma}{lemma:H_k_bounds_inverse_J_eq_plus_K} and \fref{Lemma}{lemma:aux_est_nonlinearity}
		to obtain that,
		using the notation $p_a^b (x) \vcentcolon= x^a + x^b$ for $a < b$ and $x\geqslant 0$,
		\begin{align}
			\normtypns{\pdt^j Z}{H}{2M-2j}
			\lesssim \brac{
				\normtyp{K}{H}{2M-2j+2} + \normtyp{K}{H}{2M-2j+2}^{2M-2j}
			}\brac{
				\norm{Z}{P^{2M}_{j-1}} + \norm{Z}{P^{2M}_{j-1}}^3
			}
		\nonumber
		\\
			= p_1^{2M-2j} \brac{ \normtyp{K}{H}{2M} } p_1^3 \brac{ \norm{Z}{P^{2M}_{j-1}} }
			\lesssim p_1^{2M-2j+3} \brac{ \norm{Z}{P^{2M}_{j-1}} }.
		\label{eq:control_P_2M_by_H_2M_1}
		\end{align}
		On the other hand, for $j=M$, using \fref{Lemma}{lemma:H_k_bounds_inverse_J_eq_plus_K} and \fref{Lemma}{lemma:aux_est_nonlinearity} tells us that
		\begin{equation}
		\label{eq:control_P_2M_by_H_2M_2}
			\normtypns{ \pdt^M Z }{L}{2}
			\lesssim \norm{Z}{P^{2M}_{M-1}} + \norm{Z}{P^{2M}_{M-1}}
			= p_1^3 \brac{ \norm{Z}{P^{2M}_{j-1}} }.
		\end{equation}
		Combining \eqref{eq:control_P_2M_by_H_2M_1} and \eqref{eq:control_P_2M_by_H_2M_2} and unpacking the definition of $ \norm{ \,\cdot\, }{P^{k}_{j}} $ we see that,
		for every $1\leqslant j\leqslant M$,
		\begin{equation}
			\norm{Z}{P^{2M}_{j}}
			\asymp \norm{Z}{P^{2M}_{j-1}} + \normtypns{\pdt^j Z}{H}{2M-2j}
			\lesssim p_1^{2M-2j+3} \brac{ \norm{Z}{P^{2M}_{j-1}} }.
		\end{equation}
		Iterating this inequality yields
		\begin{equation}
			\normtyp{Z}{P}{2M}
			= \norm{Z}{P^{2M}_{M}}
			\lesssim \brac{ p_1^3 \circ p_1^5 \circ \cdots \circ p_1^{2M-1} } \brac{ \norm{Z}{P^{2M}_{0}} }
		\end{equation}
		from which, since $ \norm{Z}{P^{2M}_{0}} = \normtyp{Z}{H}{2M} $, the claim follows.
		In particular note that, as in \fref{Lemma}{lemma:bounds_on_P_2M_by_H_2M}, $C_M = \prod_{j=1}^M \brac{2M-2j+3}$.
	\end{proof}

	We may now prove the second auxiliary result of this section required to prove \fref{Theorem}{thm:gwp_decay_clean} below.
	In \fref{Proposition}{prop:control_IC__energy_by_purely_spatial_deriv} we prove that the initial energy may be controlled in terms of purely spatial norms.
	Recall that $ \mathcal{E}_M$ and $ \mathcal{F}_M$ are defined in \eqref{eq:not_E_M_and_F_M}.

	\begin{prop}[Control of the full energy by purely spatial norms]
	\label{prop:control_IC__energy_by_purely_spatial_deriv}
		Let $M\geqslant 1$ be an integer.
		There exist $C_s, C_M > 0$ such that, for any time horizon $T>0$, the following holds.
		If $Z = (u, \theta, K)$ solves \eqref{eq:pertub_sys_no_ten_div}--\eqref{eq:pertub_sys_no_ten_pdt_K} and \eqref{eq:full_sys_u_Leray_proj}
		on $\sbrac{0, T}$ and satisfies $\displaystyle\sup_{0\leqslant t\leqslant T} \normtyp{ K(t) }{L}{\infty} \leqslant \frac{\lambda}{2}$, then,
		for every $0\leqslant t\leqslant T$,
		\begin{equation*}
			\mathcal{E}_M + \mathcal{F}_M
			\leqslant C_s \brac{
				\normtyp{Z}{H}{2M}^2 + \normtyp{Z}{H}{2M}^{2^{M+1} C_M}
				+ \normtyp{K}{H}{2M+1}^2 + \normtyp{K}{H}{2M+1}^{2^{M+1}}
			}.
		\end{equation*}
		In particular this holds when $t=0$.
	\end{prop}
	\begin{proof}
		We proceed in two steps. First we use \fref{Lemma}{lemma:local_in_time_aux_est_K} to show that $ \mathcal{E}_M + \mathcal{F}_M $ can be controlled by
		$ \normtyp{Z}{P}{2M} $ and $ \normtyp{K}{H}{2M+1}$, then we use \fref{Lemma}{lemma:control_P_2M_by_H_2M} to show that $ \normtyp{Z}{P}{2M} $ can be controlled by $ \normtyp{Z}{H}{2M} $.

		Before we being the proof in earnest, note that we may write
		\begin{equation*}
			\mathcal{E}_M + \mathcal{F}_M \asymp \normtyp{Z}{P}{2M}^2 + \normtyp{K}{H}{2M+1}^2 + \norm{K}{P^{2M+2}_{1,M}}^2.
		\end{equation*}
		It then follows immediately from \fref{Lemma}{lemma:local_in_time_aux_est_K} that
		\begin{equation*}
			\mathcal{E}_M + \mathcal{F}_M
			\lesssim \normtyp{Z}{P}{2M}^2 + \normtyp{Z}{P}{2M}^{2^{M+1}} + \normtyp{K}{H}{2M+1}^2 + \normtyp{K}{H}{2M+1}^{2^{M+1}}.
		\end{equation*}
		We may combine this inequality with \fref{Lemma}{lemma:control_P_2M_by_H_2M} to conclude that indeed
		\begin{equation*}
			\mathcal{E}_M + \mathcal{F}_M \lesssim \normtyp{Z}{H}{2M}^2 + \normtyp{Z}{H}{2M}^{2^{M+1} C_M} + \normtyp{K}{H}{2M+1}^2 + \normtyp{K}{H}{2M+1}^{2^{M+1}}
		\end{equation*}
		for $C_M > 0$ as in \fref{Lemma}{lemma:control_P_2M_by_H_2M}.
	\end{proof}

	We may now prove the main result of this paper.
	In order to do so, recall first that $ \mathcal{E}_\text{low}$, $ \mathcal{E}_M$ and $ \mathcal{F}_M$, and $ \mathcal{D}_M$ are defined in
	\eqref{eq:not_E_low}, \eqref{eq:not_E_M_and_F_M}, and \eqref{eq:not_D_M}, respectively.
		
	\begin{thm}[Global well-posedness and decay]
	\label{thm:gwp_decay_clean}
		Let $M\geqslant 4$ be an integer and recall that the global assumptions of \fref{Definition}{def:global_assumptions} hold.
		There exist universal constants $\eta, C > 0$ depending only on $M$ such that the following holds.
		For any
		$
			Z_0 = \brac{u_0, \theta_0, K_0} \in L^2\brac{\T^3;\, \R^3 \times \R^3 \times \sym(3) }
		$
		satisfying
		\begin{equation*}
			\nabla\cdot u_0 = 0,\,
			\fint_{\T^3} u_0 = 0 \text{ and } 
			\normtyp{Z_0}{H}{2M}^2 + \normtyp{K}{H}{2M+1}^2 < \eta
		\end{equation*}
		there exists a unique strong solution $(Z, p)$ of \eqref{eq:pertub_sys_no_ten_pdt_u}--\eqref{eq:pertub_sys_no_ten_pdt_K} where
		\begin{equation*}
			Z = (u, \theta, K) \in C^2 \brac{ \cobrac{0, \infty} \times \T^3;\, \R^3 \times \R^3 \times \sym(3) }
			\text{ and } 
			p \in C^2 \brac{ \cobrac{0, \infty} \times \T^3;\, \R}.
		\end{equation*}
		Moreover the solution satisfies the estimate
		\begin{equation}
		\label{eq:gwp_clean_small_energy_regime}
			\sup_{t\geqslant 0} \mathcal{E}_\text{low} (t) {\brac{1+t}}^{2M-1} + \mathcal{E}_M (t) + \frac{ \mathcal{F}_M }{1+t} + \int_0^\infty \mathcal{D}_M (s) ds
			\leqslant C \brac{ \normtyp{Z_0}{H}{2M}^2 + \normtyp{K_0}{H}{2M+1}^2 }.
		\end{equation}
	\end{thm}
	\begin{proof}
		The strategy of the proof is as follows.
		Coupling the local well-posedness theory of \fref{Theorem}{thm:lwp} to the auxiliary estimate for $K$ of \fref{Lemma}{lemma:local_in_time_aux_est_K},
		which allows us to account for the mismatch between the energies used for the local well-posedness and the energies used for the a priori estimates,
		we produce a solution locally-in-time on which we have enough control to invoke the a priori estimates of \fref{Theorem}{thm:a_priori}.
		This tells us that this (possibly very short-lived) solution lives in the small energy regime defined by \eqref{eq:gwp_clean_small_energy_regime}.
		Continuing this solution globally-in-time then follows immediately from leveraging the continuation argument of \fref{Theorem}{thm:cont}.

		We begin by defining the smallness parameter $\eta > 0$.
		We pick $0 < \eta \leqslant 1$ satisfying
		\begin{enumerate}
			\item	\label{it:gwp_smallness_cond_1}
				$\eta^{1/2} < \sigma\brac{ \delta^{loc}_{ap} }$ for $\delta^{loc}_{ap}$ as in \fref{Lemma}{lemma:reduced_a_priori_est}
				and $\sigma = \sigma\brac{\delta}$ as in the beginning of \fref{Section}{sec:lwp},
			\item	\label{it:gwp_smallness_cond_2}
				$\brac{2 + C_K} p_1^{2M} \brac{ \brac{\rho_e + \rho_d} \brac{\eta} } < \delta_{ap}$ for $\delta_{ap}$ as in \fref{Theorem}{thm:a_priori},
				$C_K$ as in \fref{Lemma}{lemma:local_in_time_aux_est_K}, $\rho_e$ and $\rho_d$ as in \fref{Theorem}{thm:lwp},
				and $p_1^{2M} (x) \vcentcolon= x + x^{2M}$ for all $x\geqslant 0$,
			\item 	\label{it:gwp_smallness_cond_3}
				$C_s \eta \leqslant \eta_{ap}$ for $C_s$ as in \fref{Proposition}{prop:control_IC__energy_by_purely_spatial_deriv}
				and $\eta_{ap}$ as in \fref{Theorem}{thm:a_priori}, and
			\item	\label{it:gwp_smallness_cond_4}
				$C_s \eta \leqslant \eta_{cont}$ for $\eta_{cont}$ as in \fref{Theorem}{thm:cont}.
		\end{enumerate}

		We may now construct a solution locally-in-time which lives in the small energy regime, as defined by \eqref{eq:gwp_clean_small_energy_regime}.
		In light of the smallness condition \eqref{it:gwp_smallness_cond_1}, which tells us that $ \normtyp{K_0}{H}{3} < \eta^{1/2} < \sigma\brac{\delta^{loc}_{ap}}$,
		the local well-posedness theory of \fref{Theorem}{thm:lwp} shows that there exists $T_\text{lwp} > 0$ and a strong solution
		\begin{equation*}
			Z = (u, \theta, K) \in C^2 \brac{ \sbrac{0, T_\text{lwp}} \times \T^3;\, \R^3 \times \R^3 \times \sym(3) }
			\text{ and }
			p \in C^2 \brac{ \sbrac{0, T_\text{lwp}} \times \T^3;\, \R}
		\end{equation*}
		which satisfies
		\begin{align}
			&\sup_{0\leqslant t\leqslant T_\text{lwp}} \normtyp{K(t)}{H}{3} < \eta^{1/2} \text{ and } 
			\label{eq:gwp_clean_1}\\
			&\sup_{0\leqslant t\leqslant T_\text{lwp}} \normtyp{Z(t)}{P}{2M}^2 + \int_0^{T_\text{lwp}} \overline{\mathcal{D}}_M (s) ds
			\leqslant \brac{\rho_e + \rho_d} \brac{ \normtyp{Z_0}{H}{2M} },
			\label{eq:gwp_clean_2}
		\end{align}
		where recall that $ \overline{\mathcal{D}}_M$ is defined in \eqref{eq:not_D_M}.
		In particular the auxiliary estimate for $K$ of \fref{Lemma}{lemma:local_in_time_aux_est_K} tells us that, in light of \eqref{eq:gwp_clean_2}
		and recalling that $ \mathcal{E}_M^{(K)}$ is defined in \eqref{eq:not_E_M_K},
		\begin{align}
			\sup_{0\leqslant t\leqslant T_\text{lwp}} \normtyp{Z(t)}{P}{2M}^2 + \mathcal{E}_M^{(K)} (t)
			\leqslant \brac{1 + C_K} p_1^{2M} \brac{ \sup_{0\leqslant t\leqslant T_\text{lwp}} \normtyp{Z(t)}{P}{2M} }
		\nonumber
		\\
			\leqslant \brac{1 + C_K} p_1^{2M} \brac{ \brac{\rho_e + \rho_d} \brac{ \normtyp{Z_0}{H}{2M} }}.
		\label{eq:gwp_clean_4}
		\end{align}
		Putting \eqref{eq:gwp_clean_2} and \eqref{eq:gwp_clean_4} together tells us that, by virtue of the smallness condition \eqref{it:gwp_smallness_cond_2},
		\begin{equation}
		\label{eq:gwp_clean_5}
			\sup_{0\leqslant t\leqslant T_\text{lwp}} \mathcal{E}_M (t) + \int_0^{T_\text{lwp}} \overline{\mathcal{D}}_M (s) ds
			\leqslant \brac{2 + C_K} p_1^{2M} \brac{ \brac{\rho_e + \rho_d} \brac{ \normtyp{Z_0}{H}{2M} } }
			\leqslant \delta_{ap}.
		\end{equation}
		Note also that we may deduce from the smallness condition \eqref{it:gwp_smallness_cond_3} and \fref{Proposition}{prop:control_IC__energy_by_purely_spatial_deriv} that,
		since $\eta \leqslant 1$,
		\begin{equation}
		\label{eq:gwp_clean_6}
			\brac{ \mathcal{E}_M + \mathcal{F}_M } (0)
			\leqslant C_s \brac{ \normtyp{Z_0}{H}{2M}^2 + \normtyp{K_0}{H}{2M+1}^2 }
			\leqslant \eta_{ap}.
		\end{equation}
		Combining \eqref{eq:gwp_clean_5} and \eqref{eq:gwp_clean_6} allows us to use the a priori estimate of \fref{Theorem}{thm:a_priori},
		from which we deduce that
		\begin{equation}
		\label{eq:gwp_clean_7}
			\sup_{0\leqslant t\leqslant T_\text{lwp}}
				\mathcal{E}_\text{low} (t) {\brac{1+t}}^{2M-2} + \mathcal{E}_M (t) + \frac{ \mathcal{F}_M (t) }{1+t} + \int_0^{T_\text{lwp}} \mathcal{D}_M (s) ds
			\leqslant C_{ap} \brac{ \mathcal{E}_M + \mathcal{F}_M } (0).
		\end{equation}

		To conclude we employ a standard continuation argument revolving around the continuation argument of \fref{Theorem}{thm:cont}.
		Let us define, for any $T\in\ocbrac{0, \infty}$,
		\begin{equation*}
			\mathcal{G}(T) \vcentcolon= \sup_{0\leqslant t\leqslant T}
				\mathcal{E}_\text{low} (t) {\brac{1+t}}^{2M-2} + \mathcal{E}_M (t) + \frac{ \mathcal{F}_M (t) }{1+t} + \int_0^{T} \mathcal{D}_M (s) ds
		\end{equation*}
		which we use to define the maximal time of existence
		\begin{equation*}
			T_\text{max} \vcentcolon= \sup \cbrac{ T > 0 :
				\exists ! \text{ strong solution on } \sbrac{0,T} \text{ and } \mathcal{G}(T) \leqslant C_{ap} \brac{ \mathcal{E}_M + \mathcal{F}_M } (0) 
			}.
		\end{equation*}
		By virtue of \eqref{eq:gwp_clean_7}, we know that $T_\text{max} > T_\text{lwp} > 0$.
		Crucially: \fref{Theorem}{thm:cont} tells us that, in light of the smallness condition \eqref{it:gwp_smallness_cond_4} and \eqref{eq:gwp_clean_6},
		and since $T_\text{max} > 0$,
		$T_\text{max}$ cannot be finite. So indeed the solution exists globally-in-time and, since $T_\text{max} = \infty$,
		we appeal to \fref{Proposition}{prop:control_IC__energy_by_purely_spatial_deriv} one last time to deduce that
		\begin{equation*}
			\mathcal{G} \brac{\infty}
			\leqslant C_{ap} \brac{ \mathcal{E}_M + \mathcal{F}_M } (0)
			\lesssim \normtyp{Z_0}{H}{2M}^2 + \normtyp{K_0}{H}{2M+1}^2,
		\end{equation*}
		i.e. indeed \eqref{eq:gwp_clean_small_energy_regime} holds.
	\end{proof}

	In order to deduce the decay of $K$ from \fref{Theorem}{thm:gwp_decay_clean} above we need the quantitative rigidity estimate of \fref{Proposition}{prop:quant_rigidity} below.
	Note that the term \emph{quantitative rigidity} is motivated by contrast with the following \emph{qualitative rigidity} result:
	if $a=0$ and $ \normtyp{J - J_{eq}}{L}{\infty}  < \nu - \lambda$ then $J = J_{eq}$
	(this can be seen by noticing that if $a=0$ then $J_{33}$ must be an eigenvalue of $J$, and it cannot be that $J_{33} = \lambda$
	since that would contradict the condition $ \normtyp{J-J_{eq}}{L}{\infty}  < \nu-\lambda$).

	\begin{prop}[Quantitative rigidity]
	\label{prop:quant_rigidity}
		Let $T>0$ be a time horizon and suppose that
		\begin{equation}
		\label{eq:quant_rigidity_assumption}
			\sup_{0\leqslant t\leqslant T} \normtyp{(u,\theta)(t)}{H}{3} + \normtyp{J(t)}{H}{3} + \normtyp{\pdt(u,\theta)}{H}{2}^2 + \normtyp{\pdt J}{H}{2}^2 < \infty.
		\end{equation}
		If $ \displaystyle\sup_{0\leqslant t\leqslant T} \normtyp{K}{L}{\infty} \leqslant \nu-\lambda$
		then $ \displaystyle\sup_{0\leqslant t\leqslant T} \normtyp{K}{L}{p} \leqslant 2 \displaystyle\sup_{0\leqslant t\leqslant T} \normtyp{a}{L}{p}$ for any $1\leqslant p\leqslant\infty$.
	\end{prop}
	\begin{proof}
		Since \eqref{eq:quant_rigidity_assumption} holds we know from \fref{Proposition}{prop:persist_spec_sols_adv_rot_eqtns}
		that $J(t,x)$ is pointwise symmetric with spectrum $\cbrac{\lambda, \lambda, \nu}$.
		The key observation now is that we may therefore find a unit vector field $n(t,x)$ such that $J = \nu n\otimes n + \lambda (I - n\otimes n)$ pointwise
		(indeed we may simply take $n$ to be the unit eigenvector of $J$ corresponding to the eigenvalue $\nu$).
		Writing $J_{eq} = \nu e_3\otimes e_3 + \lambda (I-e_3\otimes e_3)$ we may then compute that
		\begin{equation*}
			{\abs{ J-J_{eq} }}^2
			= {(\nu-\lambda)}^2 {\abs{ n\otimes n - e_3\otimes e_3}}^2
			= 2 {(\nu-\lambda)}^2 (1 - n_3^2)
			= 2 {(\nu-\lambda)}^2 \abs{\bar{n}}^2.
		\end{equation*}
		In particular, if $ \normtyp{K}{L}{\infty}  \leqslant \nu - \lambda$ then we may deduce that $n_3^2 \geqslant \frac{1}{2}$ pointwise.
		To conclude we note that since $J_{ij} = J e_j \cdot e_i$ we may compute that $a = (\nu - \lambda) n_3 \bar{n}$.
		So finally
		$
			{\abs{K}}^2 = \frac{2 \abs{a}^2}{n_3^2} \leqslant 4\abs{a}^2,
		$
		from which the claim follows.
	\end{proof}

	In light of this quantitative rigidity result we may deduce the decay of $K$, and hence $\pdt K$, from the decay of $a$.
	As discussed in \fref{Section}{sec:intro}, this argument could be iterated further in order to derive the decay of higher-order temporal derivatives of $K$,
	but this is not done here since that decay is not used in the scheme of a priori estimates.
	Recall that $ \overline{\mathcal{K}}_I$ is defined in \eqref{eq:not_K_bar}.

	\begin{prop}[Rates of decay of $K$ and $\pdt K$]
	\label{prop:rates_decay_K_pdt_K}
		Let $M\geqslant 3$ be an integer and let $T>0$ be a time horizon.
		There exists $C_1 > 0$ such that the following holds.
		If $(u, \theta, K)$ solves \eqref{eq:pertub_sys_no_ten_pdt_K} and satisfies
		\begin{align}
		\label{eq:rates_decay_K_assumption}
			C \vcentcolon= \sup_{0\leqslant t \leqslant T}
			\normtyp{K(t)}{L}{2}^2 {\brac{1+t}}^{2M-4} + \overline{\mathcal{K}}_2 {\brac{1+t}}^{2M-2} + \frac{\mathcal{F}_M (t)}{1+t}
			< \infty
		\end{align}
		then, for $j=0,\,1$,
		$
			\displaystyle \sup_{0\leqslant s \leqslant s_j} \sup_{0\leqslant t\leqslant T} \normtypns{\pdt^j K}{H}{s}^2 {\brac{1+t}}^{2M-4-(2M-3)/s_j} \leqslant C_1 C,
		$
		where $s_0 = 2M+1$ and $s_1 = 2M$.
	\end{prop}

	\begin{proof}
		We interpolate between the decay of $ \normtyp{K}{L}{2}^2$ and the growth of $ \mathcal{F}_M$ in \eqref{eq:rates_decay_K_assumption}
		to deduce the bounds on $ \normtyp{K}{H}{s}^2$ recorded here.
		To obtain the bounds on $ \normtyp{\pdt K}{H}{s}^2 $ we first use \eqref{eq:pertub_sys_no_ten_pdt_K} to read off the $L^2$ bound on $\pdt K$ using H\"{o}lder's inequality
		and \eqref{eq:rates_decay_K_assumption} and then interpolate between this $L^2$ bound and $ \mathcal{F}_M $.
	\end{proof}

	\begin{remark}
		Note that the estimates recorded above are not all decay estimates.
		To be precise: $ \normtyp{K(t)}{H}{s}$ decays when $s < 2M - \frac{4}{2M-3}$ whereas $ \normtyp{\pdt K}{H}{s} $ decays when $s < 2M-1-\frac{3}{2M-3}$.
		In particular these regularity cut-offs approach $2M$ and $2M-1$, respectively, asymptotically from below as $M\to +\infty$.
	\end{remark}

	We conclude this section by proving \fref{Corollary}{cor:add_decay}, which records the precise decay rates of the unknowns and their temporal derivatives.
	\begin{proof}[Proof of \fref{Corollary}{cor:add_decay}]
		It suffices to combine \fref{Theorem}{thm:gwp_decay_clean} and \fref{Propositions}{prop:decay_intermediate_norms}, \ref{prop:quant_rigidity}, and \ref{prop:rates_decay_K_pdt_K}.
	\end{proof}
	
%----------------------------------------------------------------------------------------------------
%	APPENDIX
%----------------------------------------------------------------------------------------------------
\appendix

\section{Identities involving the microinertia}
\label{sec:app_id_microinertia}

	In this section we record various computations and identities involving the microinertia tensor $J$ which are used throughout the paper.

	We now record two lemmas that are used in the proof of \fref{Proposition}{prop:persist_spec_sols_adv_rot_eqtns} below.
	This proposition is essential to our scheme of a priori estimates,
	and is a fundamental feature of the micropolar fluid model.
	It shows that if the solution is regular enough, then the spectrum of the microinertia is propagated by the flow.
	First we record a well-known result showing that the advective derivative is simply a time derivative up to a change of variables using the flow map (i.e. with respect to Eulerian coordinates).

	\begin{lemma}[Calculus of advective derivatives]
	\label{lemma:calc_adv_deriv}
		Let $\eta \in C^2 \brac{ \cobrac{0, T} \times \R^n;\, \R^n}$ be a flow map, i.e., for all $0 \leqslant t < T$, $\eta_t \vcentcolon= \eta\brac{t, \,\cdot\,}$ is a $C^1$-diffeomorphism,
		with velocity $u \in C^1 \brac{ \cobrac{0, T} \times \R^n,\; \R^n}$ defined by $u(t,x) \vcentcolon= \pdt\eta \brac{t, \eta_t\inv (x) }$. Then
		$
			\pdt\brac{\det\nabla\eta} = \brac{\brac{\nabla\cdot u}\circ\eta} \det\nabla\eta
		$
		and, for every $f\in C^1 \brac{ \cobrac{0, T} \times \R^n;\, \R}$,
		$
			\pdt\brac{f\circ\eta} = \brac{\brac{\pdt + u\cdot\nabla} f}\circ\eta
		$
		where for any $g:\cobrac{0, T} \times \R^n \to \R$ we write $g\circ\eta$ to denote the composition $(g\circ\eta) (t,x) \vcentcolon= g\brac{t, \eta(t,x) }$.
	\end{lemma}
	\begin{proof}
		The first identity is the well-known Liouville Theorem and the second identity follows from the first by the Chain Rule.
	\end{proof}

	We continue our progress towards a proof of \fref{Proposition}{prop:persist_spec_sols_adv_rot_eqtns} below
	with an ODE result recorded in \fref{Lemma}{lemma:two_sided_ODE_commutators}.
	This lemma provides an equivalent characterization of the ODE satisfied by the microinertia (denoted by $S$ in \fref{Lemma}{lemma:two_sided_ODE_commutators})
	in Lagrangian coordinates in terms of the ODE satisfied by its rotation matrix (denoted by $Q$ in \fref{Lemma}{lemma:two_sided_ODE_commutators}).

	\begin{lemma}[Two-sided integrating factors for ODEs with commutators]
	\label{lemma:two_sided_ODE_commutators}
		Let $S, A \in C^1 \brac{ \cobrac{0,T};\, \R^{n\times n}}$ be time-dependent symmetric and anti-symmetric matrices, respectively,
		and let $S_0$ be a fixed symmetric real $n$-by-$n$ matrix. The following are equivalent.
		\begin{enumerate}
			\item 	$S$ solves the initial value problem
				$
					\pdt S = \sbrac{A,S}
				$ on $
					\brac{0,T}
				$ and $
					S(0) = S_0.
				$
			\item	There exists a time-dependent orthogonal matrix $Q \in C^1 \brac{ \cobrac{0, T};\, O(n)}$ such that $S =QS_0 Q^T$ and $Q$ solves the initial value problem
				$
					\pdt Q = AQ
				$ on $
					\brac{0,T}
				$ and $
					Q(0) = I.
				$
		\end{enumerate}
		Here $O(n)$ denotes the space of $n$-by-$n$ real orthogonal matrices.
	\end{lemma}
	\begin{proof}
		First we show that $(2) \implies (1)$. If $(2)$ holds then $\pdt Q^T = {\brac{AQ}}^T = -Q^T A$ and
		therefore $$\pdt S = \pdt Q S_0 Q^T + Q S_0 \pdt Q^T = A Q S_0 Q^T - Q S_0 \pdt Q^T A = \sbrac{A, S}.$$
		Now we show that $(1) \implies (2)$. Suppose that $(1)$ holds and let us define $Q(t) \vcentcolon= \exp\brac{\int_0^t A(s) ds}$ such that $Q$ solves the initial value problem of $(2)$.
		Since $(1)$ is a linear ODE it has a unique solution, so in order to show that $S = QS_0S^T$ it suffices to show that $QS_0S^T$ is a solution of the initial value problem of $(1)$.
		This follows immediately from the same computation as that which was carried out above in order to show that $(2) \implies(1)$.
	\end{proof}

	We are now ready to prove \fref{Proposition}{prop:persist_spec_sols_adv_rot_eqtns}
	which shows that if the velocity fields and the microinertia are sufficiently regular then the spectrum of the microinertia is propagated in time.

	\begin{prop}[Persistence of the spectrum for solutons of advection-rotation equations]
	\label{prop:persist_spec_sols_adv_rot_eqtns}
		Suppose that $u\in C^1\brac{ \cobrac{0,T} \times \R^n;\, \R^n}$ is divergence-free and consider $\Omega, J \in C^1\brac{ \cobrac{0,T} \times \R^n;\, \R^{n\times n}}$,
		where $\Omega$ is anti-symmetric. If they satisfy
		\begin{equation*}
			\pdt J + u\cdot\nabla J = \sbrac{\Omega, J}
			\text{ and } 
			J(0, \,\cdot\,) = J_0
		\end{equation*}
		for some real $n$-by-$n$ matrix $J_0$ then there exists a flow map $\eta \in C^2 \brac{ \cobrac{0, T} \times \R^n;\, \R^n}$,
		where $\eta_t \vcentcolon= \eta(t, \,\cdot\,)$ is a $C^1$-diffeomorphism for all $0\leqslant t < T$,
		and there exists an Eulerian rotation map $R \in C^1 \brac{ \cobrac{0,T} \times \R^n;\, O(n)}$ such that
		\begin{equation*}
			J = R \brac{J_0 \circ \eta\inv} R^T,
		\end{equation*}
		or, more precisely, $J(t,x) = R(t,x) J\brac{t, \eta_t\inv (x)} R^T (t,x)$.
		In particular, for every $(t,x) \in \cobrac{0,T}\times\R^n$ if we write $y = \eta_t\inv (x)$ then $J_0 (y)$ and $J(t,x)$ have the same spectrum.
	\end{prop}
	\begin{proof}
		The key ideas are that (1) by virtue of \fref{Lemma}{lemma:calc_adv_deriv}, $\pdt + u\cdot\nabla$ is nothing more than a time derivative up to a change of coordinates
		and (2) in light of \fref{Lemma}{lemma:two_sided_ODE_commutators} solutions of $\pdt = \sbrac{\Omega, \,\cdot\,}$ are pointwise conjugate to their initial conditions
		by some rotation matrix with angular velocity $\Omega$.

		\paragraph{\textbf{Step 1.}} We define the flow map $\eta$ to be the solution of
		$
			\pdt\eta = u\circ\eta
		$ with initial condition $
			\eta\brac{t=0} = \id.
		$
		As a consequence of $u$ being divergence-free it follows from \fref{Lemma}{lemma:calc_adv_deriv} that
		$
			\pdt\brac{\det\nabla\eta} = 0
		$
		and hence $\det\nabla\eta = \det\nabla\eta\brac{t=0} \equiv 1$, so indeed $\eta_t$ is invertible at all times $t$.
		Finally we deuce that $\eta_t$ is a $C^1$-diffeomorphism for all times $t$ from the fact that $\nabla\brac{\eta\inv} = {\brac{\nabla\eta}}\inv \circ \eta\inv$.

		\paragraph{\textbf{Step 2.}} Let us define $\mathcal{J}$ and $\Theta$ to be the Lagrangian counterparts of $J$ and $\Omega$ respectively,
		i.e. $\mathcal{J} \vcentcolon= J\circ\eta$ and $\Theta \vcentcolon= \Omega\circ\eta$
		Then, by \fref{Lemma}{lemma:calc_adv_deriv},
		\begin{equation*}
			\pdt\mathcal{J} = \pdt\brac{J\circ\eta} = \brac{\brac{\pdt+u\cdot\nabla}J}\circ\eta = \sbrac{\Theta, \mathcal{J}}
			\text{ and } \mathcal{J}\brac{0, \,\cdot\,} = J_0 \circ\eta_0 = J_0.
		\end{equation*}
		So $\mathcal{J}$ solves
		$
			\pdt\mathcal{J} = \sbrac{\Theta, \mathcal{J}}
		$
		with initial condition
		$
			\mathcal{J}\brac{0, \,\cdot\,} = J_0.
		$

		\paragraph{\textbf{Step 3.}} We define the Lagrangian rotation map $Q\brac{t,y} \vcentcolon= \exp\brac{\displaystyle\int_0^t \Theta\brac{s,y} ds}$ such that,
		by \fref{Lemma}{lemma:two_sided_ODE_commutators}, $\mathcal{J} = QJ_0 Q^T$.
		So finally, if we introduce the Eulerian rotation map $R \vcentcolon= Q \circ \eta\inv$ we may conclude that
		$
			J = R\brac{J_0 \circ \eta\inv}R^T.\qedhere
		$
	\end{proof}

	We now record some elementary identities which are useful throughout the paper.
	The first identity allows us to deal with the precession term appearing in the conservation of angular momentum when deriving energy-dissipation relations.

	\begin{lemma}
	\label{lemma:identity_comm_A_S_and_sym_A_cross_S}
		Let $A$ and $S$ be $n$-by-$n$ matrices which are anti-symmetric and symmetric, respectively.
		Then $ \frac{1}{2} \sbrac{A,S} = \sym\brac{AS}$.
		In particular: if $n=3$ and we let $a\defeq\vc A$ then $ \frac{1}{2} \sbrac{A,S} = \sym\brac{a\times S}$.
	\end{lemma}
	\begin{proof}
		This is immediate: $\sym\brac{AS} = \frac{1}{2}\brac{AS + S^T A^T} = \frac{1}{2} \brac{AS - SA} = \frac{1}{2} \sbrac{A,S}$.
	\end{proof}

	The second identity shows that one of the terms appearing in the conservation of microinertia \eqref{eq:full_sys_J} is antisymmetric (as a map on the space of symmetric matrices),
	and hence does not contribute to energy or transport estimates.

	\begin{lemma}
	\label{lemma:commut_are_antisym_maps_on_space_sym_matrices}
		Let $S$ and $M$ be real $n$-by-$n$ matrices such that $S$ is symmetric. Then $\sbrac{M,S}:S = 0$.
	\end{lemma}
	\begin{proof}
		The proof follows from the observation that $SM:S = M:S^TS = M:SS^T = MS:S$.
	\end{proof}

	Finally we record a detailed computation of the block form of $\sbrac{\Omega, J}$ which comes in handy when reading off the equation governing the dynamics of $a$.

	\begin{lemma}[Block form of $\sbrac{\Omega, J}$]
	\label{lemma:block_form_Omega_J}
		Let $J$ be a symmetric $3$-by-$3$ matrix written in $(2+1)$-by-$(2+1)$ block form as 
		$
			J = \begin{pmatrix}
				\bar{J} & a\\
				a^T & J_{33}
			\end{pmatrix}
		$
		and let $\Omega = \ten\omega$ for some $\omega\in\R^3$.
		Then we may write the commutator $\sbrac{\Omega, J}$ in $(2+1)$-by-$(2+1)$ block form as
		\begin{equation}
		\label{eq:block_form_Omega_J}
			\sbrac{\Omega, J} = \begin{pmatrix}
				\omega_3 \sbrac{R, \bar{J}} - \brac{ \bar{\omega}^\perp \otimes a + a \otimes \bar{\omega}^\perp}
				& \brac{\bar{J} - J_{33}I_2} \bar{\omega}^\perp + \omega_3 a^\perp
				\\ {\brac{ \brac{\bar{J} - J_{33}I_2} \bar{\omega}^\perp + \omega_3 a^\perp }}^T
				& 2 a\cdot\bar{\omega}^\perp
			\end{pmatrix}
		\end{equation}
		where $R = e_2 \otimes e_1 - e_1 \otimes e_2 \in \R^{2\times 2}$ denotes the (counterclockwise) $\frac{\pi}{2}$ rotation in $\R^2$.
	\end{lemma}
	\begin{proof}
		Note that we may write $\Omega$ in block form using the rotation matrix $R$ as
		$
			\Omega = \begin{pmatrix}
				\omega_3 R & -\bar{\omega}^\perp\\
				{\brac{ \bar{\omega}^\perp }}^\perp & 0
			\end{pmatrix}.
		$
		We may then compute
		\begin{equation*}
			\Omega J = \begin{pmatrix}
				\omega_3 R \bar{J} - \bar{\omega}^\perp \otimes a & \omega_3 a^\perp - J_{33} \bar{\omega}^\perp\\
				{\brac{ \bar{\omega}^\perp }}^T \bar{J} & \bar{\omega}^\perp \cdot a
			\end{pmatrix}.
		\end{equation*}
		Since $J\Omega = -{\brac{\Omega J}}^T$ we deduce that indeed \eqref{eq:block_form_Omega_J} holds.
	\end{proof}

\section{Analytical results}
\label{sec:app_analyt_results}

	In this section we record precise statements of well-known analytical results for the reader's convenience.
	First we record the Gagliardo-Nirenberg interpolation inequalities on bounded domains, which is crucial in several nonlinear estimates.

	\begin{thm}[Gagliardo-Nirenberg interpolation inequalities]
	\label{thm:Gagliardo_Nirenberg}
		Let $u\in L^q\brac{\T^n}$ with $\nabla^k u \in L^r \brac{\T^n}$ such that
		\begin{equation*}
			\frac{1}{p} - \frac{l}{n} = \theta \frac{1}{q} + (1-\theta)\brac{ \frac{1}{r} - \frac{k}{n} }
			\text{ and } (1-\theta)k \geqslant l
			\text{ for some } 0 \leqslant \theta \leqslant 1.
		\end{equation*}
		Then $\nabla^l u \in L^p \brac{\T^n}$ and we have the estimate
		$
			\normns{\nabla^l u}{L^p \brac{\T^n}} \lesssim \norm{u}{L^q \brac{\T^n}}^\theta \norm{u}{W^{k,r} \brac{\T^n}}^{1-\theta}.
		$
	\end{thm}
	\begin{proof}
		This is a standard result. See for example Section 13.3 in \cite{leoni} for a proof of this result on cubes which immediately implies the result on the torus.
	\end{proof}

	In practice the Gagliardo-Nirenberg interpolation inequality is used in the form recorded in \fref{Corollary}{cor:est_interactions_L_2_via_Gagliardo_Nirenberg}
	throughout the paper.
	In particular the second inequality recorded in \fref{Corollary}{cor:est_interactions_L_2_via_Gagliardo_Nirenberg} is a high-low estimate
	which is central to our efforts to balance out terms that grow in time and terms that decay in time when designing our scheme of a priori estimates.

	\begin{cor}[Estimate of interactions in $L^2$]
	\label{cor:est_interactions_L_2_via_Gagliardo_Nirenberg}
		Let $f, g \in H^k \brac{\T^n}$ and let $\alpha$ and $\beta$ be multi-indices satisfying $\abs{\alpha} + \abs{\beta} = k$.
		Then we have the estimates
		$
			\normns{ ( \partial^\alpha f ) ( \partial^\beta g)}{L^2}
			\lesssim \normtyp{f}{L}{\infty} \normtyp{g}{H}{k} + \normtyp{f}{H}{k} \normtyp{g}{L}{\infty}
		$
		and
		$
			\normtyp{fg}{H}{k}
			\lesssim \normtyp{f}{L}{\infty} \normtyp{g}{H}{k} + \normtyp{f}{H}{k} \normtyp{g}{L}{\infty}.
		$
	\end{cor}
	\begin{proof}
		The first estimate follows from the Gagliardo-Nirenberg inequality recorded in \fref{Theorem}{thm:Gagliardo_Nirenberg}.
		So let us define $\theta \vcentcolon= \frac{\abs{\beta}}{k}$, $\frac{1}{p} \vcentcolon= \frac{1-\theta}{2}$, and $\frac{1}{q} \vcentcolon= \frac{\theta}{2}$.
		Then, by the H\"{o}lder and Gagliardo-Nirenberg inequalities,
		\begin{equation*}
			\normns{ ( \partial^\alpha f) ( \partial^\beta g ) }{L^2}
			\leqslant \normns{ \partial^\alpha f}{L^p} \normns{ \partial^\beta g}{L^q}
			\lesssim \normtyp{f}{L}{\infty}^\theta \normtyp{f}{H}{k}^{1-\theta} \normtyp{g}{L}{\infty}^{1-\theta} \normtyp{g}{H}{k}^\theta
			\lesssim \normtyp{f}{L}{\infty} \normtyp{g}{H}{k} + \normtyp{f}{H}{k} \normtyp{g}{L}{\infty},
		\end{equation*}
		where we have used Young's inequality at the end, namely using the fact that for any $x,y\geqslant 0$ and any $0\leqslant \theta \leqslant 1$,
		$xy \leqslant \theta x^{1/\theta} + (1-\theta) y^{1/(1-\theta)}$.
		The second estimate then follows from the first by using the Leibniz rule.
	\end{proof}

	From \fref{Corollary}{cor:est_interactions_L_2_via_Gagliardo_Nirenberg} we may commutator estimates for transport and multiplication operators.

	\begin{lemma}[Commutator estimates for transport and multiplication operators]
	\label{lemma:comm_est_transp_op}
		Let $u\in H^k \brac{\T^n;\, \R^n}$, let $f,g\in H^k \brac{\T^n;\, \R}$, and let $\alpha\in\N^n$ with $\abs{\alpha} = k$.
		Then
		$
			\normtyp{ \sbrac{g, \partial^\alpha } f}{L}{2}
			\lesssim \normtyp{g}{L}{\infty} \normtyp{f}{H}{k} + \normtyp{g}{H}{k} \normtyp{f}{L}{\infty}
		$
		and
		$
			\normtyp{ \sbrac{u\cdot\nabla, \partial^\alpha } f}{L}{2}
			\lesssim \normtyp{ \nabla u }{L}{\infty} \normtyp{\nabla f}{H}{k-1} + \normtyp{\nabla u}{H}{k-1} \normtyp{ \nabla f }{L}{\infty}.
		$
	\end{lemma}
	\begin{proof}
		This follows from \fref{Corollary}{cor:est_interactions_L_2_via_Gagliardo_Nirenberg} and the Leibniz rule.
	\end{proof}

	We conclude this section with other well-known analytical results.
	First, a product estimate in $H^s$ spaces.

	\begin{lemma}[Product estimate]
	\label{lemma:prod_est}
		Let $s > \frac{n}{2}$ and let $0\leqslant t\leqslant s$.
		There exists $C = C(s,t) > 0$ such that, for every $f\in H^s\brac{\T^n}$ and every $g\in H^t\brac{\T^n}$,
		$
			\normtyp{fg}{H}{t} \leqslant C \normtyp{f}{H}{s} \normtyp{g}{H}{t}.
		$
		In other words: $H^s$ is a continuous multiplier on $H^t$.
	\end{lemma}
	\begin{proof}
		The key observation is that $H^t$ is the interpolation space of order $t/s$ of the pair $(L^2, H^s)$.
		Since $H^s \hookrightarrow L^\infty$ and $H^s$ is a Banach algebra we know that $g\mapsto fg$ is bounded on both $L^2$ and $H^s$.
		The result then follows by interpolation.
	\end{proof}

	We also record a nonlinear Gronwall-type argument which is crucial in closing the energy estimates at the low level when developing the scheme of a priori estimates,
	in obtaining uniform bounds on the approximate solutions when building the local-well posedness,
	and in deriving the reduced energy estimates necessary to produce the continuation argument that glues the a priori and the local well-posedness together.

	\begin{lemma}[Bihari's Lemma]
	\label{lemma:Bihari}
		Let $f : \cobrac{0, \infty} \to \cobrac{0, \infty}$ be non-decreasing and continuous such that $f > 0$ on $\brac{0, \infty}$ and $\int_1^\infty 1/f < \infty$.
		Let $F$ be the anti-derivative of $-1/f$ which vanishes at $+\infty$.
		For every continuous function $y : \cobrac{0, \infty} \to \cobrac{0, \infty}$, if there exists $\alpha_0 > 0$ such that
		\begin{equation*}
			y\brac{t} + \int_0^t f\brac{y\brac{s}} ds \leqslant \alpha_0
			\text{ for every } t \geqslant 0
		\end{equation*}
		then, for every $t \geqslant 0$,
		$
			y\brac{t} \leqslant F\inv\brac{t + F\brac{\alpha_0}}.
		$
	\end{lemma}
	\begin{proof}
		This is proven in Lemma II.4.12 of \cite{boyer_fabrie}
	\end{proof}

	We conclude this section with an elementary result which is very handy when it comes to ensuring that
	derivatives do not accumulate unduly on a single term in the nonlinear interactions.

	\begin{lemma}
	\label{lemma:from_ineq_of_sums_to_overlap_ineq_3_terms}
		Let $x, y, z, C_x, C_y, C_z$ be real numbers such that $x, y, z \geqslant 0$.
		If $$x + y + z \leqslant \min\brac{C_x + C_y, C_y + C_z, C_z + C_x}$$ then either
		$
			(1)\; x \leqslant C_x \text{ and } y\leqslant C_y,\;
			(2)\; y \leqslant C_y \text{ and } z\leqslant C_z, \text{ or } 
			(3)\; z \leqslant C_z \text{ and } x\leqslant C_x.
		$
	\end{lemma}
	\begin{proof}
		This can be seen to be true by contraposition.
		The key observation is that
		\begin{equation*}
			\text{(1) or (2) or (3) } \iff
			(x \leqslant C_x \text{ or } y \leqslant C_y) \text{ and } 	
			(y \leqslant C_y \text{ or } z \leqslant C_z) \text{ and } 	
			(z \leqslant C_z \text{ or } x \leqslant C_x).
		\end{equation*}
		We may then use this equivalence to rewrite the negation of the conclusion of the Lemma and deduce that the contrapositive holds.
	\end{proof}

%%%% The wrapper around the 'Conflict of interest' section is there so that the references do not appear in the table of contents
%%%\begingroup
%%%\renewcommand{\addcontentsline}[3]{} % Remove functionality of \addcontentsline
%%%\section*{Conflict of interest}The authors declare that they have no conflict of interest.
%%%\endgroup

%----------------------------------------------------------------------------------------------------
%	BILBLIOGRAPHY
%----------------------------------------------------------------------------------------------------

% The wrapper around the bibliography is there so that the references do not appear in the table of contents
\begingroup
\renewcommand{\addcontentsline}[3]{} % Remove functionality of \addcontentsline
	% Bibliography
	\bibliographystyle{alpha-bis}
	\bibliography{main}
\endgroup
\end{document}